\documentclass[english]{article}
\usepackage[T1]{fontenc}
\usepackage[latin9]{inputenc}
\usepackage{geometry}
\usepackage{babel}
\usepackage{verbatim}
\usepackage{float}
\usepackage{mathtools}
\usepackage{bm}
\usepackage{amsmath}
\usepackage{amssymb}
\usepackage{graphicx}
\usepackage[unicode=true,
 bookmarks=false,
 breaklinks=false,pdfborder={0 0 1},colorlinks=false]
 {hyperref}
\hypersetup{
 colorlinks,linkcolor=red,anchorcolor=blue,citecolor=blue}

\makeatletter

\providecommand{\tabularnewline}{\\}
\floatstyle{ruled}
\newfloat{algorithm}{tbp}{loa}
\providecommand{\algorithmname}{Algorithm}
\floatname{algorithm}{\protect\algorithmname}

\usepackage{cite}\usepackage{amsthm}\usepackage{dsfont}\usepackage{array}\usepackage{mathrsfs}\usepackage{comment}\onecolumn

\usepackage{color}\usepackage{babel}

\allowdisplaybreaks

\usepackage{xfrac}
\usepackage{enumitem}
\setlist[itemize]{leftmargin=1.5em}
\setlist[enumerate]{leftmargin=1.5em}

\usepackage{babel}
\usepackage{algorithm}
\usepackage{algorithmic}
\usepackage{arydshln}

\DeclareMathOperator{\ind}{\mathds{1}}  

\definecolor{yxc}{RGB}{255,0,0}
\newcommand{\yxc}[1]{\textcolor{yxc}{[YXC: #1]}}

\definecolor{yc}{RGB}{255,2,255}

\definecolor{cxc}{RGB}{0, 150, 0}

\usepackage[outdir=./]{epstopdf}

\makeatother

\begin{document}
\theoremstyle{plain} \newtheorem{lemma}{\textbf{Lemma}} \newtheorem{prop}{\textbf{Proposition}}\newtheorem{theorem}{\textbf{Theorem}}\setcounter{theorem}{0}
\newtheorem{corollary}{\textbf{Corollary}} \newtheorem{example}{\textbf{Example}}
\newtheorem{definition}{\textbf{Definition}} \newtheorem{fact}{\textbf{Fact}}
\newtheorem{claim}{\textbf{Claim}}\newtheorem{assumption}{\textbf{Assumption}}

\theoremstyle{definition}

\theoremstyle{remark}\newtheorem{remark}{\textbf{Remark}}\newtheorem{conjecture}{Conjecture}\newtheorem{condition}{\textbf{Condition}}
\title{Subspace Estimation from Unbalanced and Incomplete Data Matrices:
$\ell_{2,\infty}$ Statistical Guarantees\footnotetext{Corresponding
author: Yuxin Chen.}}
\author{Changxiao Cai\thanks{Department of Electrical Engineering, Princeton University, Princeton,
NJ 08544, USA; email: \texttt{\{ccai,poor,yuxin.chen\}@princeton.edu}.} \and Gen Li\thanks{Department of Electronic Engineering, Tsinghua University, Beijing,
China 100086, USA; email: \texttt{g-li16@mails.tsinghua.edu.cn}.} \and Yuejie Chi\thanks{Department of Electrical and Computer Engineering, Carnegie Mellon
University, Pittsburgh, PA 15213, USA; email: \texttt{yuejiechi@cmu.edu}.} \and H.~Vincent Poor\footnotemark[1] \and Yuxin Chen\footnotemark[1]}

\date{October 2019; \quad Revised: June 2020}

\maketitle
\begin{abstract}
This paper is concerned with estimating the column space of an unknown
low-rank matrix $\bm{A}^{\star}\in\mathbb{R}^{d_{1}\times d_{2}}$,
given noisy and partial observations of its entries. There is no shortage
of scenarios where the observations --- while being too noisy to
support faithful recovery of the entire matrix --- still convey 
sufficient information to enable reliable estimation of the column
space of interest. This is particularly evident and crucial for the
highly unbalanced case where the column dimension $d_{2}$ far exceeds
the row dimension $d_{1}$, which is the focal point of the current paper.

We investigate an efficient spectral method, which operates upon the
sample Gram matrix with diagonal deletion. While this algorithmic idea has been studied before, we establish new statistical
guarantees for this method in terms of both $\ell_{2}$ and $\ell_{2,\infty}$
estimation accuracy, which improve upon prior results
if $d_{2}$ is substantially larger than $d_{1}$. To illustrate the
effectiveness of our findings, we derive matching minimax lower bounds with respect to the noise levels, and develop consequences of our general
theory for three applications of practical importance: (1) tensor
completion from noisy data, (2) covariance estimation\,/\,principal component analysis
with missing data, and (3) community recovery in bipartite graphs.
Our theory leads to improved performance guarantees for all three cases.  
\end{abstract}

\noindent\textbf{Keywords:} spectral method, principal component analysis with missing data, tensor completion, covariance estimation, spectral clustering, leave-one-out analysis

\tableofcontents{}

\section{Introduction\label{sec:Introduction}}

Consider the problem of estimating the \emph{column space }of a low-rank
matrix $\bm{A}^{\star}=[A_{i,j}^{\star}]_{1\leq i\leq d_{1},1\leq j\leq d_{2}}$,
based on noisy and highly incomplete observations of its entries.
To set the stage, suppose that we observe
\begin{equation}
A_{i,j}=A_{i,j}^{\star}+N_{i,j},\qquad\forall\left(i,j\right)\in\Omega,\label{eq:A-model}
\end{equation}
where $\Omega\subseteq\left\{ 1,\cdots,d_{1}\right\} \times\left\{ 1,\cdots,d_{2}\right\} $
is the sampling set, and $A_{i,j}$ denotes the observed entry at
location $\left(i,j\right)$, which is corrupted by noise $N_{i,j}$.
In contrast to the classical matrix completion problem that aims to
fill in all missing entries \cite{candes2009exact,KesMonSew2010,chi2018nonconvex,chen2018harnessing},
the current paper focuses solely on estimating the column space of
$\bm{A}^{\star}$, which is oftentimes a less stringent requirement.

\paragraph{Motivating applications.} A problem of this kind arises
in numerous applications. We immediately point out several representative
examples as follows, with precise descriptions postponed to Section~\ref{subsec:examples}.
\begin{itemize}
\item \emph{Tensor completion and estimation.} 
Imagine we seek to estimate a low-rank symmetric tensor from partial observations of its entries
\cite{barak2016noisy,xia2017statistically,montanari2018spectral},
a task that spans various applications like visual data inpainting \cite{liu2013tensor}
and medical imaging \cite{semerci2014tensor}. Consider, for example,
an order-$3$ tensor $\bm{T}^{\star}=\sum_{s=1}^{r}\bm{w}_{s}^{\star}\otimes\bm{w}_{s}^{\star}\otimes\bm{w}_{s}^{\star}\in\mathbb{R}^{d\times d\times d}$,   
where $\left\{ \bm{w}_{s}^{\star}\right\}$
represents a collection of tensor factors.\footnote{For any vectors $\bm{a},\bm{b},\bm{c}\in\mathbb{R}^{d}$,
we use $\bm{a}\otimes\bm{b}\otimes\bm{c}$ to denote a $d\times d\times d$
array whose $\left(i,j,k\right)$-th entry is given by $a_{i}b_{j}c_{k}$.
} An alternative representation of $\bm{T}^{\star}$ can be obtained
by unfolding the tensor of interest into a low-rank structured matrix  $\bm{A}^{\star}\in \mathbb{R}^{d\times d^2}$. 
Consequently, estimation of the subspace spanned by $\left\{ \bm{w}_{s}^{\star}\right\} $
from partial noisy entries of $\bm{T}^{\star}$ --- which serves
as a common and crucial step for tensor completion \cite{montanari2018spectral,cai2019nonconvex}
 ---  is equivalent to estimating the column space of $\bm{A}^{\star}$
from incomplete and noisy data; see Section~\ref{subsec:tensor-completion} for details. Notably, the unfolded matrix 
becomes extremely fat as the dimension $d$ grows.

\item \emph{Covariance estimation\,/\,principal component analysis with missing data.}
Suppose we have available a sequence of $n$ independent sample vectors
$\left\{ \bm{x}_{i}\in\mathbb{R}^{d}\right\} _{i=1}^{n}$, whose covariance matrix exhibits certain low-dimensional structure. 
Several  statistical models fall in the same
vein of this model, e.g.~the generalized spiked
model \cite{bai2012sample} and the factor model \cite{fan2018robust}. An important task amounts to estimating
the principal subspace of the covariance matrix of interest, possibly in the presence of missing
data (where we only get to see highly incomplete entries of $\{\bm{x}_i\}$).
If a substantial amount of data are missing, then individual sample vectors cannot possibly be recovered, thus enabling privacy protection for individual data. Fortunately, one might still hope to estimate the principal subspace, provided that a large number of sample vectors are queried (which might yield a fat data matrix $\bm{X}:=[\bm{x}_1,\cdots,\bm{x}_n]$). 
See Section~\ref{subsec:Covariance-estimation}.

\item \emph{Community recovery in bipartite graphs. }Community recovery
is often concerned with clustering a collection of individuals or
nodes into different communities, based on similarities
between pairs of nodes. In many complex networks, such pairwise interactions
might only occur when the two nodes involved belong to two disjoint
groups (denoted by $\mathcal{U}$ and $\mathcal{V}$ respectively). 
This calls for community recovery in bipartite networks (sometimes referred to as biclustering) \cite{dhillon2001co,larremore2014efficiently,alzahrani2014community,zhou2018optimal}.
As we shall detail in Section~\ref{subsec:BSBM}, the biclustering problem is tightly connected to subspace estimation; for instance, the column subspace of some bi-adjacency matrix $\bm{A}\in\mathbb{R}^{|\mathcal{U}|\times|\mathcal{V}|}$  
 (which is a noisy copy of a low-rank matrix) reveals the  community memberships  in $\mathcal{U}$. When the size of $\mathcal{V}$ is substantially larger than that of $\mathcal{U}$,
one might encounter a situation where only the nodes in $\mathcal{U}$
(rather than those in $\mathcal{V}$) can be reliably clustered. This
calls for development of ``one-sided'' community recovery algorithms,
that is, the type of algorithms that guarantee reliable clustering
of $\mathcal{U}$ without worrying about the clustering accuracy in
$\mathcal{V}$. 
\end{itemize}

\paragraph{Contributions.} Since we concentrate primarily on estimating the column space of $\bm{A}^{\star}$,
it is natural to expect a reduced sample complexity as well as a weaker
requirement on the signal-to-noise ratio, in comparison to the conditions required
for reliable reconstruction of the whole matrix --- particularly
for those highly unbalanced problems with drastically different dimensions $d_{1}$
and $d_{2}$.  Focusing on a spectral method applied to the Gram matrix $\bm{A}\bm{A}^{\top}$ with diagonal deletion (whose variants have been studied in multiple contexts \cite{lounici2014high,florescu2016spectral,loh2012high,montanari2018spectral,elsener2019sparse}),
we establish new statistical guarantees in terms of the sample complexity
and the estimation accuracy, both of which strengthen prior theory. Our results
deliver optimal $\ell_{2,\infty}$ estimation risk bounds with respect to the noise level, which are previously unavailable. 
All this is accomplished via a powerful leave-one-out\,/\,leave-two-out analysis framework.   Further, we develop minimax lower bounds under Gaussian noise, revealing that the sample complexity and the signal-to-noise ratio (SNR) required for spectral methods to achieve consistent estimation are both minimax optimal (up to some logarithmic factor).   
Finally, we develop concrete
consequences of our general theory for all three applications mentioned
above, leading to optimal performance guarantees that improve upon prior
literature.

It is worth noting that low-rank subspace estimation from noisy and incomplete data  has been extensively studied in a large number of prior work (e.g.~\cite{lounici2014high,abbe2017entrywise, cho2017asymptotic, zhang2018heteroskedastic,zhu2019high}). While many of these prior results allow $d_1$ and $d_2$ to differ, they typically fall short of establishing optimal dependency on $d_1$ and $d_2$ in the highly unbalanced scenarios. The focal point of this paper is thus to characterize the effect of such unbalancedness (as reflected by the aspect ratio $d_2/d_1$) upon consistent subspace estimation.

\paragraph{Paper organization.} The rest of this paper is organized as follows. Section~\ref{sec:Problem-formulation}
formulates the problem and introduces basic definitions and notations.
In Section~\ref{sec:Main-results}, we present our theoretical guarantees for a spectral method, as well as minimax lower bounds. Section~\ref{subsec:examples} applies our general theorem to the aforementioned applications, and corroborate our theory by numerical experiments. Section~\ref{sec:Prior-art}
provides an overview of related prior works. The proof of our main theory is outlined
in Section~\ref{sec:Analysis}, with the proofs of many auxiliary
lemmas postponed to the appendix. We conclude the paper with a discussion
of future directions in Section~\ref{sec:discussion}.

\section{Problem formulation\label{sec:Problem-formulation}}

\subsection{Models}

\label{subsec:Models}

\paragraph{Low-rank matrix.} Suppose that the unknown matrix $\bm{A}^{\star}\in\mathbb{R}^{d_{1}\times d_{2}}$
is rank-$r$, where the row dimension $d_{1}$ and the column dimension
$d_{2}$ are allowed to be drastically different. Assume that the
(compact) singular value decomposition (SVD) of $\bm{A}^{\star}$
is given by
\begin{align}
	\bm{A}^{\star} & =\bm{U}^{\star}\bm{\Sigma}^{\star}\bm{V}^{\star\top}=\sum_{i=1}^{r}\sigma_{i}^{\star}\bm{u}_{i}^{\star}\bm{v}_{i}^{\star\top}.
	\label{def:A_star}
\end{align}
Here, $\sigma_{1}^{\star}\geq\sigma_{2}^{\star}\geq\cdots\geq\sigma_{r}^{\star}>0$
represent the $r$ nonzero singular values of $\bm{A}^{\star}$, and
$\bm{\Sigma}^{\star}\in\mathbb{R}^{r\times r}$ is a diagonal matrix
whose diagonal entries are given by $\{\sigma_{1}^{\star},\cdots,\sigma_{r}^{\star}\}$.
The columns of $\bm{U}^{\star}=\left[\bm{u}_{1}^{\star},\cdots,\bm{u}_{r}^{\star}\right]\in\mathbb{R}^{d_{1}\times r}$
(resp.~$\bm{V}^{\star}=\left[\bm{v}_{1}^{\star},\cdots,\bm{v}_{r}^{\star}\right]\in\mathbb{R}^{d_{2}\times r}$)
are orthonormal, which are the top-$r$ left (resp.~right) singular vectors
of $\bm{A}^{\star}$. We define and denote the condition number of
$\bm{A}^{\star}$ as follows
\begin{equation}
	\kappa:=\sigma_{1}^{\star}\,/\,\sigma_{r}^{\star}, 
	\label{eq:defn-kappa-d}
\end{equation}
and take
\begin{equation}
	d:=\max\left\{ d_{1},d_{2}\right\} .
	\label{eq:defn-d}
\end{equation}

\paragraph{Incoherence.} Further, we impose certain incoherence conditions
on the unknown matrix $\bm{A}^{\star}$, which are commonly adopted
in the matrix completion literature (e.g.~\cite{candes2009exact,KesMonSew2010,chi2018nonconvex}).

\begin{definition}[Incoherence parameters]
\label{definition-mu0-mu1-mu2}
Define the incoherence parameters $\mu_{0}$, $\mu_{1}$ and $\mu_{2}$
as follows
\begin{align}
\mu_{0}:=\frac{d_{1}d_{2}\underset{1\leq i\leq d_{1},1\leq j\leq d_{2}}{\max}\left|A_{i,j}^{\star}\right|^{2}}{\left\Vert \bm{A}^{\star}\right\Vert _{\mathrm{F}}^{2}} & ,\label{eq:def-col-mu0}\\
\mu_{1}:=\frac{d_{1}}{r}\max_{1\leq i\leq d_{1}}\left\Vert \bm{U}^{\star\top}\bm{e}_{i}\right\Vert _{2}^{2}\qquad & \text{and}\qquad\mu_{2}:=\frac{d_{2}}{r}\max_{1\leq i\leq d_{2}}\left\Vert \bm{V}^{\star\top}\bm{e}_{i}\right\Vert _{2}^{2},\label{def:coh}
\end{align}
where $\bm{e}_{i}$ is the $i$-th standard basis vector of compatible
dimensionality. 
\end{definition}

Intuitively, when $\mu_{0}$, $\mu_{1}$ and $\mu_{2}$ are all small,
the energies of the matrices $\bm{A}^{\star}$, $\bm{U}^{\star}$
and $\bm{V}^{\star}$ are (nearly) evenly spread out across all entries,
rows, and columns. For notational simplicity, we shall set 
\begin{equation}
\mu:=\max\left\{ \mu_{0},\mu_{1},\mu_{2}\right\} .\label{eq:defn-mu}
\end{equation}

\paragraph{Random sampling and random noise.} Suppose that we have only
collected noisy observations of the entries of $\bm{A}^{\star}$ over a
sampling set $\Omega\subseteq\{1,\cdots,d_{1}\}\times\{1,\cdots,d_{2}\}$.
Specifically, we observe
\begin{equation}
A_{i,j}=\begin{cases}
A_{i,j}^{\star}+N_{i,j},\qquad & (i,j)\in\Omega,\\
0, & \text{else},
\end{cases}\label{eq:A-observation}
\end{equation}
where $N_{i,j}$ denotes the noise at location $(i,j)$. For notational
simplicity, we shall write
\begin{align}
\bm{A}=\mathcal{P}_{\Omega}(\bm{A}) & =\mathcal{P}_{\Omega}(\bm{A}^{\star})+\mathcal{P}_{\Omega}(\bm{N}),\label{def:A}
\end{align}
where $\mathcal{P}_{\Omega}$ represents the Euclidean projection
onto the subspace of matrices supported on $\Omega$. In addition,
this paper concentrates on random sampling and random noise as follows.

\begin{assumption}[Random sampling]\label{assumption:random-sampling}Each
$(i,j)$ is included in the sampling set $\Omega$ independently with probability $p$.
\end{assumption}

\begin{assumption}[Random noise]\label{assumption:random-noise}The
noise $N_{i,j}$'s are independent random variables and satisfy the
following conditions: for each $1\leq i\leq d_{1},1\leq j\leq d_{2}$,

\begin{itemize}

\item[(1)] \textbf{ (Zero mean) $\mathbb{E}\left[N_{i,j}\right]=0$};

\item[(2)] \textbf{ (Variance)} $\mathsf{Var}\left(N_{i,j}\right)\leq\sigma^{2}$;

\item[(3)] \textbf{ (Magnitude) }Each $N_{i,j}$ satisfies either
of the following condition:

\begin{itemize} 

\item[(a)]  $\left|N_{i,j}\right|\leq R$;

\item[(b)]  $N_{i,j}$ has a symmetric distribution satisfying $\mathbb{P}\left\{ \left|N_{i,j}\right|>R\right\} \leq c_{\mathrm{r}}d^{-12}$
for some universal constant $c_{\mathrm{r}}>0$.

\end{itemize}

Here, $R$ is some quantity obeying\textbf{
\begin{equation}
\frac{R^{2}}{\sigma^{2}}\leq C_{\mathrm{r}}\frac{\min\left\{ p\sqrt{d_{1}d_{2}},\,pd_{2}\right\} }{\log d}\label{eq:assumption-noise-spike}
\end{equation}
}for some universal constant $C_{\mathrm{r}}>0$.

\end{itemize}

\end{assumption}
\noindent As a remark, Assumption \ref{assumption:random-noise}
allows the largest possible size $R$ of each noise component to be
substantially larger than its typical size $\sigma$. For example,
if $p\asymp1$, then $R$ can be $\min{\{(d_{1}d_{2})^{1/4},\sqrt{d_{2}}\}}$
times larger than $\sigma$ (ignoring any logarithmic factor). In
addition, the noise $N_{i,j}$'s do not necessarily have identical
variance; in fact, our formulation allows us to accommodate the heteroscedasticity
of noise (i.e.~the scenario where the noise has location-varying
variance). 

\paragraph{Goal.} Given incomplete and noisy observations about
$\bm{A}^{\star}\in\mathbb{R}^{d_{1}\times d_{2}}$ (cf.~(\ref{eq:A-observation})),
we seek to estimate $\bm{U}^{\star}\in\mathbb{R}^{d_{1}\times r}$
modulo some global rotation. We emphasize once again that the aim
here is not to estimate the entire matrix. In truth, there are many
unbalanced cases with $d_{2}\gg d_{1}$ such that (1) reliable estimation
of $\bm{U}^{\star}$ is feasible, but (2) faithful estimation of the
whole matrix $\bm{A}^{\star}$ is information theoretically impossible.

\subsection{Notation}

\label{subsec:Notations}

We denote $\left[n\right]:=\{1,\cdots,n\}$. For any matrix $\bm{A}\in\mathbb{R}^{d_{1}\times d_{2}}$,
we use $\sigma_{i}\left(\bm{A}\right)$ and $\lambda_{i}\left(\bm{A}\right)$
to represent the $i$-th largest singular value and the $i$-th largest
eigenvalue of $\bm{A}$, respectively. Let $\bm{A}_{i,:}$ and $\bm{A}_{:,j}$
denote respectively the $i$-th row and the $j$-th column of $\bm{A}$.
Let $\left\Vert \bm{A}\right\Vert $ (resp.~$\left\Vert \bm{A}\right\Vert _{\mathrm{F}}$)
represent the spectral norm (resp.~the Frobenius norm) of $\bm{A}$.
We also denote by $\left\Vert \bm{A}\right\Vert _{2,\infty}:=\max_{i\in[d_{1}]}\|\bm{A}_{i,:}\|_{2}$
and $\left\Vert \bm{A}\right\Vert _{\infty}:=\max_{i\in[d_{1}],j\in[d_{2}]}\big|A_{i,j}\big|$
the $\ell_{2,\infty}$ norm and the entrywise $\ell_{\infty}$ norm
of $\bm{A}$, respectively. Similarly, for any tensor $\bm{T}$, we use $\|\bm{T}\|_{\infty}$ to represent the largest magnitude of the entries of $\bm{T}$.  
Moreover, we denote by $\mathcal{P}_{\mathsf{diag}}$
the projection onto the subspace that vanish outside the diagonal,
and define $\mathcal{P}_{\mathsf{off}\text{-}\mathsf{diag}}$ such
that $\mathcal{P}_{\mathsf{off}\text{-}\mathsf{diag}}(\bm{A}):=\bm{A}-\mathcal{P}_{\mathsf{diag}}(\bm{A})$.
Let $\mathcal{O}^{r\times r}$ stand for the set of $r\times r$ orthonormal
matrices. In addition, we use $\mathsf{diag}\left(\bm{a}\right)$
to represent a diagonal matrix whose $(i,i)$-th entry is equal to
$a_{i}$. Throughout this paper, the notations $C,C_{1},\cdots,c,c_{1},\cdots$
denote absolute positive constants whose values may change from line
to line. 

Furthermore, for any real-valued functions $f(d_{1},d_{2})$ and $g(d_{1},d_{2})$,
$f(d_{1},d_{2})\lesssim g(d_{1},d_{2})$ or $f(d_{1},d_{2})=O\left(g(d_{1},d_{2})\right)$
mean that $\left|f(d_{1},d_{2})/g(d_{1},d_{2})\right|\leq C_{1}$
for some constant $C_{1}>0$; $f(d_{1},d_{2})\gtrsim g(d_{1},d_{2})$
means that $\left|f(d_{1},d_{2})/g(d_{1},d_{2})\right|\geq C_{2}$
for some universal constant $C_{2}>0$; $f(d_{1},d_{2})\asymp g(d_{1},d_{2})$
means that $C_{1}\leq\left|f(d_{1},d_{2})/g(d_{1},d_{2})\right|\leq C_{2}$
for some universal constants $C_{1},C_{2}>0$; $f(d_{1},d_{2})=o\left(g(d_{1},d_{2})\right)$ means that $f(d_{1},d_{2})/g(d_{1},d_{2})\rightarrow0$ as $\min\left\{ d_{1},d_{2}\right\} \rightarrow\infty$.
In addition, $f(d_{1},d_{2})\ll g(d_{1},d_{2})$ (resp.~$f(d_{1},d_{2})\gg g(d_{1},d_{2})$)
means that there exists some sufficiently small (resp.~large) constant
$c_{1}>0$ (resp.~$c_{2}>0$) such that $f(d_{1},d_{2})\leq c_{1}g(d_{1},d_{2})$
(resp.~$f(d_{1},d_{2})\geq c_{2}g(d_{1},d_{2})$) holds true for
all sufficiently large $d_{1}$ and $d_{2}$.

\section{Main results\label{sec:Main-results}}

\subsection{Algorithm: a spectral method with diagonal deletion\label{subsec:Algorithm}}

Recall that $\bm{A}=[A_{i,j}]_{1\leq i\leq d_{1},1\leq j\leq d_{2}}$
is the zero-padded data matrix (see (\ref{eq:A-observation})). It
is easily seen that, under our random sampling model (i.e.~Assumption
\ref{assumption:random-sampling}), $p^{-1}\bm{A}$ serves as an unbiased
estimator of $\bm{A}^{\star}$. One might thus expect the left singular
subspace of $\bm{A}$ to form a reasonably good estimator of the subspace
spanned by $\bm{U}^{\star}$. As it turns out, when $\bm{A}^{\star}$
is a very fat matrix (namely, $d_{2}\gg d_{1}$), this approach might
fail to work when the sample complexity is not sufficiently large
or when the noise size is not sufficiently small.

This paper adopts an alternative route by resorting to the sample Gram matrix
$\bm{A}\bm{A}^{\top}$ (properly rescaled). Straightforward calculation
reveals that
\begin{equation}
\frac{1}{p^{2}}\mathbb{E}\left[\bm{A}\bm{A}^{\top}\right]=\bm{A}^{\star}\bm{A}^{\star\top}+\underset{\text{a diagonal matrix}}{\underbrace{\left(\frac{1}{p}-1\right)\mathcal{P}_{\mathsf{diag}}\left(\bm{A}^{\star}\bm{A}^{\star\top}\right)+\frac{1}{p}\mathsf{\mathsf{diag}}\Bigg(\Bigg[\sum_{j=1}^{d_{2}}\mathsf{Var}(N_{i,j})\Bigg]_{1\leq i\leq d_{1}}\Bigg)}},\label{eq:E-AA-calculation}
\end{equation}
where $\mathsf{diag}\left(\bm{a}\right)$ with $\bm{a}\in\mathbb{R}^{d_{1}}$
represents a diagonal matrix whose $(i,i)$-th entry equals $a_{i}$.
The identity (\ref{eq:E-AA-calculation}) implies that the diagonal
components of $p^{-2}\mathbb{E}\left[\bm{A}\bm{A}^{\top}\right]$
are significantly inflated, which might need to be properly suppressed.

In order to remedy the above-mentioned diagonal inflation issue, we
adopt a simple strategy that zeros out all diagonal entries; that
is, performing the spectral method on the following matrix 
\begin{equation}
\bm{G}=\frac{1}{p^{2}}\mathcal{P}_{\mathsf{off}\text{-}\mathsf{diag}}\big(\bm{A}\bm{A}^{\top}\big)\label{eq:defn-G}
\end{equation}
with $\mathcal{P}_{\mathsf{off}\text{-}\mathsf{diag}}(\bm{M}):=\bm{M}-\mathcal{P}_{\mathsf{diag}}(\bm{M})$
denoting projection onto the set of zero-diagonal matrices. This clearly
satisfies
\[
\mathbb{E}\left[\bm{G}\right]=\mathcal{P}_{\mathsf{off}\text{-}\mathsf{diag}}\left(\bm{A}^{\star}\bm{A}^{\star\top}\right)=\mathcal{P}_{\mathsf{off}\text{-}\mathsf{diag}}\left(\bm{U}^{\star}\bm{\Sigma}^{\star2}\bm{U}^{\star\top}\right).
\]
If the diagonal entries of $\bm{A}^{\star}\bm{A}^{\star\top}$ are
not too large, then one has $\bm{A}^{\star}\bm{A}^{\star\top}\approx\mathcal{P}_{\mathsf{off}\text{-}\mathsf{diag}}\left(\bm{A}^{\star}\bm{A}^{\star\top}\right)$
and, as a result, the rank-$r$ eigen-subspace of $\bm{G}$ might
form a reliable estimate of the subspace spanned by $\bm{U}^{\star}$. 
The procedure is summarized in Algorithm~\ref{alg:spectral}.

\begin{algorithm}
\caption{The spectral method on the diagonal-deleted Gram matrix}
\label{alg:spectral}
\begin{algorithmic}[1]
\STATE{{\bf Input:} sampling  set $\Omega$, observed entries $\left\{ A_{i, j} \mid \left( i, j \right) \in \Omega \right\}$,  sampling rate $p$, rank $r$.}
\STATE{{\bf Compute} the (truncated) rank-$r$ eigen-decomposition  $\bm{U}\bm{\Lambda}\bm{U}^\top$ of $\bm{G}$, where $\bm{U}\in \mathbb{R}^{d_1\times r}$, $\bm{\Lambda}\in \mathbb {R}^{r\times r}$, and
\begin{align}
\label{def:G}
\bm{G} := \mathcal{P}_{\mathsf{off}\text{-}\mathsf{diag}} \left( \tfrac{1}{p^2} \bm{A}\bm{A}^\top \right).
\end{align}
Here, $\bm{A}$ is defined in \eqref{def:A} and $\mathcal{P}_{\mathsf{off}\text{-}\mathsf{diag}} \left(\bm{M}\right)$ zeros out the diagonal entries of a matrix $\bm{M}$.}
\STATE{{\bf Output} $\bm{U}$ as the subspace estimate, and $\bm{\Sigma} = \bm{\Lambda}^{1/2}$ as the spectrum estimate.}

\end{algorithmic}
\end{algorithm}

We remark that this is clearly not a new algorithmic idea. In fact, proper handling of the diagonal entries (e.g.~diagonal deletion, diagonal reweighting) has already been recommended in several different applications, including bipartite stochastic block models \cite{florescu2016spectral}, covariance estimation \cite{lounici2013sparse,lounici2014high,loh2012high,elsener2019sparse}, tensor completion \cite{montanari2018spectral}, to name just a few. 

\subsection{Theoretical guarantees\label{subsec:Theoretical-guarantees}}

In general, one can only hope to estimate $\bm{U}^{\star}$ up to
global rotation. With this in mind, we introduce the following rotation
matrix
\begin{equation}
\bm{R}\,:=\,\underset{\bm{Q}\in\mathcal{O}^{r\times r}}{\arg\min}\text{ }\left\Vert \bm{U}\bm{Q}-\bm{U}^{\star}\right\Vert _{\mathrm{F}},\label{def:R}
\end{equation}
where $\mathcal{O}^{r\times r}$ stands for the set of $r\times r$
orthonormal matrices. In words, $\bm{R}$ is the global rotation matrix
that best aligns $\bm{U}$ and $\bm{U}^{\star}$. Equipped with this
notation, the following theorem delivers upper bounds on the difference
between the obtained estimate $\bm{U}$ and the ground truth $\bm{U}^{\star}$.
The proof is postponed to Section \ref{sec:Analysis}.

\begin{theorem}\label{thm:U_loss}Assume that the following conditions
hold
\begin{equation}
p\geq c_{0}\max\left\{ \frac{\mu\kappa^{4}r\log^{2}d}{\sqrt{d_{1}d_{2}}},\frac{\mu\kappa^{8}r\log^{2}d}{d_{2}}\right\} ,\,\,\frac{\sigma}{\sigma_{r}^{\star}}\leq c_{1}\min\left\{ \frac{\sqrt{p}}{\kappa\sqrt[4]{d_{1}d_{2}}\,\sqrt{\log d}},\frac{1}{\kappa^{3}}\sqrt{\frac{p}{d_{1}\log d}}\right\} \,\,\text{and}\,\,r\leq\frac{c_{2}d_{1}}{\mu_{1}\kappa^{4}},\label{asmp}
\end{equation}
where $c_{0}>0$ is some sufficiently large constant and $c_{1},c_{2}>0$
are some sufficiently small constants. Then with probability at least
$1-O\left(d^{-10}\right)$, the matrices $\bm{U}$ and $\bm{\Sigma}$
returned by Algorithm \ref{alg:spectral} satisfy
\begin{subequations}
\begin{align}
\left\Vert \bm{U}\bm{R}-\bm{U}^{\star}\right\Vert  & \lesssim \mathcal{E}_{\mathsf{general}},\label{claim:U_op_loss}\\
\left\Vert \bm{U}\bm{R}-\bm{U}^{\star}\right\Vert _{2,\infty} & \lesssim\kappa^{2}\sqrt{\frac{\mu r}{d_{1}}}\cdot \mathcal{E}_{\mathsf{general}},\label{claim:U_2inf_loss}\\
\left\Vert \bm{\Sigma}-\bm{\Sigma}^{\star}\right\Vert  & \lesssim\sigma_{r}^{\star}\cdot\mathcal{E}_{\mathsf{general}},\label{claim:Lambda_loss}
\end{align}
\end{subequations}
where $\bm{R}$ is defined in (\ref{def:R}), and
\begin{equation}
 \mathcal{E}_{\mathsf{general}}:=\underset{\mathrm{missing}\text{ }\mathrm{data}\text{ }\mathrm{effect}}{\underbrace{\frac{\mu\kappa^{2}r\log d}{\sqrt{d_{1}d_{2}}\,p}+\sqrt{\frac{\mu\kappa^{4}r\log d}{d_{2}p}}}}+\underset{\mathrm{noise}\text{ }\mathrm{effect}}{\underbrace{\frac{\sigma^{2}}{\sigma_{r}^{\star2}}\frac{\sqrt{d_{1}d_{2}}\,\log d}{p}+\frac{\sigma\kappa}{\sigma_{r}^{\star}}\sqrt{\frac{d_{1}\log d}{p}}}}+ \hspace{-1em} \underset{\mathrm{diagonal}\text{ }\mathrm{deletion}\text{ }\mathrm{effect}}{\underbrace{\frac{\mu_{1}\kappa^{2}r}{d_{1}}}}.
\label{eq:defn-Err-UB}
\end{equation}
\end{theorem}

\begin{remark} If there is no missing data, namely, $p=1$, then Theorem~\ref{thm:U_loss} continues to hold if the first two terms on the right-hand side of \eqref{eq:defn-Err-UB} are removed.  
\end{remark}

In a nutshell, Theorem~\ref{thm:U_loss} asserts that Algorithm~\ref{alg:spectral}
produces reliable estimates of the column subspace of $\bm{A}^{\star}$
--- with respect to both the spectral norm and the $\left\Vert \cdot\right\Vert _{2,\infty}$
norm --- under certain conditions imposed on the sample size and the noise
size. For instance, consider the settings where $\mu,\kappa \asymp 1$ and $r\ll d_1\leq d_2$. Then as long as the following condition holds:
\begin{align}
	\label{eq:consistent-estimation-spectral}
	p \gtrsim  \frac{r\log^{2}d}{\sqrt{d_{1}d_{2}}}  \qquad \text{and}\qquad \frac{\sigma^2}{\sigma_{r}^{\star 2}} =o\left(  \frac{p}{\sqrt{d_{1}d_{2}}\,{\log d}} \right) ,
\end{align}
the proposed spectral method achieves consistent estimation with high probability, namely,
\begin{align}
	\underset{\bm{Q}\in\mathcal{O}^{r\times r}}{\min} \frac{ \left\Vert \bm{U}\bm{Q}-\bm{U}^{\star}\right\Vert }{ \|\bm{U}^{\star}\| } = o \left( 1 \right),\quad
	\underset{\bm{Q}\in\mathcal{O}^{r\times r}}{\min} \frac{ \left\Vert \bm{U}\bm{Q}-\bm{U}^{\star}\right\Vert _{2,\infty} }{ \|\bm{U}^{\star}\|_{2,\infty} } = o \left( 1 \right)
	\quad
	\text{and}\quad
	 \frac{ \left\Vert \bm{\Sigma}-\bm{\Sigma}^{\star}\right\Vert }{ \| \bm{\Sigma}^{\star} \| } = o \left( 1 \right). 
	 \label{eq:consistent-estimation-defn}
\end{align}

Our upper bound (\ref{eq:defn-Err-UB}) on the spectral norm
error contains five terms. The first two terms of (\ref{eq:defn-Err-UB})
are incurred by missing data; the third and the fourth terms of (\ref{eq:defn-Err-UB})
represent the influence of observation noise; and the last term of
(\ref{eq:defn-Err-UB}) arises due to the bias caused by diagonal deletion.
In particular, the last term is expected to be vanishingly small in
the low-rank and incoherent case. 
Interestingly, both the missing data effect and the noise effect are
captured by two different terms, which we shall interpret in what follows. Note that a primary focus of this paper is to demonstrate the feasibility of obtaining a tight control of the $\ell_{2,\infty}$ statistical error. 
This is particularly evident for the low-rank,
incoherent, and well-conditioned case with $r,\mu,\kappa=O(1)$, in which our theory (cf.~\eqref{claim:U_op_loss} and \eqref{claim:U_2inf_loss}) reveals that the $\ell_{2,\infty}$ error can be a factor of $\sqrt{1/d_1}$ smaller than the spectral norm error.
The discussion below focuses on this case (namely, $r,\mu,\kappa=O(1)$), 
with all logarithmic factors omitted for simplicity of presentation. 
\begin{itemize}
\item Let us first examine the influence of observation noise, which reads
\begin{align}
\frac{\sigma^{2}}{\sigma_{r}^{\star2}}\frac{\sqrt{d_{1}d_{2}}}{p}+\frac{\sigma}{\sigma_{r}^{\star}}\sqrt{\frac{d_{1}}{p}}.
\label{eq:noise-effect-simple}
\end{align}
This contains a quadratic term as well as a linear term w.r.t.~$\sigma/\sigma_{r}^{\star}$. 
To interpret this, consider, for example, the case without missing data (i.e.~$p=1$) and decompose 
\[
\bm{A}\bm{A}^{\top}=\bm{A}^{\star}\bm{A}^{\star\top}+\underset{\text{linear perturbation}}{\underbrace{\bm{A}^{\star}\bm{N}^{\top}+\bm{N}\bm{A}^{\star\top}}}+\underset{\text{quadratic perturbation}}{\underbrace{\bm{N}\bm{N}^{\top}}},
\]
which clearly explains why eigenspace perturbation bounds depend both linearly
and quadratically on the noise magnitudes. In general, the quadratic
term $\frac{\sigma^{2}}{\sigma_{r}^{\star2}}\frac{\sqrt{d_{1}d_{2}}}{p}$
is dominant when the signal-to-noise ratio (SNR) is not large enough; as
the noise decreases to a sufficiently low level, the linear term starts
to enter the picture. See Table \ref{tab:dominant-noise} for a more precise summary. As we shall demonstrate momentarily, 
the terms \eqref{eq:noise-effect-simple} match the minimax limits (up to some logarithmic factor),  meaning that it is generally impossible to get rid of either the linear term or the quadratic term.

\begin{table}
\centering %
\begin{tabular}{c|c|c}
\hline 
 & large-noise regime (i.e.~$\sigma/\sigma_{r}^{\star}\gtrsim\sqrt{p/d_{2}}$) & small-noise regime (i.e.~$\sigma/\sigma_{r}^{\star}\lesssim\sqrt{p/d_{2}}$)\tabularnewline
\hline 
dominant term & $\frac{\sigma^{2}}{\sigma_{r}^{\star2}}\frac{\sqrt{d_{1}d_{2}}}{p}$ & $\frac{\sigma}{\sigma_{r}^{\star}}\sqrt{\frac{d_{1}}{p}}$\tabularnewline
\hline 
\end{tabular}\caption{The dominant term of the noise effect in $\frac{\sigma^{2}}{\sigma_{r}^{\star2}}\frac{\sqrt{d_{1}d_{2}}}{p}+\frac{\sigma}{\sigma_{r}^{\star}}\sqrt{\frac{d_{1}}{p}}$
if $d_{2}\protect\geq d_{1}$ (omitting logarithmic factors and assuming
$r,\kappa,\mu\asymp1$).\label{tab:dominant-noise}}
\end{table}

\item Next, we examine the influence of missing data and assume $\sigma=0$ to simplify the discussion. If we view $\bm{N}_{\mathrm{missing}}=\frac{1}{p}\bm{A}-\bm{A}^{\star}$
as a zero-mean perturbation matrix, then one can write
\[
\tfrac{1}{p^{2}}\bm{A}\bm{A}^{\top}=\bm{A}^{\star}\bm{A}^{\star\top}+\underset{\text{linear perturbation}}{\underbrace{\bm{A}^{\star}\bm{N}_{\mathrm{missing}}^{\top}+\bm{N}_{\mathrm{missing}}\bm{A}^{\star\top}}}+\underset{\text{quadratic perturbation}}{\underbrace{\bm{N}_{\mathrm{missing}}\bm{N}_{\mathrm{missing}}^{\top}}}.
\]
Similar to the above noisy case with $p=1$, this decomposition explains
why the influence of missing data contains two terms as well (see
Table \ref{tab:dominant-missing data})
\[
\underset{\text{quadratic term in }1/\sqrt{p}}{\underbrace{\frac{1}{\sqrt{d_{1}d_{2}}\,p}}}+\underset{\text{linear term in }1/\sqrt{p}}{\underbrace{\frac{1}{\sqrt{d_{2}p}}}}.
\]
\begin{table}
\centering %
\begin{tabular}{c|c|c}
\hline 
 & high-missingness regime (i.e.~$p\lesssim1/d_{1}$) & low-missingness regime (i.e.~$p\gtrsim1/d_{1}$)\tabularnewline
\hline 
dominant term & $\frac{1}{\sqrt{d_{1}d_{2}}\,p}$ & $\frac{1}{\sqrt{d_{2}p}}$\tabularnewline
\hline 
\end{tabular}\caption{The dominant term of the missing data effect in $\frac{1}{\sqrt{d_{1}d_{2}}\,p}+\frac{1}{\sqrt{d_{2}p}}$
if $d_{2}\protect\geq d_{1}$ (omitting logarithmic factors and assuming
$r,\kappa,\mu\asymp1$).\label{tab:dominant-missing data}}
\end{table}
\end{itemize}

\paragraph{Comparison with prior results.} To demonstrate the effectiveness
of our theory, we take a moment to compare them with several prior results.
Once again, the discussion below focuses on the case
with $\max\left\{ \mu,\kappa,r\right\} \asymp1$.  To be fair, it is worth noting that most papers discussed below either have different objectives (e.g.~aiming at matrix estimation rather than subspace estimation \cite{KesMonSew2010,CanTao10,chen2019noisy}), or work with different (and possibly more general) model assumptions (e.g.~square matrices \cite{abbe2017entrywise}, or heteroskedastic noise \cite{zhang2018heteroskedastic}). Our purpose here is not to argue that our results are always stronger than the previous ones, but rather to point out the insufficiency of prior theory when directly applied to some basic settings.

\begin{itemize}

\item To begin with, we compare our spectral norm bound with that required
for matrix completion \cite{KesMonSew2010,CanTao10,abbe2017entrywise,chi2018nonconvex,chen2019noisy}
in the noise-free case (i.e.~$\sigma=0$), in order to show how much saving can be harvested when we move from matrix estimation to subspace estimation.  Suppose that $d_{2}\geq d_{1}$. As is well known, for both spectral and optimization-based
methods, the sample complexities required for faithful matrix completion
need to satisfy $pd_{1}d_{2}\gtrsim d_{2}\mathrm{poly}\log d$. In
comparison, faithful estimation of the column
subspace becomes feasible under the sample size $pd_{1}d_{2}\gtrsim\sqrt{d_{1}d_{2}}\,\mathrm{poly}\log d$, which can be much lower than that required for matrix completion (i.e.~by a factor
of $\sqrt{d_{2}/d_{1}}$). 
Further, we compare our $\|\cdot\|_{2,\infty}$ bound with the
theory derived in \cite{abbe2017entrywise} when $d_{2}\gtrsim  d_{1}\log^2 d$. The theory in \cite[Theorem~3.4]{abbe2017entrywise}
requires the sample size and the noise level to satisfy $p\gtrsim d_{1}^{-1}\log d$
		and $\sigma/\sigma_{r}^{\star}\lesssim\sqrt{\frac{p}{d_{2}\log d}}$, both of which
are more stringent requirements than ours (namely, $p\gtrsim\frac{\log^2 d}{\sqrt{d_{1}d_{2}}}$
and $\sigma/\sigma_{r}^{\star}\lesssim\frac{\sqrt{p}}{\sqrt[4]{d_{1}d_{2}}\,\sqrt{\log d}}$). 
Again, this arises primarily because \cite{abbe2017entrywise} seeks to estimate the whole matrix as opposed to its column subspace.

\item We then compare our results with \cite{montanari2018spectral}, which studies a diagonal-rescaling algorithm for the noise-free case (i.e.~$\sigma = 0$). Combining \cite[Theorem 6.2]{montanari2018spectral} with the standard Davis-Kahan matrix perturbation theory, we can easily see that their spectral norm bound for subspace estimation reads 
\begin{align*}
	\frac{\mathrm{poly}\log d}{\sqrt{d_{1}d_{2}} \, p}+\frac{\mathrm{poly}\log d}{\sqrt{d_{2} p}}. 
\end{align*}
This coincides with our bound except for the last term of \eqref{eq:defn-Err-UB} (due to the bias incurred by diagonal deletion). In comparison, our theory offers additional $\ell_{2,\infty}$ statistical guarantees and covers the noisy case, thus strengthening the theory presented in \cite{montanari2018spectral}.

\item Additionally,  we compare our spectral norm bound with the results derived
in \cite{zhang2018heteroskedastic}. Consider the noiseless case where
$\sigma=0$. It is proven in \cite[Theorem~6]{zhang2018heteroskedastic} (see also the remark that follows)
that: if the sample size satisfies $pd_{1}d_{2}\gtrsim\max\big\{ d_{1}^{1/3}d_{2}^{2/3},d_{1}\big\}\mathrm{poly}\log d$,
then the $\mathsf{HeteroPCA}$ estimator is consistent in estimating
the column subspace (namely, achieving a relative $\ell_{2}$ estimation
error not exceeding $o(1)$). In comparison, our theory claims that Algorithm \ref{alg:spectral} is guaranteed to yield consistent column subspace estimation
as long as the sample size obeys $pd_{1}d_{2}\gtrsim\sqrt{d_{1}d_{2}}\,\mathrm{poly}\log d$. Consequently,
if we omit logarithmic terms, then our sample complexity improves
upon the theoretical support of $\mathsf{HeteroPCA}$ by a factor of $(d_{2}/d_{1})^{1/6}$
if $d_{2}\geq d_{1}$.  Once again, the comparison here focuses on the effect of the aspect ratio $d_{2}/d_{1}$, without accounting for the influence of other parameters like $\mu,\kappa,r$. 

\end{itemize}

\paragraph{SVD applied directly to $\bm{A}$?} Finally, another natural spectral method that comes immediately into mind is to compute the rank-$r$ SVD of $\bm{A}$, and return the matrix containing the $r$ left singular vectors as the column subspace estimate. The $\ell_2$ risk analysis of this approach is typically based on classical matrix perturbation theory like Wedin's theorem \cite{wedin1973perturbation}.  We caution, however, that this approach becomes highly sub-optimal  when the aspect ratio $d_2 / d_1$ grows. Take the case with Gaussian noise and no missing data (i.e.~$p=1$) for example: in order for Wedin's theorem to be applicable, a basic requirement is $\|\bm{N}\|< \sigma_r^{\star}$, which translates to the condition
\begin{align}
	\frac{\sigma}{\sigma_{r}^{\star}}\lesssim \frac{1}{\sqrt{d_{2}}} 
	\label{eq:sigma-UB-svd}
\end{align}
since $\|\bm{N}\|\lesssim \sigma \sqrt{d_2}$. In comparison, our theory covers the range $\frac{\sigma}{\sigma_{r}^{\star}}\lesssim\frac{1}{(d_{1}d_{2})^{1/4}}$ (modulo some log factor), which allows the noise level to be  
$(d_2/d_1)^{1/4}$ times larger than the upper bound \eqref{eq:sigma-UB-svd} derived for the above SVD approach. The sub-optimality of this approach can also be easily seen from numerical experiments as well; see Section~\ref{subsec:examples} for details.

\subsection{Minimax lower bounds}

It is natural to wonder whether our theoretical guarantees are tight, and whether there are other estimators that can potentially improve the performance of Algorithm~\ref{alg:spectral}. To answer these questions, we develop the following minimax lower bounds under Gaussian noise; the proof is deferred to Appendix~\ref{sec:proof-minimax-lower-bounds-noise}.  
\begin{theorem}\label{thm:minimax-lower-bound}
Suppose $1\leq r\leq d_{1}/2$,
and $N_{i,j}\overset{\mathrm{i.i.d.}}{\sim}\mathcal{N}(0,\sigma^{2})$.  
Define 
\[
\mathcal{M}^{\star}:=\left\{ \bm{B}\in\mathbb{R}^{d_{1}\times d_{2}}\mid\mathsf{rank}(\bm{B})=r,\text{ }\sigma_{r}(\bm{B})\in[0.9\sigma_{r}^{\star},1.1\sigma_{r}^{\star}]\right\} .
\]
Denote by $\bm{U}(\bm{B})\in\mathbb{R}^{d_{1}\times r}$ the matrix
containing the $r$ left singular vectors of $\bm{B}$. Then there
exists some universal constant $c_{\mathsf{lb}}>0$ such that
\begin{subequations}
\label{eq:minimax-lower-bounds}
\begin{align}
\inf_{\widehat{\bm{U}}}\sup_{\bm{A}^{\star}\in\mathcal{M}^{\star}}\mathbb{E}\Big[\min_{\bm{R}\in\mathcal{O}^{r\times r}}\big\|\widehat{\bm{U}}\bm{R}-\bm{U}(\bm{A}^{\star})\big\|\Big] & \geq c_{\mathsf{lb}}\min\Bigg\{\frac{\sigma^{2}}{\sigma_{r}^{\star2}}\frac{\sqrt{d_{1}d_{2}}}{p}+\frac{\sigma}{\sigma_{r}^{\star}}\sqrt{\frac{d_{1}}{p}},\,1\Bigg\},  
\label{eq:minimax-L2-subspace}\\
\inf_{\widehat{\bm{U}}}\sup_{\bm{A}^{\star}\in\mathcal{M}^{\star}}\mathbb{E}\Big[\min_{\bm{R}\in\mathcal{O}^{r\times r}}\big\|\widehat{\bm{U}}\bm{R}-\bm{U}(\bm{A}^{\star})\big\|_{2,\infty}\Big] & \geq c_{\mathsf{lb}}\min\Bigg\{\frac{\sigma^{2}}{\sigma_{r}^{\star2}}\frac{\sqrt{d_{1}d_{2}}}{p}+\frac{\sigma}{\sigma_{r}^{\star}}\sqrt{\frac{d_{1}}{p}},\,1\Bigg\}\frac{1}{\sqrt{d_{1}}},
\label{eq:minimax-Linf-subspace}
\end{align}
\end{subequations}
where the infimum is taken over all estimators for $\bm{U}(\bm{A}^{\star})$
based on the observation $\mathcal{P}_{\Omega}(\bm{A}^{\star}+\bm{N})$.
\end{theorem}
If we again consider the case where $r,\kappa,\mu \asymp 1$, then the above lower bounds \eqref{eq:minimax-lower-bounds} match the noise effect terms in Theorem~\ref{thm:U_loss} (or equivalently, \eqref{eq:noise-effect-simple}) up to logarithmic factors. This unveils a fundamental reason why the linear and the quadratic terms in \eqref{eq:noise-effect-simple} are both essential in determining the estimation risk.

Another information-theoretic limit that concerns only the influence of subsampling is supplied as follows; the proof is postponed to Appendix~\ref{sec:proof-minimax-lower-bounds-sampling}. 
\begin{theorem}
\label{thm:lower-bounds-bipartite}
Suppose  $d_1\leq d_2$ and $p<\frac{1-\epsilon}{\sqrt{d_1d_2}}$ for any small constant $0<\epsilon <1$. With probability approaching one, there exist unit vectors $\bm{u}^{\star}, \widetilde{\bm{u}}^{\star} \in \mathbb{R}^{d_1}$ and $\bm{v}^{\star}, \widetilde{\bm{v}}^{\star}\in \mathbb{R}^{d_2}$ such that
\begin{itemize}
\item 
	$\min \|  \bm{u}^{\star} \pm \widetilde{\bm{u}}^{\star} \|_2  \asymp 1$ and $\|  \bm{u}^{\star}\bm{v}^{\star\top} - \widetilde{\bm{u}}^{\star} \widetilde{\bm{v}}^{\star\top} \|_{\mathrm{F}}  \asymp 1$;
		
\item one cannot distinguish  $\bm{u}^{\star}\bm{v}^{\star\top}$ and $\widetilde{\bm{u}}^{\star}\widetilde{\bm{v}}^{\star\top}$ from the entries in $\Omega$, i.e.~$\mathcal{P}_{\Omega}( \bm{u}^{\star}\bm{v}^{\star\top} ) = \mathcal{P}_{\Omega}( \widetilde{\bm{u}}^{\star} \widetilde{\bm{v}}^{\star\top} ) $.   

\end{itemize}

\end{theorem}

In words, Theorem~\ref{thm:lower-bounds-bipartite} asserts that one cannot hope to achieve consistent subspace estimation (in the sense of \eqref{eq:consistent-estimation-defn}) at all, as soon as the sampling rate $p$ falls below the threshold $1/\sqrt{d_1d_2}$. Putting Theorem~\ref{thm:minimax-lower-bound}-\ref{thm:lower-bounds-bipartite} together reveals that: consistent estimation can by no means be guaranteed unless  
\begin{align}
p\gtrsim\frac{1}{\sqrt{d_{1}d_{2}}}\qquad\text{and}\qquad\frac{\sigma^{2}}{\sigma_{r}^{\star2}}\lesssim\frac{p}{\sqrt{d_{1}d_{2}}} ,
\end{align}
which agrees with our theoretical guarantees \eqref{eq:consistent-estimation-spectral} (up to some logarithmic term). As a result, our minimax lower bounds confirm the near optimality of Algorithm~\ref{alg:spectral} in enabling consistent estimation.

On the other hand, 
it is widely recognized that spectral methods are typically unable to
achieve exact recovery or optimal estimation accuracy in the presence 
of missing data, even in the balanced case with $d_1= d_2$.  For instance, if there is no noise, namely $\sigma=0$, 
the spectral methods fail to achieve perfect recovery as long as $p<1$ (basically the first two terms of \eqref{eq:defn-Err-UB} do not vanish) \cite{keshavan2010matrix}, whereas exact recovery might sometimes be feasible with the aid of optimization-based approaches \cite{candes2009exact}. 
More often than not, spectral methods are employed to produce a rough initial 
estimate that outperforms the random guess, which can then
be refined via other algorithms (e.g.~nonconvex optimization algorithms
like gradient descent and alternating minimization \cite{keshavan2010matrix,sun2016guaranteed,ma2017implicit,cai2019nonconvex}).

\section{Consequences for concrete applications\label{subsec:examples}}

We showcase the consequence of Theorem \ref{thm:U_loss} in three concrete applications previously
introduced in Section \ref{sec:Introduction} in relatively simple settings.  Rather than striving for full generality, our purpose is to highlight the broad applicability of our main results.


\subsection{Noisy tensor completion \label{subsec:tensor-completion}}

\paragraph{Problem settings.} We begin by considering the problem of symmetric tensor completion.
Consider an unknown order-3 tensor 
\[
\bm{T}^{\star}=\sum_{s=1}^{r}\bm{w}_{s}^{\star}\otimes\bm{w}_{s}^{\star}\otimes\bm{w}_{s}^{\star}:=\sum_{s=1}^{r}\left(\bm{w}_{s}^{\star}\right)^{\otimes3}\in\mathbb{R}^{d\times d\times d},
\]
with canonical polyadic (CP) rank $r$. The goal is to estimate the subspace
spanned by $\left\{ \bm{w}_{s}^{\star}\right\} _{s=1}^{r}$, on the
basis of the noisy tensor $\bm{T}=[T_{i,j,k}]_{1\leq i,j,k\leq d}$
obeying
\begin{equation}
T_{i,j,k}=\begin{cases}
T_{i,j,k}^{\star}+N_{i,j,k}, & \left(i,j,k\right)\in\Omega,\\
0, & \left(i,j,k\right)\notin\Omega.
\end{cases}\label{eq:TC-observation}
\end{equation}
Here, $T_{i,j,k}$ is the observed entry in location $\left(i,j,k\right)$,
$N_{i,j,k}$ is the associated independent random noise satisfying Assumption
\ref{assumption:random-noise}, and $\Omega\subseteq\left[d\right]^{3}$
stands for a sampling set obtained via uniform random sampling with
sampling rate $p$ (namely, each entry is observed independently with
probability $p$).

\paragraph{Algorithm.} Observe that the mode-1 matricization of $\bm{T}^{\star}$ is given by\footnote{We let $\bm{a}\otimes\bm{b}:=\tiny\left[\begin{array}{c}
a_{1}\bm{b}\\
\vdots\\
a_{d}\bm{b}
\end{array}\right]$ represent a $d^{2}$-dimensional vector.}
\begin{equation}
	\bm{A}^{\star}=\sum_{s=1}^{r}\bm{w}_{s}^{\star}\left(\bm{w}_{s}^{\star}\otimes\bm{w}_{s}^{\star}\right)^{\top}\in\mathbb{R}^{d\times d^{2}}, 
	\label{eq:A-unfold}
\end{equation}
indicating that the column subspace of $\bm{A}^{\star}$ is essentially the subspace spanned by the tensor factors $\{\bm{w}_s^{\star}\}_{s=1}^r$. Therefore, if we denote by $\bm{A}\in\mathbb{R}^{d\times d^{2}}$ the mode-1 matricization of $\bm{T}$, then we can invoke our general spectral method to estimate the column subspace of $\bm{A}^{\star}$ given $\bm{A}$. This procedure is summarized in Algorithm~\ref{alg:spectral-tc}. 
\begin{algorithm}[H]
\caption{The spectral method for tensor completion}
\label{alg:spectral-tc}
\begin{algorithmic}[1]
\STATE{{\bf Input:} sampling  set $\Omega$, observed entries $\left\{ T_{i, j, k} \mid \left( i, j, k \right) \in \Omega \right\}$,  sampling rate $p$, CP-rank $r$.}
\STATE{Let $\bm{A}\in\mathbb{R}^{d\times d^{2}}$ be the mode-1 matricization of the observed tensor $\bm{T}$ (see \eqref{eq:TC-observation}), namely, set $A_{i,\left(j-1\right)d+k}=T_{i,j,k}$ for each $\left(i,j,k\right)\in\left[d\right]^{3}$, and employ $\bm{A}$ as the input of Algorithm~\ref{alg:spectral}.}
\STATE{{\bf Output} $\bm{U} \in \mathbb{R}^{d\times r}$ returned by Algorithm~\ref{alg:spectral} as the subspace estimate.}
\end{algorithmic}
\end{algorithm}

\paragraph{Theoretical guarantees.} In order to provide theoretical support for Algorithm~\ref{alg:spectral-tc}, we introduce a few more notation. First, 
we introduce the following quantities
\begin{equation}
\kappa_{\mathsf{tc}}:= \lambda_{\max}^{\star} \, / \, \lambda_{\min}^{\star},\qquad\lambda_{\min}^{\star}:=\min_{1\leq i\leq r}\left\Vert \bm{w}_{i}^{\star}\right\Vert _{2}^{3},\qquad\lambda_{\max}^{\star}:=\max_{1\leq i\leq r}\left\Vert \bm{w}_{i}^{\star}\right\Vert _{2}^{3}.\label{eq:defn-TC-cond}
\end{equation}
Note that $\left\Vert \bm{w}_{i}^{\star}\right\Vert _{2}^{3}$ is precisely the Frobenius norm of the rank-$1$ tensor $\bm{w}_{i}^{\star \otimes 3}$ --- the $i$-th tensor component.  Informally, $\kappa_{\mathsf{tc}}$ captures the condition number of
the unknown tensor. Additionally, similar to matrix completion, we
introduce the following incoherence definitions  that enable efficient
tensor completion:

\begin{definition}[Incoherence] Define the incoherence parameters $\mu_3, \mu_4, \mu_5$ for the tensor $\bm{T}^{\star}$
and its tensor factors $\left\{ \bm{w}_{s}^{\star}\right\} _{s=1}^{r}$ as follows:
\begin{subequations}
\label{asmp:tensor}
\begin{align}
\mu_3 := \frac{d^3 \left\Vert \bm{T}^{\star}\right\Vert_{\infty}^2 }{ \left\Vert \bm{T}^{\star}\right\Vert _{\mathrm{F}}^2}, \qquad
\mu_4 := \max_{1\leq i \leq r}\frac{d \left\Vert\bm{w}_{i}^{\star}\right\Vert_{\infty}^2}{\left\Vert\bm{w}_{i}^{\star}\right\Vert_{2}^2}, \qquad
\mu_5 := \max_{1\leq i \neq j \leq r} \frac{d \left\langle \bm{w}_{i}^{\star},\bm{w}_{j}^{\star}\right\rangle^2}{\left\Vert \bm{w}_{i}^{\star}\right\Vert_{2}^2 \left\Vert \bm{w}_{j}^{\star}\right\Vert_{2}^2 }
\end{align}
\end{subequations}
\end{definition}

For notational convenience, we also set
\begin{equation}
\mu_{\mathsf{tc}}:=\max\left\{ \mu_{3,}\mu_{4}^{2}\right\} .\label{eq:defn-lambda-min-mu}
\end{equation}
Given that the tensor factors $\{\bm{w}_{s}^{\star}\}_{1\leq s\leq r}$
are in general not orthogonal to each other, we introduce the following
orthonormal matrix $\bm{U}^{\star}\in\mathbb{R}^{d\times r}$ to represent
the subspace spanned by $\{\bm{w}_{s}^{\star}\}_{1\leq s\leq r}$:
\begin{equation}
\bm{U}^{\star}:=\bm{W}^{\star}\big(\bm{W}^{\star\top}\bm{W}^{\star}\big)^{-1/2},\qquad\bm{W}^{\star}:=\left[\bm{w}_{1}^{\star},\cdots,\bm{w}_{r}^{\star}\right]\in\mathbb{R}^{d\times r}.\label{eq:defn-Ustar-TC}
\end{equation}
Note that the particular choice of $\bm{U}^{\star}$ in \eqref{eq:defn-Ustar-TC} is not pivotal, and can be replaced by any $d_1\times r$ orthonormal matrix that spans the same column space as  $\bm{W}^{\star}$.  With these in place, we are now ready to quantify the estimation error of this spectral
algorithm. The proof is deferred to Appendix~\ref{subsec:pf:tc}.

\begin{corollary}[Symmetric tensor completion]\label{cor:tensor-completion}
Consider the above
tensor completion model. There exist some universal constants $c_{0},c_{1},c_{2}>0$
such that if
\[
	p\geq c_{0}\max\left\{ \frac{\mu_{\mathsf{tc}}\kappa_{\mathsf{tc}}^{4}r\log^{2}d}{d^{3/2}},\frac{\mu_{\mathsf{tc}}\kappa_{\mathsf{tc}}^{8}r\log^{2}d}{d^{2}}\right\} ,\quad \,\,\frac{\sigma}{\lambda_{\min}^\star}\leq c_{1}\min\left\{ \frac{\sqrt{p}}{\kappa_{\mathsf{tc}}d^{3/4} \sqrt{\log d}},\frac{1}{\kappa_{\mathsf{tc}}^{3}}\sqrt{\frac{p}{d\log d}}\right\} \]
 
\[ 
	\text{and} \quad r\leq c_{2}\min\left\{ \frac{d}{\kappa_{\mathsf{tc}}^{4}\mu_{4}},\frac{1}{\kappa_{\mathsf{tc}}^2}\sqrt{\frac{d}{\mu_{5}}}\right\} ,
\]
then with probability exceeding $1-O(d^{-10})$, Algorithm \ref{alg:spectral-tc} yields \begin{subequations}
\begin{align}
\left\Vert \bm{U}\bm{R}-\bm{U}^{\star}\right\Vert  & \lesssim \mathcal{E}_{\mathsf{tc}},\label{claim:U_op_loss_TC}\\
\left\Vert \bm{U}\bm{R}-\bm{U}^{\star}\right\Vert _{2,\infty} & \lesssim\kappa_{\mathsf{tc}}^{2}\sqrt{\frac{\mu_{\mathsf{tc}}r}{d}}\cdot \mathcal{E}_{\mathsf{tc}},\label{claim:U_2inf_loss_TC}
\end{align}
\end{subequations}where $\bm{R}\,:=\,\underset{\bm{Q}\in\mathcal{O}^{r\times r}}{\arg\min}\text{ }\left\Vert \bm{U}\bm{Q}-\bm{U}^{\star}\right\Vert _{\mathrm{F}}$ and
\begin{equation}
\mathcal{E}_{\mathsf{tc}}:=\frac{\mu_{\mathsf{tc}}\kappa_{\mathsf{tc}}^{2}r\log d}{d^{3/2}p}+\sqrt{\frac{\mu_{\mathsf{tc}}\kappa_{\mathsf{tc}}^{4}r\log d}{d^{2}p}}+\frac{\sigma^{2}}{\lambda_{\min}^{\star2}}\frac{d^{3/2}\log d}{p}+\frac{\sigma\kappa_{\mathsf{tc}}}{\lambda_{\min}^{\star}}\sqrt{\frac{d\log d}{p}}+\frac{\mu_{4}\kappa_{\mathsf{tc}}^{2}r}{d}.\label{def:err-TC}
\end{equation}
\end{corollary}

As discussed in several related work (e.g.~\cite{xia2017statistically,jain2014provable,montanari2018spectral,xia2017polynomial,cai2019nonconvex}),
once we obtain reliable estimates of the subspace spanned by the tensor
factors, we can further exploit the tensor structure to estimate
the unknown tensor. Indeed, in many tensor completion algorithms,
subspace estimation serves as a crucial initial step for tensor completion.
Moreover, while prior works only provide $\ell_{2}$ estimation error
bounds, Corollary~\ref{cor:tensor-completion} further delivers $\ell_{2,\infty}$ statistical
guarantees, which reflect
a stronger sense of statistical accuracy.
We note that \cite[Theorem~4]{xia2019sup} derived an appealing $\ell_{2,\infty}$ statistical error bound for an algorithm called HOSVD, under the tensor de-noising setting. In comparison to the Gaussian noise considered therein, our results accommodate the case with missing data and possibly spiky noise.


%

\paragraph{Implications.} In what follows, we discuss the sample size and the signal-to-noise (SNR) required for achieving consistent tensor estimation (namely, obtaining an $o \, (1)$ relative estimation error). For convenience of presentation, we again focus on the low-rank, incoherent, and well-conditioned case with $r,\mu,\kappa_{\mathsf{tc}}\asymp1$. In this case, our
results in Corollary~\ref{cor:tensor-completion} indicate that
\begin{align}
\underset{\bm{Q}\in\mathcal{O}^{r\times r}}{\min}\left\Vert \bm{U}\bm{Q}-\bm{U}^{\star}\right\Vert  & =o \, (1)\qquad\text{and}\qquad\underset{\bm{Q}\in\mathcal{O}^{r\times r}}{\min}\left\Vert \bm{U}\bm{Q}-\bm{U}^{\star}\right\Vert _{2,\infty}=o \, \big(1/\sqrt{d}\big)
\end{align}
with high probability, provided that the sample size and the noise
satisfy
\begin{equation}
p\gtrsim \frac{\log^{2}d}{d^{3/2}}\qquad\text{and}\qquad\frac{\sigma}{\lambda_{\min}^\star} = o\left( \sqrt{\frac{p}{d^{3/2}\log d}} \right)  .\label{eq:sample-size-noise-TC}
\end{equation}
Several remarks are in order.
\begin{itemize}

\item \emph{Sample complexity.} It is widely conjectured that the sample
complexity $pd^{3}$ required to reconstruct a order-3 tensor in polynomial
time --- even in the noiseless case --- is at least $d^{3/2}$ (or
equivalently, $p\gtrsim1/d^{3/2}$) \cite{barak2016noisy,xia2017statistically,montanari2018spectral}.
Therefore, our theory reveals that spectral methods achieve consistent
estimation (w.r.t.~both $\|\cdot\|$ and $\|\cdot\|_{2,\infty}$),
as long as the sample size is slightly above the (conjectured) computational
limit. Moreover, it is easily seen that the bias incurred by deleting
the diagonal is much smaller than the error due to missing data, which
justifies the rationale that diagonal deletion does not harm the performance
by much.

\item \emph{Noise size.} We now comment on the noise size requirement. It
is easily seen that the maximum magnitude of the entries of $\bm{T}^{\star}$
in this case is $\|\bm{T}^{\star}\|_{\infty}\asymp\lambda_{\max}^\star/d^{3/2}$.
As a result, the noise size condition in (\ref{eq:sample-size-noise-TC})
is equivalent to
\[
\frac{\sigma}{\|\bm{T}^{\star}\|_{\infty}}\lesssim \sqrt{pd^{3/2}}.
\]
Taken together with our sample size requirement $p\gtrsim\frac{\log^{2}d}{d^{3/2}}$,
this condition allows the noise magnitude in each observed entry to
significantly exceed the size of the corresponding entry, which covers
a broad range of scenarios of practical interest. In addition, in
the fully-observed case (i.e.~$p=1$) with i.i.d.~Gaussian noise, the authors in \cite{zhang2018tensor}
showed that the noise size condition (\ref{eq:sample-size-noise-TC}) --- up to some log factor --- is
necessary for any polynomial-time algorithm to achieve consistent estimation, provided that a certain hypergraphic planted clique conjecture holds.   



\end{itemize}

Finally, we remark that in the fully-observed case (i.e.~$p=1$) with i.i.d.~Gaussian noise, 
it can be seen from \cite[Theorem~1]{zhang2018tensor} that (\ref{claim:U_op_loss_TC})
is suboptimal; in fact, the minimax risk  consists only of
the linear term in $\sigma$ (namely, $\frac{\sigma}{\lambda_{\min}^{\star}}\sqrt{d}$, if we omit log factors and assume $r,\mu_{\mathsf{tc}}, \mu_5, \kappa \asymp 1$).
This is a typical drawback of the spectral method for tensor estimation, since it falls short of exploiting the low-complexity structure in the row subspace. However, the spectral estimate offers a reasonably good initial estimate for this problem, and one can often employ optimization-based iterative refinement paradigms (like gradient descent \cite{cai2019nonconvex}) to obtain minimax optimal estimates.

\subsection{Covariance estimation\,/\,principal component analysis with missing data \label{subsec:Covariance-estimation}}

\paragraph{Model and algorithm.} Next, we study covariance
estimation with missing data, as previously introduced in Section~\ref{sec:Introduction}. For concreteness, imagine a set of independent sample vectors obeying
\[
\bm{x}_{i}=\bm{B}^{\star}\bm{f}_{i}^{\star}+\bm{\eta}_{i}\in\mathbb{R}^{d},\qquad\bm{f}_{i}^{\star}\overset{\mathrm{i.i.d.}}{\sim}\mathcal{N}(\bm{0},\bm{I}_{r}),\qquad1\leq i\leq n .
\]
Here $\bm{B}^{\star}\in\mathbb{R}^{d\times r}$ encodes the $r$-dimensional principal
subspace underlying the data (sometimes referred to as the factor
loading matrix in factor models \cite{lawley1962factor,fan2018robust}),
$\bm{f}_{i}^{\star}\sim\mathcal{N}(\bm{0},\bm{I}_{r})$ represents
some random coefficients, and the noise vector $\bm{\eta}_{i}=[\eta_{i,j}]_{1\leq j\leq d}$
consists of independent Gaussian components\footnote{Here, we assume $\bm{f}_i^\star$ and $\bm{\eta}_i$ to be Gaussian for simplicity of presentation. The results in this subsection continue to hold if they are sub-Gaussian random vectors.} obeying
\[
\mathbb{E}[\eta_{i,j}]=0\qquad\text{and}\qquad\mathsf{Var}\left[\eta_{i,j}\right]\leq\sigma^{2}.
\]
%
What we observe is
a partial set of entries of $\bm{x}_{i}=[x_{i,j}]_{1\leq j\leq d}$,
namely, we only observe
$
 x_{i,j}
$
for any $(i,j)\in\Omega$,
where $\Omega$ is  obtained by random sampling with rate
$p$.  The goal is to estimate the subspace spanned by
$\bm{B}^{\star}$, or even $\bm{B}^{\star}\bm{B}^{\star\top}$.

If we write $\bm{F}^{\star}=\left[\bm{f}_{1,}^{\star}\cdots,\bm{f}_{n}^{\star}\right]\in\mathbb{R}^{r\times n}$
and $\bm{N}=\left[\bm{\eta}_{1},\cdots,\bm{\eta}_{n}\right]\in\mathbb{R}^{d\times n}$,
then it boils down to estimating the column space of $\bm{A}^{\star}:=\bm{B}^{\star}\bm{F}^{\star}$ from
the data $	\mathcal{P}_{\Omega}(\bm{X}) =\mathcal{P}_{\Omega}\big( \bm{B}^{\star}\bm{F}^{\star}+\bm{N} \big)$. 
Our spectral method for covariance estimation is  summarized in Algorithm~\ref{alg:spectral-ce}.
\begin{algorithm}
\caption{The spectral method for covariance estimation}
\label{alg:spectral-ce}
\begin{algorithmic}[1]
\STATE{{\bf Input:} sampling  set $\Omega$, observed entries $\left\{ X_{i, j} \mid \left( i, j \right) \in \Omega \right\}$,  sampling rate $p$, rank $r$.}
\STATE{Let $\bm{A}=\mathcal{P}_{\Omega}(\bm{X})\in\mathbb{R}^{d\times n}$ with $\bm{X}=[\bm{x}_{1,}\cdots,\bm{x}_{n}]$,  and use $\bm{A}$ as the input of Algorithm \ref{alg:spectral}. Let $\bm{U}\in\mathbb{R}^{d\times r}$ and $\bm{\Sigma}\in\mathbb{R}^{r\times r}$ be the estimates returned by Algorithm \ref{alg:spectral}, and set $\bm{B} := \frac{1}{\sqrt{n}}\bm{U}\bm{\Sigma}$.}
\STATE{{\bf Output} $\bm{U}$ as the subspace estimate and $\bm{S}:=\bm{B}\bm{B}^{\top}$ as the covariance estimate.}
\end{algorithmic}
\end{algorithm}

\paragraph{Theoretical guarantees.} 
In order to present our theory, we make a few more definitions.  Without loss of generality, we shall define
\begin{equation}
\bm{S}^{\star}:=\bm{B}^{\star}\bm{B}^{\star\top}=\bm{U}^{\star}\bm{\Lambda}^{\star}\bm{U}^{\star\top} \qquad\text{and}\qquad \bm{B}^{\star}=\bm{U}^{\star}\bm{\Lambda}^{\star1/2} , \label{eq:defn-B-U}
\end{equation}
where $\bm{U}^{\star}\in\mathbb{R}^{d\times r}$ consists of orthonormal
columns, and $\bm{\Lambda}^{\star}=\mathsf{diag}(\lambda_{1}^{\star},\cdots,\lambda_{r}^{\star})\in\mathbb{R}^{r\times r}$
is a diagonal matrix with $\lambda_1^{\star}\geq \cdots \geq \lambda_r^{\star}\geq 0$. We also define the condition number and the incoherence parameter as 
\begin{align}
	\kappa_{\mathsf{ce}} & :=\lambda_{1}^{\star}\,/\,\lambda_{r}^{\star} \qquad \text{and} \qquad \mu_{\mathsf{ce}} := \frac{d}{r} \left\Vert \bm{U}^{\star}\right\Vert _{2,\infty}^2.
\label{def:kappa_ce}
\end{align}
%
%

We are now positioned to derive statistical estimation guarantees
using our general theorem. The following result is a consequence of
Theorem \ref{thm:U_loss}; the proof is postponed to Appendix~\ref{subsec:pf:cov}.

\begin{corollary}[Covariance estimation]\label{cor:cov}Consider
the above covariance estimation model with missing data. There exist
universal constants $c_{0},c_{1}>0$ such that if $r\leq c_{1}\frac{d}{\mu_{\mathsf{ce}}\kappa_{\mathsf{ce}}^{2}}$ and
\begin{align}
\label{eq:sample-size-assumption-ce}
	n\geq c_{0}\max\left\{ \frac{\mu_{\mathsf{ce}}^{2}\kappa_{\mathsf{ce}}^{6}r^{2}\log^{6}\left(n+d\right)}{dp^{2}},\frac{\mu_{\mathsf{ce}}\kappa_{\mathsf{ce}}^{5}r\log^{3}\left(n+d\right)}{p}, \frac{\sigma^{4}}{\lambda_{r}^{\star2}}\frac{\kappa_{\mathsf{ce}}^{2}d\log^{2}\left(n+d\right)}{p^{2}},\frac{\sigma^{2}}{\lambda_{r}^{\star}}\frac{\kappa_{\mathsf{ce}}^{3}d\log\left(n+d\right)}{p}\right\} , 
\end{align}
then with probability exceeding $1-O((n+d)^{-10})$, Algorithm \ref{alg:spectral-ce} yields 
\begin{subequations}
\begin{align}
\left\Vert \bm{U}\bm{R}-\bm{U}^{\star}\right\Vert  & \lesssim\mathcal{E}_{\mathsf{ce}},\label{claim:U_op_loss_cov}\\
\left\Vert \bm{U}\bm{R}-\bm{U}^{\star}\right\Vert _{2,\infty} & \lesssim\kappa_{\mathsf{ce}}^{3/2}\sqrt{\frac{\mu_{\mathsf{ce}}r\log \left(n+d\right)}{d}}\cdot\mathcal{E}_{\mathsf{ce}},\label{claim:U_2inf_loss_cov}\\
\left\Vert \bm{S}-\bm{S}^{\star}\right\Vert  & \lesssim\kappa_{\mathsf{ce}}\lambda_{1}^{\star}\cdot\mathcal{E}_{\mathsf{ce}},\label{claim:S_op_loss}\\
\left\Vert \bm{S}-\bm{S}^{\star}\right\Vert _{\infty} & \lesssim\frac{\kappa_{\mathsf{ce}}\mu_{\mathsf{ce}}r\log \left(n+d\right)}{d}\lambda_{1}^{\star}\cdot\mathcal{E}_{\mathsf{ce}}.\label{claim:S_inf_loss}
\end{align}
\end{subequations}
Here, $\bm{R}\,:=\,\underset{\bm{Q}\in\mathcal{O}^{r\times r}}{\arg\min}\text{ }\left\Vert \bm{U}\bm{Q}-\bm{U}^{\star}\right\Vert _{\mathrm{F}}$ and
\begin{equation}
\mathcal{E}_{\mathsf{ce}}:=\frac{\mu_{\mathsf{ce}}\kappa^2_{\mathsf{ce}}r\log^{2}\left(n+d\right)}{\sqrt{dn}\,p}+\sqrt{\frac{\mu_{\mathsf{ce}}\kappa_{\mathsf{ce}}^{3}r\log^{2}\left(n+d\right)}{np}}+\frac{\sigma^{2}}{\lambda_{r}^{\star}}\sqrt{\frac{d}{n}}\,\frac{\log\left(n+d\right)}{p}+\frac{\sigma}{\sqrt{\lambda_{r}^{\star}}}\sqrt{\frac{d}{n}}\sqrt{\frac{\kappa_{\mathsf{ce}}\log\left(n+d\right)}{p}}+\frac{\mu_{\mathsf{ce}}\kappa_{\mathsf{ce}}r}{d}.\label{def:err-CE}
\end{equation}

\end{corollary}


\begin{remark} We make note of a scaling issue that one shall bear in mind when comparing this result with our main theorem. In the settings of Theorem~\ref{thm:U_loss}, the singular values $\{ \sigma_i^\star \}_{i=1}^r$ of the truth $\bm{A}^\star$ do not change as the column dimension $d_2$ grows. In contrast, in the settings of Corollary~\ref{cor:cov}, the singular values of the sample covariance matrix keep growing as we collect more sample vectors, which is equivalent to saying that these singular values scale with the column dimension.   \end{remark}

\paragraph{Discussion.}
To facilitate interpretation, let us again focus on the case where
$\mu_{\mathsf{ce}},\kappa_{\mathsf{ce}}\asymp1$. Corollary~\ref{cor:cov}
demonstrates that for any given sampling rate $p$, we can achieve consistent estimation\footnote{Here, consistent estimation is declared if ${\min}_{\bm{Q}\in\mathcal{O}^{r\times r}}\left\Vert \bm{U}\bm{Q}-\bm{U}^{\star}\right\Vert  =o\left(1\right)$ and $\left\Vert \bm{S}-\bm{S}^{\star}\right\Vert   =o\left(\lambda_{r}^{\star}\right)$. 
}
as long as the number $n$ of samples satisfies
\begin{equation}
n\gtrsim \max\left\{ \frac{r^{2}}{dp^{2}},\,\frac{r}{p},\,\frac{\sigma^{4}d}{\lambda_{r}^{\star2}p^{2}},\,\frac{\sigma^{2}d}{\lambda_{r}^{\star}p}\right\} \mathrm{poly}\log d.\label{eq:sample-size-ce}
\end{equation}
Throughout this subsection, the sample size refers to $n$ --- the
number of sample vectors $\left\{ \bm{x}_{i}\right\} _{1\leq i\leq n}$
available.

Next, we compare our $\ell_{2,\infty}$ bounds with several prior work for the case with $r\asymp 1$. We emphasize again that the foci and model assumptions of these prior papers might be quite different from ours (e.g.~\cite{zhu2019high} is able to accommodate inhomogeneous sampling patterns),  and the advantages of our results discussed below are restricted to the settings considered in this paper.     
For simplicity, we ignore all log factors. 
\begin{itemize}
\item Suppose that $\sigma = 1$. 
In this setting, \cite[Theorem~4]{zhu2019high}
demonstrates that if
\[
n\gtrsim\max\left\{ \tfrac{1}{p^{2}},\tfrac{d^{2}}{\lambda_{r}^{\star2}p^{2}},\tfrac{d}{\lambda_{r}^{\star}p^{2}}\right\} \mathrm{poly}\log d,
\]
then with high probability one has
\[
\min_{\bm{Q}\in \mathcal{O}^{r\times r}} \left\Vert \bm{U}\bm{Q}-\bm{U}^{\star}\right\Vert _{2,\infty}\lesssim \tfrac{1}{p\sqrt{n}}\left(\tfrac{1}{\sqrt{\lambda_{r}^{\star}}}+\tfrac{1}{\lambda_{r}^{\star}}\right)\left(1+\tfrac{\sqrt{d}}{\lambda_{r}^{\star}}\right) \mathrm{poly}\log d
\]
In comparison, our sample size requirement for consistent
estimation improves upon \cite[Theorem 4]{zhu2019high} by a factor of $\min\left\{ d,p^{-1}\right\} $. 
Moreover, our estimation error bound improves upon \cite[Theorem 4]{zhu2019high}
by a factor of $\min\big\{\sqrt{\lambda_{r}^{\star}},\frac{1}{\sqrt{p}}\big\}$
if $\sqrt{d}\ll\lambda_{r}^{\star}\lesssim d$, by a factor of $\frac{\sqrt{d}}{\lambda_{r}^{\star}}$
when $\lambda_{r}^{\star}\lesssim1$, and by a factor of $\min\left\{ \frac{\sqrt{d}}{\lambda_{r}^{\star}\sqrt{p}},\sqrt{\frac{d}{\lambda_{r}^{\star}}}\right\} $
if $1\ll\lambda_{r}^{\star}\lesssim\sqrt{d}$. 

\item In the absence of missing data, the $\ell_{2,\infty}$ error bound presented in \cite[Theorem~1.1]{cape2019two}
reads (ignoring logarithmic terms)
\[
  \min_{\bm{Q}\in \mathcal{O}^{r\times r}} \left\Vert \bm{U}\bm{Q}-\bm{U}^{\star}\right\Vert _{2,\infty}\lesssim\begin{cases}
\sqrt{\frac{1}{nd}}, & \text{for}\;\frac{\sigma}{\sqrt{\lambda_{r}^{\star}}}\lesssim\frac{1}{\sqrt{d}} \\
\frac{\sigma^{2}}{\lambda_{r}^{\star}}\sqrt{\frac{d}{n}} & \text{for}\;\frac{1}{\sqrt{d}}\ll\frac{\sigma}{\sqrt{\lambda_{r}^{\star}}}\lesssim 1 .
\end{cases}
\]
 Consequently, our result improves upon the above error
bound by a factor of $\frac{\sigma\sqrt{d}}{\sqrt{\lambda_{r}^{\star}}}$
if $\frac{1}{\sqrt{d}}\ll\frac{\sigma}{\sqrt{\lambda_{r}^{\star}}}\lesssim1$,
while being able to handle the case with larger noise (namely, $\frac{\sigma}{\sqrt{\lambda_{r}^{\star}}}\gg1$).

\end{itemize}

\subsection{Community recovery in bipartite stochastic block models\label{subsec:BSBM}}

 As it turns out, if we
denote by $\bm{A}\in\mathbb{R}{}^{|\mathcal{U}|\times|\mathcal{V}|}$
the bi-adjacency matrix of the observed random bipartite graph or
its centered version, then $\bm{A}^{\star}:=\mathbb{E}[\bm{A}]$ exhibits
a low-rank structure (as we shall elaborate momentarily). Perhaps
more importantly, the column subspace of $\bm{A}^{\star}$ reveals
the community memberships of all nodes in $\mathcal{U}$. As a result,
this biclustering problem is tightly connected to subspace estimation
given noisy observations of a low-rank matrix. In particular, when
the size of $\mathcal{V}$ is substantially larger than that of $\mathcal{U}$,
one might encounter a situation where only the nodes in $\mathcal{U}$
(rather than those in $\mathcal{V}$) can be reliably clustered. This
calls for development of ``one-sided'' community recovery algorithms,
that is, the type of algorithms that guarantee reliable clustering
of $\mathcal{U}$ without worrying about the clustering accuracy in
$\mathcal{V}$.

\paragraph{Model.} This subsection investigates the problem of biclustering, by considering a bipartite stochastic block model (BSBM) with two
disjoint groups of nodes $\mathcal{U}$ and $\mathcal{V}$. Suppose
that the nodes in $\mathcal{U}$ (resp.~$\mathcal{V}$) form two
clusters. For each pair of nodes $(i,j)\in(\mathcal{U},\mathcal{V})$,
there is an edge connecting them with probability depending only on
the community memberships of $i$ and $j$. To be more specific:

\begin{itemize}
\item \emph{Biclustering structure.} Consider two disjoint collections of
nodes $\mathcal{U}$ and $\mathcal{V}$, which are of size $n_{u}$
and $n_{v}$, respectively. Suppose that each collection of nodes
can be clustered into two communities. To be more precise, let $\mathcal{I}_{1}\subseteq\mathcal{U}$
and $\mathcal{I}_{2}=\mathcal{U}\backslash\mathcal{I}_{1}$ (resp.~$\mathcal{J}_{1}\subseteq\mathcal{V}$
and $\mathcal{J}_{2}=\mathcal{V}\backslash\mathcal{J}_{1}$) be two
non-overlapping communities in $\mathcal{U}$ (resp.~$\mathcal{V}$)
that contain $n_{u}/2$ (resp.~$n_{v}/2$) nodes each. Without loss
of generality, we assume that $\mathcal{I}_{1}$ contains the first
$n_{u}/2$ nodes of $\mathcal{U}$, and $\mathcal{J}_{1}$ contains
the first $n_{v}/2$ nodes of $\mathcal{V}$; these are of course
\emph{a priori} unknown.
\item \emph{Measurement model.} What we observe is a random bipartite graph
generated based on the community memberships of the nodes. In the
simplest version of BSBMs, a pair of nodes $(i,j)\in(\mathcal{U},\mathcal{V})$
is connected by an edge independently with probability $q_{\mathsf{in}}$
if either $(i,j)\in(\mathcal{I}_{1},\mathcal{J}_{1})$ or $(i,j)\in(\mathcal{I}_{2},\mathcal{J}_{2})$
holds, and with probability $q_{\mathsf{out}}$ otherwise. Here, $0\leq q_{\mathsf{out}}\leq q_{\mathsf{in}}\leq1$
represent the edge densities. If we denote by $\bm{C}\in\{0,1\}^{n_{u}\times n_{v}}$
the bi-adjacency matrix of this random bipartite graph, then one has
\[
\mathbb{P}\left\{ C_{i,j}=1\right\} \overset{\text{ind.}}{=}\begin{cases}
q_{\mathsf{in}}, & \text{if}\;\;(i,j)\in(\mathcal{I}_{1},\mathcal{J}_{1})\;\;\text{or}\;\;(i,j)\in(\mathcal{I}_{2},\mathcal{J}_{2}),\\
q_{\mathsf{out}}, & \text{\text{otherwise}.}
\end{cases}
\]
\end{itemize}
Our goal is to recover the community memberships of the nodes in $\mathcal{U}$, based on the above random bipartite graph. In what follows, we define
\begin{equation}
n:=n_{u}+n_{v},\label{eq:defn-n-BSBM}
\end{equation}
and declare exact community recovery of $\mathcal{U}$ if the partition
of the nodes returned by our algorithm coincides precisely with the
true partition $(\mathcal{I}_{1},\mathcal{I}_{2})$.

While our theory covers a broad range of $n_{u}$ and $n_{v}$, we
emphasize the case where $n_{v}\gg n_{u}$ (namely, $\mathcal{V}$
contains far more nodes than $\mathcal{U}$). In such a case, it is
not uncommon to encounter a situation where one can only hope to recover
the community memberships of the nodes in $\mathcal{U}$ but not those
in $\mathcal{V}$.

\paragraph{Algorithm.} To attempt community recovery, we look at a centered version of the
bi-adjacency matrix\footnote{Here, we assume prior knowledge about $q_{\mathsf{in}}$ and $q_{\mathsf{out}}$.
Otherwise, the quantity $\frac{q_{\mathsf{in}}+q_{\mathsf{out}}}{2}$
can also be easily estimated.}
\begin{align}
\bm{A} & :=\bm{C}-\frac{q_{\mathsf{in}}+q_{\mathsf{out}}}{2}\bm{1}_{n_{u}}\bm{1}_{n_{v}}^{\top}.\label{eq:A-defn-biclustering}
\end{align}
Recognizing that
\begin{equation}
\bm{A}^{\star}:=\mathbb{E}[\bm{A}]=\frac{q_{\mathsf{in}}-q_{\mathsf{out}}}{2}\left[\begin{array}{cc}
\bm{1}_{n_{u}/2}\bm{1}_{n_{v}/2}^{\top}, & -\bm{1}_{n_{u}/2}\bm{1}_{n_{v}/2}^{\top}\\
-\bm{1}_{n_{u}/2}\bm{1}_{n_{v}/2}^{\top}, & \bm{1}_{n_{u}/2}\bm{1}_{n_{v}/2}^{\top}
\end{array}\right]=\frac{q_{\mathsf{in}}-q_{\mathsf{out}}}{2}\left[\begin{array}{c}
\bm{1}_{n_{u}/2}\\
-\bm{1}_{n_{u}/2}
\end{array}\right]\left[\bm{1}_{n_{v}/2}^{\top},-\bm{1}_{n_{v}/2}^{\top}\right],\label{eq:mean-A-biclustering}
\end{equation}
we see that the leading singular vectors of $\bm{A}^{\star}$ reveals the
community memberships of all nodes. Motivated by this observation,
our algorithm for recovering the community memberships in $\mathcal{U}$
proceeds as follows:
\begin{algorithm}
\caption{The spectral method for BSBM}
\label{alg:spectral-bsbm}
\begin{algorithmic}[1]
\STATE{{\bf Input:} observed bi-adjacency matrix $\bm{C}$, edge probabilities $q_{\mathsf{in}}, q_{\mathsf{out}}$.}
\STATE{Employ $\bm{A}$ (cf.~\eqref{eq:A-defn-biclustering}) as the input of Algorithm~\ref{alg:spectral}, and let $\bm{u}=[u_{i}]\in\mathbb{R}^{n_{u}}$
be the output returned by Algorithm~\ref{alg:spectral} (which serves
as the estimate of the leading left singular subspace of $\bm{A}^{\star}$.}
\STATE{{\bf Output:} for any $i\in\mathcal{U}$, we claim that $i$ belongs
to the first community if $u_{i}>0$, and the second community otherwise.}
\end{algorithmic}
\end{algorithm}

\paragraph{Theoretical guarantees and implications.} We are now ready to invoke our general theory to
demonstrate the effectiveness of the above algorithm, as asserted
by the following result.

\begin{corollary}[Bipartite stochastic block model]\label{cor:BSBM}Consider
the above bipartite stochastic block model. There exists some universal
constant $c_{0}>0$ such that if
\begin{equation}
\frac{(q_{\mathsf{in}}-q_{\mathsf{out}})^{2}}{q_{\mathsf{in}}}\geq c_{0}\max\left\{ \frac{\log n}{\sqrt{n_{u}n_{v}}},\frac{\log n}{n_{v}}\right\} ,\label{eq:BSBM-assump}
\end{equation}
then Algorithm \ref{alg:spectral-bsbm} achieves exact community recovery
of $\mathcal{U}$ with probability exceeding $1-O(n^{-10})$.

\end{corollary}

We then take a moment to discuss the implications of Corollary~\ref{cor:BSBM}.
For simplicity of presentation, we shall focus on the scenario with
$q_{\mathsf{in}}\asymp q_{\mathsf{out}}=o\left(1\right)$ and $n_{u}\leq n_{v}$.
\begin{itemize}
\item \emph{Exact recovery via the spectral method alone.} Consider the
following sparse regime, where 
\[
q_{\mathsf{in}}=\frac{a\log n}{\sqrt{n_{u}n_{v}}}\qquad\text{and}\qquad q_{\mathsf{out}}=\frac{b\log n}{\sqrt{n_{u}n_{v}}}
\]
for some absolute positive constants $a\geq b$. Corollary~\ref{cor:BSBM}
demonstrates that we can achieve exact recovery when $\frac{\left(a-b\right)^{2}}{a}\gtrsim1$.
This improves upon prior results presented in \cite{florescu2016spectral}.
More specifically, the results in \cite{florescu2016spectral} only
guaranteed \emph{almost} exact recovery of community memberships (namely,
obtaining correct community memberships for a fraction $1-o(1)$ of
the nodes). In comparison, our results assert that the spectral estimates
alone are sufficient to reveal exact community memberships for all
nodes in $\mathcal{U}$; there is no need to invoke further refinement
procedures to clean up the remaining errors.

\item \emph{Near optimality. }In the balanced case where $n_{u} \asymp n_{v}$, the
condition $\frac{\left(a-b\right)^{2}}{a}\gtrsim1$ above is known
to be information-theoretically optimal up to a constant factor. In
the  unbalanced case with $n_v \geq n_u$, prior work has identified a sharp threshold for \emph{detection} --- the problem of 
recovering  a fraction $1/2+\epsilon$ of the community memberships for an arbitrarily small fixed constant $\epsilon>0$. Specifically, such results reveal a fundamental lower limit that requires $\frac{(q_{\mathsf{in}}-q_{\mathsf{out}})^{2}}{q_{\mathsf{in}}}\gtrsim\frac{1}{\sqrt{n_{u}n_{v}}}$ \cite{feldman2015subsampled,florescu2016spectral}, 
thus implying the information-theoretic optimality of the spectral method (up to a logarithmic factor).%
\end{itemize}

\subsection{Numerical experiments}
\label{subsec:Numerical-experiments}

To confirm the applicability of our algorithm and the theoretical
findings, we conduct a series of numerical experiments. All results
reported in this subsection are averaged over 100 independent Monte
Carlo trials. For the sake of comparisons, we also report the numerical
performance of the vanilla spectral method (namely, returning the
$r$-dimensional principal column subspace of $\bm{A}$ directly without
proper diagonal deletion).

\paragraph{Subspace estimation for random low-rank data matrices.} 

We start with subspace estimation for a randomly generated matrix
$\bm{A}^{\star}$. Specifically, generate $\bm{A}^{\star}=\bm{Z}_{1}\bm{Z}_{2}^{\top}$,
where $\bm{Z}_{1}\in\mathbb{R}^{d_{1}\times r},\bm{Z}_{2}\in\mathbb{R}^{d_{2}\times r}$
consist of i.i.d.~standard Gaussian entries. The noise matrix contains
i.i.d.~Gaussian entries, namely, $N_{i,j}\overset{\mathrm{i.i.d.}}{\sim}\mathcal{N}\left(0,\sigma^{2}\right)$
for each $\left(i,j\right)\in[d_{1}]\times[d_{2}]$. Figures~\ref{fig:subspace-est} and \ref{fig:subspace-est-abs} plot respectively the numerical
estimation errors of the estimate $\bm{U}$ vs.~the sampling rate
$p$, the column dimension $d_{2}$, and the standard deviation $\sigma$
of noise.  Two types of estimation errors are reported: (1) the absolute spectral norm error $\left\Vert \bm{U}\bm{R}-\bm{U}^{\star}\right\Vert$ and the relative
spectral norm error $\left\Vert \bm{U}\bm{R}-\bm{U}^{\star}\right\Vert /\left\Vert \bm{U}^{\star}\right\Vert $;
(2) the absolute $\ell_{2,\infty}$ norm error $\left\Vert \bm{U}\bm{R}-\bm{U}^{\star}\right\Vert_{2,\infty}$  and the relative $\ell_{2,\infty}$ norm error $\left\Vert \bm{U}\bm{R}-\bm{U}^{\star}\right\Vert _{2,\infty}/\left\Vert \bm{U}^{\star}\right\Vert _{2,\infty}$,
where $\bm{R}:=\arg\min_{\bm{Q}\in\mathcal{O}_{r\times r}}\|\bm{U}\bm{Q}-\bm{U}^{\star}\|_{\mathrm{F}}$.
As can be seen from the plots, Algorithm~\ref{alg:spectral} yields
reasonably good estimates in terms of both the spectral norm and the
$\ell_{2,\infty}$ norm, outperforming the vanilla spectral method
in all experiments. 

\begin{figure}[t]
\centering
\begin{tabular}{ccc}
\includegraphics[width=0.31\textwidth]{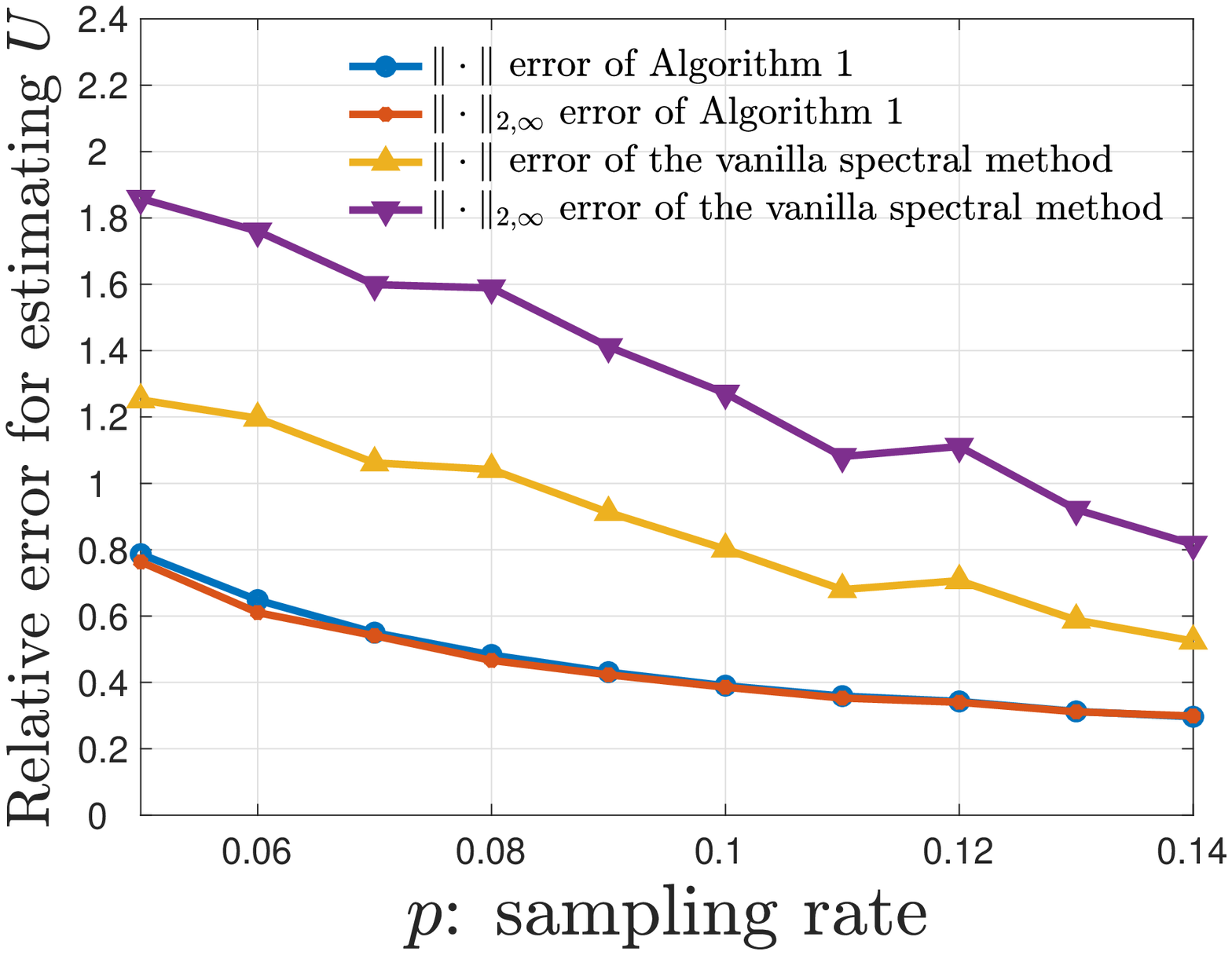} & \includegraphics[width=0.31\textwidth]{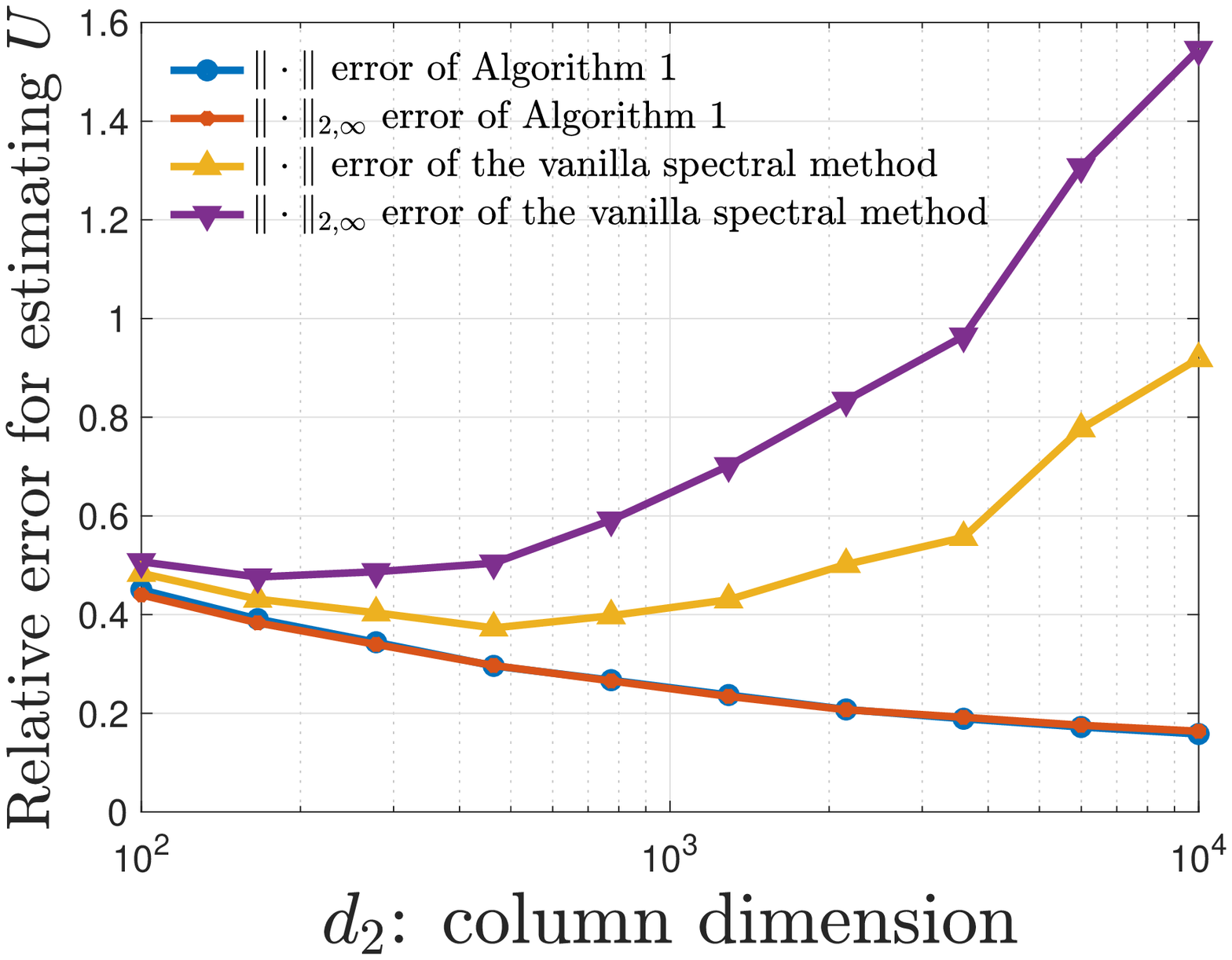} & \includegraphics[width=0.31\textwidth]{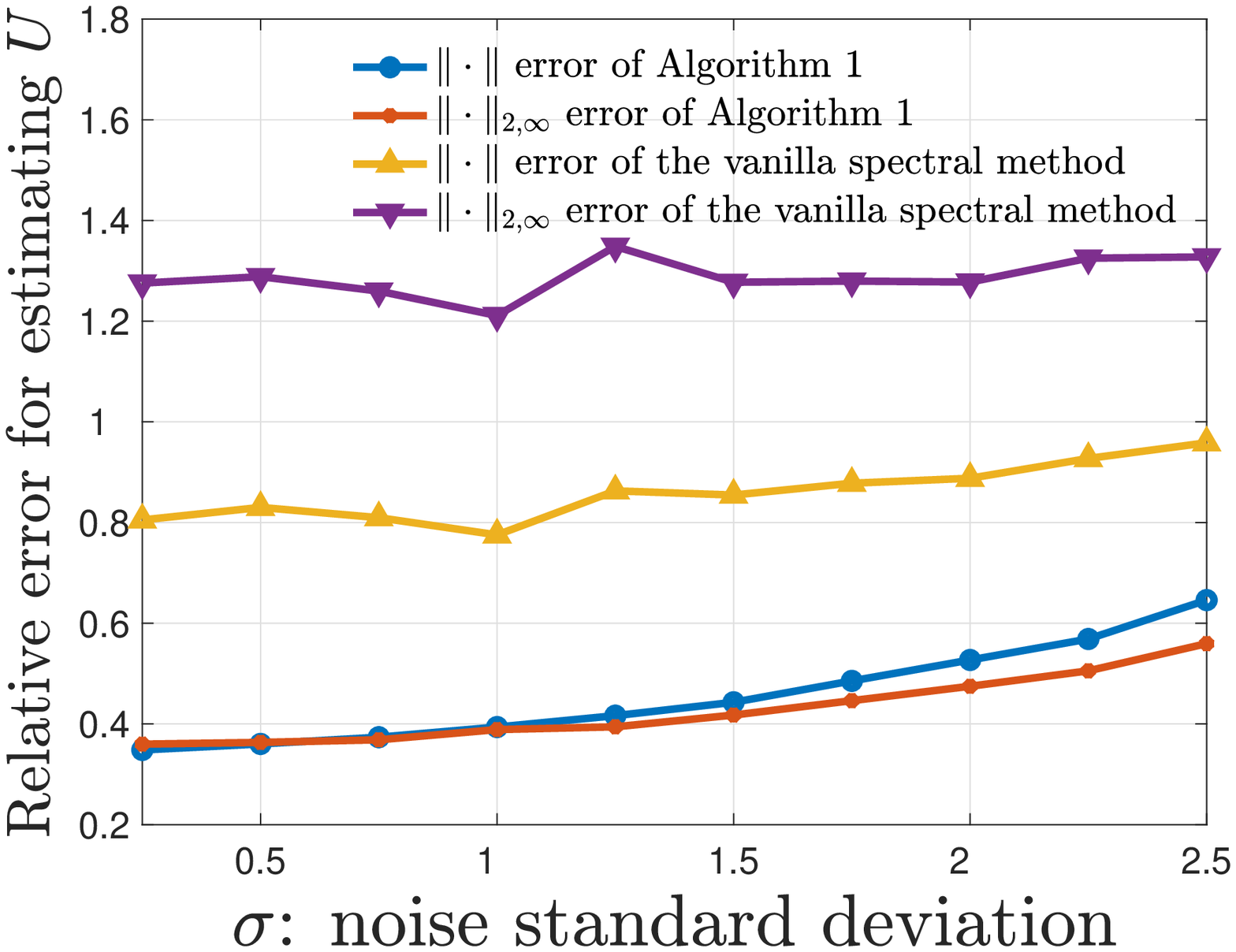}\tabularnewline
(a) & (b) & (c)\tabularnewline
\end{tabular}
\caption{Relative estimation errors of the subspace estimate $\bm{U}$ for
both Algorithm~\ref{alg:spectral} and the vanilla spectral method.
The results are reported for (a) relative error vs.~sampling rate
$p$ (where $d_{1}=100$, $d_{2}=1000$, $r=4$, $\sigma=1$), (b)
relative error vs.~column dimension $d_{2}$ (where $d_{1}=100$,
$r=4$, $\sigma=1$, $p=\frac{2r\log\left(d_{1}+d_{2}\right)}{\sqrt{d_{1}d_{2}}}$),
and (c) relative error vs.~noise standard deviation $\sigma$ (where
$d_{1}=100$, $d_{2}=1000$, $r=4$, $p=0.1$). \label{fig:subspace-est}}
\end{figure}

\begin{figure}[t]
\centering
\begin{tabular}{ccc}
\includegraphics[width=0.31\textwidth]{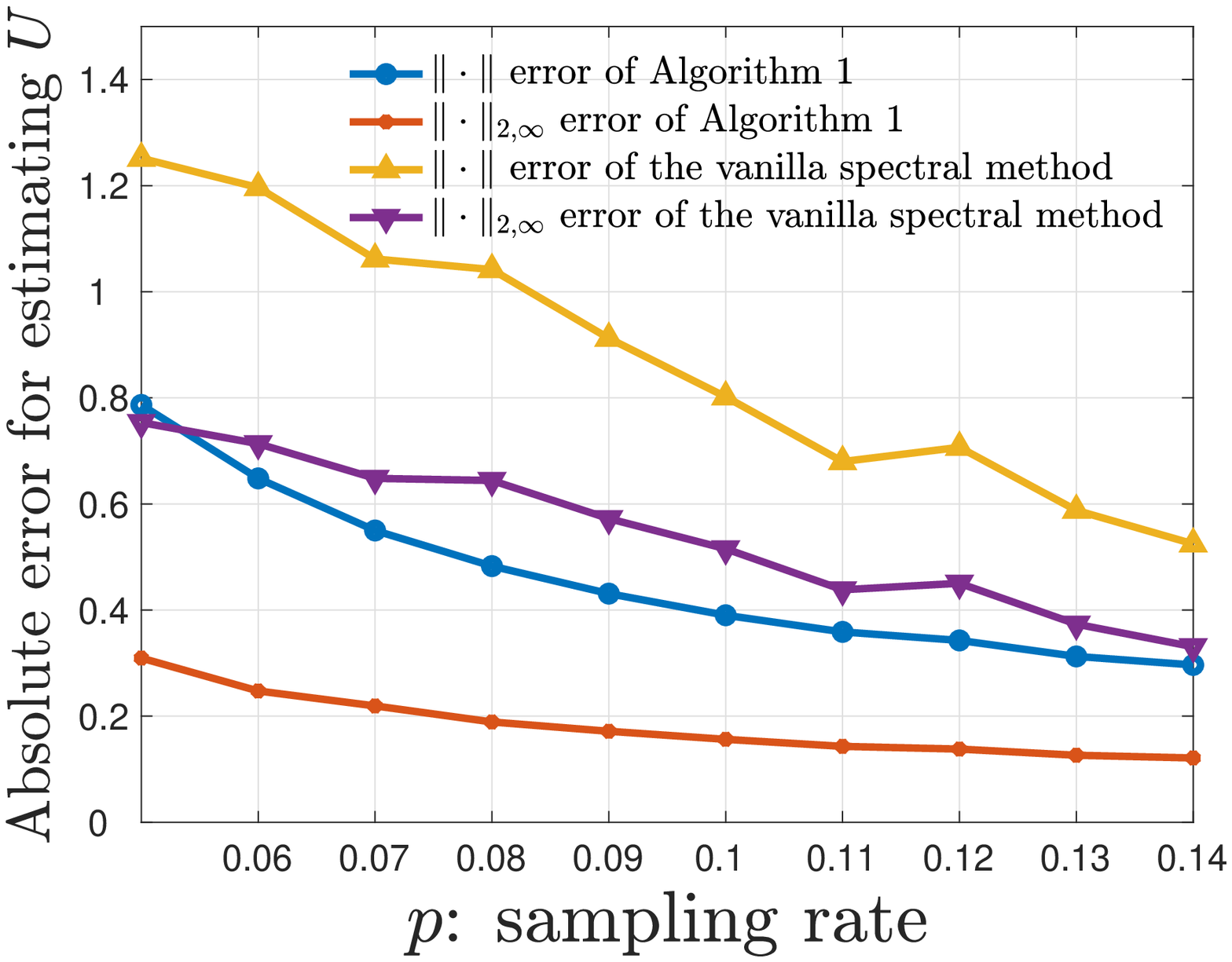} & \includegraphics[width=0.31\textwidth]{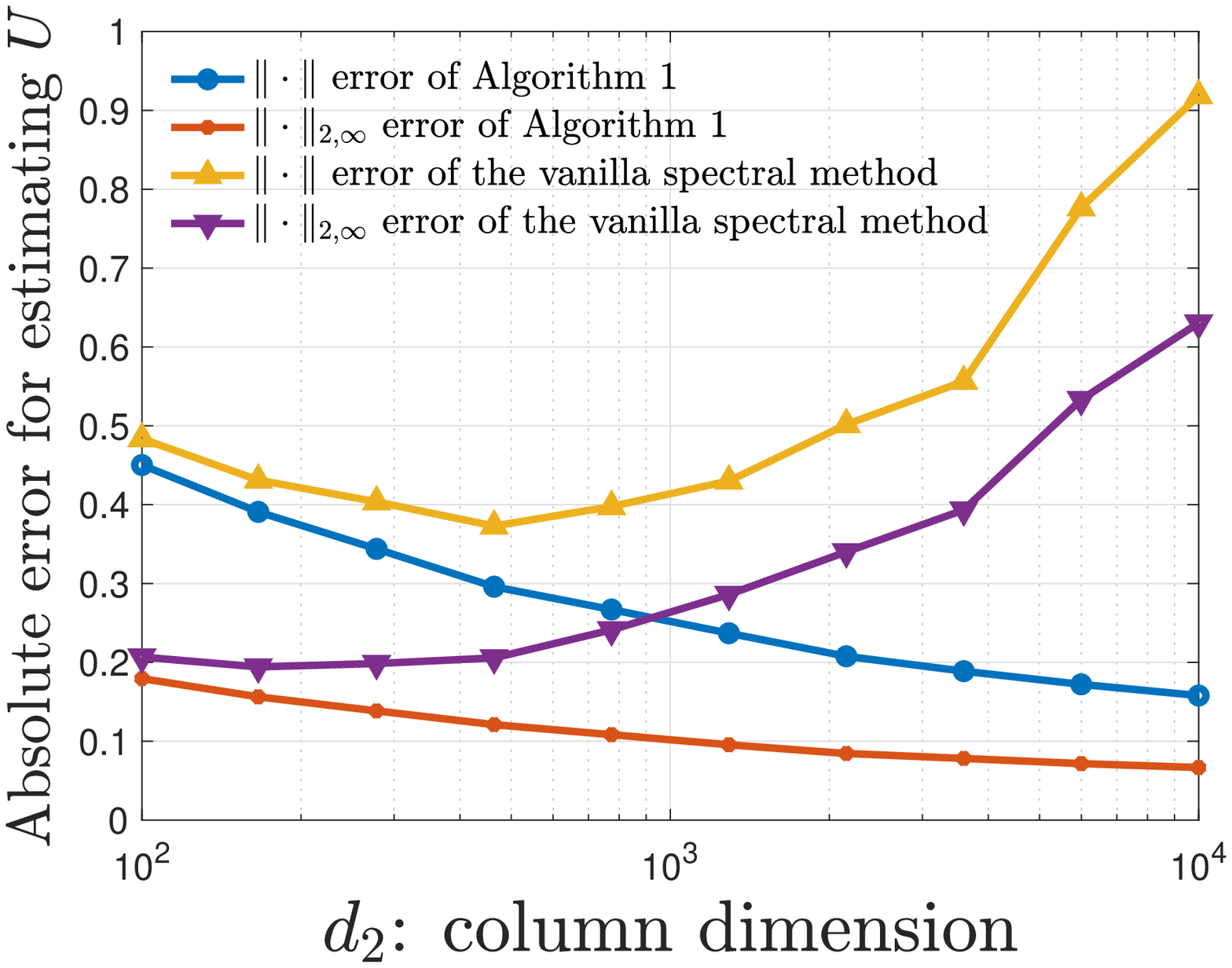} & \includegraphics[width=0.31\textwidth]{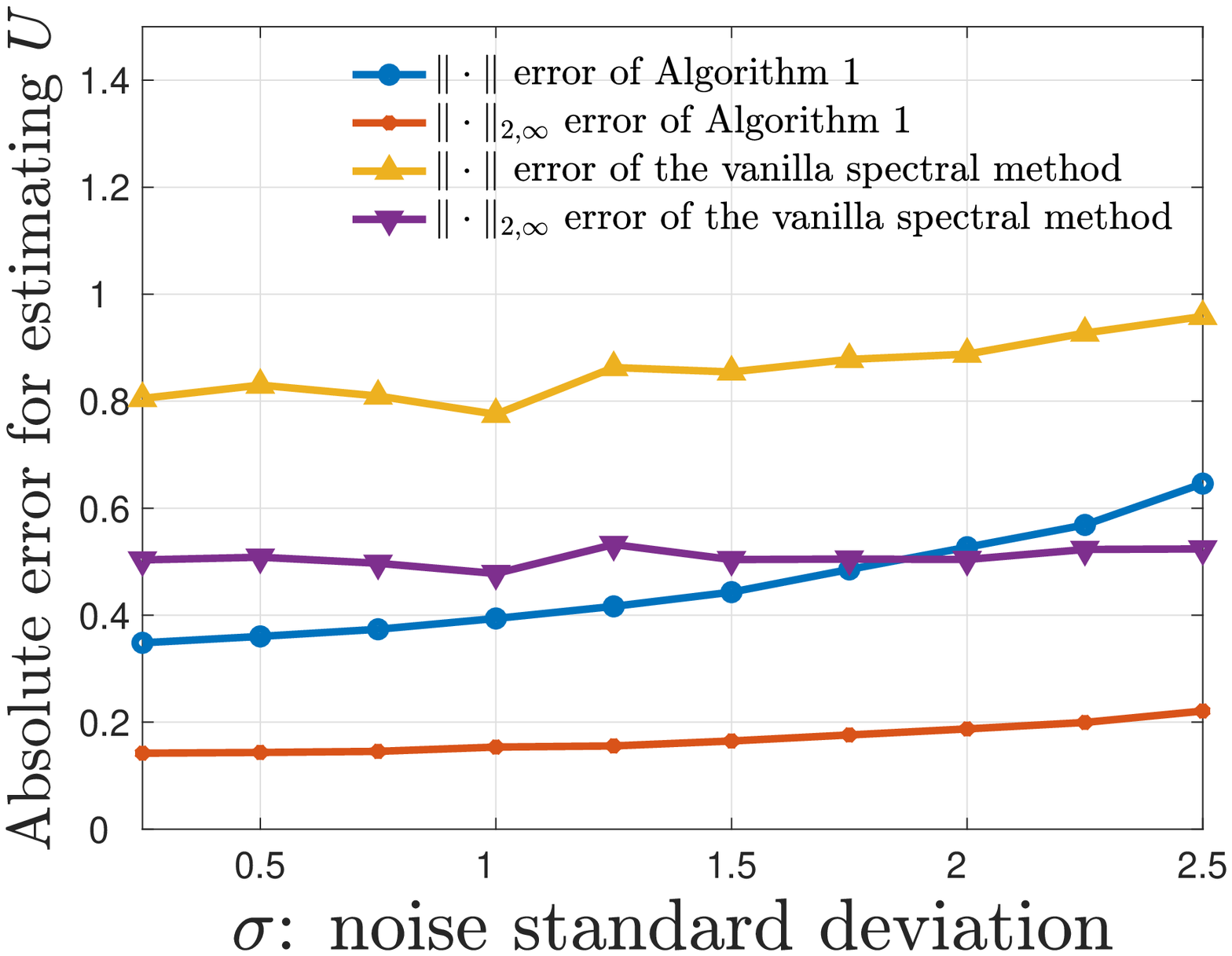}\tabularnewline
(a) & (b) & (c)\tabularnewline
\end{tabular}
\caption{Absolute estimation errors of the subspace estimate $\bm{U}$ for
both Algorithm~\ref{alg:spectral} and the vanilla spectral method.
The results are reported for (a) absolute error vs.~sampling rate
$p$ (where $d_{1}=100$, $d_{2}=1000$, $r=4$, $\sigma=1$), (b)
absolute error vs.~column dimension $d_{2}$ (where $d_{1}=100$,
$r=4$, $\sigma=1$, $p=\frac{2r\log\left(d_{1}+d_{2}\right)}{\sqrt{d_{1}d_{2}}}$),
and (c) absolute error vs.~noise standard deviation $\sigma$ (where
$d_{1}=100$, $d_{2}=1000$, $r=4$, $p=0.1$). \label{fig:subspace-est-abs}}
\end{figure}

\paragraph{Tensor completion from noise data.} Next, we consider
numerically the problem of tensor completion from noisy observations
of its entries. Recall the notations in Section~\ref{subsec:tensor-completion}.
We generate $\bm{W}^{\star}\in\mathbb{R}^{d\times r}$ with i.i.d.~standard
Gaussian entries, and generate $N_{i,j,k}\overset{\mathrm{i.i.d.}}{\sim}\mathcal{N}\left(0,\sigma^{2}\right)$
independently for each $\left(i,j,k\right)\in\left[d\right]^{3}$.
Figure~\ref{fig:tensor}(a) and Figure \ref{fig:tensor}(b) illustrate the
relative estimation errors of the subspace estimate $\bm{U}$ vs.~the
sampling rate $p$ and noise standard deviation $\sigma$, respectively.
Encouragingly, Figure~\ref{fig:tensor} shows that Algorithm~\ref{alg:spectral-tc}
accurately recovers the subspace spanned by the tensor factors of interest (with
respect to both the spectral norm and the $\ell_{2,\infty}$ norm); in particular,
it is capable of producing faithful subspace estimates even when the
vanilla spectral method fails.

\begin{figure}[t]
\centering

\begin{tabular}{cc}
\includegraphics[width=0.33\textwidth]{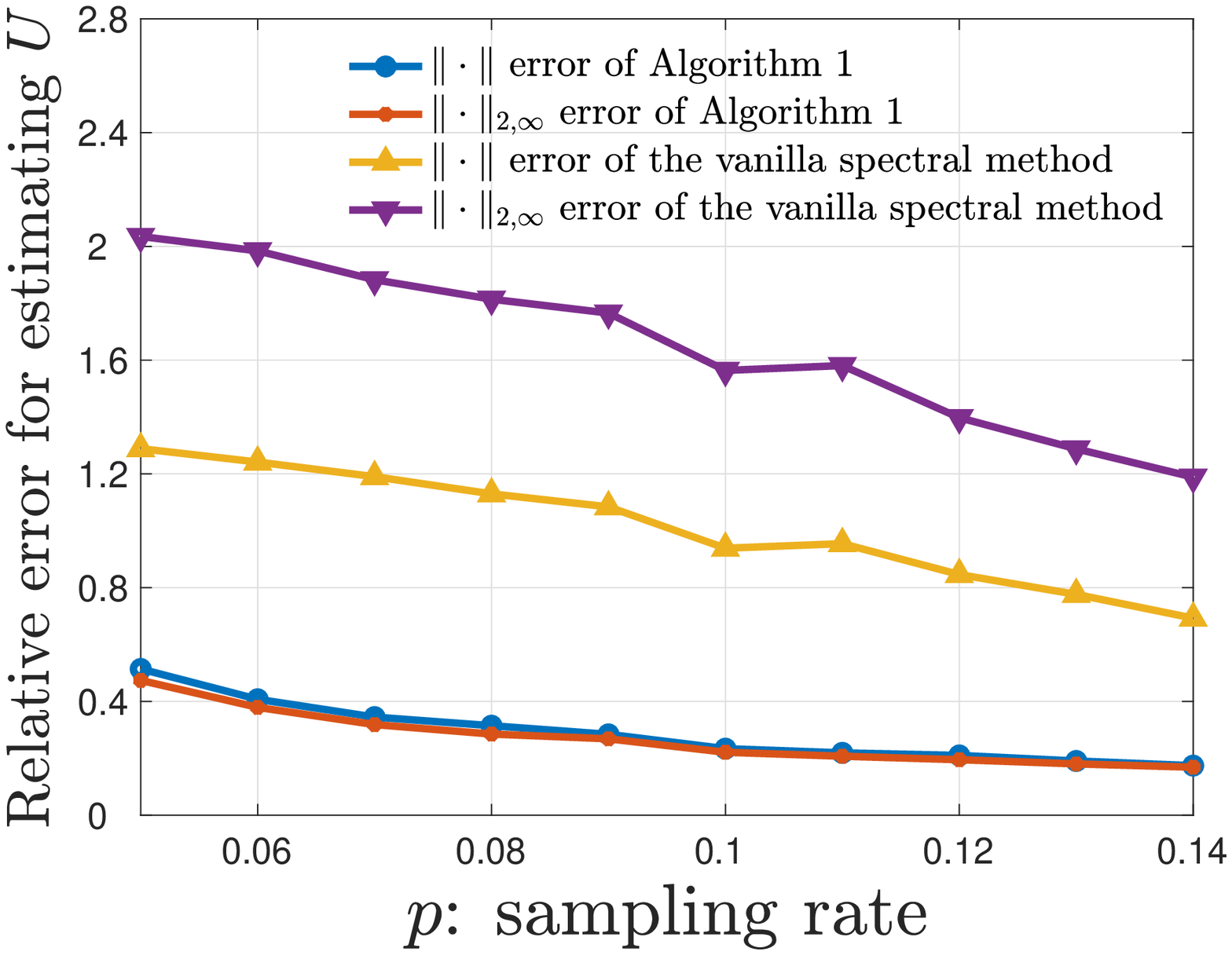} & \includegraphics[width=0.33\textwidth]{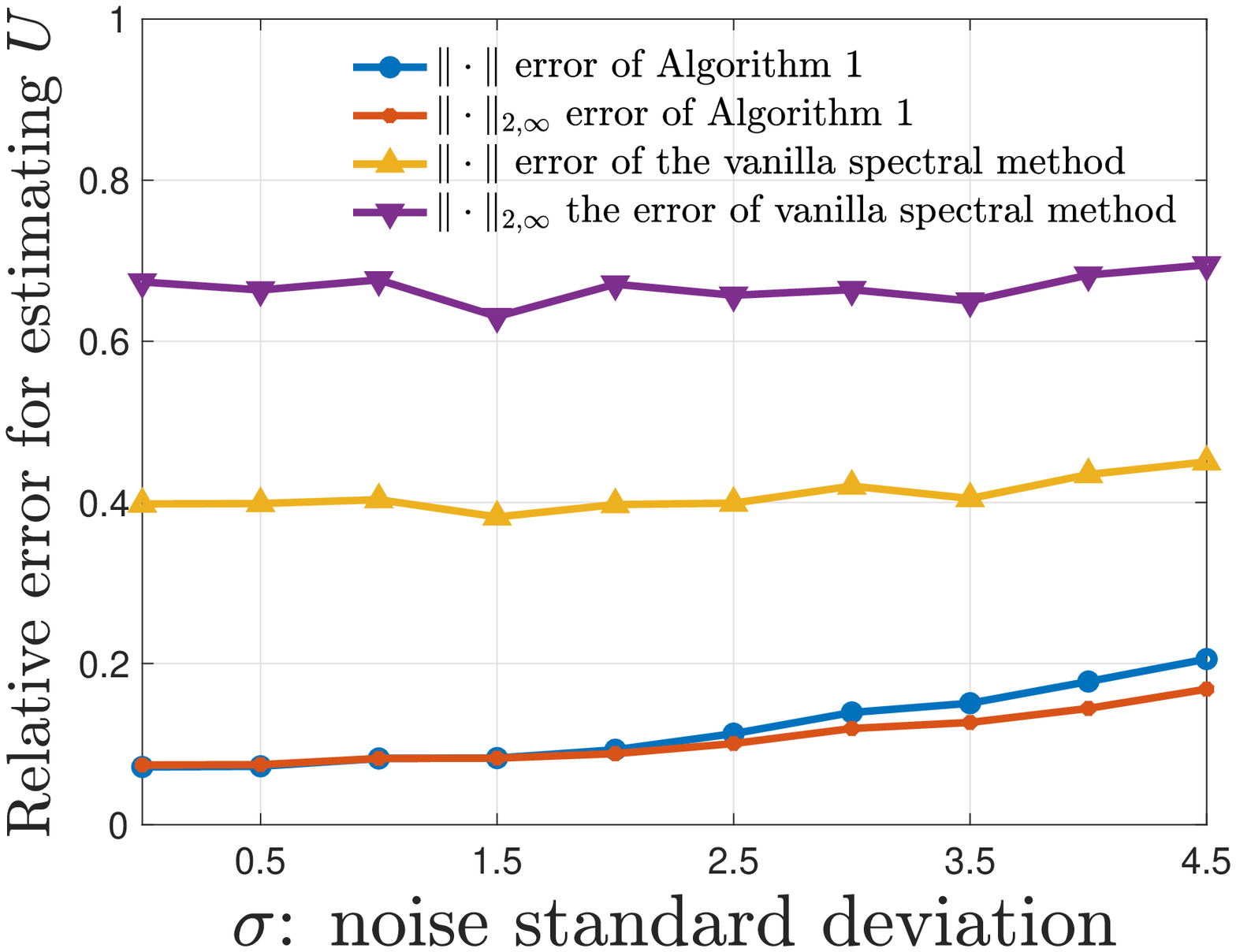}\tabularnewline
(a) & (b)\tabularnewline
\end{tabular}

\caption{Relative estimation errors of the subspace $\bm{U}$ spanned by tensor components in tensor completion. The results are
plotted for (a) relative error vs.~sampling rate $p$ (where $d=100$,
$r=4$, $\sigma=2$), and (b) relative error vs.~noise standard deviation
$\sigma$ (where $d=100$, $r=4$, $p=0.1$).\label{fig:tensor}}
\end{figure}

\paragraph{Covariance estimation with missing data.} The next series
of experiments is concerned with covariance estimation with missing
data. Recall the notations in Section~\ref{subsec:Covariance-estimation}.
We draw $\bm{x}_{i}\overset{}{\sim}\mathcal{N}\left(\bm{{0}},\bm{\Sigma}^{\star}\right)$
independently with $\bm{\Sigma}^{\star}=\bm{U}^{\star}\bm{U}^{\star\top}$,
where $\bm{U}^{\star}\in\mathbb{R}^{d\times r}$ is a i.i.d.~standard
Gaussian random matrix in $\mathbb{R}^{d\times r}$, and $\bm{\varepsilon}_{i}\overset{\mathrm{i.i.d.}}{\sim}\mathcal{N}\left(\bm{{0}},\sigma^{2}\bm{I}_{d}\right)$
for each $1\leq i\leq n$. We first consider the estimation error
of the subspace. The numerical estimation errors of the estimate $\bm{U}$
vs.~the sampling rate $p$, the sample size $n$ and the noise standard
deviation $\sigma$ are plotted in Figure \ref{fig:cov-est}(a) -- Figure \ref{fig:cov-est}(c),
respectively. We then turn to the estimation accuracy of the covariance
matrix. The numerical estimation errors of the estimate $\bm{S}$
of Algorithm~\ref{alg:spectral-ce} and the vanilla spectral method vs.~the sampling rate $p$, the
sample size $n$ and the noise standard deviation $\sigma$ are plotted
in Figure \ref{fig:cov-est-matrix-alg1},
respectively.  Similar to previous experiments,
Algorithm~\ref{alg:spectral-ce} produces reliable estimates both in
terms of the spectral norm, the $\ell_{2,\infty}$ norm and the $\ell_{\infty}$
norm accuracy.

\begin{figure}[t]
\centering

\begin{tabular}{ccc}
\includegraphics[width=0.31\textwidth]{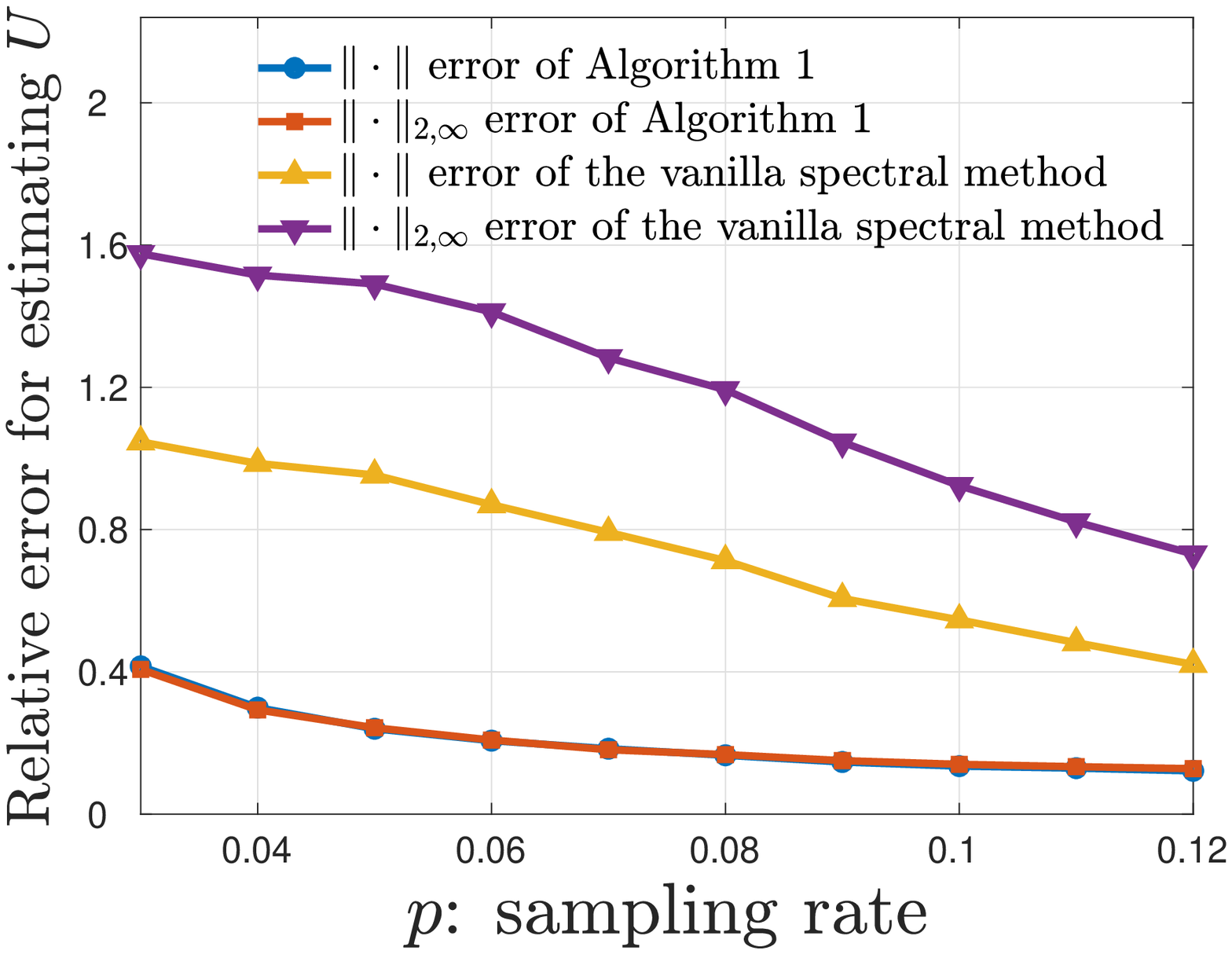} & \includegraphics[width=0.31\textwidth]{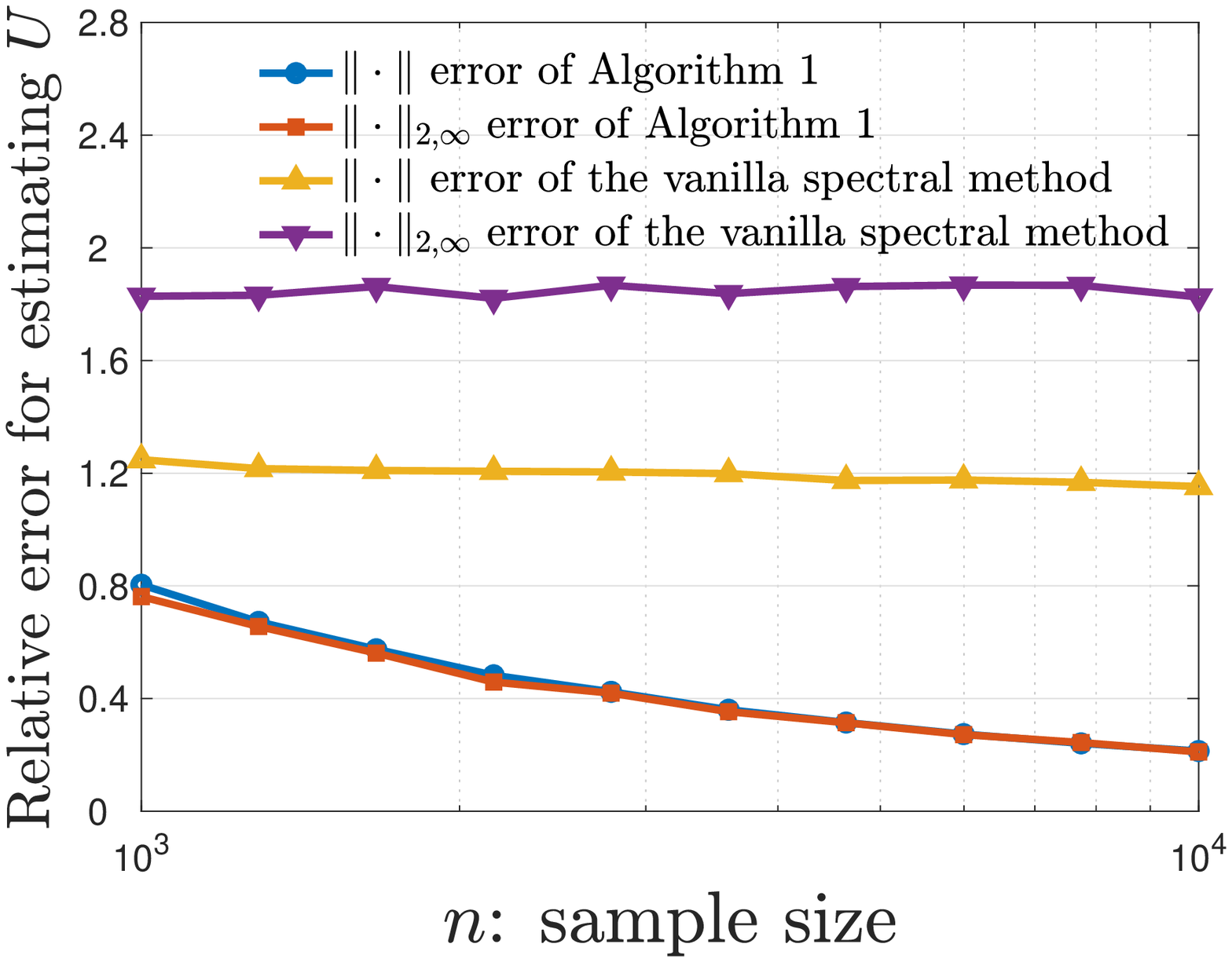} & \includegraphics[width=0.31\textwidth]{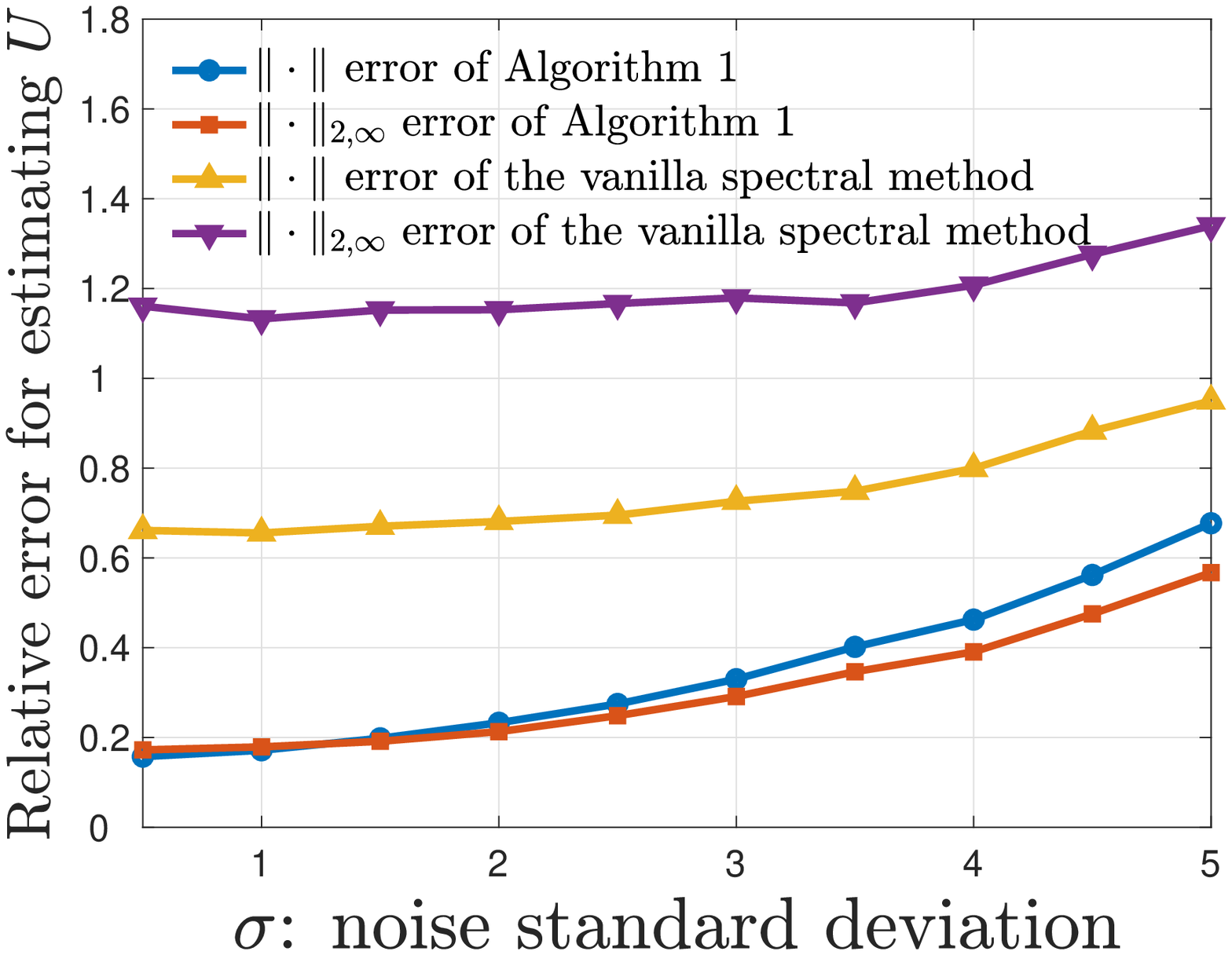}\tabularnewline
(a) & (b) & (c)\tabularnewline
\end{tabular}

\caption{Relative error of the estimate $\bm{U}$ for covariance estimation
with missing data. The results are shown for (a) relative error vs.~sampling
rate $p$ (where $d=100$, $n=5000$, $r=4$, $\sigma=1$), (b) relative
error vs.~sample size $n$ (where $d=100$, $r=4$, $\sigma=1$,
$p=0.05$), and (c) relative error vs.~noise standard deviation $\sigma$
(where $d=100$, $n=5000$, $r=4$, $p=0.1$). \label{fig:cov-est}}
\end{figure}

\begin{figure}[t]
\centering

\begin{tabular}{ccc}
\includegraphics[width=0.31\textwidth]{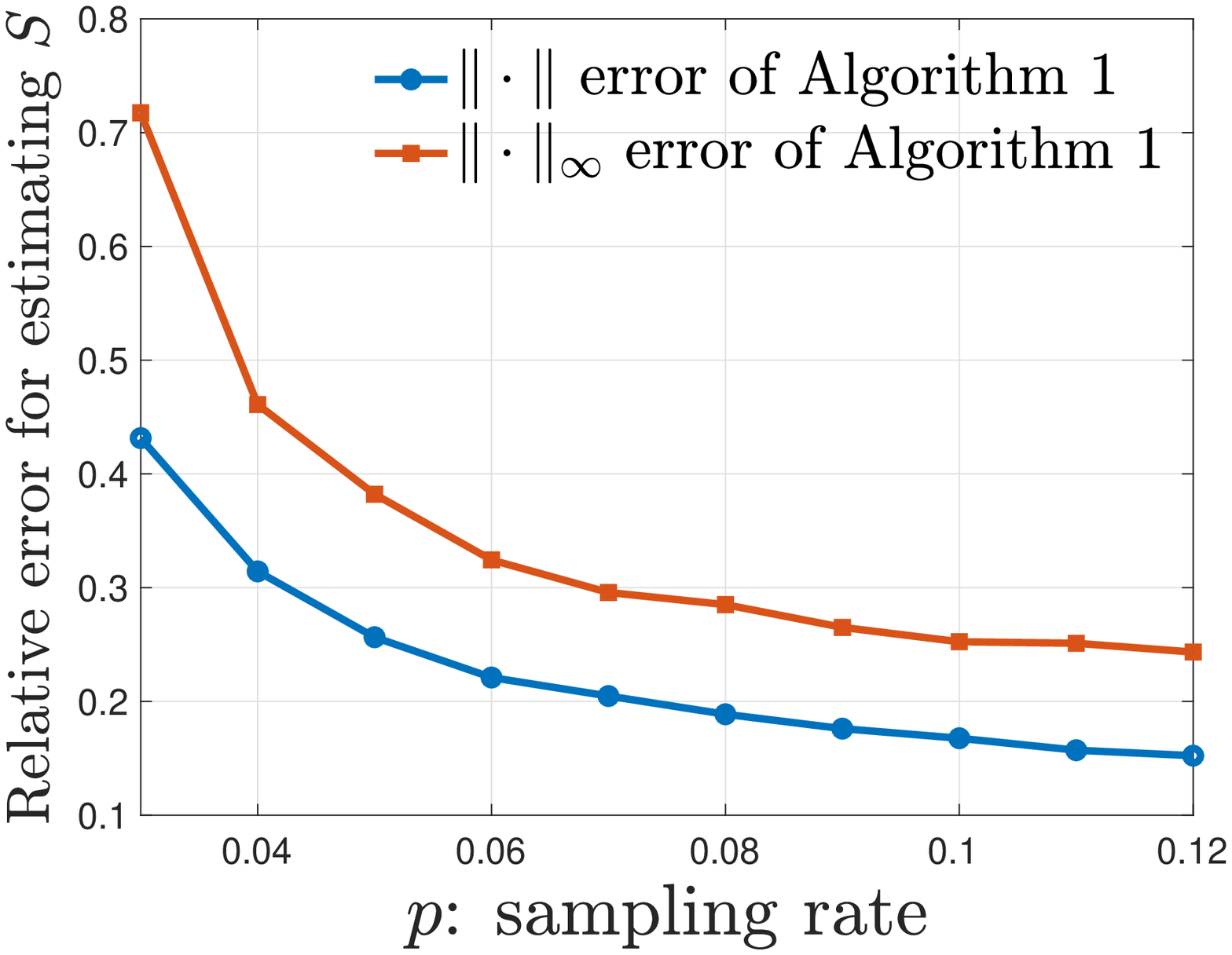} & \includegraphics[width=0.31\textwidth]{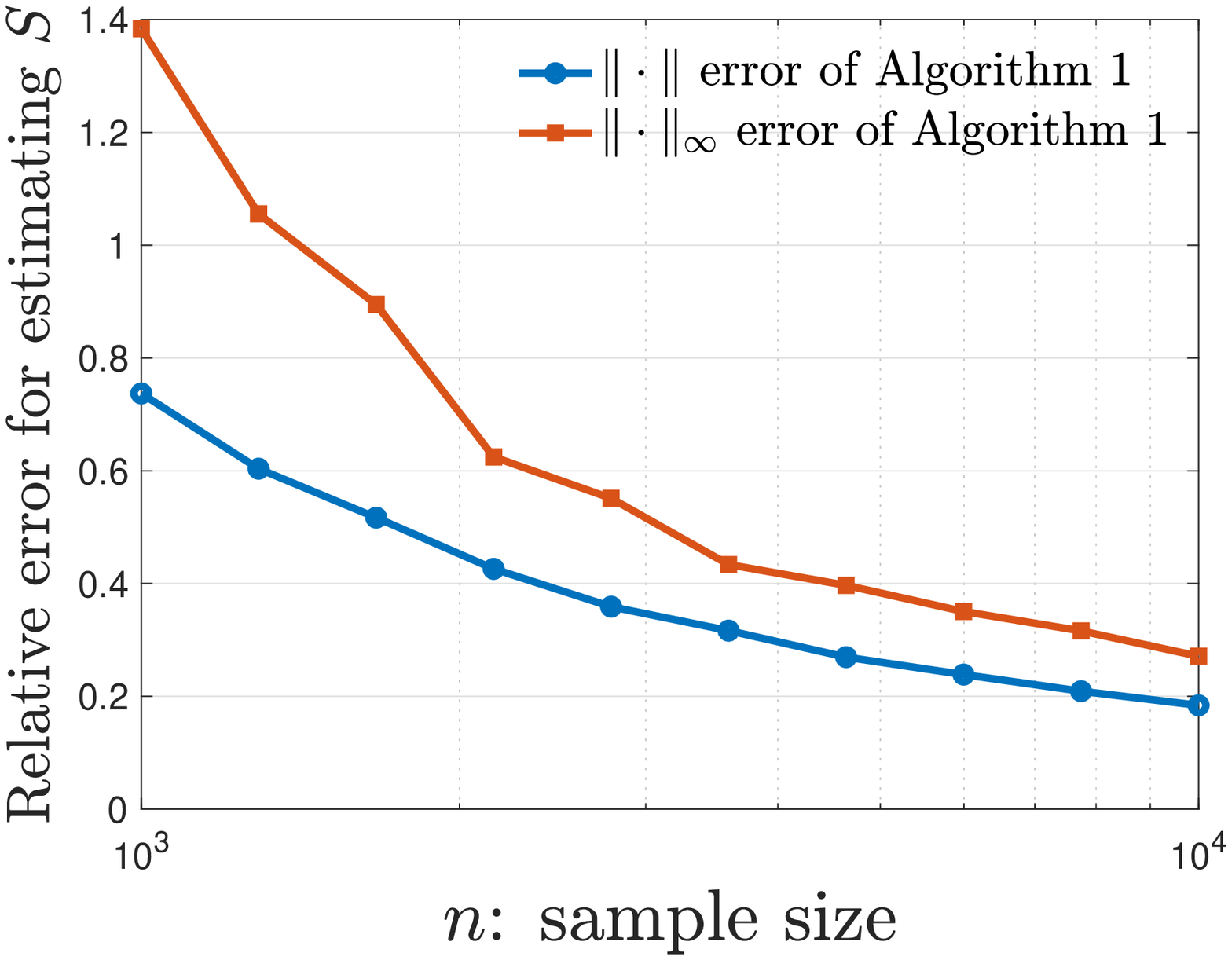} & \includegraphics[width=0.31\textwidth]{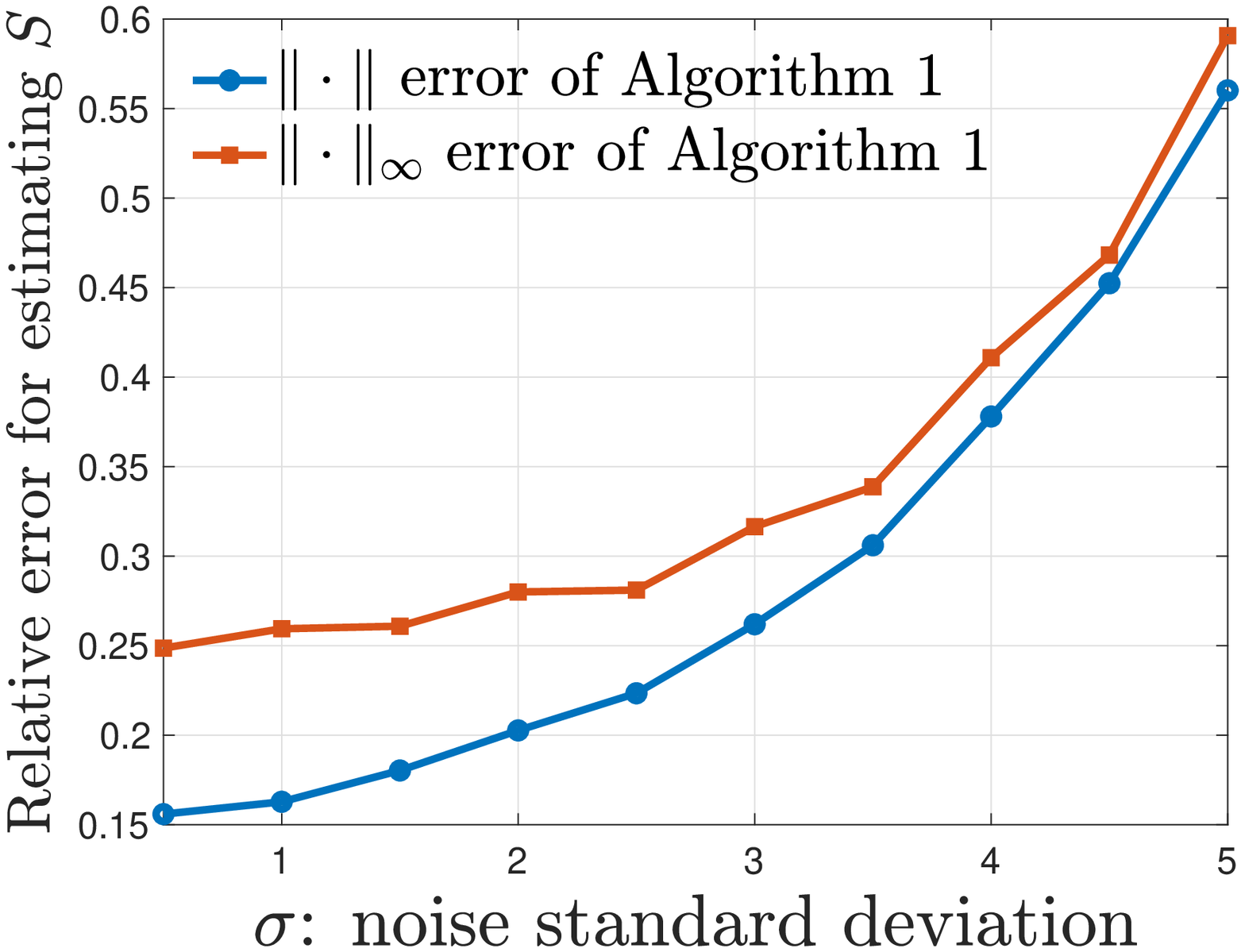}\tabularnewline
(a) & (b) & (c)\tabularnewline
\includegraphics[width=0.31\textwidth]{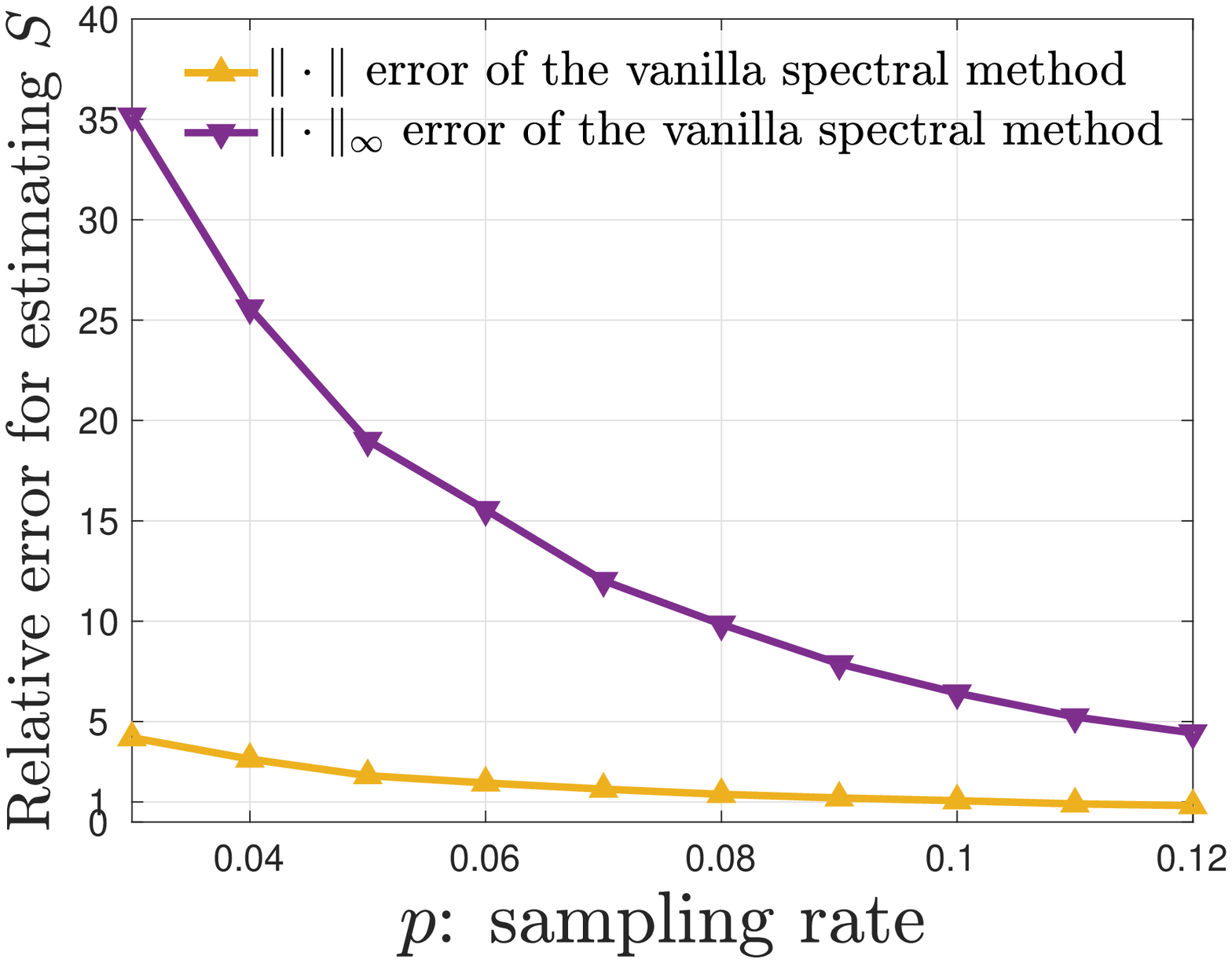} & \includegraphics[width=0.31\textwidth]{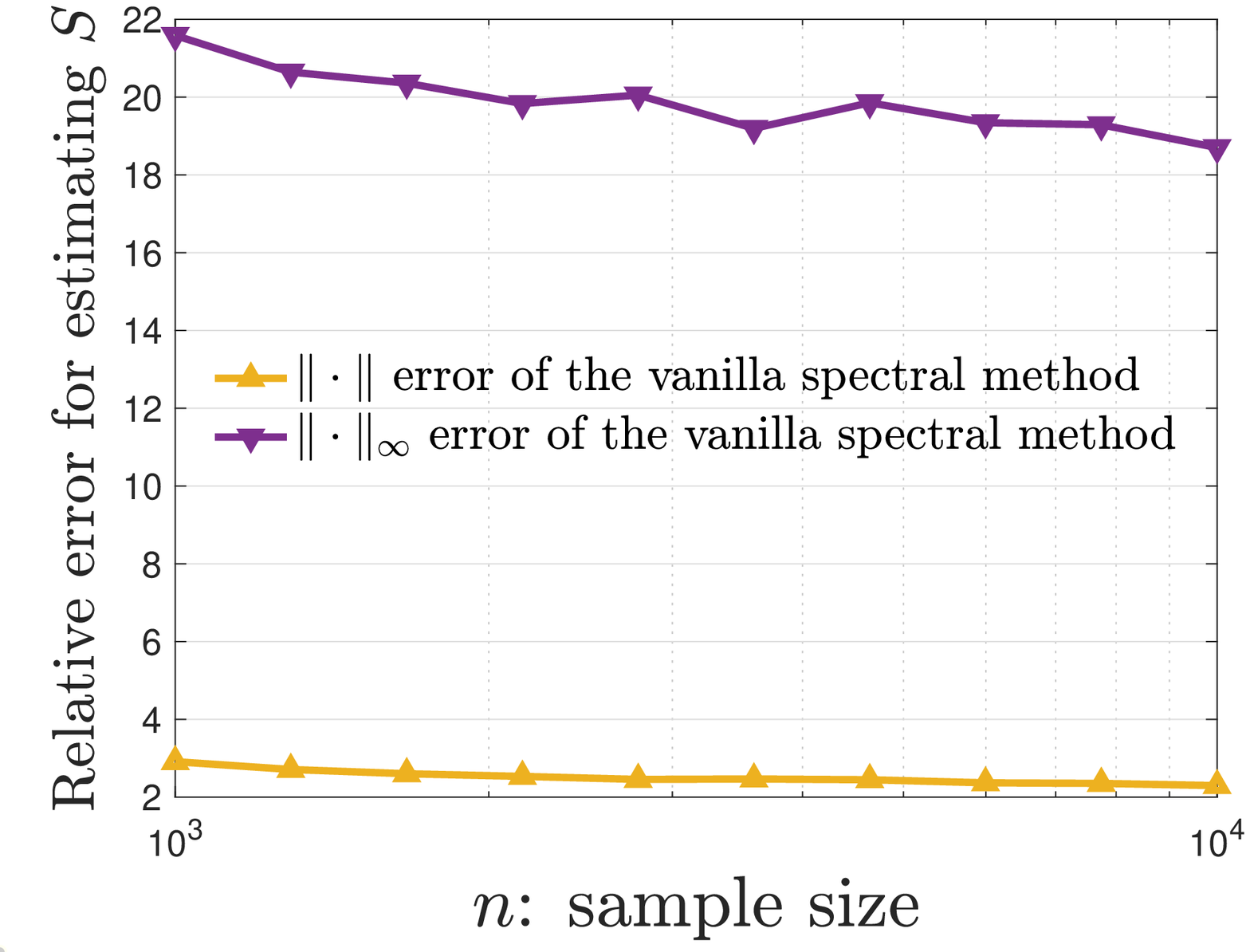} & \includegraphics[width=0.31\textwidth]{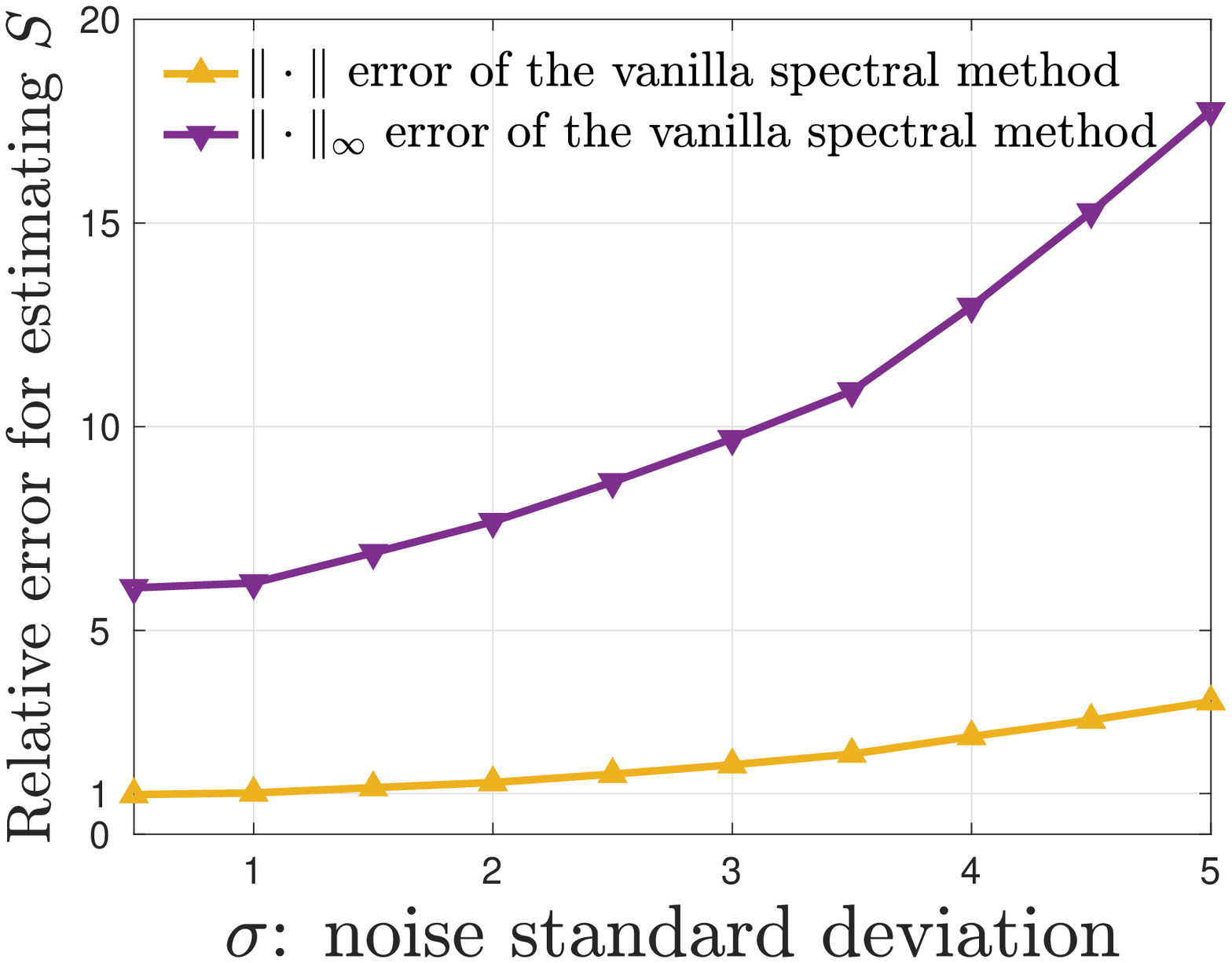}\tabularnewline
(d) & (e) & (f)\tabularnewline
\end{tabular}

\caption{Relative error of the estimate $\bm{S}$ of Algorithm~\ref{alg:spectral-ce} (top row) and the vanilla spectral method (bottom row)
for covariance estimation with missing data. The results are shown
for (a,d) relative error vs.~sampling rate $p$ (where $d=100$, $n=5000$,
$r=4$, $\sigma=1$), (b,e) relative error vs.~sample size $n$ (where
$d=100$, $r=4$, $\sigma=1$, $p=0.05$), and (c,f) relative error
vs.~noise standard deviation $\sigma$ (where $d=100$, $n=5000$,
$r=4$, $p=0.1$). It is worthnoting the different scales of the $y$-axis when plotting the errors of the two algorithms. \label{fig:cov-est-matrix-alg1}}
\end{figure}

%
%

\paragraph{Community recovery in bipartite stochastic block model.}
Finally, we conduct numerical experiments for community recovery in
bipartite stochastic block models. The parameters are chosen to be
$q_{\mathsf{in}}=\frac{a\log\left(n_{u}+n_{v}\right)}{\sqrt{n_{u}n_{v}}}$
and $q_{\mathsf{out}}=\frac{b\log\left(n_{u}+n_{v}\right)}{\sqrt{n_{u}n_{v}}}$
for some constants $a>b>0$. Figure~\ref{fig:BSBM}(a) reveals a
phase transition phenomenon concerned with exact community recovery.
As can be seen, Algorithm~\ref{alg:spectral-bsbm} always succeeds in
achieving exact recovery once $a$ --- or equivalently $q_{\mathsf{in}}$
--- exceeds a certain threshold, which outperforms the vanilla spectral
method. In Figure~\ref{fig:BSBM}(b), we vary the number $n_{v}$
nodes in $\mathcal{V}$ and plot the empirical success rates for exact
recovery. The advantage of Algorithm~\ref{alg:spectral-bsbm} compared
to the vanilla spectral method can be clearly seen from the plot.

\begin{figure}[t]
\centering

\begin{tabular}{cc}
\includegraphics[width=0.33\textwidth]{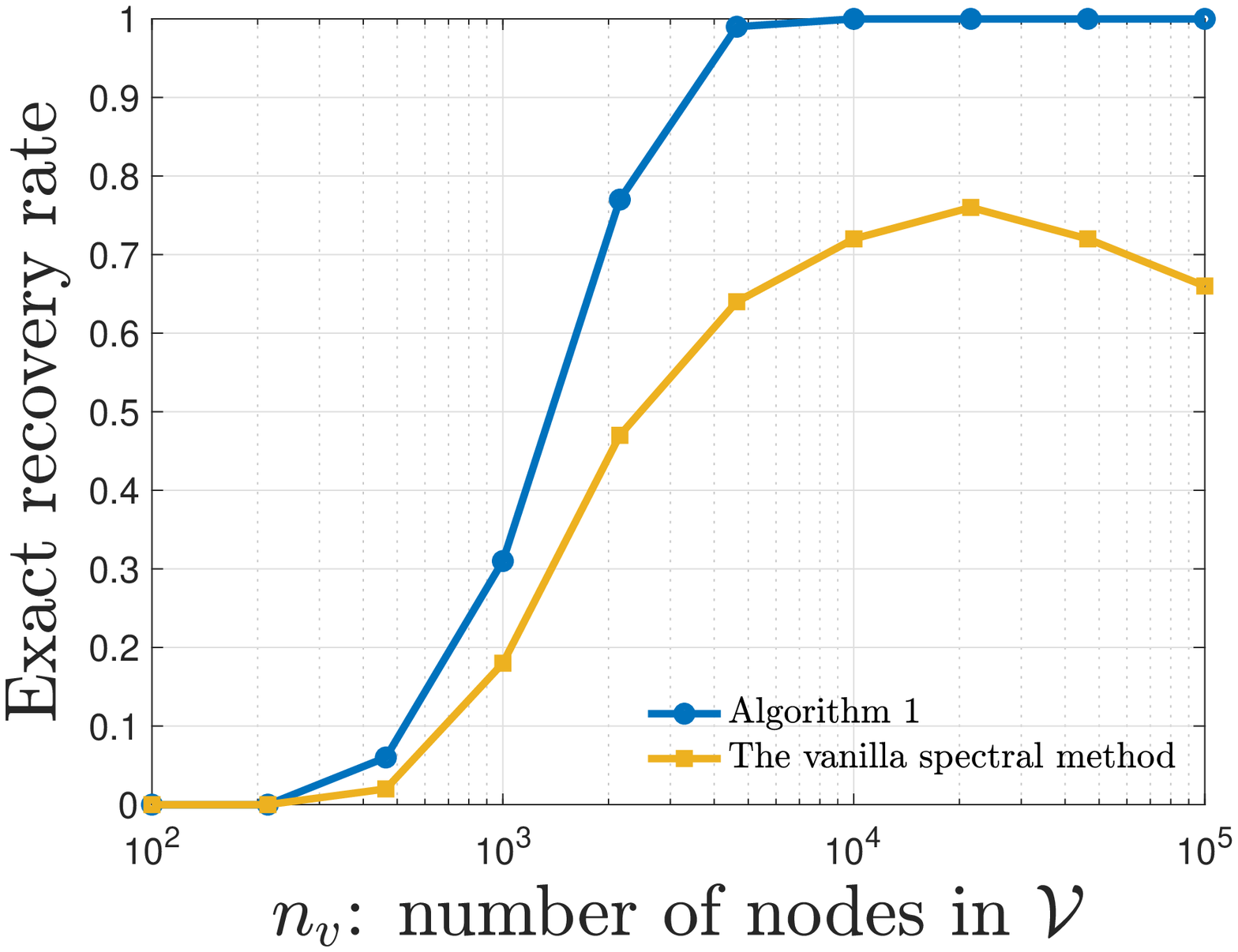} & \includegraphics[width=0.33\textwidth]{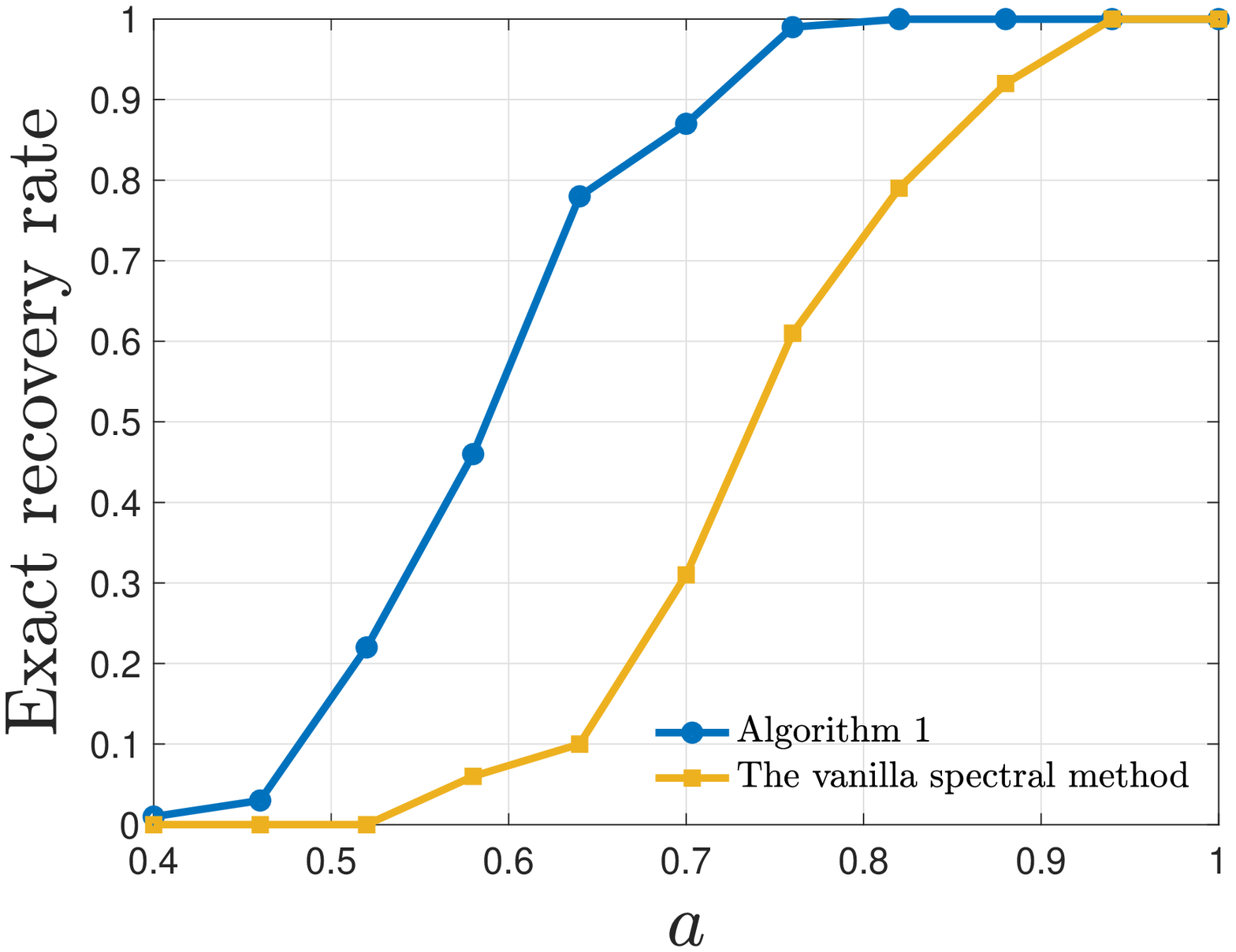}\tabularnewline
(a) & (b)\tabularnewline
\end{tabular}

\caption{Empirical success rates for exact community recovery in bipartite
stochastic block models, where $q_{\mathsf{in}}=\frac{a\log\left(n_{u}+n_{v}\right)}{\sqrt{n_{u}n_{v}}}$
and $q_{\mathsf{out}}=\frac{b\log\left(n_{u}+n_{v}\right)}{\sqrt{n_{u}n_{v}}}$.
The results are shown for (a) empirical success rate vs.~the number
$n_{v}$ of nodes in $\mathcal{V}$ (where $n_{u}=100$, $a=0.8$,
$b=0.01$), and (b) empirical exact recovery rate vs.~$a$ (where
$n_{u}=100$, $n_{v}=10000$, $b=0.01$). \label{fig:BSBM}}
\end{figure}

\section{Further related work \label{sec:Prior-art}}

A natural class of spectral algorithms to estimate the leading singular
subspace of a matrix --- when given a noisy and sub-sampled copy
of the true matrix --- is to compute the leading left singular subspace
of the observed data matrix. Despite the simplicity of this idea,
this type of spectral methods provably achieves appealing performances
for multiple statistical problems when the true matrix is (nearly)
square. A partial list of examples includes low-rank matrix estimation
and completion \cite{KesMonSew2010,jain2013low,chatterjee2015matrix,chen2015fast,ma2017implicit},
community detection \cite{rohe2011spectral,yun2014accurate,abbe2017entrywise,lei2019unified},
and synchronization and alignment \cite{singer2011angular,abbe2017entrywise,chen2016projected,shen2016normalized}.
The $\ell_{2}$ statistical analysis of such algorithms relies heavily
on the matrix perturbation theory such as the Davis-Kahan $\sin\bm{\Theta}$
theorem \cite{davis1970rotation}, the Wedin theorem \cite{wedin1972perturbation},
and their extensions \cite{vu2011singular,yu2014useful,o2018random,cai2018rate,chi2018nonconvex,zhang2018heteroskedastic}.

However, the above-mentioned approach might lead to sub-optimal performance
when the row dimension and the column dimension of the matrix differ
dramatically. This issue has already been recognized in multiple contexts,
including but not limited to unfolding-based spectral methods for
tensor estimation \cite{hopkins2016fast,xia2017polynomial,montanari2018spectral,xia2019sup,zhang2018tensor}
and spectral methods for biclustering \cite{florescu2016spectral}.
Motivated by this sub-optimality issue, an alternative is to look
at the ``sample Gram matrix'' which, as one expects, shares the
same leading left singular space as the original observed data matrix.
However, in the highly noisy or highly subsampled regime, the diagonal entries
of the sample Gram matrix are highly biased, thus requiring special
care. Several different treatments of diagonal components have been
adopted for different contexts, including proper rescaling \cite{montanari2018spectral,lounici2014high,gonen2016subspace},
deletion \cite{florescu2016spectral}, and iterative updates \cite{zhang2018heteroskedastic}.
The deletion strategy is perhaps the simplest of this kind, as it
does not require estimation of noise parameters. We note, however,
that performing more careful iterative updates might be beneficial
for certain heteroskedastic noise scenarios; see \cite{zhang2018heteroskedastic}
for detailed discussions.

An important application of our work is the problem of tensor completion
and estimation \cite{gandy2011tensor,liu2013tensor,romera2013new,kim2013robust,mu2014square,richard2014statistical,ghassemi2017stark,salmi2009sequential,yuan2016tensor,yuan2017incoherent,hao2018sparse}.
Despite its similarity to matrix completion, tensor completion is
considerably more challenging; for concreteness and simplicity, we
shall only discuss order-3 symmetric rank-$r$ tensors in $\mathbb{R}^{d^{3}}$.
Motivated by the success of matrix completion, a simple strategy for
tensor completion\,/\,denoising is to unfold the observed tensor
into a $d\times d^{2}$ matrix and to apply standard matrix completion
methods for completion. However, existing statistical guarantees derived
in the matrix completion literature \cite{candes2009exact,KesMonSew2010,Gross2011recovering}
do not lead to useful bounds unless the sample size exceeds the order
of $rd^{2}$, which far exceeds the requirement for other methods
such as the sum-of-squares (SOS) hierarchy \cite{barak2016noisy,potechin2017exact}.
The work by \cite{montanari2018spectral} demonstrates that unfolding-based
spectral algorithms can also lead to useful estimates under minimal sample complexity, as
long as we look at the ``Gram matrix'' instead. In addition, such
spectral algorithms also play an important role in initializing
other nonconvex optimization methods \cite{xia2017polynomial,xia2017statistically,cai2019nonconvex,cai2020uncertainty}.

In addition, there is an enormous literature on covariance estimation
and principal component analysis (PCA). For instance, the classical
spiked covariance model \cite{johnstone2001distribution} has been
extensively studied; in particular, the high-dimensional setting has
inspired much investigation from both algorithmic and analysis perspectives
\cite{johnstone2009consistency,paul2007asymptotics,bickel2008covariance,nadler2008finite,cai2012adaptive,ma2013sparse,cai2013sparse,cai2015optimal}.
More recently, a computationally efficient algorithm called \emph{HeteroPCA} has been proposed by \cite{zhang2018heteroskedastic} to achieve rate-optimal statistical guarantees for PCA in the presence of heteroskedastic noise. 
When it comes to incomplete data, a variety of methods have been
introduced \cite{kiers1997weighted,elsener2019sparse,josse2012handling}.
For instance, Lounici considered estimating the top eigenvector
in the setting of sparse PCA in \cite{lounici2013sparse}, and further 
proposed an estimator for the covariance matrix in \cite{lounici2014high}. In \cite{cai2016minimax},
bandable and sparse covariance matrices are considered. In addition, most of the prior work considered uniform random subsampling, and the recent work \cite{zhu2019high,pavez2019covariance} began to account for heterogeneous missingness patterns. 

Turning to the problem of community recovery or graph clustering,
we note that extensive research has been carried out on stochastic
block models or censored block models, which can be viewed as special
cases of uni-partite networks \cite{mossel2014consistency,massoulie2014community,abbe2015exact,mossel2015reconstruction,hajek2016achieving,chen2014clustering,chen2016community,chin2015stochastic,guedon2016community,chen2016information,javanmard2016phase,chen2018convexified,cai2015robust,gao2017achieving}.
The algorithms that enable exact community recovery in these block
models include two-stage approaches \cite{abbe2015exact,mossel2014consistency}
and semidefinite programming \cite{hajek2016achieving,amini2018semidefinite,agarwal2017multisection,bandeira2018random,guedon2016community}.
In addition, spectral clustering algorithms have been extensively
studied as well \cite{coja2006spectral,coja2010graph,rohe2011spectral,yun2014accurate,lei2015consistency,yun2016optimal,abbe2017entrywise,gao2017achieving,vu2018simple,o2018random,sussman2012consistent,lelarge2015reconstruction}.
While this class of algorithms was originally developed to yield almost
exact recovery (e.g.~\cite{abbe2015exact}), the recent work by \cite{abbe2017entrywise,lei2019unified}
uncovered that spectral methods alone are sufficient to achieve optimal
exact community recovery (a.k.a.~achieving strong consistency) for
stochastic block models. The interested reader is referred to \cite{abbe2018community}
for an in-depth overview. Various extensions of the SBMs have been
introduced and studied in the last few years. Our work contributes
to this growing literature by justifying the optimality of spectral
methods in bipartite stochastic block models \cite{feldman2015subsampled,florescu2016spectral,gao2016optimal}.

Further, entrywise statistical analysis has recently received significant
attention for various statistical problems \cite{fan2016ell,abbe2017entrywise,mao2017estimating,chen2017spectral,chen2018asymmetry,zhong2018near,agarwal2019robustness,lei2019unified,eldridge2018unperturbed,cape2019two,cape2019signal,xia2019sup,rudelson2015delocalization,Pananjady2019MDP}.
For instance, entrywise guarantees for spectral methods are obtained
in \cite{chen2018asymmetry,eldridge2018unperturbed} based on an algebraic
Neumann trick, while the results in \cite{zhong2018near,abbe2017entrywise,chen2017spectral}
were established based on a leave-one-out analysis. The work by \cite{koltchinskii2016asymptotics,koltchinskii2016perturbation,chen2018asymmetry}
went one step further by controlling an arbitrary linear form of the
eigenvectors or singular vectors of interest. These results, however,
typically lead to suboptimal performance guarantees when the row dimension
and the column dimension of the matrix are substantially different.

Finally, we recently became aware of an unpublished work by Abbe,
Fan and Wang \cite{abbe2019Lp}, which also considers statistical
guarantees of PCA beyond the usual $\ell_{2}$ analysis; in particular,
they develop an analysis framework that delivers tight $\ell_{p}$
perturbation bounds. Note, however, that their results are very different
from the ones presented here. For instance, the results presented
herein emphasize the scenarios with drastically different $d_{1}$
and $d_{2}$, which are not the main focus of \cite{abbe2019Lp}.

\section{Analysis}
\label{sec:Analysis}

In this section, we discuss in detail the analysis techniques employed
to establish Theorem \ref{thm:U_loss}. This is built upon a leave-one-out (as well as a leave-two-out)
analysis strategy that is particularly effective in controlling entrywise
and $\ell_{2,\infty}$ estimation errors \cite{el2013robust,el2015impact,zhong2018near,chen2017spectral,abbe2017entrywise,sur2017likelihood,chen2018gradient,chen2019inference,chen2019nonconvex,lei2018asymptotics,Pananjady2019MDP}. 

\subsection{Leave-one-out and leave-two-out estimates\label{subsec:Leave-one-out-estimates}}

In order to facilitate the analysis when bounding $\left\Vert \bm{U}\bm{R}-\bm{U}^{\star}\right\Vert _{2,\infty}$,
we introduce a set of auxiliary leave-one-out matrices --- a powerful
analysis technique that has been employed to decouple complicated
statistical dependency. It is worth emphasizing that these procedures are never executed in practice. Specifically, for each $1\leq m\leq d_{1}$,
we introduce an auxiliary matrix
\begin{align}
\bm{A}^{(m)} & =\mathcal{P}_{-m,:}\left(\bm{A}\right)+p\mathcal{P}_{m,:}\left(\bm{A}^{\star}\right),\label{def:A_loo}
\end{align}
where $\mathcal{P}_{-m,:}$ (resp.~$\mathcal{P}_{m,:}$) represents
the projection onto the subspace of matrices supported on the index
subset $\left\{ \left[d_{1}\right]\setminus\left\{ m\right\} \right\} \times\left[d_{2}\right]$
(resp.~$\left\{ m\right\} \times\left[d_{2}\right]$). In other words,
$\bm{A}^{(m)}$ is obtained by replacing all entries in the $m$-th
row by their expected values (taking into account the sampling rate).
By construction, (1) $\bm{A}^{(m)}$ is statistically independent
of the data in the $m$-th row of $\bm{A}$, and (2) $\bm{A}^{(m)}$
is expected to be quite close to $\bm{A}$, as we only discard a small
fraction of data when constructing $\bm{A}^{(m)}$. These two observations
taken together allow for optimal control of the estimation error
in the $m$-th row of $\bm{U}$.

Armed with the leave-one-out matrices, we are ready to introduce auxiliary
leave-one-out procedures for subspace estimation. Similar to the matrix
$\bm{G}$ in Algorithm \ref{alg:spectral} (whose eigenspace serves
as an estimate of the column space of $\bm{U}^{\star}$), we define
an auxiliary matrix $\bm{G}^{\left(m\right)}\in\mathbb{R}^{d_{1}\times d_{1}}$
as follows:
\begin{equation}
\bm{G}^{\left(m\right)}=\mathcal{P}_{\mathsf{off}\text{-}\mathsf{diag}}\Big(\tfrac{1}{p^{2}}\bm{A}^{\left(m\right)}\bm{A}^{\left(m\right)\top}\Big),\label{def:G_loo}
\end{equation}
where $\ensuremath{\mathcal{P}_{\mathsf{off}\text{-}\mathsf{diag}}\left(\cdot\right)}$
(as already defined in Section~\ref{subsec:Notations}) extracts
out all off-diagonal entries from a matrix. The auxiliary procedure,
which is summarized in Algorithm \ref{alg:init_loo}, is very similar
to Algorithm \ref{alg:spectral} except that it operates upon $\bm{G}^{(m)}$.

\begin{algorithm}
\caption{The $m$-th leave-one-out sequence}

\label{alg:init_loo}

\begin{algorithmic}[1]
\STATE{{\bf Input:} sampling  set $\Omega$, observed entries $\{ A_{i, j} \mid \left( i, j \right) \in \Omega \}$, true entries $\{ A^\star_{m, j}  \mid j \in \left[ d_2 \right] \}$, sampling rate $p$, rank $r$.}
	\STATE{Let $\bm{U}^{\left(m\right)} \bm{\Lambda}^{\left(m\right)} \bm{U}^{\left(m\right) \top}$ be the (truncated) rank-$r$ eigen-decomposition of $\bm{G}^{\left(m\right)}$. Here, $\bm{G}^{\left(m\right)}$ and $\bm{A}^{\left(m\right)}$ are defined respectively in \eqref{def:G_loo} and \eqref{def:A_loo}.}
\STATE{{\bf Output} $\bm{U}^{\left(m\right)}$ as the subspace estimate and $\bm{\Sigma}^{\left(m\right)} = {(\bm{\Lambda}^{\left(m\right)})}^{1/2}$ as the spectrum estimate.}
\end{algorithmic}
\end{algorithm}

Given that $\bm{A}^{(m)}$ (resp.~$\bm{G}^{(m)}$) is very close
to $\bm{A}$ (resp.~$\bm{G}$), one would naturally expect $\bm{U}^{(m)}$
--- the $r$-dimensional principal eigenspace of $\bm{G}^{(m)}$ --- to stay extremely
close to the original estimate $\bm{U}$. This fact will be formalized
shortly.

As it turns out, given that the spectral method is applied to the Gram matrix (which is a quadratic form of the original data matrix), introducing the leave-one-out sequences alone is not yet sufficient for our purpose; we still need to introduce an
additional set of ``leave-two-out'' matrices, in the hope of simultaneously
handling the row-wise and the column-wise statistical dependency.
Specifically, for each $1\leq m\leq d_{1}$ and each $1\leq l\leq d_{2}$,
define the following auxiliary matrices:\begin{subequations}\label{def:loo_col}
\begin{align}
\bm{A}^{\left(m,l\right)} & :=\mathcal{P}_{-m,-l}\left(\bm{A}\right)+p\mathcal{\mathcal{P}}_{m,l}\left(\bm{A}^{\star}\right),\label{def:A_loo_col}\\
\bm{G}^{\left(m,l\right)} & :=\mathcal{P}_{\mathsf{off}\text{-}\mathsf{diag}}\left(\tfrac{1}{p^{2}}\bm{A}^{\left(m,l\right)}\bm{A}^{\left(m,l\right)\top}\right),\label{def:G_loo_col}
\end{align}
\end{subequations} where $\mathcal{\mathcal{P}}_{-m,-l}$ (resp.~$\mathcal{\mathcal{P}}_{m,l}$)
denotes the projection onto the subspace of matrices supported on
$\left\{ \left[d_{1}\right]\setminus\left\{ m\right\} \right\} \times\left\{ \left[d_{2}\right]\setminus\left\{ l\right\} \right\} $
(resp.~$\left\{ m\right\} \times\left\{ l\right\} $). Similar to $\bm{A}^{(m)}$,
$\bm{A}^{(m,l)}$ is generated by replacing all data lying on the $m$-th
row and the $l$-th column of $\bm{A}$ by their expected values (taking into account the sampling rate). The precise procedure is
summarized in Algorithm \ref{alg:init_loo_col}. Similar to the leave-one-out
estimates, one expects the new leave-two-out estimates $\bm{U}^{\left(m,l\right)}$
to be extremely close to $\bm{U}^{\left(m\right)}$ (and hence $\bm{U}$).

\begin{algorithm}
\caption{The $\left(m,l\right)$-th leave-two-out sequence}
\label{alg:init_loo_col}
\begin{algorithmic}[1]
	\STATE{{\bf Input:} sampling  set $\Omega$, observed entries $\{ A_{i, j} \mid \left( i, j \right) \in \Omega \}$, true entries $\{ A^\star_{m, j}  \mid j \in \left[ d_2 \right] \} \cup \{ A^\star_{i, l}  \mid i \in \left[ d_1 \right] \} $, sampling rate $p$, rank $r$.}
	\STATE{Let $\bm{U}^{\left(m, l\right)} \bm{\Lambda}^{\left(m, l\right)} \bm{U}^{\left(m, l\right) \top}$ be the (truncated) rank-$r$ eigen-decomposition of $\bm{G}^{\left(m, l\right)}$. Here, $\bm{G}^{\left(m, l\right)}$ and $\bm{A}^{\left(m, l\right)}$ are defined respectively in \eqref{def:G_loo_col} and \eqref{def:A_loo_col}.}
	\STATE{{\bf Output} $\bm{U}^{\left(m, l\right)}$ as the subspace estimate and $\bm{\Sigma}^{\left(m, l\right)} = {(\bm{\Lambda}^{\left(m, l\right)})}^{1/2}$ as the spectrum estimate.}
\end{algorithmic}
\end{algorithm}

\subsection{Key lemmas\label{subsec:Key-lemmas}}

In this subsection, we provide several lemmas that play a crucial
role in establishing our main theorem. These lemmas are primarily
concerned with the proximity between the original estimate, the leave-one-out
estimates, and the ground truth. Throughout this section, we let
\begin{equation}
\bm{G}^{\star}:=\bm{A}^{\star}\bm{A}^{\star\top}=\bm{U}^{\star}\bm{\Sigma}^{\star2}\bm{U}^{\star\top}.\label{eq:defn-Gstar}
\end{equation}

To begin with, we demonstrate that $\bm{G}$ is sufficiently close
to $\bm{G}^{\star}$ when the difference is measured by the spectral
norm. In view of standard matrix perturbation theory (which we shall make precise later), the proximity of $\bm{G}$ and $\bm{G}^{\star}$ 
is crucial in bounding the difference between $\bm{U}$ and $\bm{U}^{\star}$. The proof is
deferred to Appendix \ref{subsec:pf:G_op_loss}.

\begin{lemma}\label{lemma:G_op_loss}Instate the assumptions of Theorem
\ref{thm:U_loss}. With probability at least $1-O\left(d^{-10}\right)$,
one has
\begin{align}
\left\Vert \bm{G}-\bm{G}^{\star}\right\Vert  & \lesssim\underset{=:\zeta_{\mathsf{op}}}{\underbrace{\frac{\mu r\sigma_{1}^{\star2}\log d}{\sqrt{d_{1}d_{2}}\,p}+\sqrt{\frac{\mu r\sigma_{1}^{\star4}\log d}{d_{2}p}}+\frac{\sigma^{2}\sqrt{d_{1}d_{2}}\,\log^ {}d}{p}+\sigma\sigma_{1}^{\star}\sqrt{\frac{d_{1}\log d}{p}}}}+\left\Vert \bm{A}^{\star}\right\Vert _{2,\infty}^{2}.\label{claim:G_op_loss}
\end{align}
\end{lemma}

In order to get a better sense of the term $\zeta_{\mathsf{op}}$
appearing above, we make note of a straightforward yet useful fact, which reveals that $\zeta_{\mathsf{op}}$ is much smaller than any nonzero eigenvalue of $\bm{G}^{\star}$.

\begin{fact}\label{fact:rel_err_o1}Instate the assumptions of Theorem
\ref{thm:U_loss}. Then the quantity $\zeta_{\mathsf{op}}$ as defined
in (\ref{claim:G_op_loss}) obeys
\begin{align*}
\zeta_{\mathsf{op}}+\left\Vert \bm{A}^{\star}\right\Vert _{2,\infty}^{2} & \leq\frac{\mu r\sigma_{1}^{\star2}\log d}{\sqrt{d_{1}d_{2}}\,p}+\sqrt{\frac{\mu r\sigma_{1}^{\star4}\log d}{d_{2}p}}+\frac{\sigma^{2}}{\sigma_{r}^{\star2}}\frac{\sqrt{d_{1}d_{2}}\,\log^ {}d}{p}+\sigma\sigma_{1}^{\star}\sqrt{\frac{d_{1}\log d}{p}}+\frac{\mu_{1}r\sigma_{1}^{\star2}}{d_{1}}\\
 & \ll\frac{\sigma_{r}^{\star2}}{\kappa^{2}},
\end{align*}
where $\|\bm{A}^{\star}\|_{2,\infty}^{2}\leq\frac{\mu_{1}r\sigma_{1}^{\star2}}{d_{1}}$
(cf.~Lemma \ref{lemma:incoh}).
\end{fact}

Further, the following lemma upper bounds the difference between $\bm{G}$
and $\bm{G}^{\star}$ in the $m$-th row, when projected onto the
subspace represented by $\bm{U}^{\star}$; the proof is postponed
to Appendix~\ref{subsec:pf:G_dev_W_2_norm}. This result gives a
more refined control of the difference between $\bm{G}$ and $\bm{G}^{\star}$.

\begin{lemma}\label{lemma:G_dev_W_2_norm}Instate the assumptions
of Theorem \ref{thm:U_loss}. With probability at least $1-O\left(d^{-10}\right)$,
the following holds simultaneously for all $1\leq m\leq d_{1}$:
\begin{align*}
\big\|\left(\bm{G}-\bm{G}^{\star}\right)_{m,:}\bm{U}^{\star}\big\|_{2} & \lesssim\big(\zeta_{\mathsf{op}}+\left\Vert \bm{A}^{\star}\right\Vert _{2,\infty}^{2}\big)\sqrt{\frac{\mu r}{d_{1}}},
\end{align*}
where $\zeta_{\mathsf{op}}$ is defined in \eqref{claim:G_op_loss}.\end{lemma}

The next step, which is also the most challenging and crucial step, lies in showing that: every row of $\bm{U}$, under certain global linear transformation, serves as a good approximation of the corresponding row of $\bm{U}^{\star}$. Towards this end, we begin with the following preparations: 
\begin{itemize}
\item We first introduce the following matrix $\bm{H}$ to represent the linear transformation we have in mind:
\begin{equation}
\bm{H}:=\bm{U}^{\top}\bm{U}^{\star}.\label{eq:defn-H}
\end{equation}
While this is not a rotation matrix, it is quite close to the rotation matrix $\bm{R}$
defined in \eqref{def:R}. 
\item In addition, we find it convenient to express
\[
\bm{U}^{\star}=\bm{G}^{\star}\bm{U}^{\star}\left(\bm{\Sigma}^{\star}\right)^{-2}.
\]
Combining this with Lemma \ref{lemma:G_dev_W_2_norm}, one would expect $\bm{U}^{\star}$ and $\bm{G}\bm{U}^{\star}\left(\bm{\Sigma}^{\star}\right)^{-2}$
to be reasonably close, namely,
\begin{equation}
	\bm{U}^{\star} \approx \bm{G}\bm{U}^{\star}\left(\bm{\Sigma}^{\star}\right)^{-2}. 
\end{equation}
\end{itemize}


With these in hand, the following lemma (together with Lemma \ref{lemma:G_dev_W_2_norm}) asserts that
\[
	\bm{U}\bm{H} \approx \bm{G}\bm{U}^{\star}\left(\bm{\Sigma}^{\star}\right)^{-2} \approx \bm{U}^{\star}
\]
in an $\ell_{2,\infty}$ sense. 

\begin{lemma}\label{lemma:UH_BUtrue} Instate the assumptions of
Theorem \ref{thm:U_loss}, and recall the definition of $\zeta_{\mathsf{op}}$
in (\ref{claim:G_op_loss}). With probability at least $1-O(d^{-10})$, one has
\begin{align*}
\big\|\bm{U}\bm{H}-\bm{G}\bm{U}^{\star}\left(\bm{\Sigma}^{\star}\right)^{-2}\big\|_{2,\infty} & \lesssim\frac{\kappa^{2}\big(\zeta_{\mathsf{op}}+\left\Vert \bm{A}^{\star}\right\Vert _{2,\infty}^{2}\big)}{\sigma_{r}^{\star2}}\left(\left\Vert \bm{U}\bm{H}\right\Vert _{2,\infty}+\sqrt{\frac{\mu r}{d_{1}}}\right).
\end{align*}
\end{lemma}
The proof of this lemma, however, goes far beyond conventional matrix perturbation theory, and requires delicate decoupling of statistical dependencies. This is accomplished via leave-one-out and leave-two-out analysis arguments. In what follows, we take a moment to explain the high-level idea. 

To establish Lemma \ref{lemma:UH_BUtrue}, we first learn
from standard matrix perturbation theory\cite[Lemma~1]{abbe2017entrywise}
that: for each $1\leq m\leq d_{1}$, 
\begin{equation}
\left\Vert \big(\bm{U}\bm{H}-\bm{G}\bm{U}^{\star}\left(\bm{\Sigma}^{\star}\right)^{-2}\big)_{m,:}\right\Vert _{2}\lesssim\frac{1}{\lambda_{r}\left(\bm{G}^{\star}\right)^2}\left\Vert \bm{G}-\bm{G}^{\star}\right\Vert \left\Vert \bm{G}_{m,:}\bm{U}^{\star}\right\Vert _{2}+\frac{1}{\lambda_{r}\left(\bm{G}^{\star}\right)}\left\Vert \bm{G}_{m,:}\left(\bm{U}\bm{H}-\bm{U}^{\star}\right)\right\Vert _{2}\label{eq:UH-UB2}
\end{equation}
holds, provided that $\bm{G}$ and $\bm{G}^{\star}$ are sufficiently
close. 

\begin{itemize}
\item
The first term on the right-hand side of (\ref{eq:UH-UB2})
can already be controlled by Lemma~\ref{lemma:G_op_loss} and Lemma~\ref{lemma:G_dev_W_2_norm}.

\item 
The second term on the right-hand side of (\ref{eq:UH-UB2}), however,
is considerably more difficult to analyze, due to the complicated
statistical dependence between $\bm{G}_{m,:}$ and $\bm{U\bm{H}}$.
In order to decouple statistical dependency, we resort to the leave-one-out
sequence $\bm{U}^{\left(m\right)}$ introduced in Algorithm~\ref{alg:init_loo} and use the triangle inequality to bound
\begin{equation}
\left\Vert \bm{G}_{m,:}\left(\bm{U}\bm{H}-\bm{U}^{\star}\right)\right\Vert _{2}\leq\big\|\bm{G}_{m,:}\big(\bm{U}\bm{H}-\bm{U}^{\left(m\right)}\bm{H}^{\left(m\right)}\big)\big\|_{2}+\big\|\bm{G}_{m,:}\big(\bm{U}^{\left(m\right)}\bm{H}^{\left(m\right)}-\bm{U}^{\star}\big)\big\|_{2},\label{eq:UH-UB3}
\end{equation}
where $\bm{H}^{(m)}:=\bm{U}^{(m)\top}\bm{U}^{\star}$. 
As mentioned before,
the leave-one-out estimate $\bm{U}^{\left(m\right)}$ enjoys two nice
properties. 
\begin{itemize}

	\item[(1)] The true estimate $\bm{U}$ and the leave-one-out estimate $\bm{U}^{(m)}$ are exceedingly close, as asserted by the following lemma (to be established in Appendix
\ref{subsec:pf:U_U_loo_dist}).   
%
%
\begin{lemma}\label{lemma:U_U_loo_dist}
Instate the assumptions of
Theorem \ref{thm:U_loss}, and recall the definition of $\bm{H}$
in (\ref{eq:defn-H}). With probability at least $1-O\left(d^{-10}\right)$,
the following holds simultaneously for all $1\leq m\leq d_{1}$:
\begin{align*}
\big\|\bm{U}^{\left(m\right)}\bm{U}^{\left(m\right)\top}-\bm{U}\bm{U}^{\top}\big\|_{\mathrm{F}} & \lesssim\frac{\zeta_{\mathsf{op}}}{\sigma_{r}^{\star2}}\left(\left\Vert \bm{U}\bm{H}\right\Vert _{2,\infty}+\sqrt{\frac{\mu r}{d_{1}}}\right),
\end{align*}
where $\zeta_{\mathsf{op}}$ is defined in (\ref{claim:G_op_loss}).\end{lemma}
This result in turn allows us to control the first term on the right-hand side of (\ref{eq:UH-UB3}).

\item[(2)] Due to the statistical independence between $\bm{A}_{m,:}$
and $\bm{U}^{(m)}$,  the matrices $\bm{G}_{m,:}$ and $\bm{U}^{(m)}$ turn out
to be nearly independent. This allows one to invoke simple concentration
inequalities to develop tight bounds for the second term on the right-hand
side of (\ref{eq:UH-UB3}). The detailed proof can be found in Appendix
\ref{subsec:pf:UH_BUtrue}.
\end{itemize}

\end{itemize}


Finally, we make a remark on a technical issue encountered in the proof of Lemma \ref{lemma:U_U_loo_dist}. Recall that $\bm{U}^{\left(m\right)}$
is obtained by simply replacing the $m$-th row of $\bm{A}$ with 
its population version, which indicates the statistical dependency between  $\bm{U}^{\left(m\right)}$ and the $m$-th row of $\bm{A}$. 
However, there is still some delicate statistical dependency between $\bm{U}^{(m)}$ and the columns of $\bm{A}$ that need to be carefully coped with. 
Fortunately, the leave-two-out estimate $\bm{U}^{(m,l)}$ --- which is obtained by dropping  not only the $m$-th row of $\bm{A}$ but also the $l$-th of its columns --- allows us to decouple the dependency between $\bm{U}^{(m,l)}$ (and hence $\bm{U}$ and $\bm{U}^{(m)}$) and the $l$-th column of $\bm{A}$.  This is precisely the main reason why we introduce additional leave-two-out estimates. 

\subsection{Proof of Theorem~\ref{thm:U_loss}\label{subsec:pf:thm:U_loss}}

We are now positioned to establish our main theorem. The proof is
split into two parts.

\subsubsection{Statistical accuracy measured by $\left\Vert \cdot\right\Vert $}

We begin by establishing the spectral norm bound \eqref{claim:Lambda_loss}.
Let $\lambda_i$ and $\lambda_i^{\star}$ be the $i$-th largest eigenvalue of $\bm{\Lambda}$ and $\bm{\Lambda^{\star}}$, respectively. 
From Lemma~\ref{lemma:G_op_loss} and
Weyl's inequality, one finds that
\begin{align}
	\label{eq:lambda-i-bound}
\max_{1\leq i\leq r}\left|\lambda_{i}-\lambda_{i}^{\star}\right|=\left\Vert \bm{\Lambda}-\bm{\Lambda}^{\star}\right\Vert  & \leq\left\Vert \bm{G}-\bm{G}^{\star}\right\Vert \lesssim\zeta_{\mathsf{op}}+\|\bm{A}^{\star}\|_{2,\infty}^{2}\leq\sigma_{r}^{\star2}\cdot\mathcal{E}_{\mathsf{general}} ,
\end{align}
where $\zeta_{\mathsf{op}}$ and $\mathcal{E}_{\mathsf{general}}$ are
defined in (\ref{claim:G_op_loss}) and (\ref{eq:defn-Err-UB}), respectively. 
Here, the last inequality arises from the simple fact that $\|\bm{A}^{\star}\|_{2,\infty}^{2}\leq\frac{\mu_{1}r\sigma_{1}^{\star2}}{d_{1}}$
(cf.~Lemma \ref{lemma:incoh}). By virtue of Fact \ref{fact:rel_err_o1}, we
know that $\left\Vert \bm{\Lambda}-\bm{\Lambda}^{\star}\right\Vert \ll\sigma_{r}^{\star2}$. 
Given that $\bm{\Lambda}^{\star}=\bm{\Sigma}^{\star2}$ and $\bm{\Lambda}=\bm{\Sigma}^{2}$,
this implies that for each $1\leq i\leq r$,
\[
\frac{1}{4}\sigma_{i}^{\star2}=\frac{1}{4}\lambda_{i}^{\star}\leq\lambda_{i}^{\star}-\left|\lambda_{i}-\lambda_{i}^{\star}\right|\leq\lambda_{i}\leq\lambda_{i}^{\star}+\left|\lambda_{i}-\lambda_{i}^{\star}\right|\leq4\lambda_{i}^{\star }=4\sigma_{i}^{\star2},
\]
thus indicating that
\begin{equation}
	\label{eq:sigma-bound}
	\frac{1}{2}\sigma_{i}^{\star}\leq\sigma_{i}\leq2\sigma_{i}^{\star}.
\end{equation}
In conclusion, 
\[
	\left\Vert \bm{\Sigma}-\bm{\Sigma}^{\star}\right\Vert =\max_{1\leq i\leq r}\left|\sigma_{i}-\sigma_{i}^{\star}\right|=\max_{1\leq i\leq r}\frac{\left|\sigma_{i}^{2}-\sigma_{i}^{\star2}\right|}{\sigma_{i}+\sigma_{i}^{\star}} \overset{\mathrm{(a)}}{\leq} \max_{1\leq i\leq r}\frac{\left\Vert \bm{\Lambda}-\bm{\Lambda}^{\star}\right\Vert }{\frac{3}{2}\sigma_{i}^{\star}}  \overset{\mathrm{(b)}}{\lesssim} \sigma_{r}^{\star}\cdot\mathcal{E}_{\mathsf{general}}
\]
as claimed. Here, (a) comes from \eqref{eq:sigma-bound}, whereas (b) follows from \eqref{eq:lambda-i-bound}.

Next, we turn attention to (\ref{claim:U_op_loss}). First, it is seen that
\begin{align}
	\label{eq:UR-U-UB}
	\left\Vert \bm{U}\bm{R}-\bm{U}^{\star}\right\Vert  \leq \sqrt{2}\left\Vert \bm{U}\bm{U}^{\top}-\bm{U}^{\star}\bm{U}^{\star\top}\right\Vert = \| \sin \bm{\Theta}( \bm{U},  \bm{U}^{\star}) \| 
\end{align}
where $\bm{R}$ is defined in (\ref{def:R}), and $\bm{\Theta}( \bm{U},  \bm{U}^{\star})$ denotes a diagonal matrix
whose $i$-th diagonal entry is the $i$-th principal angle between the two subspaces represented by $\bm{U}$ and $\bm{U}^{\star}$. Here,  the first inequality follows from a well-known inequality connecting two different subspace distance metrics \cite[Chapter II]{stewart1990matrix}, while the last identity follows from  \cite[Chapter II]{stewart1990matrix}. In addition, Lemma~\ref{lemma:G_op_loss}
and Fact \ref{fact:rel_err_o1} tell us that 
\begin{align}
\left\Vert \bm{G}-\bm{G}^{\star}\right\Vert  & \lesssim\zeta_{\mathsf{op}}+\|\bm{A}^{\star}\|_{2,\infty}^{2}\ll\sigma_{r}^{\star2}\label{eq:G_eigval}
\end{align}
with probability at least $1-O(d^{-10})$, which together with Weyl's inequality gives
\begin{align}
	\lambda_{r+1}(\bm{G}) \leq \lambda_{r+1}(\bm{G}^{\star}) + \|\bm{G}-\bm{G}^{\star}\| = \|\bm{G}-\bm{G}^{\star}\| \leq \sigma_{r}^{\star2}/2 . 
\end{align}
Therefore, \cite[Chapter V, Theorem~3.6]{stewart1990matrix} (which is a version of the celebrated Davis-Kahan $\sin\bm{\Theta}$ Theorem\cite{davis1970rotation}) reveals that 
\begin{align*}
\|\sin\bm{\Theta}(\bm{U},\bm{U}^{\star})\| & \leq\frac{\left\Vert \bm{G}-\bm{G}^{\star}\right\Vert }{\lambda_{r}\left(\bm{G}^{\star}\right)-\lambda_{r+1}\big(\bm{G}\big)}\leq\frac{\left\Vert \bm{G}-\bm{G}^{\star}\right\Vert }{\sigma_{r}^{\star2}-\sigma_{r}^{\star2}/2}=\frac{2\left\Vert \bm{G}-\bm{G}^{\star}\right\Vert }{\sigma_{r}^{\star2}} ,
\end{align*}
where we have used the fact $\lambda_r(\bm{G}^{\star})=\sigma_r^{\star 2}$. The above bounds taken collectively imply that, with probability at least
$1-O(d^{-10})$, 
\begin{align}
	\left\Vert \bm{U}\bm{R}-\bm{U}^{\star}\right\Vert  & \leq\frac{2\sqrt{2}\left\Vert \bm{G}-\bm{G}^{\star}\right\Vert }{\sigma_{r}^{\star2}} \lesssim \mathcal{E}_{\mathsf{general}}. 
	\label{eq:UR-Ustar-UB1}
\end{align}

\subsubsection{Statistical accuracy measured by $\left\Vert \cdot\right\Vert _{2,\infty}$}

Before continuing to the proof, we find it convenient to introduce a few more notations.
In addition to the rotation matrix $\bm{R}$  defined in (\ref{def:R})
and the linear transformation $\bm{H}$ defined in (\ref{eq:defn-H}),
we define
\begin{align}
\mathsf{sgn}\left(\bm{H}\right) & :=\widetilde{\bm{U}}\widetilde{\bm{V}}^{\top}\in\mathbb{R}^{d_{1}\times d_{1}},\label{def:H_sgn}
\end{align}
where the columns of $\widetilde{\bm{U}}\in\mathbb{R}^{d_{1}\times d_{1}}$
(resp.~$\widetilde{\bm{V}}\in\mathbb{R}^{d_{1}\times d_{1}}$) are
the left (resp.~right) singular vectors of $\bm{H}$. It is well-known
that \cite[Theorem~2]{ten1977orthogonal}
\begin{equation}
\bm{R}=\mathsf{sgn}\left(\bm{H}\right).\label{eq:R-sgnH}
\end{equation}

We now move on to establishing the advertised bound (\ref{claim:U_2inf_loss}).
\begin{enumerate}
\item To begin with, we claim that $\bm{U}\bm{H}$ is extremely close to
$\bm{U}\bm{R}$, provided that $\left\Vert \bm{G}-\bm{G}^{\star}\right\Vert $
is sufficiently small. To this end, recognizing that $\left\Vert \bm{G}-\bm{G}^{\star}\right\Vert \lesssim\zeta_{\mathsf{op}}\ll\sigma_{r}^{\star2}$
(according to Lemma \ref{lemma:G_op_loss} and Fact \ref{fact:rel_err_o1}),
we can apply \cite[Lemma~3]{abbe2017entrywise} to show that
\[
\left\Vert \bm{H}^{-1}\right\Vert \lesssim1\qquad\text{and}\qquad\sqrt{\left\Vert \bm{H}-\mathsf{sgn}\left(\bm{H}\right)\right\Vert }\leq\left\Vert \bm{U}\bm{U}^{\top}-\bm{U}^{\star}\bm{U}^{\star\top}\right\Vert \lesssim\frac{\zeta_{\mathsf{op}}}{\sigma_{r}^{\star2}},
\]
where the last inequality follows from (\ref{eq:UR-Ustar-UB1}). Thus,
invoke the identity (\ref{eq:R-sgnH}) to arrive at
\begin{align}
\left\Vert \bm{U}\bm{H}-\bm{U}\bm{R}\right\Vert _{2,\infty} & =\left\Vert \bm{U}\bm{H}-\bm{U}\mathsf{sgn}\left(\bm{H}\right)\right\Vert _{2,\infty}=\left\Vert \bm{U}\bm{H}\bm{H}^{-1}\big(\bm{H}-\mathsf{sgn}\left(\bm{H}\right)\big)\right\Vert _{2,\infty}\nonumber \\
 & \leq\left\Vert \bm{U}\bm{H}\right\Vert _{2,\infty}\left\Vert \bm{H}^{-1}\right\Vert \left\Vert \bm{H}-\mathsf{sgn}\left(\bm{H}\right)\right\Vert \nonumber \\
 & \lesssim\left(\frac{\zeta_{\mathsf{op}}}{\sigma_{r}^{\star2}}\right)^{2}\left\Vert \bm{U}\bm{H}\right\Vert _{2,\infty}\ll\frac{\zeta_{\mathsf{op}}}{\sigma_{r}^{\star2}}\left\Vert \bm{U}\bm{H}\right\Vert _{2,\infty}.\label{eq:UH_UsgnH_2inf_diff}
\end{align}
This in turn allows us to focus attention on bounding $\left\Vert \bm{U}\bm{H}-\bm{U}^{\star}\right\Vert _{2,\infty}$
(instead of $\left\Vert \bm{U}\bm{R}-\bm{U}^{\star}\right\Vert _{2,\infty}$).
\item Next, recall that $\bm{G}^{\star}=\bm{U}^{\star}\bm{\Sigma}^{\star2}\bm{U}^{\star\top}$
and hence $\bm{G}^{\star}\bm{U}^{\star}(\bm{\Sigma}^{\star})^{-2}=\bm{U}^{\star}$.
Invoke the triangle inequality to reach
\begin{align}
\left\Vert \bm{U}\bm{H}-\bm{U}^{\star}\right\Vert _{2,\infty} & =\left\Vert \bm{U}\bm{H}-\bm{G}\bm{U}^{\star}\left(\bm{\Sigma}^{\star}\right)^{-2}+\bm{G}\bm{U}^{\star}\left(\bm{\Sigma}^{\star}\right)^{-2}-\bm{G}^{\star}\bm{U}^{\star}\left(\bm{\Sigma}^{\star}\right)^{-2}\right\Vert _{2,\infty}\nonumber \\
 & \leq\big\|\left(\bm{G}-\bm{G}^{\star}\right)\bm{U}^{\star}\left(\bm{\Sigma}^{\star}\right)^{-2}\big\|_{2,\infty}+\big\|\bm{U}\bm{H}-\bm{G}\bm{U}^{\star}\left(\bm{\Sigma}^{\star}\right)^{-2}\big\|_{2,\infty} \nonumber \\
 & \leq\left\Vert \left(\bm{G}-\bm{G}^{\star}\right)\bm{U}^{\star}\right\Vert _{2,\infty}\big\|\left(\bm{\Sigma}^{\star}\right)^{-2}\big\|+\big\|\bm{U}\bm{H}-\bm{G}\bm{U}^{\star}\left(\bm{\Sigma}^{\star}\right)^{-2}\big\|_{2,\infty}\nonumber \\
 & \leq\frac{1}{\sigma_{r}^{\star2}}\left\Vert \left(\bm{G}-\bm{G}^{\star}\right)\bm{U}^{\star}\right\Vert _{2,\infty}+\big\|\bm{U}\bm{H}-\bm{G}\bm{U}^{\star}\left(\bm{\Sigma}^{\star}\right)^{-2}\big\|_{2,\infty}.\label{eq:UH_2inf_loss_decomp}
\end{align}
Regarding the first term of (\ref{eq:UH_2inf_loss_decomp}), Lemma~\ref{lemma:G_dev_W_2_norm}
reveals that with probability at least $1-O\left(d^{-10}\right)$,
\begin{align}
\frac{1}{\sigma_{r}^{\star2}}\left\Vert \left(\bm{G}-\bm{G}^{\star}\right)\bm{U}^{\star}\right\Vert _{2,\infty} & \lesssim\frac{\zeta_{\mathsf{op}}+\left\Vert \bm{A}^{\star}\right\Vert _{2,\infty}^{2}}{\sigma_{r}^{\star2}}\sqrt{\frac{\mu r}{d_{1}}}.\label{eq:U_loss_2inf_term1}
\end{align}
With regards to the second term of (\ref{eq:UH_2inf_loss_decomp}),
Lemma~\ref{lemma:UH_BUtrue} demonstrates that
\begin{align}
\big\|\bm{U}\bm{H}-\bm{G}\bm{U}^{\star}\left(\bm{\Sigma}^{\star}\right)^{-2}\big\|_{2,\infty}\lesssim & \frac{\kappa^{2}\big(\zeta_{\mathsf{op}}+\left\Vert \bm{A}^{\star}\right\Vert _{2,\infty}^{2}\big)}{\sigma_{r}^{\star2}}\left(\left\Vert \bm{U}\bm{H}\right\Vert _{2,\infty}+\sqrt{\frac{\mu r}{d_{1}}}\right)\label{eq:U_loss_2inf_term2}
\end{align}
with probability at least $1-O\left(d^{-10}\right)$. Combine (\ref{eq:U_loss_2inf_term1})
and (\ref{eq:U_loss_2inf_term2}) to arrive at
\begin{align}
\left\Vert \bm{U}\bm{H}-\bm{U}^{\star}\right\Vert _{2,\infty} & \lesssim\frac{\kappa^{2}\big(\zeta_{\mathsf{op}}+\left\Vert \bm{A}^{\star}\right\Vert _{2,\infty}^{2}\big)}{\sigma_{r}^{\star2}}\left(\left\Vert \bm{U}\bm{H}\right\Vert _{2,\infty}+\sqrt{\frac{\mu r}{d_{1}}}\right).\label{eq:UH_2inf_loss_temp}
\end{align}

\item As a byproduct of (\ref{eq:UH_2inf_loss_temp}) and Fact~\ref{fact:rel_err_o1}, we see that
\begin{align*}
\left\Vert \bm{U}\bm{H}-\bm{U}^{\star}\right\Vert _{2,\infty} & \ll\left\Vert \bm{U}\bm{H}\right\Vert _{2,\infty}+\sqrt{\frac{\mu r}{d_{1}}} .
\end{align*}
%
%
It then follows from the triangle inequality that
\[
\left\Vert \bm{U}\bm{H}\right\Vert _{2,\infty}\leq\left\Vert \bm{U}\bm{H}-\bm{U}^{\star}\right\Vert _{2,\infty}+\left\Vert \bm{U}^{\star}\right\Vert _{2,\infty}\lesssim o\left(1\right)\left\Vert \bm{U}\bm{H}\right\Vert _{2,\infty}+\sqrt{\frac{\mu r}{d_{1}}},
\]
thus indicating that
\begin{equation}
\left\Vert \bm{U}\bm{H}\right\Vert _{2,\infty}\leq2\sqrt{\frac{\mu r}{d_{1}}}.\label{eq:UH_2inf_UB}
\end{equation}
Substitution into (\ref{eq:UH_UsgnH_2inf_diff}) and (\ref{eq:UH_2inf_loss_temp})
gives
\[
\left\Vert \bm{U}\bm{H}-\bm{U}\bm{R}\right\Vert _{2,\infty}\ll\frac{\zeta_{\mathsf{op}}}{\sigma_{r}^{\star2}}\sqrt{\frac{\mu r}{d_{1}}}\quad\text{and}\quad\left\Vert \bm{U}\bm{H}-\bm{U}^{\star}\right\Vert _{2,\infty}\lesssim\frac{\kappa^{2}\big(\zeta_{\mathsf{op}}+\left\Vert \bm{A}^{\star}\right\Vert _{2,\infty}^{2}\big)}{\sigma_{r}^{\star2}}\sqrt{\frac{\mu r}{d_{1}}}.
\]
\end{enumerate}
Combining the above results yields
\begin{align*}
\left\Vert \bm{U}\bm{R}-\bm{U}^{\star}\right\Vert _{2,\infty} & \leq\left\Vert \bm{U}\bm{H}-\bm{U}^{\star}\right\Vert _{2,\infty}+\left\Vert \bm{U}\bm{H}-\bm{U}\bm{R}\right\Vert _{2,\infty}\\
 & \lesssim\frac{\kappa^{2}\big(\zeta_{\mathsf{op}}+\left\Vert \bm{A}^{\star}\right\Vert _{2,\infty}^{2}\big)}{\sigma_{r}^{\star2}}\sqrt{\frac{\mu r}{d_{1}}}.
\end{align*}
Substituting the value of $\zeta_{\mathsf{op}}$ into the above inequality
and using the upper bound $\|\bm{A}^{\star}\|_{2,\infty}^{2}\leq\frac{\mu_{1}r\sigma_{1}^{\star2}}{d_{1}}$
(cf.~Lemma \ref{lemma:incoh}), we conclude the proof.

\section{Discussion\label{sec:discussion}}

This paper has explored spectral methods tailored to subspace estimation
for low-rank matrices with missing entries. In comparison to prior
literature, our findings are particularly interesting when the column
dimension $d_{2}$ far exceeds the row dimension $d_{1}$. In many
scenarios, even though the observed data are either too noisy or too
incomplete to support reliable recovery of the entire matrix (so that
prior matrix completion results often become inapplicable), they might
still be informative enough if the purpose is merely to estimate the
column subspace of the unknown matrix. In fact, this suggests a potentially useful
paradigm for privacy-preserving estimation or learning: the
inability to recover the entire matrix facilitates the protection
of personal data, yet it is still possible to retrieve useful subspace
information for inference and learning.  Our main contribution lies in establishing $\ell_{2,\infty}$
statistical guarantees for subspace estimation, therefore providing a stronger
form of performance guarantees compared to the usual $\ell_{2}$ perturbation
bounds.

Moving forward, there are many directions that are worth pursuing.
For example, our current theory is likely suboptimal with respect
to the dependence on the rank $r$ and the condition number $\kappa$. For instance, the conditions~\eqref{asmp} and the risk bound \eqref{eq:defn-Err-UB} involve high-order polynomials of $\kappa$ in multiple places, and the rank $r$ in our current theory cannot exceed the order of $d_1/\kappa^4$; all of these might be improvable via more refined analysis.   In addition, it is natural to wonder whether we can extend our algorithm and theory to accommodate
more general sampling patterns. Going beyond estimation, an important
direction lies in statistical inference and uncertainty quantification
for subspace estimation, namely, how to construct valid and hopefully
optimal confidence regions that are likely to contain the unknown
column subspace? It would also be interesting to investigate how to
incorporate other structural prior (e.g.~sparsity) to further reduce
the sample complexity and/or improve the estimation accuracy. Finally,
another interesting avenue for future exploration is the  extension to distributed or decentralized settings. 

\section*{Acknowledgements}

Y.~Chen is supported in part by the grants AFOSR YIP award FA9550-19-1-0030,
ONR N00014-19-1-2120, ARO YIP award W911NF-20-1-0097, ARO W911NF-18-1-0303, NSF DMS-2014279, CCF-1907661 and IIS-1900140,
and the Princeton SEAS Innovation Award. Y.~Chi is supported in part
by the grants ONR N00014-18-1-2142 and N00014-19-1-2404, ARO W911NF-18-1-0303,
and NSF CCF-1806154 and ECCS-1818571. H.~V.~Poor is supported in
part by the NSF grant DMS-1736417, CCF-1908308, and a Princeton Schmidt Data-X Research Award. C.~Cai is supported in part by Gordon Y.~S.~Wu Fellowships in Engineering.

\appendix

\section{Proofs for corollaries}

\subsection{Proof of Corollary~\ref{cor:tensor-completion}\label{subsec:pf:tc}}

Recall that the spectral algorithm considered herein (cf.~Section
\ref{subsec:tensor-completion}) operates upon the noisy copy of the
mode-1 matricization of the tensor $\bm{T}^{\star}$, namely,
\begin{equation}
\bm{A}^{\star}=\sum_{s=1}^{r}\bm{w}_{s}^{\star}\left(\bm{w}_{s}^{\star}\otimes\bm{w}_{s}^{\star}\right)^{\top}.\label{eq:A-expression-TC}
\end{equation}
Consequently, in order to apply Theorem~\ref{thm:U_loss}, the main
step boils down to estimating the spectrum and the incoherence parameters
of $\bm{A}^{\star}$. Specifically, we need to upper bound the condition
number $\kappa$, as well as the incoherence parameters $\mu_{0},\mu_{1}$
and $\mu_{2}$ as introduced in Definition \ref{definition-mu0-mu1-mu2}.

Before proceeding, we introduce a few notations that simplify the
presentation. Define
\[
\lambda_{i}^{\star}:=\left\Vert \bm{w}_{i}^{\star}\right\Vert _{2}^{3},\qquad1\leq i\leq r,
\]
and let $\lambda_{\left(i\right)}^{\star}$ denote the $i$-th largest
value in $\left\{ \lambda_{i}^{\star}\right\} _{i=1}^{r}$. We also
recall that
\[
\lambda_{\min}^{\star}:=\min_{1\leq i\leq r}\left\Vert \bm{w}_{i}^{\star}\right\Vert _{2}^{3}\qquad\text{and}\qquad\lambda_{\max}^{\star}:=\max_{1\leq i\leq r}\left\Vert \bm{w}_{i}^{\star}\right\Vert _{2}^{3}.
\]
In addition, we define two matrices of interest
\[
\overline{\bm{W}}^{\star}:=\left[\overline{\bm{w}}_{1}^{\star},\cdots,\overline{\bm{w}}_{r}^{\star}\right]\in\mathbb{R}^{d\times r},\qquad\widetilde{\bm{W}}^{\star}:=\left[\overline{\bm{w}}_{1}^{\star}\otimes\overline{\bm{w}}_{1}^{\star},\cdots,\overline{\bm{w}}_{r}^{\star}\otimes\overline{\bm{w}}_{r}^{\star}\right]\in\mathbb{R}^{d^{2}\times r},
\]
where $\overline{\bm{w}}_{s}^{\star}:=\bm{w}_{s}^{\star}/\left\Vert \bm{w}_{s}^{\star}\right\Vert _{2}$,
and $\bm{a}\otimes\bm{b}:=\footnotesize\left[\begin{array}{c}
a_{1}\bm{b}\\
\vdots\\
a_{d}\bm{b}
\end{array}\right]$. In addition, let $\bm{D}^{\star}\in\mathbb{R}^{r\times r}$ be a
diagonal matrix with diagonal entries
\[
D_{s,s}^{\star}=\left\Vert \bm{w}_{s}^{\star}\right\Vert _{2},\qquad1\leq s\leq r.
\]
These allow us to express
\[
\bm{G}^{\star}=\bm{A}^{\star}\bm{A}^{\star\top}=\overline{\bm{W}}^{\star}\bm{D}^{\star3}\widetilde{\bm{W}}^{\star\top}\widetilde{\bm{W}}^{\star}\bm{D}^{\star3}\overline{\bm{W}}^{\star\top}.
\]

In the sequel, we begin by quantifying the spectrum of $\bm{G}^{\star}$,
which in turn allows us to understand the spectrum of $\bm{A}^{\star}$.
\begin{itemize}
\item We first look at the eigenvalues of the matrices $\overline{\bm{W}}^{\star\top}\overline{\bm{W}}^{\star}$
and $\widetilde{\bm{W}}^{\star\top}\widetilde{\bm{W}}^{\star}$. Towards
this, let us write
\begin{equation}
\overline{\bm{W}}^{\star\top}\overline{\bm{W}}^{\star}=\bm{I}_{r}+\bm{C},\qquad\text{and}\qquad\widetilde{\bm{W}}^{\star\top}\widetilde{\bm{W}}^{\star}=\bm{I}_{r}+\widetilde{\bm{C}}\label{eq:W_top_gram_decomp}
\end{equation}
for some matrices $\bm{C},\widetilde{\bm{C}}\in\mathbb{R}^{r\times r}$.
It follows immediately from the incoherence assumption~(\ref{asmp:tensor})
that
\[
\left\Vert \bm{C}\right\Vert _{\infty}\leq\sqrt{\mu_{5}/d}\qquad\text{and}\qquad\big\|\widetilde{\bm{C}}\big\|_{\infty}\leq\mu_{5}/d,
\]
thus leading to the simple bounds
\begin{equation}
\left\Vert \bm{C}\right\Vert \leq r\left\Vert \bm{C}\right\Vert _{\infty}\leq r\sqrt{\mu_{5}/d},\qquad\big\|\widetilde{\bm{C}}\big\|\leq r \, \big\|\widetilde{\bm{C}}\big\|_{\infty}\leq\mu_{5}r/d.\label{eq:C_op_UB}
\end{equation}
These taken collectively with \eqref{eq:W_top_gram_decomp} and Weyl's
inequality yield
\[
\max_{i\in\left[r\right]}\Big|\lambda_{i}\big(\overline{\bm{W}}^{\star\top}\overline{\bm{W}}^{\star}\big)-1\Big|\leq\left\Vert \bm{C}\right\Vert \leq r\sqrt{\mu_{5}/d}\qquad\text{and}\qquad\max_{i\in\left[r\right]}\Big|\lambda_{i}\big(\widetilde{\bm{W}}^{\star\top}\widetilde{\bm{W}}^{\star}\big)-1\Big|\leq\big\|\widetilde{\bm{C}}\big\|\leq\mu_{5}r/d,
\]
which essentially tell us that
\begin{equation}
\big\|\overline{\bm{W}}^{\star}\big\|=\sqrt{\lambda_{1}\big(\overline{\bm{W}}^{\star\top}\overline{\bm{W}}^{\star}\big)}\leq\sqrt{1+r\sqrt{\mu_{5}/d}}\qquad\text{and}\qquad\big\|\widetilde{\bm{W}}^{\star}\big\|=\sqrt{\lambda_{1}\big(\widetilde{\bm{W}}^{\star\top}\widetilde{\bm{W}}^{\star}\big)}\leq\sqrt{1+\mu_{5}r/d}.\label{eq:W_op_UB}
\end{equation}
\item Returning to $\bm{G}^{\star}$, one invokes the definition \eqref{eq:W_top_gram_decomp}
to deduce that
\[
\bm{G}^{\star}=\overline{\bm{W}}^{\star}\bm{D}^{\star6}\overline{\bm{W}}^{\star\top}+\overline{\bm{W}}^{\star}\bm{D}^{\star3}\widetilde{\bm{C}}\bm{D}^{\star3}\overline{\bm{W}}^{\star\top}.
\]
Observe that the eigenvalues of $\overline{\bm{W}}^{\star}\bm{D}^{\star3}\big(\overline{\bm{W}}^{\star}\bm{D}^{\star3}\big)^{\top}$
are identical to those of $\big(\overline{\bm{W}}^{\star}\bm{D}^{\star3}\big)^{\top}\overline{\bm{W}}^{\star}\bm{D}^{\star3}$,
where the latter can be further decomposed as follows (in view of
\eqref{eq:W_top_gram_decomp})
\[
\big(\overline{\bm{W}}^{\star}\bm{D}^{\star3}\big)^{\top}\overline{\bm{W}}^{\star}\bm{D}^{\star3}=\bm{D}^{\star3}\overline{\bm{W}}^{\star\top}\overline{\bm{W}}^{\star}\bm{D}^{\star3}=\bm{D}^{\star6}+\bm{D}^{\star3}\bm{C}\bm{D}^{\star3}.
\]
This taken together with Weyl's inequality, \eqref{eq:C_op_UB} and
\eqref{eq:W_op_UB} shows that
\[
\left|\lambda_{i}\big(\overline{\bm{W}}^{\star}\bm{D}^{\star6}\overline{\bm{W}}^{\star\top}\big)-\lambda_{i}\big(\bm{D}^{\star6}\big)\right|\leq\left\Vert \bm{D}^{\star3}\bm{C}\bm{D}^{\star3}\right\Vert \leq\left\Vert \bm{D}^{\star}\right\Vert ^{6}\left\Vert \bm{C}\right\Vert \leq r\sqrt{\frac{\mu_{5}}{d}}\,\lambda_{\max}^{\star2}
\]
for each $1\leq i\leq r$. In addition, 
\[
\left|\lambda_{i}\big(\bm{G}^{\star}\big)-\lambda_{i}\big(\overline{\bm{W}}^{\star}\bm{D}^{\star6}\overline{\bm{W}}^{\star\top}\big)\right|\leq\big\|\overline{\bm{W}}^{\star}\bm{D}^{\star3}\widetilde{\bm{C}}\bm{D}^{\star3}\overline{\bm{W}}^{\star\top}\big\|\leq\big\|\overline{\bm{W}}^{\star}\big\|^{2}\left\Vert \bm{D}^{\star}\right\Vert ^{6}\big\|\widetilde{\bm{C}}\big\|\leq\frac{\mu_{5}r}{d}\left(1+r\sqrt{\frac{\mu_{5}}{d}}\right)\lambda_{\max}^{\star2}.
\]
As a result, invoke the triangle inequality to see that 
\begin{align*}
\left|\lambda_{i}\left(\bm{G}^{\star}\right)-\lambda_{i}\left(\bm{D}^{\star6}\right)\right| & \leq\left|\lambda_{i}\big(\bm{G}^{\star}\big)-\lambda_{i}\big(\overline{\bm{W}}^{\star}\bm{D}^{\star6}\overline{\bm{W}}^{\star\top}\big)\right|+\left|\lambda_{i}\big(\overline{\bm{W}}^{\star}\bm{D}^{\star6}\overline{\bm{W}}^{\star\top}\big)-\lambda_{i}\big(\bm{D}^{\star6}\big)\right|\\
 & \leq\frac{\mu_{5}r}{d}\left(1+r\sqrt{\frac{\mu_{5}}{d}}\right)\lambda_{\max}^{\star2}+r\sqrt{\frac{\mu_{5}}{d}}\,\lambda_{\max}^{\star2}\leq3r\sqrt{\frac{\mu_{5}}{d}}\,\lambda_{\max}^{\star2}
\end{align*}
for each $1\leq i\leq r$, where the last inequality holds under the
assumption that $r\sqrt{\mu_{5}/d}\leq1$. This means
\begin{align*}
\left|\lambda_{i}\left(\bm{G}^{\star}\right)-\lambda_{\left(i\right)}^{\star2}\right| & \leq3r\sqrt{\frac{\mu_{5}}{d}}\,\lambda_{\max}^{\star2},
\end{align*}
where $\lambda_{\left(i\right)}^{\star}$ denotes the $i$-th largest
value in $\left\{ \lambda_{i}^{\star}\right\} _{i=1}^{r}$.
\item Recalling that $\mu_{\mathsf{tc}}:=\max\left\{ \mu_{3},\mu_{4}^{2} \right\} $,
$\kappa_{\mathsf{tc}}:=\lambda_{\max}^{\star}/\lambda_{\min}^{\star}$
and the rank assumption $r\ll \kappa_{\mathsf{tc}}^{-2}\sqrt{d/\mu_{5}}$,
we find that
\begin{align}
\lambda_{i}\left(\bm{G}^{\star}\right)=\lambda_{\left(i\right)}^{\star2}+O\left(r\sqrt{\frac{\mu_{5}}{d}}\right)\lambda_{\max}^{\star2}\qquad\text{and}\qquad\sigma_{i}\left(\bm{A}^{\star}\right)=\lambda_{\left(i\right)}^{\star}\left(1+o\left(1\right)\right). \label{eq:A-tc-spectrum}
\end{align}
As a result, we immediately arrive at
\[
\qquad\sigma_{1}\left(\bm{A}^{\star}\right)=\lambda_{\max}^{\star}\left(1+o\left(1\right)\right),\qquad\sigma_{r}\left(\bm{A}^{\star}\right)=\lambda_{\min}^{\star}\left(1+o\left(1\right)\right),\qquad\text{and}\qquad\kappa=\frac{\sigma_{1}\left(\bm{A}^{\star}\right)}{\sigma_{r}\left(\bm{A}^{\star}\right)}\lesssim\kappa_{\mathsf{tc}}.
\]
\end{itemize}
Next, we turn attention to bounding the incoherence parameters of
$\bm{A}^{\star}$. Let $\bm{A}^{\star}=\bm{U}^{\star}\bm{\Sigma}^{\star}\bm{V}^{\star\top}$
be the (compact) SVD of $\bm{A}^{\star}$. It is seen from (\ref{eq:A-expression-TC})
that the column space of $\bm{U}^{\star}$ (resp.~$\bm{V}^{\star}$)
coincides with the column space of $\overline{\bm{W}}^{\star}$ (resp.~$\widetilde{\bm{W}}^{\star}$).
Therefore, there exist orthonormal matrices $\bm{H}_{1}$ and $\bm{H}_{2}$
such that
\[
\bm{U}^{\star}\bm{H}_{1}=\overline{\bm{W}}^{\star}\big(\overline{\bm{W}}^{\star\top}\overline{\bm{W}}^{\star}\big)^{-1/2}\qquad\text{and}\qquad\bm{V}^{\star}\bm{H}_{2}=\widetilde{\bm{W}}^{\star}\big(\widetilde{\bm{W}}^{\star\top}\widetilde{\bm{W}}^{\star}\big)^{-1/2}.
\]
These allow us to bound
\begin{align*}
\left\Vert \bm{U}^{\star}\right\Vert _{2,\infty} & =\left\Vert \bm{U}^{\star}\bm{H}_{1}\right\Vert _{2,\infty}\leq\big\|\overline{\bm{W}}^{\star}\big\|_{2,\infty}\big\|\big(\overline{\bm{W}}^{\star\top}\overline{\bm{W}}^{\star}\big)^{-1/2}\big\|\leq\sqrt{\frac{\mu_{4}r}{d}}\sqrt{\frac{1}{\lambda_{r}\big(\overline{\bm{W}}^{\star\top}\overline{\bm{W}}^{\star}\big)}}\lesssim\sqrt{\frac{\mu_{4}r}{d}}\sqrt{\frac{1}{1-1/3}}\leq\sqrt{\frac{2\mu_{4}r}{d}},\\
\left\Vert \bm{V}^{\star}\right\Vert _{2,\infty} & =\left\Vert \bm{V}^{\star}\bm{H}_{2}\right\Vert _{2,\infty}\leq\big\|\widetilde{\bm{W}}^{\star}\big\|_{2,\infty}\big\|\big(\widetilde{\bm{W}}^{\star\top}\widetilde{\bm{W}}^{\star}\big)^{-1/2}\big\|\leq\sqrt{\frac{\mu_{4}^{2}r}{d^2}}\sqrt{\frac{1}{\lambda_{r}\big(\widetilde{\bm{W}}^{\star\top}\widetilde{\bm{W}}^{\star}\big)}}\leq\sqrt{\frac{\mu_{4}^{2}r}{d^2}}\sqrt{\frac{1}{1-1/3}}\leq\sqrt{\frac{2\mu_{4}^{2}r}{d^2}},
\end{align*}
which follow from \eqref{eq:W_op_UB} and the assumption that $r\ll\sqrt{d/\mu_{5}}$.
Moreover, the incoherence assumption~(\ref{asmp:tensor}) gives that
\[
\mu_0 = \frac{d^{3} \left\| \bm{A}^\star \right\|_{\infty}^2 }{\left\| \bm{A}^\star \right\|_{\mathrm{F}}^2 } = \frac{d^{3} \left\| \bm{T}^\star \right\|_{\infty}^2 }{\left\| \bm{T}^\star \right\|_{\mathrm{F}}^2 } \leq \mu_3.
\]

To conclude, the above analysis reveals that
\[
	\mu_{0}\leq \mu_{3},\qquad\mu_{1}\lesssim\mu_{4},\qquad\mu_{2}\lesssim\mu_{4}^{2},\qquad\mu\lesssim\max\left\{ \mu_{3},\mu_{4}^{2}\right\} = \mu_{\mathsf{tc}} \qquad\text{and}\qquad\kappa\lesssim\kappa_{\mathsf{tc}},
\]
where $\mu=\max\left\{ \mu_{0},\mu_{1},\mu_{2}\right\} $ and $\kappa=\sigma_{1}\left(\bm{A}^{\star}\right)/\sigma_{r}\left(\bm{A}^{\star}\right)$.
With these estimates in place, Corollary~\ref{cor:tensor-completion}
follows immediately from Theorem~\ref{thm:U_loss}.

\subsection{Proof of Corollary~\ref{cor:cov}\label{subsec:pf:cov}}

In the problem of covariance estimation with missing data, the ground
truth $\bm{A}^{\star}$ is effectively given by $\bm{B}^{\star}\bm{F}^{\star}$,
which obeys
\[
\bm{A}^{\star}=\bm{B}^{\star}\bm{F}^{\star}=\bm{U}^{\star}\bm{\Lambda}^{\star1/2}\bm{F}^{\star}\in\mathbb{R}^{d\times n},\qquad\bm{F}^{\star}=\left[\bm{f}_{1}^{\star},\cdots,\bm{f}_{n}^{\star}\right]\in\mathbb{R}^{r\times n}
\]
with $\bm{f}_{i}^{\star}\overset{\mathrm{i.i.d.}}{\sim}\mathcal{N}(\bm{0},\bm{I}_{r})$.We
note that by our assumption on the sample size, one has $n\gg\kappa_{\mathsf{ce}}^{2}\left(r+\log d\right)$,
where $\kappa_{\mathsf{ce}}=\lambda_{1}^{\star}/\lambda_{r}^{\star}$.
In addition, we note that under the assumption of Corollary~\ref{cor:cov},
one has
\begin{equation}
	\mathcal{E}_{\mathsf{ce}}\ll \kappa_{\mathsf{ce}}^{-1} \leq1,\label{eq:err-CE-UB}
\end{equation}
where $\mathcal{E}_{\mathsf{ce}}$ and $\kappa_{\mathsf{ce}}$
are defined in \eqref{def:err-CE} and \eqref{def:kappa_ce}, respectively.

\subsubsection{Estimation error of the principle subspace}

In this section, we will prove \eqref{claim:U_op_loss_cov} and \eqref{claim:U_2inf_loss_cov}. To begin the proof,  we verify the condition of the random
noise (cf.~\eqref{eq:assumption-noise-spike}). From standard
Gaussian concentration results, one is allowed to choose $R\asymp\sigma\sqrt{\log (n+d)}$, so that $|\eta_{i,j}|\leq R$ for all $i$ and $j$  
with probability $1-O((n+d)^{-12})$. Under our sample size condition
that $n\gg\max\big\{\frac{\log^{6}\left(n+d\right)}{dp^{2}}, \frac{\log^3\left(n+d\right)}{p}\big\}$, the requirement \eqref{eq:assumption-noise-spike}
is satisfied, namely, 
\[
	\frac{R^{2}}{\sigma^{2}}\asymp\log (n+d) \lesssim\min\left\{\frac{p\sqrt{dn}}{\log\left(n+d\right)}, \frac{pn}{\log\left(n+d\right)} \right\}.
\]

Next, we turn to the properties of $\bm{B}^{\star}\bm{F}^{\star}$
and start by looking at its spectrum. Define
\[
\bm{C}:=\bm{F}^{\star}\bm{F}^{\star\top}-n\bm{I}_{r},
\]
which allows us to write
\begin{equation}
\bm{G}^{\star}=\bm{B}^{\star}\bm{F}^{\star}\left(\bm{B}^{\star}\bm{F}^{\star}\right)^{\top}=\bm{U}^{\star}\bm{\Lambda}^{\star1/2}\bm{F}^{\star}\bm{F}^{\star\top}\bm{\Lambda}^{\star1/2}\bm{U}^{\star\top}=n\underbrace{\bm{U}^{\star}\bm{\Lambda}^{\star}\bm{U}^{\star\top}}_{= \bm{S}^{\star}}+\underbrace{\bm{U}^{\star}\bm{\Lambda}^{\star1/2}\bm{C}\bm{\Lambda}^{\star1/2}\bm{U}^{\star\top}}_{=: \, \bm{\Delta}}.\label{eq:G_decomp_ce}
\end{equation}
 Using standard results on Gaussian random matrices \cite{Vershynin2012},
one obtains
\begin{align}
\left\Vert \bm{C}\right\Vert  & \lesssim\max\left\{ \sqrt{n}\,\big(\sqrt{r}+\sqrt{\log\left(n+d\right)}\big),r+\log\left(n+d\right)\right\} \asymp\sqrt{n}\,\big(\sqrt{r}+\sqrt{\log \left(n+d\right)}\big),\nonumber \\
\left|\sigma_{i}\left(\bm{F}^{\star}\right)-\sqrt{n}\right| & \lesssim\sqrt{r}+\sqrt{\log \left(n+d\right)}\label{eq:F_spectrum}
\end{align}
with probability at least $1-O\big(\left(n+d\right)^{-10}\big)$,  provided that
$n\gg r+\log (n+d)$. It then follows from Weyl's inequality that 
\begin{align}
\left|\lambda_{i}\left(\bm{G}^{\star}\right)-\lambda_{i}\left(n\bm{S}^{\star}\right)\right|=\left|\lambda_{i}\left(\bm{G}^{\star}\right)-\lambda_{i}^{\star}n\right| & \leq\big\|\bm{\Delta}\big\|\leq\left\Vert \bm{C}\right\Vert \left\Vert \bm{U}^{\star}\right\Vert ^{2}\left\Vert \bm{\Lambda}^{\star}\right\Vert \lesssim\lambda_{1}^{\star}\sqrt{n}\,\big(\sqrt{r}+\sqrt{\log \left(n+d\right)}\big).\label{eq:Delta_op_ce}
\end{align}
Under the sample size assumption $n\gg\kappa_{\mathsf{ce}}^{2}\left(r+\log \left(n+d\right)\right)$,
we conclude that 
\begin{equation}
\lambda_{i}\left(\bm{G}^{\star}\right)=\lambda_{i}^{\star}n\left(1+o\left(1\right)\right)\qquad\text{and}\qquad\sigma_{i}\left(\bm{B}^{\star}\bm{F}^{\star}\right)=\sqrt{\lambda_{i}^{\star}n}\left(1+o\left(1\right)\right),\label{eq:BF_spectrum}
\end{equation}
and hence
\begin{align}
\kappa\left(\bm{G}^{\star}\right)=\frac{\lambda_{1}\left(\bm{G}^{\star}\right)}{\lambda_{r}\left(\bm{G}^{\star}\right)}\asymp\kappa_{\mathsf{ce}}\qquad\text{and}\qquad\kappa\left(\bm{B}^{\star}\bm{F}^{\star}\right)=\frac{\sigma_{1}\left(\bm{B}^{\star}\bm{F}^{\star}\right)}{\sigma_{r}\left(\bm{B}^{\star}\bm{F}^{\star}\right)}\asymp\sqrt{\kappa_{\mathsf{ce}}}.
	\label{eq:kappa-ce-bound}
\end{align}

Further, we look at the entrywise infinity norm of $\bm{B}^{\star}\bm{F}^{\star}$.
From standard Gaussian concentration inequalities,
\begin{align*}
\left\Vert \bm{B}^{\star}\bm{F}^{\star}\right\Vert _{\infty} & =\max\nolimits _{i,j}\left|\left\langle \big(\bm{U}^{\star}\bm{\Lambda}^{\star1/2}\big)_{i,:},\bm{f}_{j}^{\star}\right\rangle \right|\lesssim\big\|\bm{U}^{\star}\bm{\Lambda}^{\star1/2}\big\|_{2,\infty}\sqrt{\log \left(n+d\right)}\\
	& \leq\left\Vert \bm{U}^{\star}\right\Vert _{2,\infty}\left\Vert \bm{\Lambda}^{\star}\right\Vert ^{1/2}\sqrt{\log \left(n+d\right)}
        \leq\sqrt{\frac{\lambda_{1}^{\star}\mu_{\mathsf{ce}}r\log \left(n+d\right)}{d}}
\end{align*}
holds with probability at least $1-O\big(\left(n+d\right)^{-10}\big)$. Meanwhile, one has
\begin{align*}
	\left\Vert \bm{B}^{\star}\bm{F}^{\star}\right\Vert_{\mathrm{F}} \geq \left\Vert \bm{B}^{\star}\right\Vert_{\mathrm{F}} \sigma_r (\bm{F}^\star) = \big\Vert \bm{\Lambda}^{\star 1/2}\big\Vert_{\mathrm{F}} \sigma_r (\bm{F}^\star)  \gtrsim \sqrt{\lambda_r^\star r n} = \sqrt{ \frac{1}{{\kappa_{\mathsf{ce}}}} \lambda_1^\star r n},
\end{align*}
where the last step follows from \eqref{eq:BF_spectrum} and \eqref{eq:kappa-ce-bound}. 
As a result,
\begin{equation}
	\left\Vert \bm{B}^{\star}\bm{F}^{\star}\right\Vert _{\infty} \leq  \sqrt{\frac{\mu_{\mathsf{ce}}\kappa_{\mathsf{ce}}\log(n+d)}{nd}} \left\Vert \bm{B}^{\star}\bm{F}^{\star}\right\Vert_{\mathrm{F}}
\end{equation}
Recalling the
definition of $\mu_{0}$ in \eqref{eq:def-col-mu0}, one obtains
\begin{equation}
	\label{eq:mu0-ub-ce}
\mu_{0}\lesssim \mu_{\mathsf{ce}} \kappa_{\mathsf{ce}} \log \left(n+d\right).
\end{equation}

When it comes to the incoherence parameters $\mu_{1}$ and $\mu_{2}$ (cf.~\eqref{def:coh}),
it can be easily verified that 
\[
\mu_{1}=\frac{d}{r}\left\|\bm{U}^{\star}\right\|_{2,\infty}^{2}=\mu_{\mathsf{ce}}.
\]
In addition, recognizing the existence of an orthonormal matrix $\bm{H}_2$
such that $\bm{V}^{\star}\bm{H}_2 = \bm{F}^{\star\top}(\bm{F}^{\star}\bm{F}^{\star\top})^{-1/2}$,
we can bound
\begin{align*}
\left\Vert \bm{V}^{\star}\right\Vert _{2,\infty} & =\left\Vert \bm{V}^{\star}\bm{H}_{2}\right\Vert _{2,\infty}\leq\left\Vert \bm{F}^{\star\top}\right\Vert _{2,\infty}\big\Vert \left(\bm{F}^{\star}\bm{F}^{\star\top}\right)^{-1/2}\big\Vert \overset{(\text{i})}{\leq}\frac{\sqrt{r}+\sqrt{\log \left(n+d\right)}}{\sigma_{r}\left(\bm{F}^{\star}\right)} \\
	& \overset{\left(\mathrm{ii}\right)}{\lesssim}\frac{\sqrt{r}+\sqrt{\log \left(n+d\right)}}{\sqrt{n}-\sqrt{r}-\sqrt{\log \left(n+d\right)}}\overset{\left(\mathrm{iii}\right)}{\lesssim}\frac{\sqrt{r\log \left(n+d\right)}}{\sqrt{n}},
\end{align*}
where (i) follows from the standard Gaussian concentration result
that $\big|\left\Vert \bm{F}^{\star\top}\right\Vert _{2,\infty}-\sqrt{r}\big|\lesssim\sqrt{\log \left(n+d\right)}$
with probability $1-O\big(\left(n+d\right)^{-20}\big)$, (ii) arises from \eqref{eq:F_spectrum},
and (iii) holds true under our sample size assumption. Consequently, we obtain
\[
\mu_{2}=\frac{n}{r}\left\Vert \bm{V}^{\star}\right\Vert _{2,\infty}^{2}\lesssim\log \left(n+d\right).
\]

Thus far,  we have shown that
\[
\sigma_r^\star \asymp \sqrt{\lambda_r^\star n}, \quad \mu_{0}\lesssim\mu_{\mathsf{ce}}\kappa_{\mathsf{ce}}\log \left(n+d\right),\quad\mu_{1}=\mu_{\mathsf{ce}},
\quad\mu_{2}\lesssim\log \left(n+d\right),\quad\mu\lesssim\mu_{\mathsf{ce}}\kappa_{\mathsf{ce}}\log \left(n+d\right)
\quad\text{and}\quad
\kappa\asymp\sqrt{\kappa_{\mathsf{ce}}},
\]
where $\sigma_r^\star = \sigma_r^\star \left(\bm{B}^{\star}\bm{F}^{\star}\right)$, $\mu=\max\left\{ \mu_{0},\mu_{1},\mu_{2}\right\} $ and $\kappa=\kappa\left(\bm{B}^{\star}\bm{F}^{\star}\right)$. Under the sample size assumption~\eqref{eq:sample-size-assumption-ce} and the rank condition $r \ll \frac{d}{\mu_{\mathsf{ce}} \kappa_{\mathsf{ce}}^2}$, it is straightforward to verify the condition~\eqref{asmp} is satisfied, i.e.
\begin{align*}
p & \gg \max \left\{ \frac{\mu_{\mathsf{ce}} \kappa_{\mathsf{ce}}^3 r \log^3 \left( n + d \right) }{\sqrt{dn}}, \, \frac{\mu_{\mathsf{ce}} \kappa_{\mathsf{ce}}^5 r \log^3 \left( n + d \right)}{n} \right\} 
\gtrsim \max \left\{ \frac{\mu \kappa^4 r \log^2 \left( n + d \right) }{\sqrt{dn}}, \, \frac{\mu \kappa^8 r \log^2 \left( n + d \right)}{n} \right\}, \\
\frac{\sigma}{\sigma_r^\star} & \asymp \frac{\sigma}{\sqrt{\lambda_r^\star n}} \ll \min \left\{ \frac{\sqrt{p}}{\sqrt{\kappa_{\mathsf{ce}}}\sqrt[4]{dn} \sqrt{\log\left(n+d\right)}} ,\, \sqrt{\frac{p}{\kappa_{\mathsf{ce}}^3 d \log \left(n+d\right)}}  \right\} \lesssim \min\left\{ \frac{\sqrt{p}}{\kappa\sqrt[4]{d_{1}d_{2}}\,\sqrt{\log d}},\frac{1}{\kappa^{3}}\sqrt{\frac{p}{d_{1}\log d}}\right\}, \\
r &\ll \frac{d}{\mu_{\mathsf{ce}} \kappa_{\mathsf{ce}}^2} \asymp \frac{d}{\mu_1 \kappa^4}.
\end{align*}
Applying Theorem \ref{thm:U_loss} immediately establishes the claims
\eqref{claim:U_op_loss_cov} and \eqref{claim:U_2inf_loss_cov} in
Corollary \ref{cor:cov}. Along the way, we have also established the following upper bound (see Lemma \ref{lemma:G_op_loss}), which will be useful in the sequel:
\begin{align}
	\left\Vert \bm{G}-\bm{G}^{\star}\right\Vert  & \lesssim\lambda_{r}\left(\bm{G}^{\star}\right)\cdot\mathcal{E}_{\mathsf{ce}} \asymp \lambda_r^\star n \cdot \mathcal{E}_{\mathsf{ce}}. 
	\label{eq:G_op_loss_ce}
\end{align}
Here, we recall that $\bm{G}=\frac{1}{p^{2}}\mathcal{P}_{\mathsf{off}\text{-}\mathsf{diag}}\left(\mathcal{P}_{\Omega}(\bm{X})\mathcal{P}_{\Omega}(\bm{X})^{\top}\right)$.

\subsubsection{Estimation error of the covariance matrix}
It remains to prove \eqref{claim:S_op_loss} and \eqref{claim:S_inf_loss}.
Before proceeding, we first recall that $\bm{U}\bm{\Lambda}\bm{U}^\top$ is the top-$r$ eigendecomposition of $\bm{G}$, 
\begin{align}
	\label{eq:defn-Sigma-B-R}
	\bm{\Sigma} = \bm{\Lambda}^{1/2}, \quad \bm{B}=\frac{1}{\sqrt{n}}\bm{U}\bm{\Sigma}, \quad
	\bm{B}^{\star}=\bm{U}^{\star}\bm{\Lambda}^{\star1/2} 
	\quad \text{and} \quad 
	\bm{R}=\underset{\bm{Q}\in\mathcal{O}^{r\times r}}{\arg\min}\left\Vert \bm{U}\bm{Q}-\bm{U}^{\star}\right\Vert _{\mathrm{F}}.
\end{align}
Let us also define
\[
\bm{K}:=\underset{\bm{Q}\in\mathcal{O}^{r\times r}}{\arg\min}\left\Vert \bm{B}\bm{Q}-\bm{B}^{\star}\right\Vert _{\mathrm{F}}.
\]
It is well known that the minimizer $\bm{K}$ is given by \cite{ten1977orthogonal}
\[
\bm{K}=\mathsf{sgn}\left(\bm{B}^{\top}\bm{B}^{\star}\right),
\]
where the $\mathsf{sgn}(\cdot)$ function is defined in \eqref{def:H_sgn}. Since $\bm{K}$
is an orthonormal matrix, one can express
\begin{align}
\bm{S}-\bm{S}^{\star}=\left(\bm{B}\bm{K}\right)\left(\bm{B}\bm{K}\right)^{\top}-\bm{B}^{\star}\bm{B}^{\star\top}=\left(\bm{B}\bm{K}-\bm{B}^{\star}\right)\left(\bm{B}\bm{K}\right)^{\top}+\bm{B}^{\star}\left(\bm{B}\bm{K}-\bm{B}^{\star}\right)^{\top}.
	\label{eq:S_S_star_diff}
\end{align}
As a result, everything boils down to controlling $\left\Vert \bm{B}\bm{K}-\bm{B}^{\star}\right\Vert $
and $\left\Vert \bm{B}\bm{K}-\bm{B}^{\star}\right\Vert _{2,\infty}$.
To this end, we use \eqref{eq:defn-Sigma-B-R} to reach the following useful decomposition
\begin{equation}
\bm{BK}-\bm{B}^{\star}=\frac{1}{\sqrt{n}}\bm{U}\bm{\Lambda}^{1/2}\left(\bm{K}-\bm{R}\right)+\bm{U}\left(\frac{1}{\sqrt{n}}\bm{\Lambda}^{1/2}\bm{R}-\bm{R}\bm{\Lambda}^{\star1/2}\right)+\left(\bm{U}\bm{R}-\bm{U}^{\star}\right)\bm{\Lambda}^{\star1/2} .
	\label{eq:BK_B_star_decomp}
\end{equation}

Given that $\bm{U}\frac{1}{n}\bm{\Lambda}\bm{U}^{\top}$ is the top-$r$
eigendecomposition of $\frac{1}{n}\bm{G}$, an important step lies in controlling the
difference between $\frac{1}{n}\bm{G}$ and $\bm{S}^{\star}$. Recalling the matrix
$\bm{\Delta}$ as defined in \eqref{eq:G_decomp_ce}, one can use \eqref{eq:Delta_op_ce},
\eqref{eq:G_op_loss_ce} as well as the definition of $\mathcal{E}_{\mathsf{ce}}$
(cf.~\eqref{def:err-CE}) to obtain
\begin{align*}
\left\Vert \frac{1}{n}\bm{G}-\bm{S}^{\star}\right\Vert  & \leq\frac{1}{n}\left\Vert \bm{G}-\bm{G}^{\star}\right\Vert +\left\Vert \frac{1}{n}\bm{G}^{\star}-\bm{S}^{\star}\right\Vert =\frac{1}{n}\left\Vert \bm{G}-\bm{G}^{\star}\right\Vert +\frac{1}{n}\left\Vert \bm{\Delta}\right\Vert \\
 & \lesssim\frac{1}{n}\lambda_{r}\left(\bm{G}^{\star}\right)\cdot\mathcal{E}_{\mathsf{ce}}+\frac{\lambda_{1}^{\star}}{\sqrt{n}}\big(\sqrt{r}+\sqrt{\log \left(n+d\right)}\big)\asymp\lambda_{r}^{\star}\cdot\mathcal{E}_{\mathsf{ce}},
\end{align*}
where the last inequality makes use of the identity $\lambda_r(\bm{G}^{\star}) \asymp n\lambda_r^{\star}$. 
Hence,  apply \cite[Lemma~46, Lemma~47]{ma2017implicit} (with slight
modification on $\kappa$) and Weyl's inequality to show that
\begin{align}
\left\Vert \frac{1}{\sqrt{n}}\bm{\Lambda}^{1/2}\bm{R}-\bm{R}\bm{\Lambda}^{\star1/2}\right\Vert  & \lesssim\frac{\kappa\left(\bm{S}^{\star}\right)}{\sqrt{\lambda_{r}\left(\bm{S}^{\star}\right)}}\left\Vert \frac{1}{n}\bm{G}-\bm{S}^{\star}\right\Vert \lesssim\kappa_{\mathsf{ce}}\sqrt{\lambda_{r}^{\star}}\cdot\mathcal{E}_{\mathsf{ce}};\label{eq:Lamb_R_R_Lamb_star}\\
\left\Vert \bm{K}-\bm{R}\right\Vert  & \lesssim\frac{\sqrt{\kappa\left(\bm{S}^{\star}\right)}}{\lambda_{r}\left(\bm{S}^{\star}\right)}\left\Vert \frac{1}{n}\bm{G}-\bm{S}^{\star}\right\Vert \lesssim\sqrt{\kappa_{\mathsf{ce}}}\cdot\mathcal{E}_{\mathsf{ce}}. 
	\label{eq:K_R_op_diff}
\end{align}
In addition, it follows from Weyl's inequality that
\[
\left\Vert \frac{1}{n}\bm{\Lambda}-\bm{\Lambda}^{\star}\right\Vert \leq\left\Vert \frac{1}{n}\bm{G}-\bm{S}^{\star}\right\Vert \lesssim\lambda_{r}^{\star}\cdot\mathcal{E}_{\mathsf{ce}},
\]
which combined with \eqref{eq:err-CE-UB} gives
\begin{equation}
\frac{1}{n}\left\Vert \bm{\Lambda}\right\Vert \leq\left\Vert \frac{1}{n}\bm{\Lambda}-\bm{\Lambda}^{\star}\right\Vert +\left\Vert \bm{\Lambda}^{\star}\right\Vert \lesssim\lambda_{r}^{\star}\cdot\mathcal{E}_{\mathsf{ce}}+\lambda_{1}^{\star}\asymp\lambda_{1}^{\star}\label{eq:Lambda_UB}
\end{equation}
under our assumptions. 

We are ready to upper bound the difference between $\bm{BK}-\bm{B}^{\star}$.
Plugging \eqref{claim:U_op_loss_cov}, \eqref{eq:Lamb_R_R_Lamb_star}
\eqref{eq:K_R_op_diff} and \eqref{eq:Lambda_UB} into \eqref{eq:BK_B_star_decomp}
shows that
\begin{align}
\left\Vert \bm{BK}-\bm{B}^{\star}\right\Vert  & \leq\frac{1}{\sqrt{n}}\left\Vert \bm{\Lambda}\right\Vert ^{1/2}\left\Vert \bm{K}-\bm{R}\right\Vert +\left\Vert \frac{1}{\sqrt{n}}\bm{\Lambda}^{1/2}\bm{R}-\bm{R}\bm{\Lambda}^{\star1/2}\right\Vert +\left\Vert \bm{U}\bm{R}-\bm{U}^{\star}\right\Vert \left\Vert \bm{\Lambda}^{\star}\right\Vert ^{1/2}\nonumber \\
 & \lesssim\sqrt{\kappa_{\mathsf{ce}}\lambda_{1}^{\star}}\cdot\mathcal{E}_{\mathsf{ce}}+\kappa_{\mathsf{ce}}\sqrt{\lambda_{r}^{\star}}\cdot\mathcal{E}_{\mathsf{ce}}+\sqrt{\lambda_{1}^{\star}}\cdot\mathcal{E}_{\mathsf{ce}}\nonumber \\
 & \lesssim\kappa_{\mathsf{ce}}\sqrt{\lambda_{1}^{\star}}\cdot\mathcal{E}_{\mathsf{ce}}.\label{eq:BK_B_star_op_loss}
\end{align}
Since $\bm{K}\in\mathcal{O}^{r\times r}$,
this also implies that
\begin{equation}
\left\Vert \bm{B}\right\Vert =\left\Vert \bm{BK}\right\Vert \leq\left\Vert \bm{BK}-\bm{B}^{\star}\right\Vert +\left\Vert \bm{B}^{\star}\right\Vert \lesssim\kappa_{\mathsf{ce}}\sqrt{\lambda_{1}^{\star}}\cdot\mathcal{E}_{\mathsf{ce}}+\sqrt{\lambda_{1}^{\star}}\asymp\sqrt{\lambda_{1}^{\star}},
	\label{eq:B_op_UB}
\end{equation}
where the last step results from \eqref{eq:err-CE-UB}. In addition, \eqref{claim:U_2inf_loss_cov}, 
\eqref{eq:err-CE-UB} and the fact that $\bm{R}\in\mathcal{O}^{r\times r}$
guarantees that
\[
\left\Vert \bm{U}\right\Vert _{2,\infty}=\left\Vert \bm{U}\bm{R}\right\Vert _{2,\infty}\leq\left\Vert \bm{U}\bm{R}-\bm{U}^{\star}\right\Vert _{2,\infty}+\left\Vert \bm{U}^{\star}\right\Vert _{2,\infty}\lesssim\kappa_{\mathsf{ce}}^{3/2}\mathcal{E}_{\mathsf{ce}}\sqrt{\frac{\mu_{\mathsf{ce}}r\log \left(n+d\right)}{d}}+\sqrt{\frac{\mu_{\mathsf{ce}}r}{d}}\lesssim\sqrt{\frac{\kappa_{\mathsf{ce}}\mu_{\mathsf{ce}}r\log \left(n+d\right)}{d}}.
\]
Consequently, it follows from the decomposition \eqref{eq:BK_B_star_decomp} that 
\begin{align}
\left\Vert \bm{BK}-\bm{B}^{\star}\right\Vert _{2,\infty} & \leq\left\Vert \bm{U}\right\Vert _{2,\infty}\frac{1}{\sqrt{n}}\left\Vert \bm{\Lambda}\right\Vert ^{1/2}\left\Vert \bm{K}-\bm{R}\right\Vert +\left\Vert \bm{U}\right\Vert _{2,\infty}\left\Vert \frac{1}{\sqrt{n}}\bm{\Lambda}^{1/2}\bm{R}-\bm{R}\bm{\Lambda}^{\star1/2}\right\Vert +\left\Vert \bm{U}\bm{R}-\bm{U}^{\star}\right\Vert _{2,\infty}\left\Vert \bm{\Lambda}^{\star}\right\Vert ^{1/2}\nonumber \\
	& \lesssim\sqrt{\frac{\kappa_{\mathsf{ce}}^2\mu_{\mathsf{ce}}r\lambda_{1}^{\star}\log(n+d)}{d}}\,\mathcal{E}_{\mathsf{ce}}
	+\sqrt{\frac{\kappa_{\mathsf{ce}}^3\mu_{\mathsf{ce}}r\lambda_{r}^{\star}\log(n+d)}{d}}\,\mathcal{E}_{\mathsf{ce}}
	+\sqrt{\frac{\kappa_{\mathsf{ce}}^3\mu_{\mathsf{ce}}r\lambda_{1}^{\star}\log \left(n+d\right)}{d}}\,\mathcal{E}_{\mathsf{ce}}\nonumber \\
 & \asymp\sqrt{\frac{\kappa_{\mathsf{ce}}^3\mu_{\mathsf{ce}}r\lambda_{1}^{\star}\log \left(n+d\right)}{d}}\, \mathcal{E}_{\mathsf{ce}}.\label{eq:BK_B_star_2inf_loss}
\end{align}
Combining \eqref{eq:BK_B_star_2inf_loss} and \eqref{eq:err-CE-UB}
gives that
\begin{align}
\left\Vert \bm{B}\right\Vert _{2,\infty} & =\left\Vert \bm{BK}\right\Vert _{2,\infty}\leq\left\Vert \bm{BK}-\bm{B}^{\star}\right\Vert _{2,\infty}+\left\Vert \bm{B}^{\star}\right\Vert _{2,\infty}\lesssim\sqrt{\frac{\kappa_{\mathsf{ce}}^3\mu_{\mathsf{ce}}r\lambda_{1}^{\star}\log \left(n+d\right)}{d}}\, \mathcal{E}_{\mathsf{ce}}+\sqrt{\frac{\mu_{\mathsf{ce}}r\lambda_{1}^{\star}}{d}} \nonumber \\ 
& \lesssim \sqrt{\frac{\mu_{\mathsf{ce}}r\lambda_{1}^{\star}\log \left(n+d\right)}{d}},\label{eq:B_2inf_UB}
\end{align}
where we use the fact that $\left\Vert \bm{B}^{\star}\right\Vert _{2,\infty}=\left\Vert \bm{U}^{\star}\bm{\Lambda}^{\star1/2}\right\Vert _{2,\infty}\lesssim\left\Vert \bm{U}^{\star}\right\Vert _{2,\infty}\left\Vert \bm{\Lambda}^{\star}\right\Vert ^{1/2}\lesssim\sqrt{\mu_{\mathsf{ce}}r\lambda_{1}^{\star}/d}$ and $\mathcal{E}_{\mathsf{ce}} \ll \kappa_{\mathsf{ce}}^{-1}$.

To finish up, we substitute \eqref{eq:BK_B_star_op_loss} and \eqref{eq:B_op_UB}
into \eqref{eq:S_S_star_diff} to find that
\begin{align*}
\left\Vert \bm{S}-\bm{S}^{\star}\right\Vert  & \leq\left\Vert \bm{B}\bm{K}-\bm{B}^{\star}\right\Vert \left(\left\Vert \bm{B}^{\star}\right\Vert +\left\Vert \bm{B}\bm{K}\right\Vert \right)\leq\left\Vert \bm{B}\bm{K}-\bm{B}^{\star}\right\Vert \left(\left\Vert \bm{B}^{\star}\right\Vert +\left\Vert \bm{B}\right\Vert \right)\\
 & \lesssim\kappa_{\mathsf{ce}}\lambda_{1}^{\star}\cdot\mathcal{E}_{\mathsf{ce}}.
\end{align*}
Combining \eqref{eq:BK_B_star_2inf_loss} and \eqref{eq:B_2inf_UB}
reveals that
\begin{align*}
\left\Vert \bm{S}-\bm{S}^{\star}\right\Vert _{\infty} & \leq\left\Vert \bm{B}\bm{K}-\bm{B}^{\star}\right\Vert _{2,\infty}\big(\left\Vert \bm{B}\bm{K}\right\Vert _{2,\infty}+\left\Vert \bm{B}^{\star}\right\Vert _{2,\infty}\big)\leq\left\Vert \bm{B}\bm{K}-\bm{B}^{\star}\right\Vert _{2,\infty}\big(\left\Vert \bm{B}\right\Vert _{2,\infty}+\left\Vert \bm{B}^{\star}\right\Vert _{2,\infty}\big)\\
 & \lesssim\frac{\kappa_{\mathsf{ce}}^2\mu_{\mathsf{ce}}r\lambda_{1}^{\star}\log \left(n+d\right)}{d}\cdot\mathcal{E}_{\mathsf{ce}}.
\end{align*}
We have therefore established all claims.

\subsection{Proof of Corollary~\ref{cor:BSBM}\label{subsec:pf:bsbm}}

Recall from our calculation (\ref{eq:mean-A-biclustering}) that
\[
\bm{A}^{\star}:=\mathbb{E}[\bm{A}]=\frac{\left(q_{\mathsf{in}}-q_{\mathsf{out}}\right)\sqrt{n_{u}n_{v}}}{2}\bm{u}^{\star}\bm{v}^{\star\top}
\]
is a rank-1 matrix, where
\[
\bm{u}^{\star}:=\frac{1}{\sqrt{n_{u}}}\left[\begin{array}{c}
\bm{1}_{n_{u}/2}\\
-\bm{1}_{n_{u}/2}
\end{array}\right]\qquad\text{and}\qquad\bm{v}^{\star}:=\frac{1}{\sqrt{n_{v}}}\left[\begin{array}{c}
\bm{1}_{n_{v}/2}\\
-\bm{1}_{n_{v}/2}
\end{array}\right].
\]
Let $\bm{u}\in\mathbb{R}^{n_{u}}$ be the leading eigenvector of $\bm{G}$
(cf.~(\ref{eq:defn-G}) and Algorithm~\ref{alg:spectral-bsbm}). To establish Corollary~\ref{cor:BSBM},
the main step boils down to showing that, under the conditions of
Corollary~\ref{cor:BSBM}, 
\begin{align}
\min\left\{ \left\Vert \bm{u}-\bm{u}^{\star}\right\Vert _{\infty},\left\Vert \bm{u}+\bm{u}^{\star}\right\Vert _{\infty}\right\}  
	& \lesssim\frac{1}{\sqrt{n_{u}}}  \mathcal{E}_{{\mathsf{bsbm}}},
	\label{claim:U_2inf_loss_bsbm}
\end{align}
holds with probability exceeding $1-O\left(n^{-10}\right)$, where
\begin{equation}
	\mathcal{E}_{{\mathsf{bsbm}}} :=\frac{q_{\mathsf{in}}}{\left(q_{\mathsf{in}}-q_{\mathsf{out}}\right)^{2}}\frac{\log n}{\sqrt{n_{u}n_{v}}}+\frac{\sqrt{q_{\mathsf{in}}}}{q_{\mathsf{in}}-q_{\mathsf{out}}}\sqrt{\frac{\log n}{n_{v}}}+\frac{1}{\sqrt{n_{u}}}.
	\label{def:err-bsbm}
\end{equation}
If this claim (\ref{claim:U_2inf_loss_bsbm}) holds, then under our
condition (\ref{eq:BSBM-assump}) one has $\mathcal{E}_{{\mathsf{bsbm}}} \ll 1$, and hence
\[
\min\big\{ \left\Vert \bm{u}-\bm{u}^{\star}\right\Vert _{\infty},\left\Vert \bm{u}+\bm{u}^{\star}\right\Vert _{\infty}\big\} 
\lesssim\frac{1}{\sqrt{n_{u}}} \mathcal{E}_{{\mathsf{bsbm}}} <\frac{1}{\sqrt{n_{u}}}.
\]
In other words, one has either $\mathsf{sign}(u_{i})=\mathsf{sign}(u_{i}^{\star})$
for all $1\leq i\leq n_{u}$, or $\mathsf{sign}(u_{i})=-\mathsf{sign}(u_{i}^{\star})$
for all $1\leq i\leq n_{u}$. This tells us that the entrywise rounding
operation applied to $\bm{u}$ is sufficient to recover exactly the
community memberships of all nodes in $\mathcal{U}$.

The rest of the proof is devoted to establishing the claim (\ref{claim:U_2inf_loss_bsbm}).
In order to apply Theorem~\ref{thm:U_loss}, it suffices to estimate
the spectrum and the incoherence parameters of $\bm{A}^{\star}$,
as well as some simple statistical properties of $\bm{N}:=\bm{A}-\bm{A}^{\star}$.
\begin{itemize}
\item We begin by looking at $\bm{A}^{\star}$, which has rank 1 and satisfies
\[
\sigma_{1}\left(\bm{A}^{\star}\right)=\frac{\left(q_{\mathsf{in}}-q_{\mathsf{out}}\right)\sqrt{n_{u}n_{v}}}{2},\quad\left\Vert \bm{A}^{\star}\right\Vert _{\infty}=\frac{q_{\mathsf{in}}-q_{\mathsf{out}}}{2},\quad\left\Vert \bm{u}^{\star}\right\Vert _{\infty}=\frac{1}{\sqrt{n_{u}}},\quad\left\Vert \bm{v}^{\star}\right\Vert _{\infty}=\frac{1}{\sqrt{n_{v}}}.
\]
Recalling the definition of $\mu_{0},\mu_{1},\mu_{2}$ in \eqref{eq:def-col-mu0}
and \eqref{def:coh}, we obtain
\[
\mu_{0}=\frac{n_{u}n_{v}}{\left\Vert \bm{A}^{\star}\right\Vert _{\mathrm{F}}^{2}}\left\Vert \bm{A}^{\star}\right\Vert _{\infty}^{2}=1,\qquad\mu_{1}=n_{u}\left\Vert \bm{u}^{\star}\right\Vert _{2,\infty}^{2}=1,\qquad\mu_{2}=n_{v}\left\Vert \bm{v}^{\star}\right\Vert _{2,\infty}^{2}=1,\qquad\kappa=1.
\]
\item Next, we consider the maximum magnitude $R$ and the maximum variance $\sigma^{2}$
of all entries of $\bm{N}$ (see Assumption~\ref{assumption:random-noise}).
Clearly, one has
\begin{align*}
R & =\text{\ensuremath{\max_{i,j}}}\left|N_{i,j}\right|\leq1,\\
\sigma^{2} & =\max_{i,j}\mathbb{\mathsf{Var}}(N_{i,j})=\max\left\{ q_{\mathsf{in}}(1-q_{\mathsf{in}}),q_{\mathsf{out}}(1-q_{\mathsf{out}})\right\} 
	\leq \max\left\{ q_{\mathsf{in}},q_{\mathsf{out}}\right\} \asymp q_{\mathsf{in}},
\end{align*}
which follows since $N_{i,j}$ is a centered Bernoulli random variable
with parameter either $q_{\mathsf{in}}$ or $q_{\mathsf{out}}$. From
the assumption~\eqref{eq:BSBM-assump} and the fact $q_{\mathsf{in}}^{2}\geq(q_{\mathsf{in}}-q_{\mathsf{out}})^{2}$,
we know that
\[
q_{\mathsf{in}}\geq\frac{\left(q_{\mathsf{in}}-q_{\mathsf{out}}\right)^{2}}{q_{\mathsf{in}}}\gg\frac{\log n}{\sqrt{n_{u}n_{v}}}+\frac{\log n}{n_{v}}.
\]
Putting the above estimates together, we can straightforwardly verify
the random noise requirement~\eqref{eq:assumption-noise-spike},
namely,
\[
\frac{R^{2}}{\sigma^{2}}\lesssim\frac{1}{q_{\mathsf{in}}}\lesssim\frac{\min\left\{ \sqrt{n_{u}n_{v}},n_{v}\right\} }{\log n}.
\]
\end{itemize}
With the preceding bounds in place, Corollary \ref{cor:BSBM} is an
immediate consequence of Theorem \ref{thm:U_loss}.

\section{Proofs for key lemmas}

\label{sec:Proofs-key-lemmas}

This section aims to establish the key lemmas listed in Section \ref{subsec:Key-lemmas}.

\subsection{Auxiliary quantities, notation, and preliminary facts}

To simplify our treatment, the proofs shall consider the influence
of missing data and that of noise altogether. Specifically, throughout
this section, we shall define a rescaled version of $\bm{A}$ as follows
\begin{equation}
\bm{A}^{\mathsf{s}}:=\tfrac{1}{p}\bm{A}=\bm{A}^{\star}+\bm{E}\in\mathbb{R}^{d_{1}\times d_{2}},\label{eq:A-E}
\end{equation}
where the matrix $\bm{E}$ represents the aggregate perturbation
\begin{equation}
\bm{E}:=\tfrac{1}{p}\mathcal{\mathcal{P}}_{\Omega}\left(\bm{A}^{\star}\right)-\bm{A}^{\star}+\tfrac{1}{p}\mathcal{\mathcal{P}}_{\Omega}\left(\bm{N}\right).\label{eq:defn-E}
\end{equation}
Clearly, $\bm{E}\in\mathbb{R}^{d_{1}\times d_{2}}$ is a random matrix
with independent zero-mean entries and $\mathbb{E}\left[\bm{A}^{\mathsf{s}}\right]=\bm{A}^{\star}$.
In addition, we define the corresponding leave-one-out and leave-two-out
versions
\begin{align}
\bm{A}^{\mathsf{s},\left(m\right)} & :=\tfrac{1}{p}\bm{A}^{\left(m\right)},\\
\bm{A}^{\mathsf{s},\left(m,l\right)} & :=\tfrac{1}{p}\bm{A}^{\left(m,l\right)},
\end{align}
for each $1\leq m\leq d_{1},1\leq l\leq d_{2}$.

As we shall see momentarily, it is convenient to introduce the following
quantities regarding the above perturbation matrix $\bm{E}$: (1)
$\max_{i\in\left[d_{1}\right],j\in\left[d_{2}\right]}\left|E_{i,j}\right|$;
(2) $\max_{i\in\left[d_{1}\right],j\in\left[d_{2}\right]}\sqrt{\mathbb{E}\left[E_{i,j}^{2}\right]}$;
(3) $\max_{i\in\left[d_{1}\right]}\sqrt{\sum_{j\in\left[d_{2}\right]}\mathbb{E}\left[E_{i,j}^{2}\right]}$;
(4) $\max_{j\in\left[d_{2}\right]}\sqrt{\sum_{i\in\left[d_{1}\right]}\mathbb{E}\left[E_{i,j}^{2}\right]}$.
In our settings, it is easy to verify --- using the definition of incoherence
parameters (cf.~Definition \ref{definition-mu0-mu1-mu2}), the assumptions
of the random noise (cf.~Assumption \ref{assumption:random-noise}),
and Lemma \ref{lemma:incoh} --- that the quantities defined above admit
the following upper bounds \begin{subequations}\label{eq:noise_value}
\begin{align}
\max_{i\in\left[d_{1}\right],j\in\left[d_{2}\right]}\left|E_{i,j}\right| & \leq\frac{\left\Vert \bm{A}^{\star}\right\Vert _{\infty}+R}{p}\lesssim\frac{\sqrt{\mu r}\,\sigma_{1}^{\star}}{\sqrt{d_{1}d_{2}}\,p}+\frac{\sigma\min\left\{ \sqrt[4]{d_{1}d_{2}},\,\sqrt{d_{2}}\right\} }{\sqrt{p\log d}}=:B,\label{def:sigma_B_UB}\\
\max_{i\in\left[d_{1}\right],j\in\left[d_{2}\right]}\sqrt{\mathbb{E}\left[E_{i,j}^{2}\right]} & \leq\frac{\left\Vert \bm{A}^{\star}\right\Vert _{\infty}+\sigma}{\sqrt{p}}\leq\sigma_{1}^{\star}\sqrt{\frac{\mu r}{d_{1}d_{2}p}}+\frac{\sigma}{\sqrt{p}}=:\sigma_{\infty},\label{def:sigma_inf_UB}\\
\max_{i\in\left[d_{1}\right]}\sqrt{\sum_{j\in\left[d_{2}\right]}\mathbb{E}\left[E_{i,j}^{2}\right]} & \leq\frac{\left\Vert \bm{A}^{\star}\right\Vert _{2,\infty}+\sigma\sqrt{d_{2}}}{\sqrt{p}}\leq\sigma_{1}^{\star}\sqrt{\frac{\mu r}{d_{1}p}}+\sigma\sqrt{\frac{d_{2}}{p}}=:\sigma_{\mathsf{row}},\label{def:sigma_row_UB}\\
\max_{j\in\left[d_{2}\right]}\sqrt{\sum_{i\in\left[d_{1}\right]}\mathbb{E}\left[E_{i,j}^{2}\right]} & \leq\frac{\left\Vert \bm{A}^{\star\top}\right\Vert _{2,\infty}+\sigma\sqrt{d_{1}}}{\sqrt{p}}\leq\sigma_{1}^{\star}\sqrt{\frac{\mu r}{d_{2}p}}+\sigma\sqrt{\frac{d_{1}}{p}}=:\sigma_{\mathsf{col}},\label{def:sigma_col_UB}
\end{align}
\end{subequations}with probability exceeding $1-O(d^{-12})$. Further,
the following lemma singles out a few other useful properties about
these quantities (to be established in Appendix \ref{subsec:pf:noise_cond}),
which will be useful throughout the proof.

\begin{lemma}\label{lemma:noise_cond}Instate the assumptions of
Theorem \ref{thm:U_loss}. Then with probability at least $1-O\left(d^{-12}\right)$,
we have\begin{subequations}\label{eq:noise_cond}
\begin{align}
B & \lesssim\frac{\min\left\{ \sqrt{\sigma_{\mathsf{row}}\sigma_{\mathsf{col}}},\,\sigma_{\mathsf{row}}\right\} }{\sqrt{\log d}};\label{cond:B_UB}\\
\sigma_{\infty}^{2} & \lesssim B\log d\left\Vert \bm{A}^{\star}\right\Vert \sqrt{\frac{\mu r}{d_{2}}}\lesssim\sigma_{\mathsf{col}}\sqrt{\log d}\left\Vert \bm{A}^{\star}\right\Vert \sqrt{\frac{\mu r}{d_{1}}};\label{cond:col_B_rel}\\
\sqrt{\frac{\mu r}{d_{1}}} & \gtrsim\frac{B\log^{3/2}d\left\Vert \bm{A}^{\star}\right\Vert _{\infty}}{\left\Vert \bm{A}^{\star}\right\Vert ^{2}};\label{cond:row_B_rel}\\
\sigma_{r}^{\star2} & \gg\max\Big\{\kappa^{2}\sigma_{\mathsf{col}}\sigma_{\mathsf{row}}\log d,\,\kappa^{2}\sigma_{\mathsf{col}}\sqrt{\log d}\left\Vert \bm{A}^{\star}\right\Vert ,\,\kappa^{2}\left\Vert \bm{A}^{\star}\right\Vert _{2,\infty}^{2},\,\sigma_{\mathsf{row}}\sqrt{\frac{d_{1}\log d}{\mu r}}\left\Vert \bm{A}^{\star\top}\right\Vert _{2,\infty},\,B\log d\left\Vert \bm{A}^{\star}\right\Vert _{\infty}\Big\}.\label{cond:eig_o1}
\end{align}
\end{subequations}\end{lemma}

\subsection{Proof of Lemma~\ref{lemma:G_op_loss}\label{subsec:pf:G_op_loss}}

The main component of the proof is to demonstrate that
\begin{equation}
\left\Vert \bm{G}-\bm{G}^{\star}\right\Vert \lesssim\underset{=:\delta_{\mathsf{op}}}{\underbrace{\left(\sigma_{\mathsf{row}}+\sigma_{\mathsf{col}}\right)\big(\sigma_{\mathsf{col}}+\left\Vert \bm{A}^{\star\top}\right\Vert _{2,\infty}\big)\log d+\sigma_{\mathsf{col}}\sqrt{\log d}\left\Vert \bm{A}^{\star}\right\Vert +\left\Vert \bm{A}^{\star}\right\Vert _{2,\infty}^{2}}}.\label{def:delta_op}
\end{equation}
By substituting the values of $\sigma_{\mathsf{row}}$ and $\sigma_{\mathsf{col}}$
(cf.~(\ref{eq:noise_value})) into the above expression, one derives
\begin{equation}
\delta_{\mathsf{op}}\lesssim\zeta_{\mathsf{op}}+\left\Vert \bm{A}^{\star}\right\Vert _{2,\infty}^{2},\label{eq:delta-opt-zeta-op}
\end{equation}
where $\zeta_{\mathsf{op}}$ is defined in (\ref{claim:G_op_loss}).
Therefore, Lemma~\ref{lemma:G_op_loss} is an immediate consequence
of (\ref{def:delta_op}) and (\ref{eq:delta-opt-zeta-op}). The remainder
of the proof amounts to justifying (\ref{def:delta_op}).

Recall the definitions of $\bm{G}$ and $\bm{G}^{\star}$ in (\ref{def:G})
and (\ref{eq:defn-Gstar}), respectively. Given that $\bm{A}^{\mathsf{s}}=\bm{A}^{\star}+\bm{E}$,
we can expand 
\begin{align}
\bm{G}-\bm{G}^{\star} & =\mathcal{P}_{\mathsf{off}\text{-}\mathsf{diag}}\left(\bm{A}^{\mathsf{s}}\bm{A}^{\mathsf{s}\top}\right)-\bm{A}^{\star}\bm{A}^{\star\top}=\mathcal{P}_{\mathsf{off}\text{-}\mathsf{diag}}\left(\bm{A}^{\mathsf{s}}\bm{A}^{\mathsf{s}\top}-\bm{A}^{\star}\bm{A}^{\star\top}\right)-\mathcal{P}_{\mathsf{diag}}\left(\bm{A}^{\star}\bm{A}^{\star\top}\right)\nonumber \\
 & =\mathcal{P}_{\mathsf{off}\text{-}\mathsf{diag}}\left(\bm{E}\bm{E}^{\top}\right)+\mathcal{P}_{\mathsf{off}\text{-}\mathsf{diag}}\left(\bm{A}^{\star}\bm{E}^{\top}+\bm{E}\bm{A}^{\star\top}\right)-\mathcal{P}_{\mathsf{diag}}\left(\bm{A}^{\star}\bm{A}^{\star\top}\right),\label{eq:G-Gstar-expand}
\end{align}
where $\mathcal{P}_{\mathsf{off}\text{-}\mathsf{diag}}$ and $\mathcal{P}_{\mathsf{diag}}$
are defined in Section \ref{subsec:Notations}. In what follows, we
control these three terms separately.

\subsubsection{Step 1: bounding the term $\mathcal{P}_{\mathsf{off}\text{-}\mathsf{diag}}\left(\bm{E}\bm{E}^{\top}\right)$}

We first consider the term $\mathcal{P}_{\mathsf{off}\text{-}\mathsf{diag}}\left(\bm{E}\bm{E}^{\top}\right)$.
Since $\left\{ E_{i,j}\right\} _{i\in\left[d_{1}\right],j\in\left[d_{2}\right]}$
are independent zero-mean random variables, we can express

\begin{equation}
\mathcal{P}_{\mathsf{off}\text{-}\mathsf{diag}}\left(\bm{E}\bm{E}^{\top}\right)=\sum\nolimits _{1\leq l \leq d_2}\left(\bm{E}_{:,l}\bm{E}_{:,l}^{\top}-\bm{D}_{l}\right)\label{eq:Poffdiag-sum}
\end{equation}
as a sum of independent zero-mean random matrices, where $\bm{D}_{l}$
is a random diagonal matrix in $\mathbb{R}^{d_{1}\times d_{1}}$ with
entries 
\begin{align}
\big(\bm{D}_{l}\big)_{i,i}=E_{i,l}^{2}. \label{def:D_l}
\end{align}

We intend to invoke the truncated matrix Bernstein inequality \cite[Proposition~A.7]{hopkins2016fast}
to control the spectral norm of (\ref{eq:Poffdiag-sum}). To this
end, we need to look at a few quantities.
\begin{itemize}
\item We first bound the spectral norm of the following covariance matrix
\[
\bm{\Sigma}_{\mathsf{ns}}:=\sum\nolimits _{1\leq l \leq d_2}\mathbb{E}\left[\left(\bm{E}_{:,l}\bm{E}_{:,l}^{\top}-\bm{D}_{l}\right)^{2}\right]\in\mathbb{R}^{d_{1}\times d_{1}}.
\]
Straightforward computation reveals that $\bm{\Sigma}_{\mathsf{ns}}$
is a diagonal matrix with entries
\[
\left(\bm{\Sigma}_{\mathsf{ns}}\right)_{i,i}=\sum\nolimits _{1\leq l \leq d_2}\mathbb{E}\left[E_{i,l}^{2}\right]\sum\nolimits _{m: m \neq i}\mathbb{E}\left[E_{m,l}^{2}\right]\leq\sigma_{\mathsf{row}}^{2}\sigma_{\mathsf{col}}^{2}
\]
for each $1\leq i\leq d_{1}$. This immediately reveals that
\begin{equation}
V_{\mathsf{ns}}:=\left\Vert \bm{\Sigma}_{\mathsf{ns}}\right\Vert \leq\sigma_{\mathsf{row}}^{2}\sigma_{\mathsf{col}}^{2}.\label{eq:V_noise}
\end{equation}
\item Next, we turn to upper bounding the spectral norm of each summand
$\bm{E}_{:,l}\bm{E}_{:,l}^{\top}-\bm{D}_{l}$. As shown in the proof
of Lemma~\ref{lemma:row_col_2_norm}, one has
\begin{align*}
\mathbb{P}\left\{ \left|\left\Vert \bm{E}_{:,l}\right\Vert _{2}^{2}-M_{1}\right|\geq t\right\}  & \leq2\exp\left(-\frac{3}{8}\min\left\{ \frac{t^{2}}{V_{1}},\frac{t}{L_{1}}\right\} \right),\quad t>0,
\end{align*}
where $M_{1},L_{1}$ and $V_{1}$ are given respectively by
\begin{align*}
M_{1} & :=\mathbb{E}\left[\left\Vert \bm{E}_{:,l}\right\Vert _{2}^{2}\right]\leq\sigma_{\mathsf{col}}^{2},\\
L_{1} & :=\max_{1\leq i \leq d_1}\left|E_{i,l}^{2}-\mathbb{E}\left[E_{i,l}^{2}\right]\right|\leq2B^{2},\\
V_{1} & :=\sum\nolimits _{1\leq i \leq d_1}\mathsf{Var}\left(E_{i,l}^{2}\right)\leq B^{2}\sigma_{\mathsf{col}}^{2}.
\end{align*}
In addition, with probability exceeding $1-O\left(d^{-20}\right)$,
\begin{align}
\left\Vert \bm{E}_{:,l}\right\Vert _{2}^{2} & \lesssim M_{1}+L_{1}\log d+\sqrt{V_{1}\log d}\lesssim\sigma_{\mathsf{col}}^{2}+B^{2}\log d+\sqrt{B^{2}\sigma_{\mathsf{col}}^{2}\log d}\nonumber \\
 & \asymp B^{2}\log d+\sigma_{\mathsf{col}}^{2},\label{eq:n_l_2_norm}
\end{align}
where the last line comes from the AM-GM inequality $2\sqrt{B^{2}\sigma_{\mathsf{col}}^{2}\log d}\leq B^{2}\log d+\sigma_{\mathsf{col}}^{2}$.
This together with the definition $\bm{D}_{l}:=\mathsf{diag}\big(E_{i,l}^{2},\dots,E_{d_{1},l}^{2}\big)$
gives
\[
\left\Vert \bm{E}_{:,l}\bm{E}_{:,l}^{\top}-\bm{D}_{l}\right\Vert \leq\left\Vert \bm{E}_{:,l}\right\Vert _{2}^{2}+\left\Vert \bm{D}_{l}\right\Vert \leq2\left\Vert \bm{E}_{:,l}\right\Vert _{2}^{2}\lesssim B^{2}\log d+\sigma_{\mathsf{col}}^{2}.
\]
Therefore, if we set
\begin{equation}
L_{\mathsf{ns}}:=C\left(B^{2}\log d+\sigma_{\mathsf{col}}^{2}\right)\label{eq:L_noise}
\end{equation}
 for some sufficiently large constant $C>0$, then the above argument
reveals that
\begin{align*}
L_{\mathsf{ns}} & \geq\frac{C}{3}\left(M_{1}+L_{1}\log d+\sqrt{V_{1}\log d}\right)\geq\frac{C}{3}\max\left\{ \sqrt{V_{1}\log d},L_{1}\log d\right\} .
\end{align*}
\item In addition, one can easily bound that
\begin{align*}
\mathbb{E}\left[\left\Vert \bm{E}_{:,l}\right\Vert _{2}^{2}\mathds{1}\left\{ \left\Vert \bm{E}_{:,l}\right\Vert _{2}^{2}\geq L_{\mathsf{ns}}\right\} \right] & \leq L_{\mathsf{ns}}\mathbb{P}\left\{ \left\Vert \bm{E}_{:,l}\right\Vert _{2}^{2}\geq L_{\mathsf{ns}}\right\} +\int_{L_{\mathsf{ns}}}^{\infty}\mathbb{P}\left\{ \left\Vert \bm{E}_{:,l}\right\Vert _{2}^{2}\geq t\right\} \mathrm{d}t\\
 & \leq O\left(d^{-20}\right)L_{\mathsf{ns}}+\int_{L_{\mathsf{ns}}}^{\infty}\mathbb{P}\left\{ \left\Vert \bm{E}_{:,l}\right\Vert _{2}^{2}\geq t\right\} \mathrm{d}t.
\end{align*}
Moreover, we know that $\min\big\{ t^{2}/V_{1},t/L_{1}\big\} \geq t/\max\big\{ \sqrt{V_{1}/\log d},L_{1}\big\} $
for any $t\geq L_{\mathsf{ns}}/2$. As a result, for sufficiently
large $d$, we have
\begin{align*}
\int_{L_{\mathsf{ns}}}^{\infty}\mathbb{P}\left\{ \left\Vert \bm{E}_{:,l}\right\Vert _{2}^{2}\geq t\right\} \mathrm{d}t & \leq\int_{L_{\mathsf{ns}}}^{\infty}\mathbb{P}\left\{ \left|\left\Vert \bm{E}_{:,l}\right\Vert _{2}^{2}-M_{1}\right|>\frac{t}{2}\right\} \mathrm{d}t\\
 & \leq4\int_{\frac{1}{2}L_{\mathsf{ns}}}^{\infty}\exp\left(-\frac{3}{8}\min\left\{ \frac{t^{2}}{V_{1}},\frac{t}{L_{1}}\right\} \right)\mathrm{d}t\\
 & \leq4\int_{\frac{1}{2}L_{\mathsf{ns}}}^{\infty}\exp\left(-\frac{3}{8}\frac{t}{\max\big\{ \sqrt{V_{1}/\log d},L_{1}\big\} }\right)\mathrm{d}t\\
 & \lesssim\max\big\{ \sqrt{V_{1}/\log d},L_{1}\big\}\exp\left(-\frac{3}{16}\frac{L_{\mathsf{ns}}}{\max\big\{ \sqrt{V_{1} / \log d},L_{1}\big\} }\right)\\
 & \leq\max\big\{ \sqrt{V_{1}/\log d},L_{1}\big\}\exp\left(-\frac{3C}{32}\log d\right)\\
 & \ll\frac{L_{\mathsf{ns}}}{d^{2}},
\end{align*}
provided that $C>0$ is sufficiently large. Consequently, we have
\begin{align}
R_{\mathsf{ns}} & :=\mathbb{E}\left[\left\Vert \bm{E}_{:,l}\bm{E}_{:,l}^{\top}-\bm{D}_{l}\right\Vert \mathds{1}\left\{ \left\Vert \bm{E}_{:,l}\bm{E}_{:,l}^{\top}-\bm{D}_{l}\right\Vert \geq L_{\mathsf{ns}}\right\} \right]\nonumber \\
 & \leq\mathbb{E}\left[2\left\Vert \bm{E}_{:,l}\right\Vert _{2}^{2}\mathds{1}\left\{ 2\left\Vert \bm{E}_{:,l}\right\Vert _{2}^{2}\geq L_{\mathsf{ns}}\right\} \right]\ll\frac{L_{\mathsf{ns}}}{d^{2}}.\label{eq:R_noise}
\end{align}
\end{itemize}
With estimates (\ref{eq:V_noise}), (\ref{eq:L_noise}) and (\ref{eq:R_noise})
in place, we are ready to apply the truncated matrix Bernstein inequality
\cite[Proposition~A.7]{hopkins2016fast} to obtain that, with probability
at least $1-O\left(d^{-10}\right)$,
\begin{align}
\left\Vert \mathcal{P}_{\mathsf{off}\text{-}\mathsf{diag}}\left(\bm{E}\bm{E}^{\top}\right)\right\Vert  & =\Big\|\sum\nolimits _{1\leq l\leq d_{2}}\bm{E}_{:,l}\bm{E}_{:,l}^{\top}-\bm{D}_{l}\Big\|\lesssim d_{2}R_{\mathsf{ns}}+L_{\mathsf{ns}}\log d+\sqrt{V_{\mathsf{ns}}\log d}\nonumber \\
 & \asymp L_{\mathsf{ns}}\log d+\sqrt{V_{\mathsf{ns}}\log d}\nonumber \\
 & \lesssim B^{2}\log^{2}d+\sigma_{\mathsf{col}}^{2}\log d+\sigma_{\mathsf{row}}\sigma_{\mathsf{col}}\sqrt{\log d}\nonumber \\
 & \lesssim\sigma_{\mathsf{col}}\left(\sigma_{\mathsf{row}}+\sigma_{\mathsf{col}}\right)\log d,\label{eq:delta_noise_op_norm_UB}
\end{align}
where the last line results from the identity $B^{2}\log d\lesssim\sigma_{\mathsf{row}}\sigma_{\mathsf{col}}$
(See (\ref{cond:B_UB})).

\subsubsection{Step 2: bounding the term $\mathcal{P}_{\mathsf{off}\text{-}\mathsf{diag}}\left(\bm{A}^{\star}\bm{E}^{\top}+\bm{E}\bm{A}^{\star\top}\right)$}

Next, we turn attention to $\mathcal{P}_{\mathsf{off}\text{-}\mathsf{diag}}\left(\bm{A}^{\star}\bm{E}^{\top}+\bm{E}\bm{A}^{\star\top}\right)$.
By symmetry, it suffices to control to the spectral norm of $\mathcal{P}_{\mathsf{off}\text{-}\mathsf{diag}}\left(\bm{A}^{\star}\bm{E}^{\top}\right)$.
To this end, we first express

\begin{equation}
\mathcal{P}_{\mathsf{off}\text{-}\mathsf{diag}}\left(\bm{A}^{\star}\bm{E}^{\top}\right)=\sum\nolimits _{1\leq l\leq d_{2}}\big(\bm{A}_{:,l}^{\star}\bm{E}_{:,l}^{\top}-\widehat{\bm{D}}_{l}\big)\label{eq:Poff-diag-AE-expand}
\end{equation}
as a sum of independent zero-mean random matrices, where $\widehat{\bm{D}}_{l}$
is a diagonal matrix obeying 
\begin{align}
\big(\widehat{\bm{D}}_{l}\big)_{i,i}=A_{i,l}^{\star}E_{i,l}. \label{def:D_l_hat}
\end{align}

To control (\ref{eq:Poff-diag-AE-expand}), we need to first look
at two matrices defined as follows
\begin{align*}
\widehat{\bm{\Sigma}}_{\mathsf{crs}} & :=\sum\nolimits _{1\leq l\leq d_{2}}\mathbb{E}\Big[\big(\bm{A}_{:,l}^{\star}\bm{E}_{:,l}^{\top}-\widehat{\bm{D}}_{l}\big)\big(\bm{A}_{:,l}^{\star}\bm{E}_{:,l}^{\top}-\widehat{\bm{D}}_{l}\big)^{\top}\Big];\\
\widetilde{\bm{\Sigma}}_{\mathsf{crs}} & :=\sum\nolimits _{1\leq l\leq d_{2}}\mathbb{E}\Big[\big(\bm{A}_{:,l}^{\star}\bm{E}_{:,l}^{\top}-\widehat{\bm{D}}_{l}\big)^{\top}\big(\bm{A}_{:,l}^{\star}\bm{E}_{:,l}^{\top}-\widehat{\bm{D}}_{l}\big)\Big].
\end{align*}
Straightforward computation shows that 
\begin{align*}
\big(\widehat{\bm{\Sigma}}_{\mathsf{crs}}\big)_{i,i} & =\sum\nolimits _{1\leq l\leq d_{2}}A_{i,l}^{\star2}\mathbb{E}\left[\left\Vert \bm{E}_{:,l}\right\Vert _{2}^{2}-E_{i,l}^{2}\right],\quad\quad i\in\left[d_{1}\right],\\
\big(\widehat{\bm{\Sigma}}_{\mathsf{crs}}\big)_{i,j} & =\sum\nolimits _{1\leq l\leq d_{2}}A_{i,l}^{\star}A_{j,l}^{\star}\mathbb{E}\left[\left\Vert \bm{E}_{:,l}\right\Vert _{2}^{2}-E_{i,l}^{2}-E_{j,l}^{2}\right],\quad\quad i\neq j,
\end{align*}
and $\widetilde{\bm{\Sigma}}_{\mathsf{crs}}\in\mathbb{R}^{d_{1}\times d_{1}}$
is a diagonal matrix with entries
\begin{align*}
\big(\widetilde{\bm{\Sigma}}_{\mathsf{crs}}\big)_{i,i} & =\sum\nolimits _{1\leq l\leq d_{2}}\left(\left\Vert \bm{A}_{:,l}^{\star}\right\Vert _{2}^{2}-A_{i,l}^{\star2}\right)\mathbb{E}\left[E_{i,l}^{2}\right],\quad\quad i\in\left[d_{1}\right].
\end{align*}
Hence we have
\begin{equation}
\big\|\widetilde{\bm{\Sigma}}_{\mathsf{crs}}\big\|\leq\max_{1\leq i \leq d_1}\big|\big(\widetilde{\bm{\Sigma}}_{\mathsf{crs}}\big)_{i,i}\big|\lesssim\sigma_{\mathsf{row}}^{2}\left\Vert \bm{A}^{\star\top}\right\Vert _{2,\infty}^{2}.\label{eq:cov_cross_term1}
\end{equation}
To control the spectral norm of $\widehat{\bm{\Sigma}}_{\mathsf{crs}}$,
we further decompose it as $\widehat{\bm{\Sigma}}_{\mathsf{crs}}=\widehat{\bm{\Sigma}}_{\mathsf{crs}}^{'}-\widehat{\bm{\Sigma}}_{\mathsf{crs}}^{''}$,
where
\begin{align*}
\big(\widehat{\bm{\Sigma}}_{\mathsf{crs}}^{'}\big)_{i,i} & =\sum\nolimits _{1\leq l\leq d_{2}}A_{i,l}^{\star2}\mathbb{E}\left[\left\Vert \bm{E}_{:,l}\right\Vert _{2}^{2}\right],\quad i\in\left[d_{1}\right],\\
\big(\widehat{\bm{\Sigma}}_{\mathsf{crs}}^{'}\big)_{i,j} & =\sum\nolimits _{1\leq l\leq d_{2}}A_{i,l}^{\star}A_{j,l}^{\star}\mathbb{E}\left[\left\Vert \bm{E}_{:,l}\right\Vert _{2}^{2}\right],\quad i\neq j,
\end{align*}
and
\begin{align*}
\big(\widehat{\bm{\Sigma}}_{\mathsf{crs}}^{''}\big)_{i,i} & =\sum\nolimits _{1\leq l\leq d_{2}}A_{i,l}^{\star2}\mathbb{E}\left[E_{i,l}^{2}\right],\quad i\in\left[d_{1}\right],\\
\big(\widehat{\bm{\Sigma}}_{\mathsf{crs}}^{''}\big)_{i,j} & =\sum\nolimits _{1\leq l\leq d_{2}}A_{i,l}^{\star}A_{j,l}^{\star}\mathbb{E}\left[E_{i,l}^{2}+E_{j,l}^{2}\right],\quad i\neq j.
\end{align*}

\begin{itemize}
\item The spectral norm of $\widehat{\bm{\Sigma}}_{\mathsf{crs}}^{'}$ can
be easily upper bounded by
\begin{align}
\big\|\widehat{\bm{\Sigma}}_{\mathsf{crs}}^{'}\big\| & \leq\max_{1\leq l \leq d_2}\mathbb{E}\left[\left\Vert \bm{E}_{:,l}\right\Vert _{2}^{2}\right]\left\Vert \bm{A}^{\star}\bm{A}^{\star\top}\right\Vert \leq\sigma_{\mathsf{col}}^{2}\left\Vert \bm{A}^{\star}\right\Vert ^{2}.\label{eq:Sigma-hat-crs-tilde}
\end{align}
\item Regarding $\widehat{\bm{\Sigma}}_{\mathsf{crs}}^{''}$, we first decompose
$\widehat{\bm{\Sigma}}_{\mathsf{crs}}^{''}+\mathcal{P}_{\mathsf{diag}}\big(\widehat{\bm{\Sigma}}_{\mathsf{crs}}^{''}\big)=\bm{B}_{1}+\bm{B}_{2}$,
where the diagonal entries of $\bm{B}_{1}$ and $\bm{B}_{2}$ are
identical and equal to $\sum\nolimits _{1\leq l\leq d_{2}}A_{i,l}^{\star2}\mathbb{E}\big[E_{i,l}^{2}\big],\left(1\leq i\leq d_{2}\right)$
while their off-diagonal parts are given by
\[
\left(\bm{B}_{1}\right)_{i,j}=\sum\nolimits _{1\leq l\leq d_{2}}A_{i,l}^{\star}\mathbb{E}\left[E_{i,l}^{2}\right]A_{j,l}^{\star}\qquad\text{and}\qquad\left(\bm{B}_{2}\right)_{i,j}=\sum\nolimits _{1\leq l\leq d_{2}}A_{i,l}^{\star}A_{j,l}^{\star}\mathbb{E}\left[E_{j,l}^{2}\right],\qquad i\neq j.
\]
Let $\bm{C}$ be a matrix in $\mathbb{R}^{d_{1}\times d_{2}}$ with
entries $C_{i,j}=A_{i,j}^{\star}\mathbb{E}\left[E_{i,j}^{2}\right]$.
One can easily check that $\bm{B}_{1}=\sum_{1\leq l\leq d_{2}}\bm{C}_{:,l}\bm{A}_{:,l}^{\star\top}=\bm{C}\bm{A}^{\star\top}$
and develop an upper bound 
\begin{align*}
\left\Vert \bm{B}_{1}\right\Vert  & \leq\left\Vert \bm{C}\right\Vert \left\Vert \bm{A}^{\star}\right\Vert \leq\left\Vert \bm{C}\right\Vert _{\mathrm{F}}\left\Vert \bm{A}^{\star}\right\Vert \leq\sigma_{\infty}^{2}\left\Vert \bm{A}^{\star}\right\Vert _{\mathrm{F}}\left\Vert \bm{A}^{\star}\right\Vert .
\end{align*}
Note that the same bound also holds for $\bm{B}_{2}$. Therefore,
we arrive at
\begin{align*}
\big\|\widehat{\bm{\Sigma}}_{\mathsf{crs}}^{''}\big\| & \leq\big\|\mathcal{P}_{\mathsf{diag}}\big(\widehat{\bm{\Sigma}}_{\mathsf{crs}}^{''}\big)\big\|+\big\|\mathcal{P}_{\mathsf{diag}}\big(\widehat{\bm{\Sigma}}_{\mathsf{crs}}^{''}\big)+\widehat{\bm{\Sigma}}_{\mathsf{crs}}^{''}\big\|\leq\big\|\mathcal{P}_{\mathsf{diag}}\big(\widehat{\bm{\Sigma}}_{\mathsf{crs}}^{''}\big)\big\|+\left\Vert \bm{B}_{1}\right\Vert +\left\Vert \bm{B}_{2}\right\Vert \\
 & \leq\sigma_{\infty}^{2}\left\Vert \bm{A}^{\star}\right\Vert _{2,\infty}^{2}+\sigma_{\infty}^{2}\left\Vert \bm{A}^{\star}\right\Vert _{\mathrm{F}}\left\Vert \bm{A}^{\star}\right\Vert \\
 & \leq\sigma_{\infty}^{2}\left\Vert \bm{A}^{\star}\right\Vert _{2,\infty}^{2}+\sigma_{\infty}^{2}\sqrt{r}\left\Vert \bm{A}^{\star}\right\Vert ^{2}\\
 & \leq2\sigma_{\infty}^{2}\sqrt{r}\left\Vert \bm{A}^{\star}\right\Vert ^{2},
\end{align*}
where we have used the facts that $\left\Vert \bm{A}^{\star}\right\Vert _{2,\infty}\leq\left\Vert \bm{A}^{\star}\right\Vert $
and $\left\Vert \bm{A}^{\star}\right\Vert _{\mathrm{F}}\leq\sqrt{r}\left\Vert \bm{A}^{\star}\right\Vert $.
Consequently, the above bounds taken collectively yield
\begin{equation}
\big\|\widehat{\bm{\Sigma}}_{\mathsf{crs}}\big\|\leq\big\|\widehat{\bm{\Sigma}}_{\mathsf{crs}}^{'}\big\|+\big\|\widehat{\bm{\Sigma}}_{\mathsf{crs}}^{''}\big\|\lesssim\left(\sigma_{\mathsf{col}}^{2}+\sigma_{\infty}^{2}\sqrt{r}\right)\left\Vert \bm{A}^{\star}\right\Vert ^{2}\asymp\sigma_{\mathsf{col}}^{2}\left\Vert \bm{A}^{\star}\right\Vert ^{2},\label{eq:cov_cross_term2}
\end{equation}
where the last step uses (\ref{def:sigma_inf_UB}) and (\ref{def:sigma_col_UB}).
\end{itemize}
Putting (\ref{eq:cov_cross_term1}), (\ref{eq:Sigma-hat-crs-tilde})
and (\ref{eq:cov_cross_term2}) together yields
\begin{align}
V_{\mathsf{crs}} & :=\max\left\{ \big\|\widehat{\bm{\Sigma}}_{\mathsf{crs}}\big\|,\big\|\widetilde{\bm{\Sigma}}_{\mathsf{crs}}\big\|\right\} \lesssim\sigma_{\mathsf{col}}^{2}\left\Vert \bm{A}^{\star}\right\Vert ^{2}+\sigma_{\mathsf{row}}^{2}\left\Vert \bm{A}^{\star\top}\right\Vert _{2,\infty}^{2}.\label{eq:V_cross}
\end{align}

Second, we turn to the spectral norm of each summand $\bm{A}_{:,l}^{\star}\bm{E}_{:,l}^{\top}-\widehat{\bm{D}}_{l}$.
Recalling the definition that $\widehat{\bm{D}}_{l}=\mathsf{diag}\big(A_{1,l}^{\star}E_{1,l},\dots,A_{d_{1},l}^{\star}E_{d_{1},l}\big)$,
we can obtain
\[
\big\|\bm{A}_{:,l}^{\star}\bm{E}_{:,l}^{\top}-\widehat{\bm{D}}_{l}\big\|\leq\left\Vert \bm{A}_{:,l}^{\star}\right\Vert _{2}\left\Vert \bm{E}_{:,l}\right\Vert _{2}+\big\|\widehat{\bm{D}}_{l}\big\|\leq2\left\Vert \bm{A}_{:,l}^{\star}\right\Vert _{2}\left\Vert \bm{E}_{:,l}\right\Vert _{2}\leq2\left\Vert \bm{A}^{\star\top}\right\Vert _{2,\infty}\left\Vert \bm{E}_{:,l}\right\Vert _{2}.
\]
Set
\begin{align}
L_{\mathsf{crs}} & :=C\sqrt{L_{\mathsf{ns}}}\left\Vert \bm{A}^{\star\top}\right\Vert _{2,\infty}\asymp\big(\sigma_{\mathsf{col}}+B\sqrt{\log d}\big)\left\Vert \bm{A}^{\star\top}\right\Vert _{2,\infty},\label{eq:L_cross}
\end{align}
where $L_{\mathsf{ns}}$ is defined in (\ref{eq:L_noise}) and $C>0$
is some sufficiently large universal constant. Then with probability
at least $1-O\left(d^{-20}\right)$, one has
\begin{align*}
\big\|\bm{A}_{:,l}^{\star}\bm{E}_{:,l}^{\top}-\widehat{\bm{D}}_{l}\big\| & \leq2\left\Vert \bm{A}^{\star\top}\right\Vert _{2,\infty}\left\Vert \bm{E}_{:,l}\right\Vert _{2}\lesssim L_{\mathsf{crs}},
\end{align*}
where the last inequality comes from (\ref{eq:n_l_2_norm}).

Third, we need to control
\[
R_{\mathsf{crs}}:=\mathbb{E}\left[\big\|\bm{A}_{:,l}^{\star}\bm{E}_{:,l}^{\top}-\widehat{\bm{D}}_{l}\big\|\,\mathds{1}\big\{ \big\|\bm{A}_{:,l}^{\star}\bm{E}_{:,l}^{\top}-\widehat{\bm{D}}_{l}\big\|\geq L_{\mathsf{crs}}\big\} \right].
\]
From Jensen's inequality and (\ref{eq:R_noise}), we know that
\begin{align*}
\mathbb{E}\left[\left\Vert \bm{E}_{:,l}\right\Vert _{2}\mathds{1}\big\{ \left\Vert \bm{E}_{:,l}\right\Vert _{2}\geq\sqrt{L_{\mathsf{ns}}}\big\} \right] & \leq\sqrt{\mathbb{E}\left[\left\Vert \bm{E}_{:,l}\right\Vert _{2}^{2}\mathds{1}\big\{ \left\Vert \bm{E}_{:,l}\right\Vert _{2}^{2}\geq L_{\mathsf{ns}}\big\} \right]}\ll\frac{\sqrt{L_{\mathsf{ns}}}}{d}.
\end{align*}
By the definition of $L_{\mathsf{crs}}$ in (\ref{eq:L_cross}) and
the fact that 
\begin{align*}
\left\{ \left\Vert \bm{A}_{:,l}^{\star}\right\Vert _{2}\left\Vert \bm{E}_{:,l}\right\Vert _{2}\geq L_{\mathsf{crs}}\right\}  & \subset\left\{ \left\Vert \bm{A}^{\star\top}\right\Vert _{2,\infty}\left\Vert \bm{E}_{:,l}\right\Vert _{2}\geq L_{\mathsf{crs}}\right\} =\left\{ \left\Vert \bm{E}_{:,l}\right\Vert _{2}\geq C\sqrt{L_{\mathsf{ns}}}\right\} ,
\end{align*}
one has
\begin{align}
R_{\mathsf{crs}} & \leq\mathbb{E}\left[2\left\Vert \bm{A}_{:,l}^{\star}\right\Vert _{2}\left\Vert \bm{E}_{:,l}\right\Vert _{2}\mathds{1}\big\{ 2\left\Vert \bm{A}_{:,l}^{\star}\right\Vert _{2}\left\Vert \bm{E}_{:,l}\right\Vert _{2}\geq L_{\mathsf{crs}}\big\} \right]\nonumber \\
 & \lesssim\left\Vert \bm{A}^{\star\top}\right\Vert _{2,\infty}\mathbb{E}\Big[\left\Vert \bm{E}_{:,l}\right\Vert _{2}\mathds{1}\big\{ \left\Vert \bm{E}_{:,l}\right\Vert _{2}\geq\frac{C}{2}\sqrt{L_{\mathsf{ns}}}\big\} \Big]\nonumber \\
 & \ll\frac{1}{d}\left\Vert \bm{A}^{\star\top}\right\Vert _{2,\infty}\sqrt{L_{\mathsf{ns}}}\asymp\frac{L_{\mathsf{crs}}}{d}.\label{eq:R_cross}
\end{align}

With (\ref{eq:V_cross}), (\ref{eq:L_cross}) and (\ref{eq:R_cross})
in place, we can apply the truncated matrix Bernstein inequality to
obtain that, with with probability at least $1-O\left(d^{-10}\right)$,
\begin{align}
\left\Vert \mathcal{P}_{\mathsf{off}\text{-}\mathsf{diag}}\left(\bm{A}^{\star}\bm{E}^{\top}\right)\right\Vert  & \leq\big\|\sum\nolimits _{1\leq l\leq d_{2}}\big(\bm{A}_{:,l}^{\star}\bm{E}_{:,l}^{\top}-\widehat{\bm{D}}_{l}\big)\big\|\lesssim d_{2}R_{\mathsf{crs}}+L_{\mathsf{crs}}\log d+\sqrt{V_{\mathsf{crs}}\log d}\nonumber \\
 & \asymp L_{\mathsf{crs}}\log d+\sqrt{V_{\mathsf{crs}}\log d}\nonumber \\
 & \lesssim\big(\sigma_{\mathsf{col}}\log d+B\log^{3/2}d\big)\left\Vert \bm{A}^{\star\top}\right\Vert _{2,\infty}+\sigma_{\mathsf{row}}\sqrt{\log d}\left\Vert \bm{A}^{\star\top}\right\Vert _{2,\infty}+\sigma_{\mathsf{col}}\sqrt{\log d}\left\Vert \bm{A}^{\star}\right\Vert \nonumber \\
 & \lesssim\left(\sigma_{\mathsf{row}}+\sigma_{\mathsf{col}}\right)\log d\left\Vert \bm{A}^{\star\top}\right\Vert _{2,\infty}+\sigma_{\mathsf{col}}\sqrt{\log d}\left\Vert \bm{A}^{\star}\right\Vert \label{eq:delta_cross_op_norm_UB}
\end{align}
under the condition~(\ref{cond:B_UB}) that $B\sqrt{\log d}\lesssim\sqrt{\sigma_{\mathsf{row}}\sigma_{\mathsf{col}}}\leq\max\left\{ \sigma_{\mathsf{row}},\sigma_{\mathsf{col}}\right\} $.

\subsubsection{Step 3: combining Step 1 and Step 2}

Taking together (\ref{eq:delta_noise_op_norm_UB}), (\ref{eq:delta_cross_op_norm_UB})
and (\ref{eq:G-Gstar-expand}), we conclude that
\begin{align*}
\left\Vert \bm{G}-\bm{G}^{\star}\right\Vert  & \lesssim\left(\sigma_{\mathsf{row}}+\sigma_{\mathsf{col}}\right)\big(\sigma_{\mathsf{col}}+\left\Vert \bm{A}^{\star\top}\right\Vert _{2,\infty}\big)\log d+\sigma_{\mathsf{col}}\sqrt{\log d}\left\Vert \bm{A}^{\star}\right\Vert +\left\Vert \mathcal{P}_{\mathsf{diag}}\left(\bm{A}^{\star}\bm{A}^{\star\top}\right)\right\Vert \\
 & \lesssim\left(\sigma_{\mathsf{row}}+\sigma_{\mathsf{col}}\right)\big(\sigma_{\mathsf{col}}+\left\Vert \bm{A}^{\star\top}\right\Vert _{2,\infty}\big)\log d+\sigma_{\mathsf{col}}\sqrt{\log d}\left\Vert \bm{A}^{\star}\right\Vert +\left\Vert \bm{A}^{\star}\right\Vert _{2,\infty}^{2},
\end{align*}
where we have used the basic property $\left\Vert \mathcal{P}_{\mathsf{diag}}\left(\bm{A}^{\star}\bm{A}^{\star\top}\right)\right\Vert =\left\Vert \bm{A}^{\star}\right\Vert _{2,\infty}^{2}$.

\subsection{Proof of Lemma~\ref{lemma:G_dev_W_2_norm}\label{subsec:pf:G_dev_W_2_norm}}

We first claim that, for any fixed matrix $\bm{W}$, with probability
at least $1-O(d^{-10})$, the following holds for any $1\leq m \leq d_1$:
\begin{align}
\big\|\left(\bm{G}-\bm{G}^{\star}\right)_{m,:}\bm{W}\big\|_{2} & \lesssim\left(\sigma_{\mathsf{col}}\,\big(\sigma_{\mathsf{row}}+\left\Vert \bm{A}^{\star}\right\Vert _{2,\infty}\big)\sqrt{\log d}+B\log d\left\Vert \bm{A}^{\star}\right\Vert _{\infty}+\left\Vert \bm{A}^{\star}\right\Vert _{2,\infty}^{2}\right)\left\Vert \bm{W}\right\Vert _{2,\infty}\nonumber \\
 & \quad+\sigma_{\mathsf{row}}\sqrt{\log d}\left\Vert \bm{A}^{\star\top}\right\Vert _{2,\infty}\left\Vert \bm{W}\right\Vert .\label{eq:claim-G-W}
\end{align}
In particular, taking $\bm{W}=\bm{U}^{\star}$ gives
\[
\big\|\left(\bm{G}-\bm{G}^{\star}\right)_{m,:}\bm{U}^{\star}\big\|_{2}\lesssim\delta_{\mathsf{row}}\sqrt{\frac{\mu r}{d_{1}}},
\]
where
\begin{equation}
\delta_{\mathsf{row}}:=\sigma_{\mathsf{col}}\,\big(\sigma_{\mathsf{row}}+\left\Vert \bm{A}^{\star}\right\Vert _{2,\infty}\big)\sqrt{\log d}\,+B\log d\left\Vert \bm{A}^{\star}\right\Vert _{\infty}+\sqrt{\frac{d_{1}\log d}{\mu r}}\,\sigma_{\mathsf{row}}\left\Vert \bm{A}^{\star\top}\right\Vert _{2,\infty}+\left\Vert \bm{A}^{\star}\right\Vert _{2,\infty}^{2}.\label{def:delta_row}
\end{equation}
Using the values of $B,\sigma_{\infty},\sigma_{\mathsf{row}}$ and
$\sigma_{\mathsf{col}}$ specified in (\ref{eq:noise_value}), one
can easily verify that
\begin{align}
\delta_{\mathsf{row}}\lesssim\zeta_{\mathsf{op}}+\left\Vert \bm{A}^{\star}\right\Vert _{2,\infty}^{2}, \label{eq:delta_row_UB}
\end{align}
where $\zeta_{\mathsf{op}}$ is defined in (\ref{claim:G_op_loss}).
This leads to the advertised bound.

The rest of the proof is thus devoted to proving the claim (\ref{eq:claim-G-W}).
Recall the definitions of $\bm{G}$ and $\bm{G}^{\star}$ in \eqref{def:G}
and (\ref{eq:defn-Gstar}). For any $m,i\in[d_{1}]$, we can expand
\begin{align*}
\left(\bm{G}-\bm{G}^{\star}\right)_{m,i} & =\left\langle \bm{A}_{m,:}^{\mathsf{s}},\bm{A}_{i,:}^{\mathsf{s}}\right\rangle -\left\langle \bm{A}_{m,:}^{\star},\bm{A}_{i,:}^{\star}\right\rangle =\left\langle \bm{E}_{m,:},\bm{E}_{i,:}\right\rangle +\left\langle \bm{A}_{m,:}^{\star},\bm{E}_{i,:}\right\rangle +\big\langle\bm{E}_{m,:},\bm{A}_{i,:}^{\star}\big\rangle,\quad i\neq m;\\
\left(\bm{G}-\bm{G}^{\star}\right)_{m,m} & =-G_{m,m}^{\star}=-\left\Vert \bm{A}_{m,:}^{\star}\right\Vert _{2}^{2}.
\end{align*}
This allows us to derive
\begin{align}
\big\|\left(\bm{G}-\bm{G}^{\star}\right)_{m,:}\bm{W}\big\|_{2} & \leq\Big\|\sum\nolimits _{i:i\neq m}\left\langle \bm{E}_{m,:},\bm{E}_{i,:}\right\rangle \bm{W}_{i,:}\Big\|_{2}+\Big\|\sum\nolimits _{i:i\neq m}\left\langle \bm{A}_{m,:}^{\star},\bm{E}_{i,:}\right\rangle \bm{W}_{i,:}\Big\|_{2}\nonumber \\
 & \quad\quad+\Big\|\sum\nolimits _{i:i\neq m}\big\langle\bm{E}_{m,:},\bm{A}_{i,:}^{\star}\big\rangle\bm{W}_{i,:}\Big\|_{2}+\left\Vert G_{m,m}^{\star}\bm{W}_{m,:}\right\Vert _{2}.\label{eq:G-W-expansion}
\end{align}
We shall control each of these four terms separately.
\begin{itemize}
\item For the first term on the right-hand side of (\ref{eq:G-W-expansion}),
we know that
\[
\sum\nolimits _{i:i\neq m}\left\langle \bm{E}_{m,:},\bm{E}_{i,:}\right\rangle \bm{W}_{i,:}=\sum\nolimits _{(i,j):i\neq m}E_{m,j}E_{i,j}\bm{W}_{i,:}
\]
is a sum of independent zero-mean random vectors conditional on $\left\{ E_{m,j}\right\} _{j\in\left[d_{2}\right]}$.
In view of the matrix Bernstein inequality, it suffices to control
the following two quantities
\begin{align*}
L_{1} & :=\max_{(i,j):i\neq m}\left\Vert E_{m,j}E_{i,j}\bm{W}_{i,:}\right\Vert _{2}\leq\max_{i, j}\left|E_{i,j}\right|^{2}\left\Vert \bm{W}\right\Vert _{2,\infty}\leq B^{2}\left\Vert \bm{W}\right\Vert _{2,\infty},
\end{align*}
\begin{align*}
V_{1} & :=\sum\nolimits _{(i,j):i\neq m}E_{m,j}^{2}\mathbb{E}\left[E_{i,j}^{2}\right]\left\Vert \bm{W}_{i,:}\right\Vert _{2}^{2}\leq\left\Vert \bm{W}\right\Vert _{2,\infty}^{2}\sum\nolimits _{j}E_{m,j}^{2}\sum\nolimits _{i}\mathbb{E}\left[E_{i,j}^{2}\right]\leq\left\Vert \bm{W}\right\Vert _{2,\infty}^{2}\sigma_{\mathsf{col}}^{2}\sum\nolimits _{j}E_{m,j}^{2},
\end{align*}
where $B$ and $\sigma_{\mathsf{col}}$ are defined in (\ref{eq:noise_value}).
According to Lemma~\ref{lemma:row_col_2_norm}, the following holds
with probability at least $1-O\left(d^{-12}\right)$,
\begin{align*}
V_{1} & \lesssim\sigma_{\mathsf{col}}^{2}\left(\sigma_{\mathsf{row}}^{2}+B^{2}\log d\right)\left\Vert \bm{W}\right\Vert _{2,\infty}^{2}\asymp\sigma_{\mathsf{col}}^{2}\sigma_{\mathsf{row}}^{2}\left\Vert \bm{W}\right\Vert _{2,\infty}^{2},
\end{align*}
where we use the condition~(\ref{cond:B_UB}) (namely, $B\lesssim\sigma_{\mathsf{row}}/\sqrt{\log d}$).
Apply the matrix Bernstein inequality to demonstrate that with probability
exceeding $1-O(d^{-12})$,
\begin{align}
\Big\|\sum\nolimits _{i:i\neq m}\left\langle \bm{E}_{m,:},\bm{E}_{i,:}\right\rangle \bm{W}_{i,:}\Big\|_{2} & \lesssim L_{1}\log d+\sqrt{V_{1}\log d}\lesssim\big(B^{2}\log d+\sigma_{\mathsf{col}}\sigma_{\mathsf{row}}\sqrt{\log d}\big)\left\Vert \bm{W}\right\Vert _{2,\infty}\nonumber \\
 & \asymp\sigma_{\mathsf{col}}\sigma_{\mathsf{row}}\sqrt{\log d}\left\Vert \bm{W}\right\Vert _{2,\infty},\label{eq:B_dev_W_2_norm_term1}
\end{align}
where the last line follows from ~(\ref{cond:B_UB}) (i.e.~$B\lesssim\sqrt{\sigma_{\mathsf{row}}\sigma_{\mathsf{col}}/\log d}$).
\item Regarding the second term on the right-hand side of (\ref{eq:G-W-expansion}),
apply the same argument as above to show that 
\begin{equation}
\Big\|\sum\nolimits _{i:i\neq m}\left\langle \bm{A}_{m,:}^{\star},\bm{E}_{i,:}\right\rangle \bm{W}_{i,:}\Big\|_{2}\lesssim\big(\sigma_{\mathsf{col}}\sqrt{\log d}\left\Vert \bm{A}^{\star}\right\Vert _{2,\infty}+B\log d\left\Vert \bm{A}^{\star}\right\Vert _{\infty}\big)\left\Vert \bm{W}\right\Vert _{2,\infty}\label{eq:B_dev_W_2_norm_term2}
\end{equation}
holds with probability at least $1-O\left(d^{-12}\right)$.
\item Turning to the third term on the right-hand side of (\ref{eq:G-W-expansion}),
we have
\[
\Big\|\sum\nolimits _{i:i\neq m}\big\langle\bm{E}_{m,:},\bm{A}_{i,:}^{\star}\big\rangle\bm{W}_{i,:}\Big\|_{2}\leq\Big\|\sum\nolimits _{1\leq j \leq d_2}E_{m,j}\left(\bm{A}^{\star\top}\bm{W}\right)_{j,:}\Big\|_{2},
\]
where the summands are independent zero-mean random vectors. Let us
compute that
\begin{align*}
L_{2} & :=\max_{j\in\left[d_{2}\right]}\Big\| E_{m,j}\left(\bm{A}^{\star\top}\bm{W}\right)_{j,:}\Big\|_{2}\leq B\left\Vert \bm{A}^{\star\top}\bm{W}\right\Vert _{2,\infty}\leq B\left\Vert \bm{A}^{\star\top}\right\Vert _{2,\infty}\left\Vert \bm{W}\right\Vert ;\\
V_{2} & =\sum\nolimits _{j}\mathbb{E}\left[E_{m,j}^{2}\right]\Big\|\left(\bm{A}^{\star\top}\bm{W}\right)_{j,:}\Big\|_{2}^{2}\leq\sigma_{\mathsf{row}}^{2}\left\Vert \bm{A}^{\star\top}\bm{W}\right\Vert _{2,\infty}^{2}\leq\sigma_{\mathsf{row}}^{2}\left\Vert \bm{A}^{\star\top}\right\Vert _{2,\infty}^{2}\left\Vert \bm{W}\right\Vert ^{2}.
\end{align*}
Then the matrix Bernstein inequality reveals that with probability
exceeding $1-O\left(d^{-12}\right),$
\begin{align}
\Big\|\sum\nolimits _{j}E_{m,j}\left(\bm{A}^{\star\top}\bm{W}\right)_{j,:}\Big\|_{2} & \lesssim L_{2}\log d+\sqrt{V_{2}\log d}\lesssim\big(B\log d+\sigma_{\mathsf{row}}\sqrt{\log d}\big)\left\Vert \bm{A}^{\star\top}\right\Vert _{2,\infty}\left\Vert \bm{W}\right\Vert \nonumber \\
 & \asymp\sigma_{\mathsf{row}}\sqrt{\log d}\left\Vert \bm{A}^{\star\top}\right\Vert _{2,\infty}\left\Vert \bm{W}\right\Vert ,\label{eq:B_dev_W_2_norm_term3}
\end{align}
where the last line follows from the condition~(\ref{cond:B_UB})
(i.e.~$B\lesssim\sigma_{\mathsf{row}}/\sqrt{\log d}$).
\item The last term on the right-hand side of (\ref{eq:G-W-expansion})
can simply be upper bounded by
\begin{equation}
\left\Vert G_{m,m}^{\star}\bm{W}_{m,:}\right\Vert _{2}\leq\left\Vert \bm{A}^{\star}\right\Vert _{2,\infty}^{2}\left\Vert \bm{W}\right\Vert _{2,\infty}.\label{eq:B_dev_W_2_norm_term4}
\end{equation}
\end{itemize}
Putting (\ref{eq:B_dev_W_2_norm_term1}), (\ref{eq:B_dev_W_2_norm_term2}),
(\ref{eq:B_dev_W_2_norm_term3}) and (\ref{eq:B_dev_W_2_norm_term4})
together yields 
\begin{align*}
\big\|\left(\bm{G}-\bm{G}^{\star}\right)_{m,:}\bm{W}\big\|_{2} & \lesssim\left(\sigma_{\mathsf{col}}\big(\sigma_{\mathsf{row}}+\left\Vert \bm{A}^{\star}\right\Vert _{2,\infty}\big)\sqrt{\log d}+B\log d\left\Vert \bm{A}^{\star}\right\Vert _{\infty}+\left\Vert \bm{A}^{\star}\right\Vert _{2,\infty}^{2}\right)\left\Vert \bm{W}\right\Vert _{2,\infty}\\
 & \quad+\sigma_{\mathsf{row}}\sqrt{\log d}\left\Vert \bm{A}^{\star\top}\right\Vert _{2,\infty}\left\Vert \bm{W}\right\Vert 
\end{align*}
as claimed.

\subsection{Proof of Lemma~\ref{lemma:UH_BUtrue}\label{subsec:pf:UH_BUtrue}}

We claim for the moment that
\begin{align}
\big\|\bm{U}\bm{H}-\bm{G}\bm{U}^{\star}\left(\bm{\Sigma}^{\star}\right)^{-2}\big\|_{2,\infty} & \lesssim\frac{\left(\delta_{\mathsf{op}}+\delta_{\mathsf{loo}}\right)\kappa^{2}}{\sigma_{r}^{\star2}}\left(\left\Vert \bm{U}\bm{H}\right\Vert _{2,\infty}+\sqrt{\frac{\mu r}{d_{1}}}\right),\label{eq:claim4}
\end{align}
where $\delta_{\mathsf{op}}$ is defined in (\ref{def:delta_op}), and $\delta_{\mathsf{loo}}$ is defined as follows:
\begin{align}
\delta_{\mathsf{loo}} := \sigma_{\mathsf{col}}\sigma_{\mathsf{row}}\log d+\sigma_{\mathsf{col}}\left\Vert \bm{A}^{\star}\right\Vert \sqrt{\log d}. \label{def:delta_loo}
\end{align}
Using the values of $\sigma_{\mathsf{row}}$ and $\sigma_{\mathsf{col}}$
specified in (\ref{eq:noise_value}), one can easily see that $\delta_{\mathsf{loo}}\lesssim\zeta_{\mathsf{op}}$, where $\zeta_{\mathsf{op}}$ is defined in (\ref{claim:G_op_loss}). In addition, recall that we have already shown that $\delta_{\mathsf{op}}\lesssim\zeta_{\mathsf{op}}+\|\bm{A}^{\star}\|_{2,\infty}^{2}$. Putting these together establishes the lemma.

We now start to prove the claim (\ref{eq:claim4}). To this end, consider
an arbitrary $m\in[d_{1}]$. In view of \cite[Lemma~1]{abbe2017entrywise},
we can decompose
\begin{equation}
\left\Vert \big(\bm{U}\bm{H}-\bm{G}\bm{U}^{\star}\left(\bm{\Sigma}^{\star}\right)^{-2}\big)_{m,:}\right\Vert _{2}\lesssim\frac{1}{\sigma_{r}^{\star4}}\left\Vert \bm{G}-\bm{G}^{\star}\right\Vert \left\Vert \bm{G}_{m,:}\bm{U}^{\star}\right\Vert _{2}+\frac{1}{\sigma_{r}^{\star2}}\left\Vert \bm{G}_{m,:}\left(\bm{U}\bm{H}-\bm{U}^{\star}\right)\right\Vert _{2}.\label{eq:UH-GU-expand}
\end{equation}

\begin{itemize}
\item To bound the first term of (\ref{eq:UH-GU-expand}), we apply Lemma~\ref{lemma:G_dev_W_2_norm}, \eqref{eq:delta_row_UB} and Fact~\ref{fact:rel_err_o1} to reach
\[
\left\Vert \left(\bm{G}-\bm{G}^{\star}\right)\bm{U}^{\star}\right\Vert _{2,\infty}\lesssim\delta_{\mathsf{row}}\sqrt{\frac{\mu r}{d_{1}}}\ll\sigma_{r}^{\star2}\sqrt{\frac{\mu r}{d_{1}}}.
\]
The triangle inequality then gives
\begin{equation}
\left\Vert \bm{G}\bm{U}^{\star}\right\Vert _{2,\infty}\leq\left\Vert \left(\bm{G}-\bm{G}^{\star}\right)\bm{U}^{\star}\right\Vert _{2,\infty}+\left\Vert \bm{G}^{\star}\right\Vert \left\Vert \bm{U}^{\star}\right\Vert _{2,\infty}\lesssim\sigma_{1}^{\star2}\sqrt{\frac{\mu r}{d_{1}}}.\label{eq:B_U_true_2inf_norm}
\end{equation}
This taken collectively with the upper bound on $\left\Vert \bm{G}-\bm{G}^{\star}\right\Vert $
(cf.~(\ref{def:delta_op})) gives
\begin{align}
\frac{1}{\sigma_{r}^{\star4}}\left\Vert \bm{G}-\bm{G}^{\star}\right\Vert \left\Vert \bm{G}_{m,:}\bm{U}^{\star}\right\Vert _{2} & \leq\frac{1}{\sigma_{r}^{\star4}}\left\Vert \bm{G}-\bm{G}^{\star}\right\Vert \left\Vert \bm{G}\bm{U}^{\star}\right\Vert _{2,\infty}\lesssim\frac{\delta_{\mathsf{op}}\kappa^{2}}{\sigma_{r}^{\star2}}\sqrt{\frac{\mu r}{d_{1}}}.\label{eq:UH_BUtrue_term1}
\end{align}
\item Turning to the second term of (\ref{eq:UH-GU-expand}), we start with
the following bound
\[
\left\Vert \bm{G}_{m,:}\left(\bm{U}\bm{H}-\bm{U}^{\star}\right)\right\Vert _{2}\leq\big\|\bm{G}_{m,:}\big(\bm{U}\bm{H}-\bm{U}^{\left(m\right)}\bm{H}^{\left(m\right)}\big)\big\|_{2}+\big\|\bm{G}_{m,:}\big(\bm{U}^{\left(m\right)}\bm{H}^{\left(m\right)}-\bm{U}^{\star}\big)\big\|_{2}.
\]
Lemma~\ref{lemma:G_dev_W_2_norm} tells us that
\[
\big\|\bm{G}_{m,:}\big\|_{2}\leq\big\|\left(\bm{G}-\bm{G}^{\star}\right)_{m,:}\big\|_{2}+\left\Vert \bm{G}^{\star}\right\Vert _{2,\infty}\lesssim\sigma_{1}^{\star2}\left\Vert \bm{U}^{\star}\right\Vert _{2,\infty}\leq\sigma_{1}^{\star2}\left\Vert \bm{U}^{\star}\right\Vert =\sigma_{1}^{\star2},
\]
which makes use of the fact that $\left\Vert \bm{G}^{\star}\right\Vert _{2,\infty}\leq\|\bm{U}^{\star}\|_{2,\infty}\|\bm{\Sigma}^{\star}\|^{2}\|\bm{U}^{\star\top}\|=\sigma_{1}^{\star2}\|\bm{U}^{\star}\|_{2,\infty}$.
		This combined with (\ref{eq:claim3}) (to be established shortly in the proof of Lemma~\ref{lemma:U_U_loo_dist}) and the definitions of $\bm{H}$
and $\bm{H}^{(m)}$ gives
\begin{align}
\big\|\bm{G}_{m,:}\big(\bm{U}\bm{H}-\bm{U}^{\left(m\right)}\bm{H}^{\left(m\right)}\big)\big\|_{2} & \leq\big\|\bm{G}_{m,:}\big\|_{2}\big\|\bm{U}^{\left(m\right)}\bm{H}^{\left(m\right)}-\bm{U}\bm{H}\big\|\nonumber \\
 & =\big\|\bm{G}_{m,:}\big\|_{2}\big\|\big(\bm{U}^{\left(m\right)}\bm{U}^{\left(m\right)\top}-\bm{U}\bm{U}^{\top}\big)\bm{U}^{\star}\big\|\nonumber \\
 & \leq\big\|\bm{G}_{m,:}\big\|_{2}\big\|\bm{U}^{\left(m\right)}\bm{U}^{\left(m\right)\top}-\bm{U}\bm{U}^{\top}\big\|\nonumber \\
 & \lesssim\frac{\sigma_{1}^{\star2}}{\sigma_{r}^{\star2}}\delta_{\mathsf{loo}}\left(\left\Vert \bm{U}\bm{H}\right\Vert _{2,\infty}+\sqrt{\frac{\mu r}{d_{1}}}\right)\nonumber \\
 & \leq\delta_{\mathsf{loo}}\kappa^{2}\left(\left\Vert \bm{U}\bm{H}\right\Vert _{2,\infty}+\sqrt{\frac{\mu r}{d_{1}}}\right).\label{eq:UH_BUtrue_term2}
\end{align}
 In addition, Lemma~\ref{lemma:G_row_UH_Utrue} shows that with probability
at least $1-O\left(d^{-11}\right)$,
\begin{align}
\big\|\bm{G}_{m,:}\big(\bm{U}^{\left(m\right)}\bm{H}^{\left(m\right)}-\bm{U}^{\star}\big)\big\|_{2} & \lesssim\delta_{\mathsf{loo}}\,\big\|\bm{U}^{\left(m\right)}\bm{H}^{\left(m\right)}-\bm{U}^{\star}\big\|_{2,\infty}+\delta_{\mathsf{op}}\kappa^{2}\sqrt{\frac{\mu r}{d_{1}}}\nonumber \\
 & \leq\delta_{\mathsf{loo}}\,\big\|\bm{U}^{\left(m\right)}\bm{H}^{\left(m\right)}\big\|_{2,\infty}+\delta_{\mathsf{loo}}\left\Vert \bm{U}^{\star}\right\Vert _{2,\infty}+\delta_{\mathsf{op}}\kappa^{2}\sqrt{\frac{\mu r}{d_{1}}}\nonumber \\
 & \lesssim\delta_{\mathsf{loo}}\left\Vert \bm{U}\bm{H}\right\Vert _{2,\infty}+\big(\delta_{\mathsf{loo}}+\delta_{\mathsf{op}}\kappa^{2}\big)\sqrt{\frac{\mu r}{d_{1}}},\label{eq:UH_BUtrue_term3}
\end{align}
where the inequality (\ref{eq:UH_BUtrue_term3}) results from (\ref{eq:U_loo_2inf_UB}) (also established shortly in the proof of Lemma~\ref{lemma:U_U_loo_dist}).
\end{itemize}
Then claim immediately follows from (\ref{eq:UH_BUtrue_term1}), (\ref{eq:UH_BUtrue_term2}),
(\ref{eq:UH_BUtrue_term3}) and the union bound.

\subsection{Proof of Lemma~\ref{lemma:U_U_loo_dist}\label{subsec:pf:U_U_loo_dist}}

To begin with, recalling the definition of $\delta_{\mathsf{loo}}$ (cf.~\eqref{def:delta_loo}), we claim that
\begin{align}
\big\|\bm{U}^{\left(m\right)}\bm{U}^{\left(m\right)\top}-\bm{U}\bm{U}^{\top}\big\|_{\mathrm{F}} = \frac{1}{\sigma_{r}^{\star2}}  \underbrace{\big( \sigma_{\mathsf{col}}\sigma_{\mathsf{row}}\log d+\sigma_{\mathsf{col}}\left\Vert \bm{A}^{\star}\right\Vert \sqrt{\log d}\big)}_{=\delta_{\mathsf{loo}}} \left(\left\Vert \bm{U}\bm{H}\right\Vert _{2,\infty}+\sqrt{\frac{\mu r}{d_{1}}}\right) \label{eq:claim3}.
\end{align}
As mentioned before, one has $\delta_{\mathsf{loo}} \lesssim \zeta_{\mathsf{op}}$, from which the lemma follows immediately. The rest of the proof thus boils down
to proving the claim (\ref{eq:claim3}).

We shall apply the Davis-Kahan $\sin\bm{\Theta}$ theorem \cite{davis1970rotation}
to derive 
\begin{equation}
\big\|\bm{U}\bm{U}^{\top}-\bm{U}^{\left(m\right)}\bm{U}^{\left(m\right)\top}\big\|_{\mathrm{F}}\leq\frac{\left\Vert \big(\bm{G}-\bm{G}^{\left(m\right)}\big)\bm{U}^{\left(m\right)}\right\Vert _{\mathrm{F}}}{\lambda_{r}\left(\bm{G}^{\left(m\right)}\right)-\lambda_{r+1}\big(\bm{G}\big)}\leq\frac{2\left\Vert \big(\bm{G}-\bm{G}^{\left(m\right)}\big)\bm{U}^{\left(m\right)}\right\Vert _{\mathrm{F}}}{\sigma_{r}^{\star2}}.\label{eq:U_loo_DK}
\end{equation}
Here, the last inequality follows since, by Weyl's inequality,
\begin{align}
\lambda_{r}\big(\bm{G}^{\left(m\right)}\big)-\lambda_{r+1}\big(\bm{G}\big) & \geq\lambda_{r}\left(\bm{G}^{\star}\right)-\big\|\bm{G}^{\left(m\right)}-\bm{G}^{\star}\big\|-\lambda_{r+1}\left(\bm{G}^{\star}\right)-\left\Vert \bm{G}-\bm{G}^{\star}\right\Vert \nonumber \\
 & =\sigma_{r}^{\star2}-\big\|\bm{G}^{\left(m\right)}-\bm{G}^{\star}\big\|-\big\|\bm{G}-\bm{G}^{\star}\big\|\nonumber \\
 & \geq\sigma_{r}^{\star2}/2,\label{eq:U_loo_DK_deno}
\end{align}
where the last line follows since $\big\|\bm{G}^{\left(m\right)}-\bm{G}^{\star}\big\|\lesssim\delta_{\mathsf{op}}\ll\sigma_{r}^{\star2}$
--- an immediate consequence of Lemma~\ref{lemma:G_loo_op_diff}
and Condition~(\ref{cond:eig_o1}). As a side note, the fact $\big\|\bm{G}^{\left(m\right)}-\bm{G}^{\star}\big\|\ll\sigma_{r}^{\star2}$
also implies (according to \cite[Lemma~3]{abbe2017entrywise})
\begin{equation}
\big\|\big(\bm{H}^{\left(m\right)}\big)^{-1}\big\|\lesssim1,\label{eq:H_inverse_op_UB}
\end{equation}
which will be useful later.

It remains to control the term $\left\Vert \big(\bm{G}-\bm{G}^{\left(m\right)}\big)\bm{U}^{\left(m\right)}\right\Vert _{\mathrm{F}}$
in (\ref{eq:U_loo_DK}). Recall the definitions of $\bm{G}$ and $\bm{G}^{\left(m\right)}$
in \eqref{def:G} and (\ref{def:G_loo}), respectively. It is straightforward
to see that $\bm{G}-\bm{G}^{\left(m\right)}$ is a rank-$2$ symmetric
matrix with nonzero entries located only in the $m$-th row and the
$m$-th column. Simple calculation reveals that
\begin{align}
\big(\bm{G}-\bm{G}^{\left(m\right)}\big)_{m,i} & =\big\langle\bm{E}_{m,:},\bm{A}_{i,:}^{\mathsf{s}}\big\rangle,\quad i\neq m;\label{eq:G_G_loo_entry}\\
\big(\bm{G}-\bm{G}^{\left(m\right)}\big)_{m,m} & =0. \label{eq:G_G_loo_entry_m}
\end{align}
We can then derive
\begin{align}
\big\|\big(\bm{G}-\bm{G}^{\left(m\right)}\big)\bm{U}^{\left(m\right)}\big\|_{\mathrm{F}} & =\left\Vert \big(\bm{G}-\bm{G}^{\left(m\right)}\big)\bm{U}^{\left(m\right)}\bm{H}^{\left(m\right)}\big(\bm{H}^{\left(m\right)}\big)^{-1}\right\Vert _{\mathrm{F}}\leq\big\|\big(\bm{G}-\bm{G}^{\left(m\right)}\big)\bm{U}^{\left(m\right)}\bm{H}^{\left(m\right)}\big\|_{\mathrm{F}}\big\|\big(\bm{H}^{\left(m\right)}\big)^{-1}\big\|\nonumber \\
 & \lesssim\big\|\big(\bm{G}-\bm{G}^{\left(m\right)}\big)\bm{U}^{\left(m\right)}\bm{H}^{\left(m\right)}\big\|_{\mathrm{F}}\nonumber \\
 & \leq\big\|\mathcal{P}_{m,:}\big(\bm{G}-\bm{G}^{\left(m\right)}\big)\bm{U}^{\left(m\right)}\bm{H}^{\left(m\right)}\big\|_{\mathrm{F}}+\big\|\mathcal{P}_{:,m}\big(\bm{G}-\bm{G}^{\left(m\right)}\big)\bm{U}^{\left(m\right)}\bm{H}^{\left(m\right)}\big\|_{\mathrm{F}},\label{eq:G-Gm-diff-Um}
\end{align}
where the second line arises due to (\ref{eq:H_inverse_op_UB}), and
$\mathcal{P}_{m,:}$ (resp.~$\mathcal{P}_{:,m}$) is the projection
onto the subspace of matrix supported on $\left\{ m\right\} \times\left[d_{2}\right]$
(resp.~$\left[d_{1}\right]\times\left\{ m\right\} $).

To bound the first term of (\ref{eq:G-Gm-diff-Um}), we make the observation
(using (\ref{eq:G_G_loo_entry}) and \eqref{eq:G_G_loo_entry_m}) that
\[
\big\|\mathcal{P}_{m,:}\big(\bm{G}-\bm{G}^{\left(m\right)}\big)\bm{U}^{\left(m\right)}\bm{H}^{\left(m\right)}\big\|_{\mathrm{F}}=\big\|\left(\bm{A}^{\mathsf{s}}-\bm{A}^{\star}\right)_{m,:}\big[\mathcal{\mathcal{P}}_{-m,:}\left(\bm{A}^{\mathsf{s}}\right)\big]^{\top}\bm{U}^{\left(m\right)}\bm{H}^{\left(m\right)}\big\|_{2}.
\]
Controlling this quantity requires the assistance of leave-two-out matrices. Here, we only state our bound:  with probability at least $1-O\left(d^{-11}\right)$, one has
\begin{align}
 & \big\|\mathcal{P}_{m,:}\big(\bm{G}-\bm{G}^{\left(m\right)}\big)\bm{U}^{\left(m\right)}\bm{H}^{\left(m\right)}\big\|_{\mathrm{F}}\nonumber \\
 & \qquad\lesssim\sigma_{\mathsf{col}}\big(\sigma_{\mathsf{row}}\log d+\left\Vert \bm{A}^{\star}\right\Vert \sqrt{\log d}\big)\left(\big\|\bm{U}^{\left(m\right)}\bm{H}^{\left(m\right)}\big\|_{\mathrm{2,\infty}}+\sqrt{\frac{\mu r}{d_{1}}}\right).\label{eq:B_B_loo_row_U}
\end{align}
This bound will be restated in Lemma~\ref{lemma:A_dev_row_A_U_2_norm} and established in Appendix \ref{sec:proof-auxiliary-lemmas}. 
Turning to the second term of (\ref{eq:G-Gm-diff-Um}), we apply Lemma~\ref{lemma:G_loo_op_diff} (also established in Appendix \ref{sec:proof-auxiliary-lemmas})
to obtain 
\begin{align}
\big\|\mathcal{P}_{:,m}\big(\bm{G}-\bm{G}^{\left(m\right)}\big) & \bm{U}^{\left(m\right)}\bm{H}^{\left(m\right)}\big\|_{\mathrm{F}}=\Big(\sum\nolimits _{1\leq i \leq d_1}\big(\bm{G}-\bm{G}^{\left(m\right)}\big)_{i,m}^{2}\Big)^{1/2}\left\Vert \big(\bm{U}^{\left(m\right)}\bm{H}^{\left(m\right)}\big)_{m,:}\right\Vert _{2}\nonumber \\
 & \leq\big\|\bm{G}-\bm{G}^{\left(m\right)}\big\|\big\|\bm{U}^{\left(m\right)}\bm{H}^{\left(m\right)}\big\|_{2,\infty}\nonumber \\
 & \lesssim\big(\sigma_{\mathsf{col}}\big(\sigma_{\mathsf{row}}+\left\Vert \bm{A}^{\star}\right\Vert _{2,\infty}\big)\sqrt{\log d}\big)\big\|\bm{U}^{\left(m\right)}\bm{H}^{\left(m\right)}\big\|_{2,\infty}\label{eq:B_B_loo_col_U}
\end{align}
with probability at least $1-O\left(d^{-11}\right)$. Hence, we can
combine (\ref{eq:B_B_loo_row_U}), (\ref{eq:B_B_loo_col_U})
and (\ref{eq:U_loo_DK}) to yield
\begin{align}
 & \big\|\bm{U}^{\left(m\right)}\bm{U}^{\left(m\right)\top}-\bm{U}\bm{U}^{\top}\big\|_{\mathrm{F}}\nonumber \\
 & \qquad\lesssim\left(\sigma_{\mathsf{col}}\sigma_{\mathsf{row}}\log d+\sigma_{\mathsf{col}}\left\Vert \bm{A}^{\star}\right\Vert \sqrt{\log d}\right)\left(\big\|\bm{U}^{\left(m\right)}\bm{H}^{\left(m\right)}\big\|_{\mathrm{2,\infty}}+\sqrt{\frac{\mu r}{d_{1}}}\right)\nonumber \\
 & \qquad=\frac{\delta_{\mathsf{loo}}}{\sigma_{r}^{\star2}}\left(\big\|\bm{U}^{\left(m\right)}\bm{H}^{\left(m\right)}\big\|_{\mathrm{2,\infty}}+\sqrt{\frac{\mu r}{d_{1}}}\right),\label{eq:U_loo_dist_temp}
\end{align}
where $\delta_{\mathsf{loo}}$ is defined in (\ref{def:delta_loo}).
As a result, the proof is complete as long as we can show that
\begin{equation}
\big\|\bm{U}^{\left(m\right)}\bm{H}^{\left(m\right)}\big\|_{2,\infty}\lesssim\left\Vert \bm{U}\bm{H}\right\Vert _{2,\infty}+\sqrt{\frac{\mu r}{d_{1}}}.\label{eq:U_loo_2inf_UB}
\end{equation}

To finish up, it remains to justify this inequality (\ref{eq:U_loo_2inf_UB}).
To this end, from the definitions of $\bm{H}^{(m)}$ and $\bm{H}$
we have 
\begin{align}
\big\|\bm{U}^{\left(m\right)}\bm{H}^{\left(m\right)}\big\|_{2,\infty} & \leq\big\|\bm{U}^{\left(m\right)}\bm{H}^{\left(m\right)}-\bm{U}\bm{H}\big\|_{2,\infty}+\left\Vert \bm{U}\bm{H}\right\Vert _{2,\infty}\nonumber \\
 & =\big\|\big(\bm{U}^{\left(m\right)}\bm{U}^{\left(m\right)\top}-\bm{U}\bm{U}^{\top}\big)\bm{U}^{\star}\big\|_{2,\infty}+\left\Vert \bm{U}\bm{H}\right\Vert _{2,\infty}\nonumber \\
 & \leq\big\|\bm{U}^{\left(m\right)}\bm{U}^{\left(m\right)\top}-\bm{U}\bm{U}^{\top}\big\|_{\mathrm{F}}\|\bm{U}^{\star}\|+\left\Vert \bm{U}\bm{H}\right\Vert _{2,\infty}\nonumber \\
 & =\big\|\bm{U}^{\left(m\right)}\bm{U}^{\left(m\right)\top}-\bm{U}\bm{U}^{\top}\big\|_{\mathrm{F}}+\left\Vert \bm{U}\bm{H}\right\Vert _{2,\infty}.\label{eq:UmHm-UB1}
\end{align}
Under the condition (\ref{cond:eig_o1}), it is easily seen that $\delta_{\mathsf{loo}}\ll\sigma_{r}^{\star2}$.
This together with (\ref{eq:U_loo_dist_temp}) gives
\begin{equation}
\big\|\bm{U}^{\left(m\right)}\bm{U}^{\left(m\right)\top}-\bm{U}\bm{U}^{\top}\big\|_{\mathrm{F}}\leq0.5\big\|\bm{U}^{\left(m\right)}\bm{H}^{\left(m\right)}\big\|_{2,\infty}+0.5\sqrt{\frac{\mu r}{d_{1}}},\label{eq:U_loo_dist_o1}
\end{equation}
which combined with (\ref{eq:UmHm-UB1}) yields
\begin{align}
\big\|\bm{U}^{\left(m\right)}\bm{H}^{\left(m\right)}\big\|_{2,\infty} & \leq2\left\Vert \bm{U}\bm{H}\right\Vert _{2,\infty}+\sqrt{\frac{\mu r}{d_{1}}}\label{eq:UmHm-UB1-1}
\end{align}
as claimed.

\subsection{Proof of Lemma~\ref{lemma:noise_cond}\label{subsec:pf:noise_cond}}

We start with (\ref{cond:B_UB}). In view of the definitions of $B,\sigma_{\infty},\sigma_{\mathsf{row}}$
and $\sigma_{\mathsf{col}}$ in (\ref{eq:noise_value}), we have with
probability at least $1-O\left(d^{-12}\right)$,
\begin{align*}
B^{2} & =\frac{\mu r\sigma_{1}^{\star2}}{d_{1}d_{2}p^{2}}+\frac{\sigma^{2}\min\left\{ \sqrt{d_{1}d_{2}},d_{2}\right\} }{p\log d},\\
\sigma_{\mathsf{row}}^{2} & =\frac{\mu r\sigma_{1}^{\star2}}{d_{1}p}+\frac{\sigma^{2}d_{2}}{p},\\
\sigma_{\mathsf{row}}\sigma_{\mathsf{col}} & =\frac{\mu r\sigma_{1}^{\star2}}{\sqrt{d_{1}d_{2}}\,p}+\frac{2\sigma\sigma_{1}^{\star}\sqrt{\mu r}}{p}+\frac{\sigma^{2}\sqrt{d_{1}d_{2}}}{p}\asymp\frac{\mu r\sigma_{1}^{\star2}}{\sqrt{d_{1}d_{2}}\,p}+\frac{\sigma^{2}\sqrt{d_{1}d_{2}}}{p},
\end{align*}
where we have used the AM-GM inequality (i.e.~$2\sigma\sigma_{1}^{\star}\sqrt{\mu r}\leq\frac{\mu r\sigma_{1}^{\star2}}{\sqrt{d_{1}d_{2}}}+\sigma^{2}\sqrt{d_{1}d_{2}}$)
in the last line. Therefore,
\begin{align*}
B^{2}\log d\lesssim\sigma_{\mathsf{row}}\sigma_{\mathsf{col}} & \qquad\text{and}\qquad B^{2}\log d\lesssim\sigma_{\mathsf{row}}^{2}
\end{align*}
 hold as long as $p\gtrsim\left(d_{1}d_{2}\right)^{-1/2}\log d$ and $p \gtrsim d_2^{-1}\log d$.

The next step is to establish (\ref{cond:col_B_rel}). Let us consider
the first inequality. By (\ref{eq:noise_value}), it is easily seen
that
\begin{align*}
\sigma_{\infty}^{2} & =\frac{\mu r\sigma_{1}^{\star2}}{d_{1}d_{2}p}+\frac{\sigma^{2}}{p};\\
B\log d\left\Vert \bm{A}^{\star}\right\Vert \sqrt{\frac{\mu r}{d_{2}}} & =\frac{\mu r\sigma_{1}^{\star2}\log d}{\sqrt{d_{1}}\,d_{2}p}+\sigma\sigma_{1}^{\star}\min\left\{ \sqrt[4]{d_{1}d_{2}},\,\sqrt{d_{2}}\right\} \sqrt{\frac{\mu r\log d}{d_{2}p}}.
\end{align*}
As a consequence, the first inequality holds as long as $\frac{\sigma}{\sigma_{1}^{\star}}\lesssim\min\left\{ \sqrt[4]{d_{1}d_{2}},\,\sqrt{d_{2}}\right\} \sqrt{\frac{\mu rp\log d}{d_{2}}}$,
which is satisfied by our noise assumption that $\frac{\sigma}{\sigma_{r}^{\star}}\ll\frac{\sqrt{p}}{\kappa\sqrt[4]{d_{1}d_{2}}\,\sqrt{\log d}}$.
To show the second inequality, we note that
\begin{align*}
B\sqrt{\log d}\left\Vert \bm{A}^{\star}\right\Vert \sqrt{\frac{\mu r}{d_{2}}} & \leq\left\{ \frac{\sigma_{1}^{\star2}\sqrt{\mu r\log d}}{d_{2}p}+\sigma\sigma_{1}^{\star}\sqrt{\frac{d_{1}}{p}}\right\} \sqrt{\frac{\mu r}{d_{1}}}.
\end{align*}
Recognizing that $\sigma_{\mathsf{col}}\left\Vert \bm{A}^{\star}\right\Vert =\sigma_{1}^{\star2}\sqrt{\frac{\mu r}{d_{2}p}}+\sigma\sigma_{1}^{\star}\sqrt{\frac{d_{1}}{p}}$,
we prove the second inequality provided that $p\gtrsim d_{2}^{-1}\log d$.

When it comes to (\ref{cond:row_B_rel}): by virtue of Lemma~\ref{lemma:incoh} and
(\ref{eq:noise_value}), one has
\[
\frac{B\log^{3/2}d\left\Vert \bm{A}^{\star}\right\Vert _{\infty}}{\left\Vert \bm{A}^{\star}\right\Vert ^{2}}\leq\frac{\mu r\log^{3/2}d}{d_{1}d_{2}p}+\frac{\sigma\sqrt{\mu r}\log d}{\sigma_{1}^{\star}\sqrt[4]{d_{1}d_{2}}\,\sqrt{p}}.
\]
Consequently, (\ref{cond:row_B_rel}) holds provided $p\gtrsim\frac{\sqrt{\mu r}\,\log^{3/2}d}{\sqrt{d_{1}}\,d_{2}}$
and $\frac{\sigma}{\sigma_{1}^{\star}}\lesssim\sqrt[4]{\frac{d_{2}}{d_{1}}}\frac{\sqrt{p}}{\log d}$,
which holds under our assumptions that $p\gg\frac{\mu\kappa^{4}r\log^{2}d}{\sqrt{d_{1}d_{2}}}$
and $\frac{\sigma}{\sigma_{r}^{\star}}\ll\frac{\sqrt{p}}{\kappa^{3}\sqrt[4]{d_{1}d_{2}}\,\sqrt{\log d}}$.

Finally, using the definitions in (\ref{eq:noise_value}) and Lemma~\ref{lemma:incoh},
one obtains the following bounds:
\begin{align*}
\sigma_{\mathsf{row}}\sigma_{\mathsf{col}}\log d & \asymp\frac{\mu r\sigma_{1}^{\star2}\log d}{\sqrt{d_{1}d_{2}}\,p}+\frac{\sigma^{2}\sqrt{d_{1}d_{2}}\log d}{p},\\
\sigma_{\mathsf{col}}\sqrt{\log d}\left\Vert \bm{A}^{\star}\right\Vert  & =\sigma_{1}^{\star2}\sqrt{\frac{\mu r\log d}{d_{2}p}}+\sigma\sigma_{1}^{\star}\sqrt{\frac{d_{1}\log d}{p}};\\
\left\Vert \bm{A}^{\star}\right\Vert _{2,\infty}^{2} & \lesssim\frac{\mu_1 r\sigma_{1}^{\star2}}{d_{1}};\\
\sigma_{\mathsf{row}}\sqrt{\frac{d_{1}\log d}{\mu r}}\left\Vert \bm{A}^{\star\top}\right\Vert _{2,\infty} & \leq\sigma_{1}^{\star2}\sqrt{\frac{\mu r\log d}{d_{2}p}}+\sigma\sigma_{1}^{\star}\sqrt{\frac{d_{1}\log d}{p}};\\
B\left\Vert \bm{A}^{\star}\right\Vert _{\infty} \log d & \leq\frac{\mu r\sigma_{1}^{\star2}\log d}{d_{1}d_{2}p}+\frac{\sigma\sigma_{1}^{\star}\sqrt{\mu r\log d}}{\sqrt[4]{d_{1}d_{2}}\,\sqrt{p}}.
\end{align*}
Therefore, it is easy to verify (\ref{cond:eig_o1}) under the
assumptions of Theorem \ref{thm:U_loss}.

\section{Proofs for auxiliary lemmas}
\label{sec:proof-auxiliary-lemmas}

This section establishes several useful technical lemmas useful for
proving our main theorems. Throughout this section, we shall frequently
use the quantities $B,\sigma_{\infty},\sigma_{\mathsf{row}}$ and
$\sigma_{\mathsf{col}}$ defined in (\ref{eq:noise_value}). In fact,
it suffices to bear in mind the following bounds
\begin{equation}
\begin{array}{cc}
 & B\geq\max_{i,j}|E_{i,j}|;\qquad\qquad\sigma_{\infty}\geq\max_{i,j}\sqrt{\mathbb{E}[E_{i,j}^{2}]};\\
 & \sigma_{\mathsf{row}}\geq\max_{i}\sqrt{\sum\nolimits _{j}\mathbb{E}[E_{i,j}^{2}]};\qquad\sigma_{\mathsf{col}}\geq\max_{j}\sqrt{\sum\nolimits _{i}\mathbb{E}[E_{i,j}^{2}]}.
\end{array}\label{eq:B-sigma-bound}
\end{equation}

\subsection{Auxiliary technical lemmas}

We first gather all technical lemmas to be established in this section,
and begin with the following lemma, which shows that the leave-one-out
sequence $\bm{G}^{\left(m\right)}$ is close to $\bm{G}$ and $\bm{G}^{\star}$
when measured by the spectral norm.

\begin{lemma}
\label{lemma:G_loo_op_diff}
Instate the assumptions of Theorem \ref{thm:U_loss}. With probability at least $1-O\left(d^{-11}\right)$,
one has
\begin{align}
\big\|\bm{G}^{\left(m\right)}-\bm{G}\big\| & \lesssim\sigma_{\mathsf{row}}\big(\sigma_{\mathsf{col}}+\left\Vert \bm{A}^{\star\top}\right\Vert _{2,\infty}\big)\sqrt{\log d},\label{claim:G_G_loo_diff_op_norm}\\
\big\|\bm{G}^{\left(m\right)}-\bm{G}^{\star}\big\| & \lesssim\delta_{\mathsf{op}}=\left(\sigma_{\mathsf{row}}+\sigma_{\mathsf{col}}\right)\big(\sigma_{\mathsf{col}}+\left\Vert \bm{A}^{\star\top}\right\Vert _{2,\infty}\big)\log d+\sigma_{\mathsf{col}}\sqrt{\log d}\left\Vert \bm{A}^{\star}\right\Vert +\left\Vert \bm{A}^{\star}\right\Vert _{2,\infty}^{2},\label{claim:G_loo_op_loss}
\end{align}
where $\sigma_{\mathsf{row}}$ and $\sigma_{\mathsf{col}}$ are defined
in (\ref{eq:noise_value}).\end{lemma}\begin{proof}See Appendix~\ref{subsec:pf:G_loo_op_diff}.\end{proof}

Similar to $\bm{H}$ (defined in (\ref{eq:defn-H})), we also introduce
the following matrices for each $\left(m,l\right)\in\left[d_{1}\right]\times\left[d_{2}\right]$:
\begin{subequations}
\begin{align}
\bm{H}^{\left(m\right)} & :=\bm{U}^{\left(m\right)\top}\bm{U}^{\star},\label{def:H_loo}\\
\bm{H}^{\left(m,l\right)} & :=\bm{U}^{\left(m,l\right)\top}\bm{U}^{\star},\label{def:H_loo_col}
\end{align}
\end{subequations}where $\bm{U}^{\left(m\right)}$ and $\bm{U}^{\left(m,l\right)}$
are defined in Algorithm~\ref{alg:init_loo} and Algorithm \ref{alg:init_loo_col},
respectively.

Lemma~\ref{lemma:A_dev_row_A_U_2_norm} serves a crucial step towards
proving Lemma~\ref{lemma:U_U_loo_dist}.

\begin{lemma}
\label{lemma:A_dev_row_A_U_2_norm}
Instate the assumptions
of Theorem \ref{thm:U_loss}. For any fixed $1\leq m \leq d_{1}$, with probability
at least $1-O(d^{-11})$, one has
\begin{align*}
 & \big\|\left(\bm{A}^{\mathsf{s}}-\bm{A}^{\star}\right)_{m,:}\mathcal{\mathcal{P}}_{-m,:}\left(\bm{A}^{\mathsf{s}}\right)^{\top}\bm{U}^{\left(m\right)}\bm{H}^{\left(m\right)}\big\|_{2}\\
 & \qquad\lesssim\sigma_{\mathsf{col}}\big(\sigma_{\mathsf{row}}\log d+\left\Vert \bm{A}^{\star}\right\Vert \sqrt{\log d}\big)\left(\big\|\bm{U}^{\left(m\right)}\bm{H}^{\left(m\right)}\big\|_{\mathrm{2,\infty}}+\sqrt{\frac{\mu r}{d_{1}}}\right),
\end{align*}
where $\bm{A}^{\mathsf{s}}$ is defined in (\ref{eq:A-E}), and $\sigma_{\mathsf{row}}$
and $\sigma_{\mathsf{col}}$ are both defined in (\ref{eq:noise_value}).\end{lemma}
\begin{proof}See Appendix~\ref{subsec:pf:A_dev_row_A_U_2_norm}.\end{proof}

The proof of Lemma~\ref{lemma:A_dev_row_A_U_2_norm} relies on an upper bound on the $\ell_{2,\infty}$ norm of $\mathcal{\mathcal{P}}_{-m,:}\left(\bm{A}^{\mathsf{s}}\right)^{\top}\bm{U}^{\left(m\right)}\bm{H}^{\left(m\right)}$,
which is formalized below in Lemma~\ref{lemma:W_loo_2inf_norm}. This is built upon a leave-two-out argument. 

\begin{lemma}
\label{lemma:W_loo_2inf_norm}
Instate the assumptions
of Theorem \ref{thm:U_loss}. With probability at least $1-O\big(d^{-10}\big)$,
the following holds simultaneously for all $m\in[d_{1}]$,
\begin{align*}
\big\|\mathcal{\mathcal{P}}_{-m,:}\left(\bm{A}^{\mathsf{s}}\right)^{\top}\bm{U}^{\left(m\right)}\bm{H}^{\left(m\right)}\big\|_{2,\infty} & \lesssim\big(B\log d+\sigma_{\mathsf{col}}\sqrt{\log d}\big)\big\|\bm{U}^{\left(m\right)}\bm{H}^{\left(m\right)}\big\|_{2,\infty}+\left\Vert \bm{A}^{\star}\right\Vert \sqrt{\frac{\mu r}{d_{2}}},
\end{align*}
where $\bm{A}^{\mathsf{s}}$ is defined in (\ref{eq:A-E}), and $B$
and $\sigma_{\mathsf{col}}$ are defined in (\ref{eq:noise_value}).
\end{lemma}
\begin{proof}See Appendix~\ref{subsec:pf:W_loo_2inf_norm}.\end{proof}

The proof of Lemma~\ref{lemma:W_loo_2inf_norm} requires the proximity
between $\bm{U}^{(m)}$ and $\bm{U}^{(m,l)}$, which is demonstrated
below in Lemma~\ref{lemma:U_loo_col_dist}.

\begin{lemma}\label{lemma:U_loo_col_dist}Instate the assumptions
of Theorem \ref{thm:U_loss}. With probability at least $1-O\left(d^{-10}\right)$,
the following holds simultaneously for any $m\in[d_{1}]$ and $l\in[d_{2}]$,
\begin{align}
\big\|\bm{U}^{\left(m\right)}\bm{U}^{\left(m\right)\top}-\bm{U}^{\left(m,l\right)}\bm{U}^{\left(m,l\right)\top}\big\| & \lesssim\frac{1}{\sigma_{r}^{\star2}}\big(B\log d+\sigma_{\mathsf{col}}\sqrt{\log d}\big)^{2}\big\|\bm{U}^{\left(m\right)}\bm{H}^{\left(m\right)}\big\|_{2,\infty}\nonumber \\
 & \quad+\frac{\sigma_{\infty}^{2}}{\sigma_{r}^{\star2}}+\frac{1}{\sigma_{r}^{\star2}}\big(B\log d+\sigma_{\mathsf{col}}\sqrt{\log d}\big)\left\Vert \bm{A}^{\star\top}\right\Vert _{2,\infty},\label{claim:U_loo_U_loo_col_dist}
\end{align}
where $B,\sigma_{\infty}$ and $\sigma_{\mathsf{col}}$
are defined in (\ref{eq:noise_value}). \end{lemma}\begin{proof}See
Appendix~\ref{subsec:pf:U_loo_col_dist}.\end{proof}

Finally, Lemma~\ref{lemma:G_row_UH_Utrue} stated below constitutes the main part of Lemma~\ref{lemma:UH_BUtrue}
(recalling the decomposition in (\ref{eq:UH-UB2}) and (\ref{eq:UH-UB3})).

\begin{lemma}\label{lemma:G_row_UH_Utrue}Instate the assumptions
of Theorem \ref{thm:U_loss}. For each fixed $m\in[d_{1}]$, the following
holds with probability exceeding $1-O\left(d^{-11}\right)$,
\begin{align*}
\left\Vert \bm{G}_{m,:}\big(\bm{U}^{\left(m\right)}\bm{H}^{\left(m\right)}-\bm{U}^{\star}\big)\right\Vert _{2} & \lesssim\delta_{\mathsf{loo}}\,\big\|\bm{U}^{\left(m\right)}\bm{H}^{\left(m\right)}-\bm{U}^{\star}\big\|_{2,\infty}+\delta_{\mathsf{op}}\kappa^{2}\sqrt{\frac{\mu r}{d_{1}}},
\end{align*}
where $\delta_{\mathsf{op}}$ and $\delta_{\mathsf{loo}}$ are defined
in (\ref{claim:G_op_loss}) and (\ref{def:delta_loo}), respectively.\end{lemma}\begin{proof}See
Appendix~\ref{subsec:pf:G_row_UH_Utrue}.\end{proof}

\subsection{Proof of Lemma~\ref{lemma:G_loo_op_diff}}

\label{subsec:pf:G_loo_op_diff}

Recall the definitions of $\bm{G}$ and $\bm{G}^{\left(m\right)}$
in \eqref{def:G} and \eqref{def:G_loo}. As shown in \eqref{eq:G_G_loo_entry}
in the proof of Lemma~\ref{lemma:U_U_loo_dist} in Appendix~\ref{subsec:pf:U_U_loo_dist},
we know that $\bm{G}-\bm{G}^{\left(m\right)}$ is a rank-$2$ symmetric
matrix with nonzero entries located only in the $m$-th row and the
$m$-th column. In particular, one has
\begin{align*}
\big(\bm{G}-\bm{G}^{\left(m\right)}\big)_{m,i} & =\big\langle\bm{E}_{m,:},\bm{A}_{i,:}^{\mathsf{s}}\big\rangle,\qquad i\neq m,\\
\big(\bm{G}-\bm{G}^{\left(m\right)}\big)_{m,m} & =0,
\end{align*}
thus indicating that
\[
\big(\bm{G}-\bm{G}^{\left(m\right)}\big)_{m,:}=\bm{E}_{m,:}\big[\mathcal{\mathcal{P}}_{-m,:}\left(\bm{A}^{\mathsf{s}}\right)\big]^{\top}.
\]
This allows us to upper bound
\begin{align*}
\big\|\bm{G}-\bm{G}^{\left(m\right)}\big\| & \leq\big\|\bm{G}-\bm{G}^{\left(m\right)}\big\|_{\mathrm{F}}\lesssim\Big\|\big(\bm{G}-\bm{G}^{\left(m\right)}\big)_{m,:}\Big\|_{2}\\
 & =\left\Vert \bm{E}_{m,:}\big[\mathcal{\mathcal{P}}_{-m,:}\left(\bm{A}^{\mathsf{s}}\right)\big]^{\top}\right\Vert _{2}\\
 & \lesssim\sigma_{\mathsf{row}}\big(\sigma_{\mathsf{col}}+\left\Vert \bm{A}^{\star\top}\right\Vert _{2,\infty}\big)\sqrt{\log d}.
\end{align*}
Here, the last line follows the following. First, notice $\bm{A}^{\mathsf{s}}=\bm{A}^{\star}+\bm{E}$
and
\[
\bm{E}_{m,:}\big[\mathcal{\mathcal{P}}_{-m,:}\left(\bm{A}^{\star}+\bm{E}\right)\big]^{\top}=\sum\nolimits _{i:i\neq m}\left\langle \bm{E}_{m,:},\bm{E}_{i,:}\right\rangle \bm{e}_{i}^\top+\sum\nolimits _{i:i\neq m}\left\langle \bm{E}_{m,:},\bm{A}_{i,:}^{\star}\right\rangle \bm{e}_{i}^\top,
\]
where $\bm{e}_{i}$ is the $i$-th standard basis in $\mathbb{R}^{d_1}$. It follows from \eqref{eq:B_dev_W_2_norm_term1} and \eqref{eq:B_dev_W_2_norm_term3}
shown in the proof of Lemma~\ref{lemma:G_dev_W_2_norm} (cf.~Appendix~\ref{subsec:pf:G_dev_W_2_norm})
that with probability at least $1-O(d^{-11})$,
\begin{align*}
\left\Vert \bm{E}_{m,:}\big[\mathcal{\mathcal{P}}_{-m,:}\left(\bm{A}^{\star}+\bm{E}\right)\big]^{\top}\right\Vert _{2} & \leq\Big\|\sum\nolimits _{i:i\neq m}\left\langle \bm{E}_{m,:},\bm{E}_{i,:}\right\rangle \bm{e}_{i}^\top\Big\|_{2}+\Big\|\sum\nolimits _{i:i\neq m}\left\langle \bm{E}_{m,:},\bm{A}_{i,:}^{\star}\right\rangle \bm{e}_{i}^\top\Big\|_{2}\\
 & \lesssim\sigma_{\mathsf{col}}\sigma_{\mathsf{row}}\sqrt{\log d}+\sigma_{\mathsf{row}}\sqrt{\log d}\left\Vert \bm{A}^{\star\top}\right\Vert _{2,\infty}.
\end{align*}
In addition, the above bound combined with Lemma~\ref{lemma:G_op_loss}
immediately yields \eqref{claim:G_loo_op_loss}. The proof is complete
by taking the union bound over $1\leq m\leq d_{1}$.

\subsection{Proof of Lemma~\ref{lemma:A_dev_row_A_U_2_norm}\label{subsec:pf:A_dev_row_A_U_2_norm}}

By construction, the $m$-th row of $\bm{A}^{\mathsf{s}}-\bm{A}^{\star}$
is independent of $\big[\mathcal{\mathcal{P}}_{-m,:}\left(\bm{A}^{\mathsf{s}}\right)\big]^{\top}\bm{U}^{\left(m\right)}\bm{H}^{\left(m\right)}$.
As a result,
\[
\left(\bm{A}^{\mathsf{s}}-\bm{A}^{\star}\right)_{m,:}\big[\mathcal{\mathcal{P}}_{-m,:}\left(\bm{A}^{\mathsf{s}}\right)\big]^{\top}\bm{U}^{\left(m\right)}\bm{H}^{\left(m\right)}=\sum_{j\in\left[d_{2}\right]}E_{m,j}\left(\big[\mathcal{\mathcal{P}}_{-m,:}\left(\bm{A}^{\mathsf{s}}\right)\big]^{\top}\bm{U}^{\left(m\right)}\bm{H}^{\left(m\right)}\right)_{j,:}
\]
can be viewed as a sum of independent zero-mean random vectors (where
the randomness comes from $\left\{ E_{m,j}\right\} _{j\in\left[d_{2}\right]}$).
It is straightforward to calculate that
\begin{align*}
L & :=\max_{j\in\left[d_{2}\right]}\left\Vert E_{m,j}\left(\big[\mathcal{\mathcal{P}}_{-m,:}\left(\bm{A}^{\mathsf{s}}\right)\big]^{\top}\bm{U}^{\left(m\right)}\bm{H}^{\left(m\right)}\right)_{j,:}\right\Vert _{2}\leq B\left\Vert \mathcal{\mathcal{P}}_{-m,:}\left(\bm{A}^{\mathsf{s}}\right)^{\top}\bm{U}^{\left(m\right)}\bm{H}^{\left(m\right)}\right\Vert _{2,\infty},\\
V & :=\sum_{j\in\left[d_{2}\right]}\mathbb{E}\left[E_{m,j}^{2}\right]\left\Vert \left(\big[\mathcal{\mathcal{P}}_{-m,:}\left(\bm{A}^{\mathsf{s}}\right)\big]^{\top}\bm{U}^{\left(m\right)}\bm{H}^{\left(m\right)}\right)_{j,:}\right\Vert _{2}^{2}\leq\sigma_{\infty}^{2}\left\Vert \big[\mathcal{\mathcal{P}}_{-m,:}\left(\bm{A}^{\mathsf{s}}\right)\big]^{\top}\bm{U}^{\left(m\right)}\bm{H}^{\left(m\right)}\right\Vert _{\mathrm{F}}^{2}.
\end{align*}

In view of the matrix Bernstein inequality, it boils down to controlling
$L$ and $V$. To this end, let us first bound $L$. From Lemma~\ref{lemma:W_loo_2inf_norm},
one has that with probability at least $1-O\left(d^{-11}\right)$,
\begin{align}
L & \lesssim B\big(B\log d+\sigma_{\mathsf{col}}\sqrt{\log d}\big)\big\|\bm{U}^{\left(m\right)}\bm{H}^{\left(m\right)}\big\|_{2,\infty}+B\left\Vert \bm{A}^{\star}\right\Vert \sqrt{\frac{\mu r}{d_{2}}}.\label{eq:A_noise_dev_W_L_UB}
\end{align}
Regarding $V$, Lemma~\ref{lemma:A_op_norm} guarantees the following
upper bound with probability exceeding $1-O\left(d^{-11}\right)$,
\begin{align}
\left\Vert \big[\mathcal{\mathcal{P}}_{-m,:}\left(\bm{A}^{\mathsf{s}}\right)\big]^{\top}\bm{U}^{\left(m\right)}\bm{H}^{\left(m\right)}\right\Vert _{\mathrm{F}} & \leq\big\|\mathcal{\mathcal{P}}_{-m,:}\left(\bm{A}^{\mathsf{s}}\right)\big\|\big\|\bm{U}^{\left(m\right)}\bm{H}^{\left(m\right)}\big\|_{\mathrm{F}}\nonumber \\
 & \leq\sqrt{d_{1}}\left\Vert \bm{A}^{\mathsf{s}}\right\Vert \big\|\bm{U}^{\left(m\right)}\bm{H}^{\left(m\right)}\big\|_{\mathrm{2,\infty}}\nonumber \\
 & \lesssim\sqrt{d_{1}}\,\big(B\log d + \left(\sigma_{\mathsf{row}}+\sigma_{\mathsf{col}}\right)\sqrt{\log d}+\left\Vert \bm{A}^{\star}\right\Vert \big)\big\|\bm{U}^{\left(m\right)}\bm{H}^{\left(m\right)}\big\|_{\mathrm{2,\infty}}\nonumber \\
 & \lesssim\sqrt{d_{1}}\,\big(\sigma_{\mathsf{row}}\sqrt{\log d}+\left\Vert \bm{A}^{\star}\right\Vert \big)\big\|\bm{U}^{\left(m\right)}\bm{H}^{\left(m\right)}\big\|_{\mathrm{2,\infty}},\label{eq:A_U_fro_norm}
\end{align}
where the last inequality follows from the condition~\eqref{cond:eig_o1}
that $B\log d + \sigma_{\mathsf{col}}\sqrt{\log d}\ll\sigma_{r}^{\star}$. Applying
the matrix Bernstein inequality yields that with probability at least
$1-O\left(d^{-11}\right)$: one has
\begin{align}
 & \left\Vert \left(\bm{A}^{\mathsf{s}}-\bm{A}^{\star}\right)_{m,:}\big[\mathcal{\mathcal{P}}_{-m,:}\left(\bm{A}^{\mathsf{s}}\right)\big]^{\top}\bm{U}^{\left(m\right)}\bm{H}^{\left(m\right)}\right\Vert _{2}\lesssim L\log d+\sqrt{V\log d}\nonumber \\
 & \qquad\lesssim\big(B^{2}\log^{2}d+B\sigma_{\mathsf{col}}\log^{3/2}d\big)\big\|\bm{U}^{\left(m\right)}\bm{H}^{\left(m\right)}\big\|_{2,\infty}+B\log d\left\Vert \bm{A}^{\star}\right\Vert \sqrt{\frac{\mu r}{d_{2}}}\nonumber \\
 & \qquad\quad+\sqrt{d_{1}}\,\sigma_{\infty}\big(\sigma_{\mathsf{row}}\log d+\left\Vert \bm{A}^{\star}\right\Vert \sqrt{\log d}\big)\big\|\bm{U}^{\left(m\right)}\bm{H}^{\left(m\right)}\big\|_{\mathrm{2,\infty}}\nonumber \\
 & \qquad\overset{\left(\mathrm{i}\right)}{\lesssim}\sigma_{\mathsf{col}}\big(\sigma_{\mathsf{row}}\log d+\left\Vert \bm{A}^{\star}\right\Vert \sqrt{\log d}\big)\big\|\bm{U}^{\left(m\right)}\bm{H}^{\left(m\right)}\big\|_{\mathrm{2,\infty}}+\sigma_{\mathsf{col}}\left\Vert \bm{A}^{\star}\right\Vert \sqrt{\log d}\sqrt{\frac{\mu r}{d_{1}}}\nonumber \\
 & \qquad\lesssim\sigma_{\mathsf{col}}\big(\sigma_{\mathsf{row}}\log d+\left\Vert \bm{A}^{\star}\right\Vert \sqrt{\log d}\big)\left(\big\|\bm{U}^{\left(m\right)}\bm{H}^{\left(m\right)}\big\|_{\mathrm{2,\infty}}+\sqrt{\frac{\mu r}{d_{1}}}\right),\label{eq:A_dev_row_A_U_2_norm_UB1}
\end{align}
where (i) uses \eqref{eq:noise_value} that $\sigma_{\mathsf{col}}^{2}\asymp d_{1}\sigma_{\infty}^{2}$
as well as the conditions~\eqref{cond:B_UB}, \eqref{cond:eig_o1}
and \eqref{cond:col_B_rel} (namely, $B\lesssim\sqrt{\sigma_{\mathsf{row}}\sigma_{\mathsf{col}}/\log d},$
$B\log d\ll\left\Vert \bm{A}^{\star}\right\Vert $ and $B\log d\sqrt{\mu r/d_{2}}\lesssim\sigma_{\mathsf{col}}\sqrt{\log d}\sqrt{\mu r/d_{1}}$).

\subsection{Proof of Lemma~\ref{lemma:W_loo_2inf_norm}\label{subsec:pf:W_loo_2inf_norm}}

For notational convenience, we denote
\[
\bm{A}^{\mathsf{s},\left(m\right),0}:=\mathcal{P}_{-m,:}\left(\bm{A}^{\mathsf{s}}\right).
\]
Fix an arbitrary $l\in[d_{2}]$, and we would like to upper bound
$\big\|\bm{A}_{:,l}^{\mathsf{s},\left(m\right),0\top}\bm{U}^{\left(m\right)}\bm{H}^{\left(m\right)}\big\|_{2}.$
The main difficulty here lies in the complicated statistical dependence
between $\bm{A}_{:,l}^{\mathsf{s},\left(m\right),0}$ and $\bm{U}^{\left(m\right)}\bm{H}^{\left(m\right)}$.
Recall the definitions of the auxiliary matrices $\bm{U}^{\left(m,l\right)}$
and $\bm{H}^{\left(m,l\right)}$ in Algorithm~\ref{alg:init_loo_col}
and \eqref{def:H_loo_col}, respectively. By construction, $\bm{A}_{:,l}^{\mathsf{s},\left(m\right),0}$
is independent of $\bm{U}^{\left(m,l\right)}$ and $\bm{H}^{\left(m,l\right)}$.
Moreover, Lemma~\ref{lemma:U_loo_col_dist} guarantees that $\bm{U}^{\left(m\right)}\bm{H}^{\left(m\right)}$
is extremely close to $\bm{U}^{\left(m,l\right)}\bm{H}^{\left(m,l\right)}$.
Thus, invoke the triangle inequality to upper bound
\begin{align*}
\left\Vert \bm{A}_{:,l}^{\mathsf{s},\left(m\right),0\top}\bm{U}^{\left(m\right)}\bm{H}^{\left(m\right)}\right\Vert _{2} & \leq\underbrace{\left\Vert \left(\bm{A}_{:,l}^{\mathsf{s},\left(m\right),0}-\mathbb{E}\left[\bm{A}_{:,l}^{\mathsf{s},\left(m\right),0}\right]\right)^{\top}\bm{U}^{\left(m,l\right)}\bm{H}^{\left(m,l\right)}\right\Vert _{2}}_{=:\alpha_{1}}+\underbrace{\left\Vert \mathbb{E}\left[\bm{A}_{:,l}^{\mathsf{s},\left(m\right),0}\right]^{\top}\bm{U}^{\left(m,l\right)}\bm{H}^{\left(m,l\right)}\right\Vert _{2}}_{=:\alpha_{2}}\\
 & \quad+\underbrace{\left\Vert \bm{A}_{:,l}^{\mathsf{s},\left(m\right),0}\right\Vert _{2}\left\Vert \bm{U}^{\left(m\right)}\bm{H}^{\left(m\right)}-\bm{U}^{\left(m,l\right)}\bm{H}^{\left(m,l\right)}\right\Vert }_{=:\alpha_{3}}.
\end{align*}
Before moving on, we make note of the following two useful upper bounds
on $\big\|\bm{U}^{\left(m\right)}\bm{U}^{\left(m\right)\top}-\bm{U}^{\left(m,l\right)}\bm{U}^{\left(m,l\right)\top}\big\|$
based on Lemma~\ref{lemma:U_loo_col_dist}. On the one hand, one
has with probability at least $1-O\left(d^{-13}\right)$,
\begin{align}
\big\|\bm{U}^{\left(m\right)}\bm{U}^{\left(m\right)\top}-\bm{U}^{\left(m,l\right)}\bm{U}^{\left(m,l\right)\top}\big\| & \lesssim\frac{1}{\sigma_{r}^{\star2}}\big(B\log d+\sigma_{\mathsf{col}}\sqrt{\log d}\big)^{2}\big\|\bm{U}^{\left(m\right)}\bm{H}^{\left(m\right)}\big\|_{2,\infty}\nonumber \\
 & \quad+\frac{\sigma_{\infty}^{2}}{\sigma_{r}^{\star2}}+\frac{1}{\sigma_{r}^{\star2}}\big(B\log d+\sigma_{\mathsf{col}}\sqrt{\log d}\big)\left\Vert \bm{A}^{\star\top}\right\Vert _{2,\infty}\nonumber \\
 & \overset{\left(\mathrm{i}\right)}{\text{\ensuremath{\lesssim}}}\frac{1}{\sigma_{r}^{\star2}}\left(B^{2}\log^{2}d+\sigma_{\mathsf{col}}^{2}\log d+B\log d\left\Vert \bm{A}^{\star\top}\right\Vert _{2,\infty}+\sigma_{\mathsf{col}}\sqrt{\log d}\left\Vert \bm{A}^{\star}\right\Vert \right)\nonumber \\
 & \lesssim\frac{\delta_{\mathsf{op}}}{\sigma_{r}^{\star2}}\ll1.\label{eq:U_loo_col_dist_o1_UB1}
\end{align}
where (i) follows from the facts that $\sigma_{\infty}\leq\sigma_{\mathsf{col}}$,
$\big\|\bm{U}^{\left(m\right)}\bm{H}^{\left(m\right)}\big\|_{2,\infty}\leq\big\|\bm{U}^{\left(m\right)}\bm{H}^{\left(m\right)}\big\|\leq1$
and $\left\Vert \bm{A}^{\star\top}\right\Vert _{2,\infty}\leq\left\Vert \bm{A}^{\star}\right\Vert $;
\eqref{eq:U_loo_col_dist_o1_UB1} is due to the definition of $\sigma_{\mathsf{col}}$
in \eqref{def:sigma_col_UB}, the definition of $\delta_{\mathsf{op}}$
(cf.~\eqref{claim:G_op_loss}) as well as conditions \eqref{cond:B_UB}
and \eqref{cond:eig_o1}. On the other hand, we can also bound
\begin{align}
\big\|\bm{U}^{\left(m\right)}\bm{U}^{\left(m\right)\top}-\bm{U}^{\left(m,l\right)}\bm{U}^{\left(m,l\right)\top}\big\| & \overset{\left(\mathrm{i}\right)}{\lesssim}\frac{1}{\sigma_{r}^{\star2}}\big(B\log d+\sigma_{\mathsf{col}}\sqrt{\log d}\big)^{2}\big\|\bm{U}^{\left(m\right)}\bm{H}^{\left(m\right)}\big\|_{2,\infty}\nonumber \\
 & \quad+\frac{1}{\sigma_{r}^{\star2}}\big(B\log d+\sigma_{\mathsf{col}}\sqrt{\log d}\big)\left\Vert \bm{A}^{\star}\right\Vert \sqrt{\frac{\mu r}{d_{2}}}\nonumber \\
 & =o\left(1\right)\big\|\bm{U}^{\left(m\right)}\bm{H}^{\left(m\right)}\big\|_{2,\infty}+\frac{1}{\sigma_{r}^{\star2}}\big(B\log d+\sigma_{\mathsf{col}}\sqrt{\log d}\big)\left\Vert \bm{A}^{\star}\right\Vert \sqrt{\frac{\mu r}{d_{2}}}\label{eq:U_loo_col_dist_o1_UB2}
\end{align}
where (i) arises from \eqref{cond:col_B_rel} and the inequality that
$\left\Vert \bm{A}^{\star\top}\right\Vert _{2,\infty}\leq\left\Vert \bm{A}^{\star}\right\Vert \left\Vert \bm{V}^{\star}\right\Vert _{2,\infty}\leq\left\Vert \bm{A}^{\star}\right\Vert \sqrt{\mu r/d_{2}}$,
and \eqref{eq:U_loo_col_dist_o1_UB2} is due to conditions~\eqref{cond:B_UB}
and \eqref{cond:eig_o1}. In the sequel, we control the $\alpha_{i}$'s
separately.
\begin{itemize}
\item For $\alpha_{1},$ it is easy to see that
\[
\left(\bm{A}_{:,l}^{\mathsf{s},\left(m\right),0}-\mathbb{E}\left[\bm{A}_{:,l}^{\mathsf{s},\left(m\right),0}\right]\right)^{\top}\bm{U}^{\left(m,l\right)}\bm{H}^{\left(m,l\right)}=\sum_{i: i \neq m }E_{i,l}\big(\bm{U}^{\left(m,l\right)}\bm{H}^{\left(m,l\right)}\big)_{i,:}
\]
is a sum of independent zero-mean random vectors conditional on $\big\{ E_{i,j}\big\}_{i\in[d_{1}]\setminus\left\{ m\right\} ,j\in[d_{2}]\setminus\left\{ l\right\} }$.
Straightforward calculation gives that
\begin{align*}
L & :=\max_{ i: i \neq m }\left\Vert E_{i,l}\big(\bm{U}^{\left(m,l\right)}\bm{H}^{\left(m,l\right)}\big)_{i,:}\right\Vert _{2}\leq B\,\big\|\bm{U}^{\left(m,l\right)}\bm{H}^{\left(m,l\right)}\big\|_{2,\infty},\\
V & :=\sum_{i: i \neq m }\mathbb{E}\left[E_{i,l}^{2}\right]\left\Vert \big(\bm{U}^{\left(m,l\right)}\bm{H}^{\left(m,l\right)}\big)_{i,:}\right\Vert _{2}^{2}\leq\sigma_{\mathsf{col}}^{2}\,\big\|\bm{U}^{\left(m,l\right)}\bm{H}^{\left(m,l\right)}\big\|_{2,\infty}^{2}.
\end{align*}
Then we apply the matrix Bernstein inequality to obtain that with
probability at least $1-O\big(d^{-13}\big)$,
\begin{align}
 & \left\Vert \left(\bm{A}_{:,l}^{\mathsf{s},\left(m\right),0}-\mathbb{E}\left[\bm{A}_{:,l}^{\mathsf{s},\left(m\right),0}\right]\right)^{\top}\bm{U}^{\left(m,l\right)}\bm{H}^{\left(m,l\right)}\right\Vert _{2}\lesssim L\log d+\sqrt{V\log d}\nonumber \\
 & \qquad\leq\big(B\log d+\sigma_{\mathsf{col}}\sqrt{\log d}\big)\big\|\bm{U}^{\left(m,l\right)}\bm{H}^{\left(m,l\right)}\big\|_{2,\infty}\nonumber \\
 & \qquad\leq\big(B\log d+\sigma_{\mathsf{col}}\sqrt{\log d}\big)\Big(\big\|\bm{U}^{\left(m\right)}\bm{H}^{\left(m\right)}\big\|_{2,\infty}+\big\|\bm{U}^{\left(m\right)}\bm{U}^{\left(m\right)\top}-\bm{U}^{\left(m,l\right)}\bm{U}^{\left(m,l\right)\top}\big\|\Big),\label{eq:A_dev_obs_U_2_norm}
\end{align}
where the last line results from the following observation:
\begin{align}
\big\|\bm{U}^{\left(m,l\right)}\bm{H}^{\left(m,l\right)}\big\|_{2,\infty} & \leq\big\|\bm{U}^{\left(m\right)}\bm{H}^{\left(m\right)}\big\|_{2,\infty}+\big\|\bm{U}^{\left(m\right)}\bm{U}^{\left(m\right)\top}-\bm{U}^{\left(m,l\right)}\bm{U}^{\left(m,l\right)\top}\big\|_{2,\infty}\nonumber \\
 & \leq\big\|\bm{U}^{\left(m\right)}\bm{H}^{\left(m\right)}\big\|_{2,\infty}+\big\|\bm{U}^{\left(m\right)}\bm{U}^{\left(m\right)\top}-\bm{U}^{\left(m,l\right)}\bm{U}^{\left(m,l\right)\top}\big\|.\label{eq:U_loo_col_2inf__UB_temp}
\end{align}

\item Turning to $\alpha_{2},$ we obtain the simple upper bound
\begin{align}
\left\Vert \mathbb{E}\left[\bm{A}_{:,l}^{\mathsf{s},\left(m\right),0}\right]^{\top}\bm{U}^{\left(m,l\right)}\bm{H}^{\left(m,l\right)}\right\Vert _{2} & \leq\left\Vert \bm{A}_{:,l}^{\star}\right\Vert _{2}\big\|\bm{U}^{\left(m,l\right)}\bm{H}^{\left(m,l\right)}\big\|\leq\left\Vert \bm{A}^{\star\top}\right\Vert _{2,\infty}.\label{eq:A_exp_U_2_norm}
\end{align}
\item With regards to $\alpha_{3}$, Lemma~\ref{lemma:row_col_2_norm}
 reveals that with probability
at least $1-O\left(d^{-13}\right)$,
\begin{align}
\big\Vert \bm{A}_{:,l}^{\mathsf{s},\left(m\right),0}\big\Vert _{2}\big\Vert \bm{U}^{\left(m\right)}\bm{H}^{\left(m\right)}-\bm{U}^{\left(m,l\right)}\bm{H}^{\left(m,l\right)}\big\Vert  & \leq\left\Vert \bm{A}_{:,l}\right\Vert _{2}\big\|\bm{U}^{\left(m\right)}\bm{U}^{\left(m\right)\top}-\bm{U}^{\left(m,l\right)}\bm{U}^{\left(m,l\right)\top}\big\|\nonumber \\
 & \lesssim\big(\left\Vert \bm{A}^{\star\top}\right\Vert _{2,\infty}+B\sqrt{\log d}+\sigma_{\mathsf{col}}\big)\big\|\bm{U}^{\left(m\right)}\bm{U}^{\left(m\right)\top}-\bm{U}^{\left(m,l\right)}\bm{U}^{\left(m,l\right)\top}\big\|\nonumber \\
 & \lesssim\left\Vert \bm{A}^{\star\top}\right\Vert _{2,\infty}+\big(B\sqrt{\log d}+\sigma_{\mathsf{col}}\big)\big\|\bm{U}^{\left(m\right)}\bm{U}^{\left(m\right)\top}-\bm{U}^{\left(m,l\right)}\bm{U}^{\left(m,l\right)\top}\big\|,\label{eq:A_obs_U_loo_col_diff_2_norm}
\end{align}
where we use \eqref{eq:U_loo_col_dist_o1_UB1} in the last step.
\end{itemize}
Combining \eqref{eq:A_dev_obs_U_2_norm}, \eqref{eq:A_exp_U_2_norm},
\eqref{eq:A_obs_U_loo_col_diff_2_norm} implies that with probability
greater than $1-O\left(d^{-13}\right)$,
\begin{align*}
 & \left\Vert \bm{A}_{:,l}^{\mathsf{s},\left(m\right),0\top}\bm{U}^{\left(m\right)}\bm{H}^{\left(m\right)}\right\Vert _{2}\\
 & \qquad\lesssim\big(B\log d+\sigma_{\mathsf{col}}\sqrt{\log d}\big)\left(\big\|\bm{U}^{\left(m\right)}\bm{H}^{\left(m\right)}\big\|_{2,\infty}+\big\|\bm{U}^{\left(m\right)}\bm{U}^{\left(m\right)\top}-\bm{U}^{\left(m,l\right)}\bm{U}^{\left(m,l\right)\top}\big\|\right)+\left\Vert \bm{A}^{\star\top}\right\Vert _{2,\infty}\\
 & \qquad\overset{\left(\mathrm{i}\right)}{\lesssim}\big(B\log d+\sigma_{\mathsf{col}}\sqrt{\log d}\big)\big\|\bm{U}^{\left(m\right)}\bm{H}^{\left(m\right)}\big\|_{2,\infty}+\frac{1}{\sigma_{r}^{\star2}}\big(B\log d+\sigma_{\mathsf{col}}\sqrt{\log d}\big)^{2}\left\Vert \bm{A}^{\star}\right\Vert \sqrt{\frac{\mu r}{d_{2}}}+\left\Vert \bm{A}^{\star}\right\Vert \sqrt{\frac{\mu r}{d_{2}}}\\
 & \qquad\overset{\left(\mathrm{ii}\right)}{\lesssim}\big(B\log d+\sigma_{\mathsf{col}}\sqrt{\log d}\big)\big\|\bm{U}^{\left(m\right)}\bm{H}^{\left(m\right)}\big\|_{2,\infty}+\left\Vert \bm{A}^{\star}\right\Vert \sqrt{\frac{\mu r}{d_{2}}},
\end{align*}
where (i) is by \eqref{eq:U_loo_col_dist_o1_UB2} and $\left\Vert \bm{A}^{\star\top}\right\Vert _{2,\infty}\leq\left\Vert \bm{A}^{\star}\right\Vert \sqrt{\mu r/d_{2}}$,
and (ii) follows from conditions~\eqref{cond:B_UB} and \eqref{cond:eig_o1}.
The proof is complete by taking the union bound over $1\leq l\leq d_{2}$.

\subsection{Proof of Lemma~\ref{lemma:U_loo_col_dist}}

\label{subsec:pf:U_loo_col_dist}

Fix arbitrary $m\in[d_{1}]$ and $l\in[d_{2}]$. Recalling the definitions
of $\bm{G}^{\left(m\right)}$ and $\bm{G}^{\left(m,l\right)}$ in
\eqref{def:G_loo} and \eqref{def:G_loo_col}, we see that $\bm{G}^{\left(m\right)}-\bm{G}^{\left(m,l\right)}$
is symmetric with entries
\begin{align*}
\big(\bm{G}^{\left(m\right)}-\bm{G}^{\left(m,l\right)}\big)_{i,j} & =E_{i,l}E_{j,l}+A_{i,l}^{\star}E_{j,l}+E_{i,l}A_{j,l}^{\star},\qquad i\neq m,j\neq m,i\neq j,\\
\big(\bm{G}^{\left(m\right)}-\bm{G}^{\left(m,l\right)}\big)_{i,m} & =A_{m,l}^{\star}E_{i,l},\qquad i\neq m,\\
\big(\bm{G}^{\left(m\right)}-\bm{G}^{\left(m,l\right)}\big)_{i,i} & =0,\qquad1\leq i\leq d_{1}.
\end{align*}
Note that $\bm{G}^{\left(m\right)}-\bm{G}^{\left(m,l\right)}$ depends
only on $\left\{ E_{i,l}\right\} _{i\in[d_{1}]\setminus\left\{ m\right\} }$
and is hence statistically independent of $\bm{U}^{\left(m,l\right)}$
and $\bm{H}^{\left(m,l\right)}$. In particular, we can express
\[
\mathcal{P}_{-m}\big(\bm{G}^{\left(m\right)}-\bm{G}^{\left(m,l\right)}\big)=\mathcal{P}_{\mathsf{off}\text{-}\mathsf{diag}}\mathcal{P}_{-m}\big(\bm{E}_{:,l}\bm{E}_{:,l}^{\top}+\bm{E}_{:,l}\bm{A}_{:,l}^{\star\top}+\bm{A}_{:,l}^{\star}\bm{E}_{:,l}^{\top}\big),
\]
where $\mathcal{P}_{-m}$ is the projection onto the subspace of matrices
supported on $\left\{ \left(i,j\right)\in\left[d_{1}\right]\times\left[d_{2}\right]\colon i\neq m\enskip\text{and}\enskip j\neq m\right\} $
and $\mathcal{P}_{\mathsf{off}\text{-}\mathsf{diag}}$ extracts the
off-diagonal part. In addition, 
\[
\big(\bm{G}^{\left(m\right)}-\bm{G}^{\left(m,l\right)}\big)_{m,:}=A_{m,l}^{\star}\bm{E}_{:,l}^{\top}-A_{m,l}^{\star}E_{m,l}\bm{e}_{m}^{\top},
\]
where $\bm{e}_{m}$ stands for the $m$-th standard basis in $\mathbb{R}^{d_{1}}$.

In the sequel, we shall apply the Davis-Kahan $\sin\bm{\Theta}$ theorem
to prove the claim. Towards this end, we need to control $\big\|\bm{G}^{\left(m\right)}-\bm{G}^{\left(m,l\right)}\big\|$
and \textbf{$\big\|\big(\bm{G}^{\left(m\right)}-\bm{G}^{\left(m,l\right)}\big)\bm{U}^{\left(m,l\right)}\big\|$.}

\subsubsection{Step 1: controlling $\big\|\bm{G}^{\left(m\right)}-\bm{G}^{\left(m,l\right)}\big\|$}

Recall that we have already dealt with $\mathcal{P}_{-m}\big(\bm{G}^{\left(m\right)}-\bm{G}^{\left(m,l\right)}\big)$
in the proof of Lemma~\ref{lemma:G_op_loss} in Appendix~\ref{subsec:pf:G_op_loss}.
Straightforward computation gives
\begin{align*}
\big\|\mathcal{P}_{-m}\big(\bm{G}^{\left(m\right)}-\bm{G}^{\left(m,l\right)}\big)\big\| & \lesssim\left\Vert \bm{E}_{:,l}\right\Vert _{2}^{2}+\left\Vert \bm{E}_{:,l}\right\Vert _{2}\left\Vert \bm{A}^{\star\top}\right\Vert _{2,\infty}\leq\left\Vert \bm{E}_{:,l}\right\Vert _{2}^{2}+\left\Vert \bm{E}_{:,l}\right\Vert _{2}\left\Vert \bm{A}^{\star}\right\Vert ,\\
\big\|\mathcal{P}_{m}\big(\bm{G}^{\left(m\right)}-\bm{G}^{\left(m,l\right)}\big)\big\| & \lesssim\left\Vert \big(\bm{G}^{\left(m\right)}-\bm{G}^{\left(m,l\right)}\big)_{m,:}\right\Vert _{2}\leq\left\Vert \bm{E}_{:,l}\right\Vert _{2}\left\Vert \bm{A}^{\star}\right\Vert _{\infty},
\end{align*}
where $\mathcal{P}_{m}$ is the projection onto the subspace of matrices
supported on $\left\{ \left(i,j\right)\in\left[d_{1}\right]\times\left[d_{2}\right]\colon i=m\enskip\text{or}\enskip j=m\right\} $.
In view of \eqref{eq:n_l_2_norm} (shown in the proof of Lemma~\ref{lemma:G_op_loss}
in Appendix~\ref{subsec:pf:G_op_loss}), we know that
\begin{align*}
\big\|\bm{G}^{\left(m\right)}-\bm{G}^{\left(m,l\right)}\big\| & \leq\big\|\mathcal{P}_{-m}\big(\bm{G}^{\left(m\right)}-\bm{G}^{\left(m,l\right)}\big)\big\|+\big\|\mathcal{P}_{m}\big(\bm{G}^{\left(m\right)}-\bm{G}^{\left(m,l\right)}\big)\big\|\\
 & \lesssim\left\Vert \bm{E}_{:,l}\right\Vert _{2}^{2}+\left\Vert \bm{E}_{:,l}\right\Vert _{2}\left\Vert \bm{A}^{\star}\right\Vert \\
 & \lesssim B^{2}\log d+\sigma_{\mathsf{col}}^{2}+\big(B\sqrt{\log d}+\sigma_{\mathsf{col}}\big)\left\Vert \bm{A}^{\star}\right\Vert \\
 & \ll\sigma_{r}^{\star2},
\end{align*}
where the last step results from the conditions~\eqref{cond:B_UB}
and \eqref{cond:eig_o1}. Since $\big\|\bm{G}^{\left(m\right)}-\bm{G}^{\star}\big\|\lesssim\delta_{\mathsf{op}}\ll\sigma_{r}^{\star2}$
by Lemma~\ref{lemma:G_loo_op_diff} and the condition~\eqref{cond:eig_o1},
this also implies
\begin{equation}
\big\|\bm{G}^{\left(m,l\right)}-\bm{G}^{\star}\big\|\ll\sigma_{r}^{\star2}\quad\text{and}\quad\big\|\big(\bm{H}^{\left(m,l\right)}\big)^{-1}\big\|\lesssim1,\label{eq:H_loo_col_inv_op_UB}
\end{equation}
according to \cite[Lemma~2]{abbe2017entrywise}. Moreover, it follows
from Weyl's inequality that
\begin{align}
\lambda_{r}\big(\bm{\bm{G}}^{\left(m\right)}\big)-\lambda_{r+1}\big(\bm{\bm{G}}^{\left(m\right)}\big)-\big\|\bm{G}^{\left(m\right)}-\bm{G}^{\left(m,l\right)}\big\| & \geq\lambda_{r}\big(\bm{\bm{G}}^{\star}\big)-\lambda_{r+1}\big(\bm{\bm{G}}^{\star}\big)-2\,\big\|\bm{G}^{\left(m\right)}-\bm{G}^{\star}\big\|-\big\|\bm{G}^{\left(m\right)}-\bm{G}^{\left(m,l\right)}\big\|\nonumber \\
 & \gtrsim\sigma_{r}^{\star2}.\label{eq:U_loo_col_DK_deno}
\end{align}

\subsubsection{Step 2: controlling $\big\|\big(\bm{G}^{\left(m\right)}-\bm{G}^{\left(m,l\right)}\big)\bm{U}^{\left(m,l\right)}\big\|$}

In view of \eqref{eq:H_loo_col_inv_op_UB}, we can obtain
\begin{align}
\big\|\big(\bm{G}^{\left(m\right)}-\bm{G}^{\left(m,l\right)}\big)\bm{U}^{\left(m,l\right)}\big\| & \leq\big\|\big(\bm{G}^{\left(m\right)}-\bm{G}^{\left(m,l\right)}\big)\bm{U}^{\left(m,l\right)}\bm{H}^{\left(m,l\right)}\big\|\big\|\big(\bm{H}^{\left(m,l\right)}\big)^{-1}\big\| \nonumber\\
 & \lesssim\big\|\big(\bm{G}^{\left(m\right)}-\bm{G}^{\left(m,l\right)}\big)\bm{U}^{\left(m,l\right)}\bm{H}^{\left(m,l\right)}\big\| \nonumber\\
 & \leq\underbrace{\big\|\mathcal{P}_{m}\big(\bm{G}^{\left(m\right)}-\bm{G}^{\left(m,l\right)}\big)\bm{U}^{\left(m,l\right)}\bm{H}^{\left(m,l\right)}\big\|_{\mathrm{F}}}_{=:\alpha_{1}}+\underbrace{\big\|\mathcal{P}_{-m}\big(\bm{G}^{\left(m\right)}-\bm{G}^{\left(m,l\right)}\big)\bm{U}^{\left(m,l\right)}\bm{H}^{\left(m,l\right)}\big\|}_{=:\alpha_{2}} \label{eq:G_loo_col_diff_U_loo_col_decomp}.
\end{align}
Therefore, it suffices to control $\alpha_{1}$ and $\alpha_{2}$
separately.
\begin{itemize}
\item Regarding $\alpha_{1}$, Lemma~\ref{lemma:row_col_2_norm} reveals
that, with probability at least $1-O(d^{-13})$,
\begin{align}
 & \big\|\mathcal{P}_{m}\big(\bm{G}^{\left(m\right)}-\bm{G}^{\left(m,l\right)}\big)\bm{U}^{\left(m,l\right)}\bm{H}^{\left(m,l\right)}\big\|_{\mathrm{F}}\nonumber \\
 & \quad\leq\left\Vert \big(\bm{G}^{\left(m\right)}-\bm{G}^{\left(m,l\right)}\big)_{m,:}\bm{U}^{\left(m,l\right)}\bm{H}^{\left(m,l\right)}\right\Vert _{2}+\left\Vert \big(\bm{G}^{\left(m\right)}-\bm{G}^{\left(m,l\right)}\big)_{m,:}\right\Vert _{2}\big\|\bm{U}^{\left(m,l\right)}\bm{H}^{\left(m,l\right)}\big\|_{2,\infty}\nonumber \\
 & \quad\leq\left\Vert \bm{A}^{\star}\right\Vert _{\infty}\big\|\left(\bm{E}_{:,l}\right)^{\top}\bm{U}^{\left(m,l\right)}\bm{H}^{\left(m,l\right)}\big\|_{2}+\left\Vert \bm{A}^{\star}\right\Vert _{\infty}\left\Vert \bm{E}_{:,l}\right\Vert _{2}\big\|\bm{U}^{\left(m,l\right)}\bm{H}^{\left(m,l\right)}\big\|_{2,\infty}\nonumber \\
 & \quad\lesssim\big(B\log d+\sigma_{\mathsf{col}}\sqrt{\log d}\big)\left\Vert \bm{A}^{\star}\right\Vert _{\infty}\big\|\bm{U}^{\left(m,l\right)}\bm{H}^{\left(m,l\right)}\big\|_{2,\infty}.\label{eq:B_B_loo_col_diff_UH_part1}
\end{align}
\item When it comes to $\alpha_{2}$, since the spectral norm of a submatrix is always less than that of its original matrix, we can further upper bound
\begin{align*}
\big\|\mathcal{P}_{-m}\big(\bm{G}^{\left(m\right)}-\bm{G}^{\left(m,l\right)}\big)\bm{U}^{\left(m,l\right)}\bm{H}^{\left(m,l\right)}\big\| & \leq\underbrace{\left\Vert \left(\bm{E}_{:,l}\bm{E}_{:,l}^{\top}-\bm{D}_{l}\right)\bm{U}^{\left(m,l\right)}\bm{H}^{\left(m,l\right)}\right\Vert }_{=:\beta_{1}}\\
 & \quad+\underbrace{\left\Vert \left(\bm{A}_{:,l}^{\star}\bm{E}_{:,l}^{\top}+\bm{E}_{:,l}\bm{A}_{:,l}^{\star\top}-2\widehat{\bm{D}}_{l}\right)\bm{U}^{\left(m,l\right)}\bm{H}^{\left(m,l\right)}\right\Vert }_{=:\beta_{2}},
\end{align*}
where $\bm{D}_{l}$ and $\widehat{\bm{D}}_{l}$ are defined in \eqref{def:D_l} and \eqref{def:D_l_hat} in Appendix~\ref{subsec:pf:G_op_loss}.
In what follows, let us bound $\beta_{1}$ and $\beta_{2}$.
\begin{itemize}
\item To bound $\beta_{1}$, we have
\begin{align*}
\big\|\left(\bm{E}_{:,l}\bm{E}_{:,l}^{\top}-\bm{D}_{l}\right)\bm{U}^{\left(m,l\right)}\bm{H}^{\left(m,l\right)}\big\| & \leq\left\Vert \bm{E}_{:,l}\right\Vert _{2}\big\|\left(\bm{E}_{:,l}\right)^{\top}\bm{U}^{\left(m,l\right)}\bm{H}^{\left(m,l\right)}\big\|_{2}+\big\|\bm{D}_{l}\bm{U}^{\left(m,l\right)}\bm{H}^{\left(m,l\right)}\big\|.
\end{align*}
By Lemma~\ref{lemma:row_col_2_norm}, one has with probability at
least $1-O(d^{-13})$, 
\begin{align*}
\left\Vert \bm{E}_{:,l}\right\Vert _{2}\big\|\left(\bm{E}_{:,l}\right)^{\top}\bm{U}^{\left(m,l\right)}\bm{H}^{\left(m,l\right)}\big\|_{2} & \lesssim\big(B\sqrt{\log d}+\sigma_{\mathsf{col}}\big)\big(B\log d+\sigma_{\mathsf{col}}\sqrt{\log d}\big)\big\|\bm{U}^{\left(m,l\right)}\bm{H}^{\left(m,l\right)}\big\|_{2,\infty}\\
 & \leq\big(B\log d+\sigma_{\mathsf{col}}\sqrt{\log d}\big)^{2}\big\|\bm{U}^{\left(m,l\right)}\bm{H}^{\left(m,l\right)}\big\|_{2,\infty}.
\end{align*}
As for the second term $\big\|\bm{D}_{l}\bm{U}^{\left(m,l\right)}\bm{H}^{\left(m,l\right)}\big\|$,
we first observe that 
\[
\big\|\mathbb{E}\left[\bm{D}_{l}\right]\bm{U}^{\left(m,l\right)}\bm{H}^{\left(m,l\right)}\big\|\leq\left\Vert \mathbb{E}\left[\bm{D}_{l}\right]\right\Vert \big\|\bm{U}^{\left(m,l\right)}\bm{H}^{\left(m,l\right)}\big\|\leq\max_{i\in\left[d_{1}\right]}\mathbb{E}\left[E_{i,l}^{2}\right]\big\|\bm{U}^{\left(m,l\right)}\bm{H}^{\left(m,l\right)}\big\|\leq\sigma_{\infty}^{2}.
\]
Additionally, the deviation $\big(\bm{D}_{l}-\mathbb{E}\left[\bm{D}_{l}\right]\big)\bm{U}^{\left(m,l\right)}\bm{H}^{\left(m,l\right)}=\sum_{i\in\left[d_{1}\right]}\big(E_{i,l}^{2}-\mathbb{E}[E_{i,l}^{2}]\big)\bm{e}_{i}\bm{e}_{i}^{\top}\bm{U}^{\left(m,l\right)}\bm{H}^{\left(m,l\right)}$
is a sum of independent zero-mean random matrices. By the matrix Bernstein
inequality, we have with probability at least $1-O(d^{-13})$,
\begin{align*}
 & \left\Vert \left(\bm{D}_{l}-\mathbb{E}\left[\bm{D}_{l}\right]\right)\bm{U}^{\left(m,l\right)}\bm{H}^{\left(m,l\right)}\right\Vert \\
 &\qquad \qquad	 \lesssim\Big(\max_{i\in\left[d_{1}\right]}\left|E_{i,l}^{2}-\mathbb{E}[E_{i,l}^{2}]\right|\log d+\sqrt{\sum\nolimits_{i\in\left[d_{1}\right]}\mathbb{E}\big[E_{i,l}^{4}\big]\log d}\Big)\max_{i\in \left[ d_1 \right]}\big\|\bm{e}_{i}\bm{e}_{i}^{\top}\bm{U}^{\left(m,l\right)}\bm{H}^{\left(m,l\right)}\big\|_{2,\infty}\\
 &\qquad \qquad	 \lesssim\Big(\max_{i\in\left[d_{1}\right]}\left|E_{i,l}^{2}-\mathbb{E}[E_{i,l}^{2}]\right|\log d+\sqrt{\sum\nolimits_{i\in\left[d_{1}\right]}\mathbb{E}\big[E_{i,l}^{4}\big]\log d}\Big)\big\|\bm{U}^{\left(m,l\right)}\bm{H}^{\left(m,l\right)}\big\|_{2,\infty}\\
 &\qquad \qquad	 \lesssim\Big(B^{2}\log d+\sqrt{\left(B^{2}\sigma_{\mathsf{col}}^{2}\right)\log d}\Big)\big\|\bm{U}^{\left(m,l\right)}\bm{H}^{\left(m,l\right)}\big\|_{2,\infty}\\
 &\qquad \qquad	 \overset{(\mathrm{i})}{\lesssim}\Big(B^{2}\log d+\left(B^{2}+\sigma_{\mathsf{col}}^{2}\right)\sqrt{\log d}\Big)\big\|\bm{U}^{\left(m,l\right)}\bm{H}^{\left(m,l\right)}\big\|_{2,\infty}\\
 &\qquad \qquad	 \asymp\big(B^{2}\log d+\sigma_{\mathsf{col}}^{2}\sqrt{\log d}\big)\big\|\bm{U}^{\left(m,l\right)}\bm{H}^{\left(m,l\right)}\big\|_{2,\infty},
\end{align*}
where we have used the AM-GM inequality in (i). Combining the estimates above
yields
\begin{equation}
\beta_{1}=\left\Vert \left(\bm{E}_{:,l}\bm{E}_{:,l}^{\top}-\bm{D}_{l}\right)\bm{U}^{\left(m,l\right)}\bm{H}^{\left(m,l\right)}\right\Vert \lesssim\big(B\log d+\sigma_{\mathsf{col}}\sqrt{\log d}\big)^{2}\big\|\bm{U}^{\left(m,l\right)}\bm{H}^{\left(m,l\right)}\big\|_{2,\infty}+\sigma_{\infty}^{2}.\label{eq:U_loo_col_beta1}
\end{equation}
\item Turning to $\beta_{2}$, we see from Lemma~\ref{lemma:row_col_2_norm}
that with probability at least $1-O\left(d^{-13}\right)$, one has
\begin{align*}
\left\Vert \bm{A}_{:,l}^{\star}\bm{E}_{:,l}^{\top}\bm{U}^{\left(m,l\right)}\bm{H}^{\left(m,l\right)}\right\Vert  & \leq\left\Vert \bm{A}^{\star\top}\right\Vert _{2,\infty}\big\|\bm{E}_{:,l}^{\top}\bm{U}^{\left(m,l\right)}\bm{H}^{\left(m,l\right)}\big\|\\
 & \lesssim\big(B\log d+\sigma_{\mathsf{col}}\sqrt{\log d}\big)\left\Vert \bm{A}^{\star\top}\right\Vert _{2,\infty}\big\|\bm{U}^{\left(m,l\right)}\bm{H}^{\left(m,l\right)}\big\|_{2,\infty}
\end{align*}
and
\begin{align*}
\left\Vert \bm{E}_{:,l}\bm{A}_{:,l}^{\star\top}\bm{U}^{\left(m,l\right)}\bm{H}^{\left(m,l\right)}\right\Vert  & \leq\left\Vert \bm{E}_{:,l}\right\Vert _{2}\left\Vert \bm{A}^{\star\top}\right\Vert _{2,\infty}\big\|\bm{U}^{\left(m,l\right)}\bm{H}^{\left(m,l\right)}\big\|\\
 & \lesssim\big(B\sqrt{\log d}+\sigma_{\mathsf{col}}\big)\left\Vert \bm{A}^{\star\top}\right\Vert _{2,\infty}.
\end{align*}
In addition, $\widehat{\bm{D}}_{l}\bm{U}^{\left(m,l\right)}\bm{H}^{\left(m,l\right)}=\sum_{i\in\left[d_{1}\right]}A_{i,l}^{\star}E_{i,l}\bm{e}_{i}\bm{e}_{i}^{\top}\bm{U}^{\left(m,l\right)}\bm{H}^{\left(m,l\right)}$
is a sum of independent zero-mean random matrices. It then follows
from the matrix Bernstein inequality that with probability at least
$1-O\left(d^{-13}\right)$,
\begin{align*}
\big\|\widehat{\bm{D}}_{l}\bm{U}^{\left(m,l\right)}\bm{H}^{\left(m,l\right)}\big\| & \lesssim\Big(\max_{i\in\left[d_{1}\right]}\left|A_{i,l}^{\star}E_{i,l}\right|\log d+\sqrt{\sum\nolimits_{i\in\left[d_{1}\right]}A_{i,l}^{\star2}\mathbb{E}\big[E_{i,l}^{2}\big]\log d}\Big)\big\|\bm{U}^{\left(m,l\right)}\bm{H}^{\left(m,l\right)}\big\|_{2,\infty}\\
 & \lesssim\big(B\log d+\sigma_{\mathsf{col}}\sqrt{\log d}\big)\left\Vert \bm{A}^{\star}\right\Vert _{\infty}\big\|\bm{U}^{\left(m,l\right)}\bm{H}^{\left(m,l\right)}\big\|_{2,\infty}.
\end{align*}
Hence, we know that
\begin{align}
\beta_{2} & =\left\Vert \big(\bm{A}_{:,l}^{\star}\bm{E}_{:,l}^{\top}+\bm{E}_{:,l}\bm{A}_{:,l}^{\star\top}-2\widehat{\bm{D}}_{l}\big)\bm{U}^{\left(m,l\right)}\bm{H}^{\left(m,l\right)}\right\Vert \nonumber \\
 & \lesssim\big(B\log d+\sigma_{\mathsf{col}}\sqrt{\log d}\big)\left\Vert \bm{A}^{\star\top}\right\Vert _{2,\infty}\big\|\bm{U}^{\left(m,l\right)}\bm{H}^{\left(m,l\right)}\big\|_{2,\infty}+\big(B\sqrt{\log d}+\sigma_{\mathsf{col}}\big)\left\Vert \bm{A}^{\star\top}\right\Vert _{2,\infty}\nonumber \\
 & \lesssim\big(B\log d+\sigma_{\mathsf{col}}\sqrt{\log d}\big)\left\Vert \bm{A}^{\star\top}\right\Vert _{2,\infty},\label{eq:U_loo_col_beta2}
\end{align}
which results from the facts that $\left\Vert \bm{A}^{\star}\right\Vert _{\infty}\leq\left\Vert \bm{A}^{\star\top}\right\Vert _{2,\infty}$
and $\left\Vert \bm{U}^{\left(m,l\right)}\bm{H}^{\left(m,l\right)}\right\Vert _{2,\infty}\leq\left\Vert \bm{U}^{\left(m,l\right)}\bm{H}^{\left(m,l\right)}\right\Vert \leq1$.
\end{itemize}
\end{itemize}
Putting \eqref{eq:U_loo_col_beta1} and \eqref{eq:U_loo_col_beta2}
together yields that
\begin{align}
\big\|\mathcal{P}_{-m}\big(\bm{G}^{\left(m\right)}-\bm{G}^{\left(m,l\right)}\big)\bm{U}^{\left(m,l\right)}\bm{H}^{\left(m,l\right)}\big\| & \lesssim\big(B\log d+\sigma_{\mathsf{col}}\sqrt{\log d}\big)^{2}\big\|\bm{U}^{\left(m,l\right)}\bm{H}^{\left(m,l\right)}\big\|_{2,\infty}\nonumber \\
 & \quad+\sigma_{\infty}^{2}+\big(B\log d+\sigma_{\mathsf{col}}\sqrt{\log d}\big)\left\Vert \bm{A}^{\star\top}\right\Vert _{2,\infty}.\label{eq:B_B_loo_col_diff_UH_part2}
\end{align}
This combined with \eqref{eq:B_B_loo_col_diff_UH_part1} and \eqref{eq:G_loo_col_diff_U_loo_col_decomp} implies

\begin{align}
\big\|\big(\bm{G}^{\left(m\right)}-\bm{G}^{\left(m,l\right)}\big)\bm{U}^{\left(m,l\right)}\big\| & \lesssim\big(B\log d+\sigma_{\mathsf{col}}\sqrt{\log d}\big)^{2}\big\|\bm{U}^{\left(m,l\right)}\bm{H}^{\left(m,l\right)}\big\|_{2,\infty} + \sigma_\infty^2 \nonumber \\
 & \quad+\big(B\log d+\sigma_{\mathsf{col}}\sqrt{\log d}\big)\left\Vert \bm{A}^{\star\top}\right\Vert _{2,\infty}+\big(B\log d+\sigma_{\mathsf{col}}\sqrt{\log d}\big)\left\Vert \bm{A}^{\star}\right\Vert _{\infty}\big\|\bm{U}^{\left(m,l\right)}\bm{H}^{\left(m,l\right)}\big\|_{2,\infty}\nonumber \\
 & \overset{\left(\mathrm{i}\right)}{\leq}\big(B\log d+\sigma_{\mathsf{col}}\sqrt{\log d}\big)^{2}\big\|\bm{U}^{\left(m,l\right)}\bm{H}^{\left(m,l\right)}\big\|_{2,\infty}+\sigma_{\infty}^{2}\nonumber \\
 & \quad+\big(B\log d+\sigma_{\mathsf{col}}\sqrt{\log d}\big)\left\Vert \bm{A}^{\star\top}\right\Vert _{2,\infty}+\big(B\log d+\sigma_{\mathsf{col}}\sqrt{\log d}\big)\left\Vert \bm{A}^{\star}\right\Vert _{\infty}\nonumber \\
 & \overset{\left(\mathrm{ii}\right)}{\asymp}\big(B\log d+\sigma_{\mathsf{col}}\sqrt{\log d}\big)^{2}\big\|\bm{U}^{\left(m,l\right)}\bm{H}^{\left(m,l\right)}\big\|_{2,\infty}+\sigma_{\infty}^{2}\nonumber \\
 & \quad+\big(B\log d+\sigma_{\mathsf{col}}\sqrt{\log d}\big)\left\Vert \bm{A}^{\star\top}\right\Vert _{2,\infty},\label{eq:U_loo_col_DK_nume}
\end{align}
where (i) is due to the facts that $\big\|\bm{U}^{\left(m,l\right)}\bm{H}^{\left(m,l\right)}\big\|_{2,\infty}\leq\big\|\bm{U}^{\left(m,l\right)}\bm{H}^{\left(m,l\right)}\big\|\leq1$,
and \eqref{eq:U_loo_col_DK_nume} arises from the inequality $\left\Vert \bm{A}^{\star}\right\Vert _{\infty}\leq\left\Vert \bm{A}^{\star\top}\right\Vert _{2,\infty}$.

\subsubsection{Step 3: combining Step 1 and Step 2}

From \eqref{eq:U_loo_col_DK_deno} and \eqref{eq:U_loo_col_DK_nume},
we apply the Davis-Kahan $\sin\bm{\Theta}$ theorem to obtain that
with probability exceeding $1-O\left(d^{-13}\right)$,
\begin{align*}
 & \big\|\bm{U}^{\left(m\right)}\bm{U}^{\left(m\right)\top}-\bm{U}^{\left(m,l\right)}\bm{U}^{\left(m,l\right)\top}\big\|\leq\frac{\big\|\big(\bm{G}^{\left(m\right)}-\bm{G}^{\left(m,l\right)}\big)\bm{U}^{\left(m,l\right)}\big\|}{\lambda_{r}\left(\bm{G}^{\left(m\right)}\right)-\lambda_{r+1}\left(\bm{G}^{\left(m\right)}\right)-\big\|\bm{G}^{\left(m\right)}-\bm{G}^{\left(m,l\right)}\big\|}\\
 & \qquad\lesssim\frac{1}{\sigma_{r}^{\star2}}\big\|\big(\bm{G}^{\left(m\right)}-\bm{G}^{\left(m,l\right)}\big)\bm{U}^{\left(m,l\right)}\big\|\\
 & \qquad \overset{(\mathrm{i})}{\lesssim}\frac{1}{\sigma_{r}^{\star2}}\big(B\log d+\sigma_{\mathsf{col}}\sqrt{\log d}\big)^{2}\left(\big\|\bm{U}^{\left(m\right)}\bm{H}^{\left(m\right)}\big\|_{2,\infty}+\big\|\bm{U}^{\left(m\right)}\bm{U}^{\left(m\right)\top}-\bm{U}^{\left(m,l\right)}\bm{U}^{\left(m,l\right)\top}\big\|\right)\\
 & \qquad\quad+\frac{\sigma_{\infty}^{2}}{\sigma_{r}^{\star2}}+\frac{1}{\sigma_{r}^{\star2}}\big(B\log d+\sigma_{\mathsf{col}}\sqrt{\log d}\big)\left\Vert \bm{A}^{\star\top}\right\Vert _{2,\infty}\\
 & \qquad\overset{(\mathrm{ii})}{\lesssim}\frac{1}{\sigma_{r}^{\star2}}\big(B\log d+\sigma_{\mathsf{col}}\sqrt{\log d}\big)^{2}\big\|\bm{U}^{\left(m\right)}\bm{H}^{\left(m\right)}\big\|_{2,\infty}+o\left(1\right)\big\|\bm{U}^{\left(m\right)}\bm{U}^{\left(m\right)\top}-\bm{U}^{\left(m,l\right)}\bm{U}^{\left(m,l\right)\top}\big\|\\
 & \qquad\quad+\frac{\sigma_{\infty}^{2}}{\sigma_{r}^{\star2}}+\frac{1}{\sigma_{r}^{\star2}}\big(B\log d+\sigma_{\mathsf{col}}\sqrt{\log d}\big)\left\Vert \bm{A}^{\star\top}\right\Vert _{2,\infty}.
\end{align*}
Here, we have used \eqref{eq:U_loo_col_2inf__UB_temp} in (i), and the condition~\eqref{cond:eig_o1} (i.e.~$\max\left\{ B\log d,\sigma_{\mathsf{col}}\sqrt{\log d}\right\} \ll\sigma_{r}^{\star}$) in (ii). Rearrange the inequalities and taking the union bound over $m\in\left[d_{1}\right]$
and $l\in\left[d_{2}\right]$ complete the proof.

\subsection{Proof of Lemma~\ref{lemma:G_row_UH_Utrue}}

\label{subsec:pf:G_row_UH_Utrue}

Recall the definition of $\bm{G}$ in \eqref{eq:defn-G}. We can express
\[
\bm{G}_{m,:}\big(\bm{U}^{\left(m\right)}\bm{H}^{\left(m\right)}-\bm{U}^{\star}\big)=\bm{A}_{m,:}^{\mathsf{s}}\big[\mathcal{\mathcal{P}}_{-m,:}\left(\bm{A}^{\mathsf{s}}\right)\big]^{\top}\big(\bm{U}^{\left(m\right)}\bm{H}^{\left(m\right)}-\bm{U}^{\star}\big).
\]
Consequently, one can upper bound
\begin{align*}
\left\Vert \bm{G}_{m,:}\big(\bm{U}^{\left(m\right)}\bm{H}^{\left(m\right)}-\bm{U}^{\star}\big)\right\Vert _{2} & \leq\underbrace{\left\Vert \bm{A}_{m,:}^{\star}\big[\mathcal{\mathcal{P}}_{-m,:}\left(\bm{A}^{\mathsf{s}}\right)\big]^{\top}\big(\bm{U}^{\left(m\right)}\bm{H}^{\left(m\right)}-\bm{U}^{\star}\big)\right\Vert _{2}}_{=:\beta_{1}}\\
 & \quad+\underbrace{\left\Vert \left(\bm{A}^{\mathsf{s}}-\bm{A}^{\star}\right)\big[\mathcal{\mathcal{P}}_{-m,:}\left(\bm{A}^{\mathsf{s}}\right)\big]^{\top}\big(\bm{U}^{\left(m\right)}\bm{H}^{\left(m\right)}-\bm{U}^{\star}\big)\right\Vert _{2}}_{=:\beta_{2}}.
\end{align*}
In what follows, we shall control $\beta_{1}$ and $\beta_{2}$ separately.
\begin{itemize}
\item To upper bound $\beta_{1}$, we have
\begin{align*}
\left\Vert \bm{A}_{m,:}^{\star}\big[\mathcal{\mathcal{P}}_{-m,:}\left(\bm{A}^{\mathsf{s}}\right)\big]^{\top}\big(\bm{U}^{\left(m\right)}\bm{H}^{\left(m\right)}-\bm{U}^{\star}\big)\right\Vert _{2} & \leq\Big\|\bm{A}_{m,:}^{\star}\big[\mathcal{\mathcal{P}}_{-m,:}\left(\bm{A}^{\mathsf{s}}\right)\big]^{\top}\Big\|_{2}\big\|\bm{U}^{\left(m\right)}\bm{H}^{\left(m\right)}-\bm{U}^{\star}\big\|.
\end{align*}
It is straightforward to derive
\[
\Big\|\bm{A}_{m,:}^{\star}\big[\mathcal{\mathcal{P}}_{-m,:}\left(\bm{A}^{\mathsf{s}}\right)\big]^{\top}\Big\|_{2}\leq\left\Vert \bm{A}_{m,:}^{\star}\bm{A}^{\mathsf{s}\top}\right\Vert _{2}\leq\left\Vert \bm{A}_{m,:}^{\star}\bm{A}^{\star\top}\right\Vert _{2}+\left\Vert \bm{A}_{m,:}^{\star}\bm{E}^{\top}\right\Vert _{2},
\]
whose first term can be bounded by
\[
\left\Vert \bm{A}_{m,:}^{\star}\bm{A}^{\star\top}\right\Vert _{2}\leq\left\Vert \bm{A}_{m,:}^{\star}\right\Vert _{2}\left\Vert \bm{A}^{\star}\right\Vert \leq\sigma_{1}^{\star}\left\Vert \bm{A}^{\star}\right\Vert _{2,\infty}.
\]
In addition, Lemma~\ref{lemma:sum_row_square_sum_col} indicates
that
\begin{align*}
\left\Vert \bm{A}_{m,:}^{\star}\bm{E}^{\top}\right\Vert _{2}^{2} & =\sum\nolimits _{i}\Big(\sum\nolimits _{j}A_{m,j}^{\star}E_{i,j}\Big)^{2}\lesssim\left(\sigma_{\mathsf{col}}^{2}+\sigma_{\infty}^{2}\log^{2}d\right)\left\Vert \bm{A}^{\star}\right\Vert _{2,\infty}^{2}+B^{2}\left\Vert \bm{A}^{\star}\right\Vert _{\infty}^{2}\log^{3}d\\
 & \leq\left(\sigma_{\mathsf{col}}^{2}+B^{2}\log^{2}d\right)\left\Vert \bm{A}^{\star}\right\Vert _{2,\infty}^{2}+B^{2}\left\Vert \bm{A}^{\star}\right\Vert _{\infty}^{2}\log^{3}d
\end{align*}
holds with probability at least $1-O\left(d^{-11}\right)$. Hence,
we have
\begin{align}
\Big\|\bm{A}_{m,:}^{\star}\big[\mathcal{\mathcal{P}}_{-m,:}\left(\bm{A}^{\mathsf{s}}\right)\big]^{\top}\Big\|_{2} & \leq\left\Vert \bm{A}_{m,:}^{\star}\bm{A}^{\star\top}\right\Vert _{2}+\left\Vert \bm{A}_{m,:}^{\star}\bm{E}^{\top}\right\Vert _{2}\nonumber \\
 & \lesssim\sigma_{1}^{\star}\left\Vert \bm{A}^{\star}\right\Vert _{2,\infty}+\left(\sigma_{\mathsf{col}}+B\log d\right)\left\Vert \bm{A}^{\star}\right\Vert _{2,\infty}+B\log^{3/2}d\left\Vert \bm{A}^{\star}\right\Vert _{\infty}\nonumber \\
 & \lesssim\sigma_{1}^{\star}\left\Vert \bm{A}^{\star}\right\Vert _{2,\infty}+\sigma_{1}^{\star2}\sqrt{\frac{\mu r}{d_{1}}}\lesssim\sigma_{1}^{\star2}\sqrt{\frac{\mu r}{d_{1}}},\label{eq:A_true_row_A_top_2_norm}
\end{align}
using conditions~\eqref{cond:B_UB}, \eqref{cond:row_B_rel} and \eqref{cond:eig_o1}.
Moreover, from Lemma~\ref{lemma:G_op_loss} and Lemma~\ref{lemma:G_loo_op_diff},
we know that
\begin{align}
\big\|\bm{U}^{\left(m\right)}\bm{H}^{\left(m\right)}-\bm{U}^{\star}\big\| & \lesssim\big\|\bm{U}^{\left(m\right)}\bm{U}^{\left(m\right)\top}-\bm{U}^{\star}\bm{U}^{\star\top}\big\|\leq\frac{\big\|\bm{G}^{\left(m\right)}-\bm{G}^{\star}\big\|}{\lambda_{r}\left(\bm{G}^{\star}\right)-\lambda_{r+1}\big(\bm{G}^{\left(m\right)}\big)}\nonumber \\
 & \leq\frac{\big\|\bm{G}^{\left(m\right)}-\bm{G}^{\star}\big\|}{\lambda_{r}\left(\bm{G}^{\star}\right)-\lambda_{r+1}\left(\bm{G}^{\star}\right)-\big\|\bm{G}^{\left(m\right)}-\bm{G}^{\star}\big\|}\nonumber \\
 & \lesssim\frac{1}{\sigma_{r}^{\star2}}\big\|\bm{G}^{\left(m\right)}-\bm{G}^{\star}\big\|\lesssim\frac{\delta_{\mathsf{op}}}{\sigma_{r}^{\star2}},\label{eq:UH_loo_op_loss_UB}
\end{align}
where $\delta_{\mathsf{op}}$ is defined in \eqref{claim:G_op_loss}.
Combining \eqref{eq:A_true_row_A_top_2_norm} and \eqref{eq:UH_loo_op_loss_UB}
yields
\begin{align}
\left\Vert \bm{A}_{m,:}^{\star}\mathcal{\mathcal{P}}_{-m,:}\left(\bm{A}^{\mathsf{s}}\right)^{\top}\big(\bm{U}^{\left(m\right)}\bm{H}^{\left(m\right)}-\bm{U}^{\star}\big)\right\Vert _{2} & \lesssim\big\|\bm{A}_{m,:}^{\star}\mathcal{\mathcal{P}}_{-m,:}\left(\bm{A}^{\mathsf{s}}\right)^{\top}\big\|_{2}\big\|\bm{U}^{\left(m\right)}\bm{H}^{\left(m\right)}-\bm{U}^{\star}\big\|\lesssim\delta_{\mathsf{op}}\kappa^{2}\sqrt{\frac{\mu r}{d_{1}}}.\label{eq:A_true_row_A_top_UH_Utrue_op}
\end{align}
\item Next, we look at $\beta_{2}$. Before we start, we pause to note that
by \eqref{eq:U_loo_col_dist_o1_UB2}, one has
\begin{align}
\big\|\bm{U}^{\left(m\right)}\bm{U}^{\left(m\right)\top}-\bm{U}^{\left(m,l\right)}\bm{U}^{\left(m,l\right)\top}\big\| & \lesssim\frac{1}{\sigma_{r}^{\star2}}\big(B\log d+\sigma_{\mathsf{col}}\sqrt{\log d}\big)^{2}\big(\big\|\bm{U}^{\left(m\right)}\bm{H}^{\left(m\right)}-\bm{U}^{\star}\big\|_{2,\infty}+\left\Vert \bm{U}^{\star}\right\Vert _{2,\infty}\big)\nonumber \\
 & \quad+\frac{1}{\sigma_{r}^{\star2}}\big(B\log d+\sigma_{\mathsf{col}}\sqrt{\log d}\big)\left\Vert \bm{A}^{\star}\right\Vert \sqrt{\frac{\mu r}{d_{2}}}.\label{eq:U_loo_col_diff_temp_2}
\end{align}
We now ready to control $\left(\bm{A}^{\mathsf{s}}-\bm{A}^{\star}\right)_{m,:}\big[\mathcal{\mathcal{P}}_{-m,:}\left(\bm{A}^{\mathsf{s}}\right)\big]^{\top}\big(\bm{U}^{\left(m\right)}\bm{H}^{\left(m\right)}-\bm{U}^{\star}\big)$,
which can be accomplished in the same way as in the proof of Lemma~\ref{lemma:A_dev_row_A_U_2_norm}
in Appendix~\ref{subsec:pf:A_dev_row_A_U_2_norm}. We omit the proof
details for conciseness here and only give the proof sketch. First,
we can use $\bm{U}^{\left(m,l\right)}\bm{H}^{\left(m,l\right)}-\bm{U}^{\star}$
as the surrogate for $\bm{U}^{\left(m\right)}\bm{H}^{\left(m\right)}-\bm{U}^{\star}$
to deal with the statistical dependence issue, and apply the Bernstein
inequality to show that with probability at least $1-O\left(d^{-11}\right)$,
\begin{align}
 & \big\|\big[\mathcal{\mathcal{P}}_{-m,:}\left(\bm{A}^{\mathsf{s}}\right)\big]^{\top}\big(\bm{U}^{\left(m\right)}\bm{H}^{\left(m\right)}-\bm{U}^{\star}\big)\big\|_{2,\infty}\nonumber \\
 & \qquad\lesssim\big(B\log d+\sigma_{\mathsf{col}}\sqrt{\log d}\big)\big\|\bm{U}^{\left(m\right)}\bm{H}^{\left(m\right)}-\bm{U}^{\star}\big\|_{2,\infty}\nonumber \\
 & \qquad\quad+\left\Vert \bm{A}^{\star\top}\right\Vert _{2,\infty}\big\|\bm{U}^{\left(m\right)}\bm{H}^{\left(m\right)}-\bm{U}^{\star}\big\|\nonumber \\
 & \qquad\quad+\big(\left\Vert \bm{A}^{\star\top}\right\Vert _{2,\infty}+B \sqrt{\log d}+\sigma_{\mathsf{col}}\big)\big\|\bm{U}^{\left(m\right)}\bm{U}^{\left(m\right)\top}-\bm{U}^{\left(m,l\right)}\bm{U}^{\left(m,l\right)\top}\big\|\nonumber \\
 & \qquad\overset{\left(\mathrm{i}\right)}{\lesssim}\big(B\log d+\sigma_{\mathsf{col}}\sqrt{\log d}\big)\big\|\bm{U}^{\left(m\right)}\bm{H}^{\left(m\right)}-\bm{U}^{\star}\big\|_{2,\infty}+\frac{\delta_{\mathsf{op}}}{\sigma_{r}^{\star2}}\left\Vert \bm{A}^{\star}\right\Vert \sqrt{\frac{\mu r}{d_{2}}}\nonumber \\
 & \qquad\quad+\frac{1}{\sigma_{r}^{\star2}}\big(B\log d+\sigma_{\mathsf{col}}\sqrt{\log d}\big)^{3}\big(\big\|\bm{U}^{\left(m\right)}\bm{H}^{\left(m\right)}-\bm{U}^{\star}\big\|_{2,\infty}+\left\Vert \bm{U}^{\star}\right\Vert _{2,\infty}\big)\nonumber \\
 & \qquad\quad+\frac{1}{\sigma_{r}^{\star2}}\big(B\log d+\sigma_{\mathsf{col}}\sqrt{\log d}\big)^{2}\left\Vert \bm{A}^{\star}\right\Vert \sqrt{\frac{\mu r}{d_{2}}}\nonumber \\
 & \qquad\lesssim\big(B\log d+\sigma_{\mathsf{col}}\sqrt{\log d}\big)\big\|\bm{U}^{\left(m\right)}\bm{H}^{\left(m\right)}-\bm{U}^{\star}\big\|_{2,\infty}+\frac{\delta_{\mathsf{op}}}{\sigma_{r}^{\star}}\left(\left\Vert \bm{U}^{\star}\right\Vert _{2,\infty}+\kappa\sqrt{\frac{\mu r}{d_{2}}}\right),\label{eq:A_UH_Utrue_2inf_norm}
\end{align}
where (i) follows from \eqref{eq:U_loo_col_dist_o1_UB1}, \eqref{eq:UH_loo_op_loss_UB}
and \eqref{eq:U_loo_col_diff_temp_2} and the inquality $\left\Vert \bm{A}^{\star\top}\right\Vert _{2,\infty}\leq\left\Vert \bm{A}^{\star}\right\Vert \sqrt{\mu r/d_{2}}$;
\eqref{eq:A_UH_Utrue_2inf_norm} arises from the definition of $\delta_{\mathsf{op}}$
in \eqref{def:delta_op} and conditions~\eqref{cond:B_UB} and \eqref{cond:eig_o1}
(namely, $B\log d+\sigma_{\mathsf{col}}\sqrt{\log d}\ll\sigma_{r}^{\star}/\kappa$
and $\big(B\log d+\sigma_{\mathsf{col}}\sqrt{\log d}\big)^{2}\lesssim\delta_{\mathsf{op}}\ll\sigma_{r}^{\star2}$).
Applying the matrix Bernstein inequality yields that with probability
at least $1-O\left(d^{-11}\right)$,
\begin{align*}
 & \left\Vert \left(\bm{A}^{\mathsf{s}}-\bm{A}^{\star}\right)_{m,:}\big[\mathcal{\mathcal{P}}_{-m,:}\left(\bm{A}^{\mathsf{s}}\right)\big]^{\top}\big(\bm{U}^{\left(m\right)}\bm{H}^{\left(m\right)}-\bm{U}^{\star}\big)\right\Vert _{2}\\
 & \qquad\lesssim B\log d\left\Vert \big[\mathcal{\mathcal{P}}_{-m,:}\left(\bm{A}^{\mathsf{s}}\right)\big]^{\top}\big(\bm{U}^{\left(m\right)}\bm{H}^{\left(m\right)}-\bm{U}^{\star}\big)\right\Vert _{2,\infty}\\
 & \qquad\quad+\sqrt{d_{1}}\,\sigma_{\infty}\big(\sigma_{\mathsf{row}}\log d+\left\Vert \bm{A}^{\star}\right\Vert \sqrt{\log d}\big)\big\|\bm{U}^{\left(m\right)}\bm{H}^{\left(m\right)}-\bm{U}^{\star}\big\|_{2,\infty}\\
 & \qquad\overset{\left(\mathrm{i}\right)}{\lesssim}B\log d\,\big(B\log d+\sigma_{\mathsf{col}}\sqrt{\log d}\big)\big\|\bm{U}^{\left(m\right)}\bm{H}^{\left(m\right)}-\bm{U}^{\star}\big\|_{2,\infty}\\
 & \qquad\quad+\delta_{\mathsf{op}}\frac{B\log d}{\sigma_{r}^{\star}}\left(\left\Vert \bm{U}^{\star}\right\Vert _{2,\infty}+\kappa\sqrt{\frac{\mu r}{d_{2}}}\right)\\
 & \qquad\quad+\sigma_{\mathsf{col}}\big(\sigma_{\mathsf{row}}\log d+\left\Vert \bm{A}^{\star}\right\Vert \sqrt{\log d}\big)\big\|\bm{U}^{\left(m\right)}\bm{H}^{\left(m\right)}-\bm{U}^{\star}\big\|_{2,\infty}\\
 & \qquad\overset{\left(\mathrm{ii}\right)}{\lesssim}\sigma_{\mathsf{col}}\big(\sigma_{\mathsf{row}}\log d+\left\Vert \bm{A}^{\star}\right\Vert \sqrt{\log d}\big)\big\|\bm{U}^{\left(m\right)}\bm{H}^{\left(m\right)}-\bm{U}^{\star}\big\|_{2,\infty}\\
 & \qquad\quad+\delta_{\mathsf{op}}\frac{B\log d}{\sigma_{r}^{\star}}\left(\left\Vert \bm{U}^{\star}\right\Vert _{2,\infty}+\kappa\sqrt{\frac{\mu r}{d_{2}}}\right)\\
 & \qquad\overset{\left(\mathrm{iii}\right)}{\lesssim}\sigma_{\mathsf{col}}\big(\sigma_{\mathsf{row}}\log d+\left\Vert \bm{A}^{\star}\right\Vert \sqrt{\log d}\big)\big\|\bm{U}^{\left(m\right)}\bm{H}^{\left(m\right)}-\bm{U}^{\star}\big\|_{2,\infty}+o\left(1\right)\delta_{\mathsf{op}}\left\Vert \bm{U}^{\star}\right\Vert _{2,\infty}+\delta_{\mathsf{op}}\frac{\kappa\sigma_{\mathsf{col}}\sqrt{\log d}}{\sigma_{r}^{\star}}\sqrt{\frac{\mu r}{d_{1}}}\\
 & \qquad\overset{\left(\mathrm{iv}\right)}{\lesssim}\sigma_{\mathsf{col}}\big(\sigma_{\mathsf{row}}\log d+\left\Vert \bm{A}^{\star}\right\Vert \sqrt{\log d}\big)\big\|\bm{U}^{\left(m\right)}\bm{H}^{\left(m\right)}-\bm{U}^{\star}\big\|_{2,\infty}+o\left(1\right)\delta_{\mathsf{op}}\sqrt{\frac{\mu r}{d_{1}}}.
\end{align*}
Here, (i) follows from \eqref{eq:noise_value} and \eqref{eq:A_UH_Utrue_2inf_norm}; (ii) is due
to conditions~\eqref{cond:B_UB} and \eqref{cond:eig_o1} that $B^{2}\log d\lesssim\sigma_{\mathsf{col}}\sigma_{\mathsf{row}}$,
$B\log d\ll\sigma_{r}^{\star}$ and $B\log d\left(B\log d+\sigma_{\mathsf{col}}\sqrt{\log d}\right)\lesssim B\log^{2}d+\sigma_{\mathsf{col}}^{2}\log d\leq\delta_{\mathsf{\mathsf{op}}}$;
(iii) holds true because of \eqref{cond:col_B_rel} and \eqref{cond:eig_o1}
that $B\log d\ll\sigma_{r}^{\star}$; and (iv) arises from \eqref{cond:eig_o1}
that $\sigma_{\mathsf{col}}\sqrt{\log d}\ll\sigma_{r}^{\star}/\kappa$.
Recalling the definition of $\delta_{\mathsf{loo}}$ in \eqref{def:delta_loo},
we obtain that
\begin{equation}
\left\Vert \left(\bm{A}^{\mathsf{s}}-\bm{A}^{\star}\right)_{m,:}\big[\mathcal{\mathcal{P}}_{-m,:}\left(\bm{A}^{\mathsf{s}}\right)\big]^{\top}\big(\bm{U}^{\left(m\right)}\bm{H}^{\left(m\right)}-\bm{U}^{\star}\big)\right\Vert _{2}\lesssim\delta_{\mathsf{loo}}\,\big\|\bm{U}^{\left(m\right)}\bm{H}^{\left(m\right)}-\bm{U}^{\star}\big\|_{2,\infty}+o\left(1\right)\delta_{\mathsf{op}}\sqrt{\frac{\mu r}{d_{1}}}.\label{eq:A_dev_row_A_top_UH_Utrue_op}
\end{equation}
\end{itemize}
Putting \eqref{eq:A_true_row_A_top_UH_Utrue_op} and \eqref{eq:A_dev_row_A_top_UH_Utrue_op}
together, we arrive at the advertised bound.

\section{Proofs for lower bounds}
\label{sec:proof-minimax-lower-bounds}

\subsection{Proof of Theorem~\ref{thm:minimax-lower-bound}}
\label{sec:proof-minimax-lower-bounds-noise}

Without loss of generality, it suffices to focus on the set of matrices
with $\sigma_{r}(\bm{A}^{\star})\in[0.9,1.1]$; otherwise one can
always rescale the matrices $\bm{A}^{\star}$ and $\bm{N}$ by the
same factor $1/\sigma_{r}(\bm{A}^{\star})$ simultaneously. 

Let us start with the minimax spectral norm  bound \eqref{eq:minimax-L2-subspace}. Recognizing the elementary fact that
\[
\left\Vert \bm{U}\bm{U}^{\top}-\bm{U}^{\star}\bm{U}^{\star\top}\right\Vert \asymp\min_{\bm{R}\in\mathcal{O}^{r\times r}}\left\Vert \bm{U}\bm{R}-\bm{U}^{\star}\right\Vert ,
\]
we have
\begin{align*}
\inf_{\widehat{\bm{U}}}\sup_{\bm{A}^{\star}\in\mathcal{M}^{\star}}\mathbb{E}\Big[\min_{\bm{R}\in\mathcal{O}^{r\times r}}\big\|\widehat{\bm{U}}\bm{R}-\bm{U}(\bm{A}^{\star})\big\|\Big] & \asymp\inf_{\widehat{\bm{U}}}\sup_{\bm{A}^{\star}\in\mathcal{M}^{\star}}\mathbb{E}\Big[\big\|\widehat{\bm{U}}\widehat{\bm{U}}^{\top}-\bm{U}(\bm{A}^{\star})\big(\bm{U}(\bm{A}^{\star})\big)^{\top}\big\|\Big].
\end{align*}
In light of this, we shall focus attention on bounding $\|\bm{U}\bm{U}^{\top}-\bm{U}(\bm{A}^{\star})\big(\bm{U}(\bm{A}^{\star})\big)^{\top}\|$
in the remainder of the proof. In addition, it can be easily seen
(which we omit for brevity) that it is sufficient to establish the
lower bounds for the rank-$1$ case (i.e.~$r=1$).\footnote{Suppose we wish to estimate $\bm{U}^\star \in \mathbb{R}^{d_1 \times r}$ where $\bm{U}^\star$ takes the form $\bm{U}^\star = \begin{bmatrix}
\bm{u}^\star & \bm{0}\\
\bm{0} & \bm{Q}^\star
\end{bmatrix}$, $\bm{u}^\star \in \mathbb{R}^{d_1 / 2}, \| \bm{u}^\star \|_2 = 1$ and $\bm{Q}^\star \in \mathbb{R}^{d_1/2 \times (r - 1)}$ consists of orthonormal columns. If there is an oracle informing us of $\bm{Q}^\star$, then the problem of estimating $\bm{U}^\star$ is reduced to the rank-$1$ case. This suggests that we can focus on the rank-$1$ case to derive the lower bound.}
In what follows, we assume that
\[
\bm{A}=\mathcal{P}_{\Omega}(\,\underset{=\bm{A}^{\star}}{\underbrace{\bm{u}^{\star}\bm{v}^{\star\top}}}+\bm{N}),
\]
where $\bm{v}^{\star}\sim\mathcal{N}(\bm{0},\frac{1}{d_{2}}\bm{I}_{d_{2}})$
and $N_{i,j}\overset{\mathrm{i.i.d.}}{\sim}\mathcal{N}(0,\sigma^{2})$.
Without loss of generality, we assume throughout that $d_{1}/2$ is
an integer.

\paragraph{Step 1: constructing a collection of hypotheses.} Let
us begin by constructing a family of well-separated unit vectors $\{\bm{u}^{i}\}_{1\leq i\leq M}\subseteq\mathbb{R}^{d_{1}}$.
In view of the celebrated Varshamov-Gilbert bound \cite[Lemma 4.7]{massart2007concentration},
one can find a set of vectors $\left\{ \bm{w}^{i}\right\} _{i=1}^{M}\subseteq\left\{ -1,1\right\} ^{d_{1}/2}$
obeying
\begin{align}
\log M\geq d_{1}/32\qquad\text{and}\qquad\min\left\{ \left\Vert \bm{w}^{i}\pm\bm{w}^{j}\right\Vert _{2}\right\}  & \geq\sqrt{d_{1}}/2,\quad\forall i\neq j,\label{eq:wij-distance}
\end{align}
where we denote $\min\|\bm{a}\pm\bm{b}\|_{2}=\min\{\|\bm{a}-\bm{b}\|_{2},\|\bm{a}+\bm{b}\|_{2}\}$.
For some $\delta\in(0,1)$ to be chosen later, we generate the $d_{1}$-dimensional
vectors
\begin{equation}
\bm{u}^{i}:=\frac{\delta}{\sqrt{d_{1}/2}}\begin{bmatrix}\bm{w}^{i}\\
\bm{0}
\end{bmatrix}+\sqrt{\frac{1-\delta^{2}}{d_{1}/2}}\begin{bmatrix}\bm{0}\\
\bm{1}
\end{bmatrix}\in\mathbb{R}^{d_{1}},\qquad1\leq i\leq M,\label{def:minimax-u}
\end{equation}
where $\bm{0}$ (resp.~$\bm{1}$) denotes the all-zero (resp.~all-one)
vector. By construction, it is easily seen that $\|\bm{u}^{i}\|_{2}=1$
for all $1\leq i\leq M$, and that
\begin{align}
\left\Vert \bm{u}^{i}\bm{u}^{i\top}-\bm{u}^{j}\bm{u}^{j\top}\right\Vert  & \geq\frac{1}{\sqrt{2}}\left\Vert \bm{u}^{i}\bm{u}^{i\top}-\bm{u}^{j}\bm{u}^{j\top}\right\Vert _{\mathrm{F}}=\frac{1}{\sqrt{2}}\sqrt{\mathsf{tr}\big(\bm{u}^{i}\bm{u}^{i\top}-\bm{u}^{j}\bm{u}^{j\top}\big)\big(\bm{u}^{i}\bm{u}^{i\top}-\bm{u}^{j}\bm{u}^{j\top}\big)}\nonumber \\
 & =\frac{1}{\sqrt{2}}\sqrt{2-2\langle\bm{u}^{i},\bm{u}^{j}\rangle^{2}}\geq\frac{1}{\sqrt{2}}\sqrt{2-2|\langle\bm{u}^{i},\bm{u}^{j}\rangle|}\nonumber \\
 & =\frac{1}{\sqrt{2}}\sqrt{\|\bm{u}^{i}\|_{2}^{2}+\|\bm{u}^{j}\|_{2}^{2}-2|\langle\bm{u}^{i},\bm{u}^{j}\rangle|}=\frac{1}{\sqrt{2}}\min\|\bm{u}^{i}\pm\bm{u}^{j}\|_{2}\nonumber \\
 & =\frac{1}{\sqrt{2}}\cdot\frac{\delta}{\sqrt{d_{1}/2}}\min\left\{ \left\Vert \bm{w}^{i}\pm\bm{w}^{j}\right\Vert _{2}\right\} \geq\frac{\delta}{2},\label{eq:ui-uj-spectral-norm-LB}
\end{align}
where the last inequality arises from (\ref{eq:wij-distance}). We
shall then associate each vector $\bm{u}^{i}$ $(1\leq i\leq M$)
with a hypothesis as follows: 
\[
\mathcal{H}_{i}:\quad\bm{A}=\mathcal{P}_{\Omega}(\bm{u}^{i}\bm{v}^{\star\top}+\bm{N}),\qquad1\leq i\leq M.
\]
 In the sequel, for each $1\leq i\leq M$ and $1\leq k\leq d_{2}$,
we denote
\begin{itemize}
\item $\mathbb{P}^{i}$: the distribution of $\bm{A}$ under the hypothesis
$\mathcal{H}_{i}$;
\item $\mathbb{P}_{\Omega}^{i}$: the distribution of $\bm{A}$ under the
hypothesis $\mathcal{H}_{i}$, conditional on $\Omega$; 
\item $\mathbb{P}_{\Omega,k}^{i}$: the distribution of the $k$-th column
of $\bm{A}$ under the hypothesis $\mathcal{H}_{i}$, conditional
on $\Omega$. 
\end{itemize}
Additionally, standard Gaussian concentration inequalities imply that:
with high probability, one has
\[
\|\bm{u}^{i}\bm{v}^{\star\top}\|=\left\Vert \bm{u}^{i}\right\Vert _{2}\left\Vert \bm{v}^{\star}\right\Vert _{2}=1+o(1),
\]
and hence $\bm{u}^{i}\bm{v}^{\star\top}\in\mathcal{M}^{\star}$ for
all $1\leq i\leq M$. 

\paragraph{Step 2: bounding the KL divergence between each pair of
hypotheses.} Fix any $1\leq i\neq j\leq M$. The next step lies in
upper bounding the KL divergence of $\mathbb{P}^{j}$ from $\mathbb{P}^{i}$.
Towards this, we observe that
\begin{align}
\mathsf{KL}\big(\mathbb{P}^{i}\,\|\,\mathbb{P}^{j}\big) & =\mathsf{KL}\big(\mathbb{E}_{\Omega}\big[\mathbb{P}_{\Omega}^{i}\big]\,\|\,\mathbb{E}_{\Omega}\big[\mathbb{P}_{\Omega}^{j}\big]\big)\leq\mathbb{E}_{\Omega}\big[\mathsf{KL}\big(\mathbb{P}_{\Omega}^{i}\,\|\,\mathbb{P}_{\Omega}^{j}\big)\big]\nonumber \\
 & =\mathbb{E}_{\Omega}\Big[\sum_{1\leq k\leq d_{2}}\mathsf{KL}\big(\mathbb{P}_{\Omega,k}^{i}\,\|\,\mathbb{P}_{\Omega,k}^{j}\big)\Big].\label{eq:KL-additive-identity}
\end{align}
Here, the penultimate inequality arises from the convexity of KL divergence
and Jensen's inequality, whereas the last line follows since the noise
components are independently generated and KL divergence is additive
for independent distributions. 

Before moving on, we find it convenient to introduce additional notation
to simplify presentation. For any vector $\bm{u}:=[u_{i}]_{1\leq i\leq d_{1}}$
and any index set $\mathcal{A}\subseteq[d_{1}]$, we define
\[
\bm{u}_{\mathcal{A}}:=\left[u_{i}\right]_{i\in\mathcal{A}}\in\mathbb{R}^{|\mathcal{A}|},
\]
which is obtained by maintaining only those entries of $\bm{u}$ lying
within $\mathcal{A}$. In addition, define\begin{subequations}\label{def:sampling-set}
\begin{align*}
\Omega_{k} & :=\{m\in[d_{1}]\colon(m,k)\in\Omega\};\\
\widehat{\Omega}_{k} & :=\{m\in[d_{1}/2]\colon(m,k)\in\Omega\};\\
\widetilde{\Omega}_{k} & :=\{d_{1}/2<m\leq d_{1}:(m,k)\in\Omega\};\\
\widehat{\Omega}_{k}^{(i,j),\mathsf{diff}} & :=\{m\in[d_{1}/2]\colon(m,k)\in\Omega\text{ and }u_{m}^{i}\neq u_{m}^{j}\};\\
\widehat{\Omega}_{k}^{(i,j),\mathsf{same}} & :=\{m\in[d_{1}/2]\colon(m,k)\in\Omega\text{ and }u_{m}^{i}=u_{m}^{j}\}.
\end{align*}
\end{subequations}By construction, one clearly has $\Omega_{k}=\widehat{\Omega}_{k}\cup\widetilde{\Omega}_{k}=\widehat{\Omega}_{k}^{(i,j),\mathsf{diff}}\cup\widehat{\Omega}_{k}^{(i,j),\mathsf{same}}\cup\widetilde{\Omega}_{k}$,
and 
\begin{equation}
\|\bm{u}_{\Omega_{k}}^{i}\|_{2}^{2}=\frac{2\delta^{2}}{d_{1}}|\widehat{\Omega}_{k}|+\frac{2(1-\delta^{2})}{d_{1}}\big|\widetilde{\Omega}_{k}\big|,\qquad1\leq i\leq M.\label{eq:ui-Omega-norm}
\end{equation}

With these in place, we are in a position to control the KL divergence.
We first make the observation that: conditional on the sampling set
$\Omega_{k}$ and under the hypothesis $\mathcal{H}_{i}$, the entries
of $\bm{A}_{:,k}$ within $\Omega_{k}$ follow a multivariate Gaussian
distribution $\mathcal{N}\big(\bm{0},\bm{\Sigma}_{\Omega_{k}}^{i}\big)$,
where
\[
\bm{\Sigma}_{\Omega_{k}}^{i}:=\sigma^{2}\bm{I}_{|\Omega_{k}|}+\frac{1}{d_{2}}\bm{u}_{\Omega_{k}}^{i}\bm{u}_{\Omega_{k}}^{i\top}.
\]
As a result, invoking the KL divergence for multivariate Gaussians,
we can deduce that
\begin{align}
\mathsf{KL}\big(\mathbb{P}_{\Omega,k}^{i}\,\|\,\mathbb{P}_{\Omega,k}^{j}\big) & =\mathsf{KL}\big(\mathcal{N}\big(\bm{0},\bm{\Sigma}_{\Omega_{k}}^{i}\big)\,\|\,\mathcal{N}\big(\bm{0},\bm{\Sigma}_{\Omega_{k}}^{j}\big)\big)=\frac{1}{2}\Big(\mathsf{tr}\big((\bm{\Sigma}_{\Omega_{k}}^{j})^{-1}\bm{\Sigma}_{\Omega_{k}}^{i}\big)-|\Omega_{k}|\Big).\nonumber \\
 & =\frac{1}{2}\mathsf{tr}\Big(\big(\bm{I}_{|\Omega_{k}|}+\frac{1}{\sigma^{2}d_{2}}\bm{u}_{\Omega_{k}}^{j}\bm{u}_{\Omega_{k}}^{j\top}\big)^{-1}\big(\bm{I}_{|\Omega_{k}|}+\frac{1}{\sigma^{2}d_{2}}\bm{u}_{\Omega_{k}}^{i}\bm{u}_{\Omega_{k}}^{i\top}\big)\Big)-\frac{1}{2}|\Omega_{k}|\nonumber \\
 & \overset{(\mathrm{i})}{=}\frac{1}{2}\mathsf{tr}\Big(\Big(\bm{I}_{|\Omega_{k}|}-\frac{1}{\sigma^{2}d_{2}+\big\|\bm{u}_{\Omega_{k}}^{j}\big\|_{2}^{2}}\bm{u}_{\Omega_{k}}^{j}\bm{u}_{\Omega_{k}}^{j\top}\Big)\big(\bm{I}_{|\Omega_{k}|}+\frac{1}{\sigma^{2}d_{2}}\bm{u}_{\Omega_{k}}^{i}\bm{u}_{\Omega_{k}}^{i\top}\big)\Big)-\frac{1}{2}|\Omega_{k}|\nonumber \\
 & =\frac{1}{2\sigma^{2}d_{2}}\big\|\bm{u}_{\Omega_{k}}^{i}\big\|_{2}^{2}-\frac{1}{2\big(\sigma^{2}d_{2}+\big\|\bm{u}_{\Omega_{k}}^{j}\big\|_{2}^{2}\big)}\big\|\bm{u}_{\Omega_{k}}^{j}\big\|_{2}^{2}-\frac{\big\langle\bm{u}_{\Omega_{k}}^{i},\bm{u}_{\Omega_{k}}^{j}\big\rangle^{2}}{2\sigma^{2}d_{2}\big(\sigma^{2}d_{2}+\big\|\bm{u}_{\Omega_{k}}^{j}\big\|_{2}^{2}\big)}\nonumber \\
 & \overset{(\mathrm{ii})}{=}\frac{\big\|\bm{u}_{\Omega_{k}}^{i}\big\|_{2}^{4}-\big\langle\bm{u}_{\Omega_{k}}^{i},\bm{u}_{\Omega_{k}}^{j}\big\rangle^{2}}{2\sigma^{2}d_{2}\big(\sigma^{2}d_{2}+\big\|\bm{u}_{\Omega_{k}}^{i}\big\|_{2}^{2}\big)}=\frac{\big\langle\bm{u}_{\Omega_{k}}^{i},\bm{u}_{\Omega_{k}}^{i}+\bm{u}_{\Omega_{k}}^{j}\big\rangle\big\langle\bm{u}_{\Omega_{k}}^{i},\bm{u}_{\Omega_{k}}^{i}-\bm{u}_{\Omega_{k}}^{j}\big\rangle}{2\sigma^{2}d_{2}\big(\sigma^{2}d_{2}+\big\|\bm{u}_{\Omega_{k}}^{i}\big\|_{2}^{2}\big)},\label{eq:kl-div-temp-1}
\end{align}
where (i) follows from the Woodbury matrix identity, (ii) arises since
$\big\|\bm{u}_{\Omega_{k}}^{i}\big\|_{2}=\big\|\bm{u}_{\Omega_{k}}^{j}\big\|_{2}$
(cf.~(\ref{eq:ui-Omega-norm})). Next, straightforward calculations
yield
\begin{align*}
\big\langle\bm{u}_{\Omega_{k}}^{i},\bm{u}_{\Omega_{k}}^{i}-\bm{u}_{\Omega_{k}}^{j}\big\rangle & =\frac{4\delta^{2}}{d_{1}}\big|\widehat{\Omega}_{k}^{(i,j),\mathsf{diff}}\big|,\\
\big\langle\bm{u}_{\Omega_{k}}^{i},\bm{u}_{\Omega_{k}}^{i}+\bm{u}_{\Omega_{k}}^{j}\big\rangle & =\frac{4\delta^{2}}{d_{1}}\big|\widehat{\Omega}_{k}^{(i,j),\mathsf{same}}\big|+\frac{4(1-\delta^{2})}{d_{1}}\big|\widetilde{\Omega}_{k}\big|.
\end{align*}
Substituting the above identities and the identity (\ref{eq:ui-Omega-norm})
into (\ref{eq:kl-div-temp-1}) gives
\begin{align*}
\mathsf{KL}\big(\mathbb{P}_{\Omega,k}^{i}\,\|\,\mathbb{P}_{\Omega,k}^{j}\big) & \leq\frac{\big\langle\bm{u}_{\Omega_{k}}^{i},\bm{u}_{\Omega_{k}}^{i}+\bm{u}_{\Omega_{k}}^{j}\big\rangle\big\langle\bm{u}_{\Omega_{k}}^{i},\bm{u}_{\Omega_{k}}^{i}-\bm{u}_{\Omega_{k}}^{j}\big\rangle}{2\sigma^{2}d_{2}\cdot\sigma^{2}d_{2}}\\
 & \leq\frac{16\delta^{2}\big|\widehat{\Omega}_{k}^{(i,j),\mathsf{diff}}\big|\cdot\big(\delta^{2}\big|\widehat{\Omega}_{k}^{(i,j),\mathsf{same}}\big|+(1-\delta^{2})\big|\widetilde{\Omega}_{k}\big|\big)}{\sigma^{4}d_{2}^{2}d_{1}^{2}}\\
 & \leq\frac{16\delta^{2}\big|\widehat{\Omega}_{k}^{(i,j),\mathsf{diff}}\big|\cdot\big(\big|\widehat{\Omega}_{k}^{(i,j),\mathsf{same}}\big|+\big|\widetilde{\Omega}_{k}\big|\big)}{\sigma^{4}d_{2}^{2}d_{1}^{2}},
\end{align*}
where we have used the fact that $\delta\in(0,1)$. In addition, the
elementary inequality $2\,\big\|\bm{u}_{\Omega_{k}}^{i}\big\|_{2}^{2}\geq\big\langle\bm{u}_{\Omega_{k}}^{i},\bm{u}_{\Omega_{k}}^{i}+\bm{u}_{\Omega_{k}}^{j}\big\rangle$
together with the preceding identities yields
\begin{align*}
\mathsf{KL}\big(\mathbb{P}_{\Omega,k}^{i}\,\|\,\mathbb{P}_{\Omega,k}^{j}\big) & =\frac{\big\langle\bm{u}_{\Omega_{k}}^{i},\bm{u}_{\Omega_{k}}^{i}+\bm{u}_{\Omega_{k}}^{j}\big\rangle\big\langle\bm{u}_{\Omega_{k}}^{i},\bm{u}_{\Omega_{k}}^{i}-\bm{u}_{\Omega_{k}}^{j}\big\rangle}{2\sigma^{2}d_{2}\|\bm{u}_{\Omega_{k}}^{i}\|_{2}^{2}}\leq\frac{\big\langle\bm{u}_{\Omega_{k}}^{i},\bm{u}_{\Omega_{k}}^{i}-\bm{u}_{\Omega_{k}}^{j}\big\rangle}{\sigma^{2}d_{2}}\\
 & =\frac{4\delta^{2}\big|\widehat{\Omega}_{k}^{(i,j),\mathsf{diff}}\big|}{\sigma^{2}d_{2}d_{1}}.
\end{align*}
Putting the above bounds and the inequality (\ref{eq:KL-additive-identity})
together leads to\begin{subequations}\label{eq:KL-two-upper-bounds}
\begin{align}
\mathsf{KL}\big(\mathbb{P}^{i}\,\|\,\mathbb{P}^{j}) & \leq\mathbb{E}\Bigg[\sum_{k=1}^{d_{2}}\frac{16\delta^{2}\big|\widehat{\Omega}_{k}^{(i,j),\mathsf{diff}}\big|\cdot\big(\big|\widehat{\Omega}_{k}^{(i,j),\mathsf{same}}\big|+\big|\widetilde{\Omega}_{k}\big|\big)}{\sigma^{4}d_{2}^{2}d_{1}^{2}}\Bigg]\lesssim\frac{\delta^{2}p^{2}}{\sigma^{4}d_{2}};\\
\mathsf{KL}\big(\mathbb{P}^{i}\,\|\,\mathbb{P}^{j}) & \leq\sum_{k=1}^{d_{2}}\frac{4\delta^{2}\big|\widehat{\Omega}_{k}^{(i,j),\mathsf{diff}}\big|}{\sigma^{2}d_{2}d_{1}}\lesssim\frac{\delta^{2}p}{\sigma^{2}}.
\end{align}
\end{subequations}

\paragraph{Step 3: invoking Fano's inequality.} Fano's inequality
\cite[Corollary 2.6]{tsybakov2008introduction} asserts that if
\begin{equation}
\frac{1}{M}\sum_{i=2}^{M}\mathsf{KL}\big(\mathbb{P}^{i}\,\|\,\mathbb{P}^{1})\leq\frac{\log M}{32},\label{eq:KL-16-bound}
\end{equation}
then the minimax probability of testing error necessarily obeys
\[
p_{\mathrm{e},M}:=\inf_{\psi}\max_{1\leq j\leq M}\mathbb{P}\left\{ \psi\neq j\mid\mathcal{H}_{j}\right\} \geq0.2,
\]
where the infimum is taken over all tests. In view of (\ref{eq:wij-distance})
and the upper bounds (\ref{eq:KL-two-upper-bounds}), we observe that
the bound (\ref{eq:KL-16-bound}) would hold by taking
\begin{equation}
	\delta=c_{1} \min\Bigg\{ \frac{\sigma^{2}\sqrt{d_{1}d_{2}}}{p}+\sigma\sqrt{\frac{d_{1}}{p}},\, 1\Bigg\}\label{eq:delta-condition-c1}
\end{equation}
for some sufficiently small constant $c_{1}>0$. Therefore, adopting
the standard reduction scheme as introduced in \cite[Chapter 2.2]{tsybakov2008introduction},
we arrive at
\begin{align*}
\inf_{\widehat{\bm{U}}}\sup_{\bm{A}^{\star}\in\mathcal{M}^{\star}}\mathbb{E}\Big[\big\|\widehat{\bm{U}}\widehat{\bm{U}}^{\top}-\bm{U}(\bm{A}^{\star})\big(\bm{U}(\bm{A}^{\star})\big)^{\top}\big\|\Big] & \gtrsim\min_{i\neq j}\left\Vert \bm{u}^{i}\bm{u}^{i\top}-\bm{u}^{j}\bm{u}^{j\top}\right\Vert \gtrsim\delta\\
 & \asymp\min\Bigg\{\frac{\sigma^{2}\sqrt{d_{1}d_{2}}}{p}+\sigma\sqrt{\frac{d_{1}}{p}},\,1\Bigg\},
\end{align*}
where the penultimate inequality comes from (\ref{eq:ui-uj-spectral-norm-LB}),
and the last line makes use of our choice (\ref{eq:delta-condition-c1}).
Combined with the high-probability fact that $\|\bm{u}^{i} \bm{v}^{\star\top} \| \in [0.9,1.1]$,
we establish the minimax spectral norm  bound \eqref{eq:minimax-L2-subspace}.

Given that $\|\bm{Z}\|_{2,\infty}\geq \frac{1}{\sqrt{d_1}} \|\bm{Z}\|$ holds for any $\bm{Z}\in \mathbb{R}^{d_1\times r}$, 
the advertised $\ell_{2,\infty}$ lower bound \eqref{eq:minimax-Linf-subspace} follows immediately from the spectral norm lower bound \eqref{eq:minimax-L2-subspace}.

\subsection{Proof of Theorem~\ref{thm:lower-bounds-bipartite}}
\label{sec:proof-minimax-lower-bounds-sampling}

To begin with, the sampling set $\Omega$ can be equivalently viewed
as the edge set of a random bipartite graph $\mathcal{G}(d_{1},d_{2},p)$.
Here, we recall that $\mathcal{G}(d_{1},d_{2},p)$ is generated by
(i) taking the complete bipartite graph connecting two disjoint vertex
sets $\mathcal{U}$ and $\mathcal{V}$, where $|\mathcal{U}|=d_{1}$
and $|\mathcal{V}|=d_{2}$, and (ii) removing each edge independently
with probability $1-p$. As shown in \cite[Theorem 6]{johansson2012giant},
if $p<\frac{1-\epsilon}{\sqrt{d_{1}d_{2}}}$ for some constant $0<\epsilon<1$
and if $d_{1}\leq d_{2}$, then with probability $1-o(1)$, there
is no connected component in $\mathcal{G}(d_{1},d_{2},p)$ containing
more than $O(\log d_{1})$ (resp.~$O(\sqrt{d_{1}d_{2}}\log d_{1})$)
vertices in $\mathcal{U}$ (resp.~$\mathcal{V}$). In what follows,
we let $\mathcal{C}_{1},\cdots,\mathcal{C}_{K}$ denote the collection
of connected components in $\mathcal{G}(d_{1},d_{2},p)$, and denote
by $\mathcal{U}_{i}$ (resp.~$\mathcal{V}_{i}$) the set of vertices
in $\mathcal{U}$ (resp.~$\mathcal{V}$) that reside within $\mathcal{C}_{i}$. 

Generate $\bm{u}^{\star}$ and $\bm{v}^{\star}$ such that 
\[
u_{i}^{\star}=\begin{cases}
1/\sqrt{d_{1}}, & \text{with prob. 0.5}\\
-1/\sqrt{d_{1}},\,\, & \text{else}
\end{cases}\qquad\text{and}\qquad v_{j}^{\star}=\begin{cases}
1/\sqrt{d_{2}}, & \text{with prob. 0.5}\\
-1/\sqrt{d_{2}},\,\, & \text{else}
\end{cases}
\]
for each $1\leq i\leq d_{1}$ and $1\leq j\leq d_{2}$. Letting $\bm{u}_{\mathcal{S}}\in\mathbb{R}^{|\mathcal{S}|}$
represent a vector comprising the entries of $\bm{u}$ whose indices
come from $\mathcal{S}$, we generate
\[
\widetilde{\bm{u}}_{\mathcal{C}_{i}}^{\star}=z_{i}\bm{u}_{\mathcal{C}_{i}}^{\star}\qquad\text{and }\qquad\widetilde{\bm{v}}_{\mathcal{C}_{i}}^{\star}=z_{i}\bm{v}_{\mathcal{C}_{i}}^{\star};
\]
here, $z_{i}$ is a set of independent Bernoulli variables with $z_{i}=1$
with probability $0.5$ and $z_{i}=-1$ otherwise. As one can easily
verify (which we omit for brevity), 
\begin{itemize}
\item $\bm{u}_{\mathcal{C}_{i}}^{\star}\bm{v}_{\mathcal{C}_{i}}^{\star\top}=\widetilde{\bm{u}}_{\mathcal{C}_{i}}^{\star}\widetilde{\bm{v}}_{\mathcal{C}_{i}}^{\star\top}$
for each $i$, and hence $\mathcal{P}_{\Omega}(\bm{u}^{\star}\bm{v}^{\star\top})=\mathcal{P}_{\Omega}(\widetilde{\bm{u}}^{\star}\widetilde{\bm{v}}^{\star\top})$;
\item with probability $1-o(1)$, one has $\min\|\bm{u}^{\star}\pm\widetilde{\bm{u}}^{\star}\|_{2}\asymp1$
and $\|\bm{u}^{\star}\bm{v}^{\star\top}-\widetilde{\bm{u}}^{\star}\widetilde{\bm{v}}^{\star\top}\|_{\mathrm{F}}\asymp1$.
\end{itemize}
This concludes the proof.

\section{A few more auxiliary lemmas\label{sec:more-auxiliary}}

In this section, we establish a few auxiliary facts that are useful
throughout the proof of the main theorem. We begin with some basic
properties about the truth $\bm{A}^{\star}$ and $\bm{G}^{\star}$.

\begin{lemma}\label{lemma:incoh}Recall the definition of the incoherence
parameters in Definition \ref{definition-mu0-mu1-mu2}. Then one has
\begin{align*}
 & \left\Vert \bm{A}^{\star}\right\Vert _{2,\infty}\leq\sqrt{\frac{\mu_{1}r\sigma_{1}^{\star2}}{d_{1}}},\quad\left\Vert \bm{A}^{\star\top}\right\Vert _{2,\infty}\leq\sqrt{\frac{\mu_{2}r\sigma_{1}^{\star2}}{d_{2}}},\\
 & \left\Vert \bm{G}^{\star}\right\Vert _{2,\infty}\leq\sqrt{\frac{\mu_{1}r\sigma_{1}^{\star4}}{d_{1}}},\quad\left\Vert \bm{A}^{\star}\right\Vert _{\infty}\leq\min\left\{ \sqrt{\frac{\mu_{1}\mu_{2}r^{2}}{d_{1}d_{2}}},\,\sigma_{1}^{\star}\left\Vert \bm{U}^{\star}\right\Vert _{2,\infty},\,\sigma_{1}^{\star}\left\Vert \bm{V}^{\star}\right\Vert _{2,\infty}\right\} .
\end{align*}
\end{lemma}

Next, we summarize several facts related to the matrix $\bm{E}$ defined
in (\ref{eq:defn-E}), which contains independent zero-mean entries.

\begin{lemma}\label{lemma:row_col_2_norm}Fix any matrices $\bm{W}_{1}$
and $\bm{W}_{2}$. With probability greater than $1-O\left(d^{-20}\right)$,
the following holds
\begin{align*}
\max_{i\in\left[d_{1}\right]}\sum_{j\in\left[d_{2}\right]}E_{i,j}^{2} & \lesssim B^{2}\log d+\sigma_{\mathsf{row}}^{2},\\
\max_{i\in\left[d_{2}\right]}\left\Vert \bm{E}_{i,:}\bm{W}_{1}\right\Vert _{2} & \lesssim\left(B\log d+\sigma_{\mathsf{row}}\sqrt{\log d}\right)\left\Vert \bm{W}_{1}\right\Vert _{2,\infty},\\
\max_{j\in\left[d_{2}\right]}\sum_{i\in\left[d_{1}\right]}E_{i,j}^{2} & \lesssim B^{2}\log d+\sigma_{\mathsf{col}}^{2},\\
\max_{j\in\left[d_{2}\right]}\left\Vert \big(\bm{E}_{:,j}\big)^{\top}\bm{W}_{2}\right\Vert _{2} & \lesssim\left(B\log d+\sigma_{\mathsf{col}}\sqrt{\log d}\right)\left\Vert \bm{W}_{2}\right\Vert _{2,\infty},
\end{align*}
where $\sigma_{\mathsf{row}}$, $\sigma_{\mathsf{col}}$, and $B$
are respectively upper bounds on $\max_{i\in\left[d_{1}\right]}\sqrt{\sum_{j\in\left[d_{2}\right]}\mathbb{E}\left[E_{i,j}^{2}\right]}$,
$\max_{j\in\left[d_{2}\right]}\sqrt{\sum_{i\in\left[d_{1}\right]}\mathbb{E}\left[E_{i,j}^{2}\right]}$,
and $\max_{i\in\left[d_{1}\right],j\in\left[d_{2}\right]}\left|E_{i,j}\right|$;
see (\ref{eq:noise_value}) for precise definitions. As a result,
one has
\begin{align*}
\left\Vert \bm{E}\right\Vert _{2,\infty} & \lesssim B\sqrt{\log d}+\sigma_{\mathsf{row}},\\
\left\Vert \bm{A}^{\mathsf{s}}\right\Vert _{2,\infty} & \lesssim\left\Vert \bm{A}^{\star}\right\Vert _{2,\infty}+B\sqrt{\log d}+\sigma_{\mathsf{row}},\\
\left\Vert \bm{E}^{\top}\right\Vert _{2,\infty} & \lesssim B\sqrt{\log d}+\sigma_{\mathsf{col}},\\
\left\Vert \bm{A}^{\mathsf{s}\top}\right\Vert _{2,\infty} & \lesssim\left\Vert \bm{A}^{\star\top}\right\Vert _{2,\infty}+B\sqrt{\log d}+\sigma_{\mathsf{col}},
\end{align*}
where $\bm{A}^{\mathsf{s}}=\bm{A}^{\star}+\bm{E}$ is defined in (\ref{eq:A-E}).\end{lemma}

\begin{lemma}\label{lemma:A_op_norm} With probability greater than
$1-O\left(d^{-20}\right)$, one has
\[
\left\Vert \bm{E}\right\Vert \lesssim B\log d+\left(\sigma_{\mathsf{row}}+\sigma_{\mathsf{col}}\right)\sqrt{\log d},
\]
where $B$, $\sigma_{\mathsf{row}}$ and $\sigma_{\mathsf{col}}$
are defined in (\ref{eq:noise_value}).\end{lemma}

\begin{lemma}\label{lemma:sum_row_square_sum_col}Fix any vector
$\bm{w}\in\mathbb{R}^{d_{2}}$. With probability at least $1-O\left(d^{-20}\right)$,
one has
\[
\sum_{i\in\left[d_{1}\right]}\Bigg(\sum_{j\in\left[d_{2}\right]}w_{j}E_{i,j}\Bigg)^{2}\lesssim\left\Vert \bm{w}\right\Vert _{2}^{2}\left(\sigma_{\mathsf{col}}^{2}+\sigma_{\infty}^{2}\log^{2}d\right)+\left\Vert \bm{w}\right\Vert _{\infty}^{2}B^{2}\log^{3}d,
\]
where $B$, $\sigma_{\infty}$ and $\sigma_{\mathsf{col}}$ are defined
in (\ref{eq:noise_value}).\end{lemma}

\subsection{Proof of Lemma~\ref{lemma:incoh}}

\label{subsec:pf:lemma:incoh}

Given the SVD of $\bm{A}^{\star}=\bm{U}^{\star}\bm{\Sigma}^{\star}\bm{V}^{\star\top}$,
one has $\bm{G}^{\star}=\bm{A}^{\star}\bm{A}^{\star\top}=\bm{U}^{\star}\bm{\Sigma}^{\star2}\bm{U}^{\star\top}$.
Using the definition of the incoherence parameters, one can derive
\begin{align*}
\left\Vert \bm{A}^{\star}\right\Vert _{2,\infty} & =\max_{i\in\left[d_{1}\right]}\left\Vert \bm{U}_{i,:}^{\star}\bm{\Sigma}^{\star}\bm{V}^{\star\top}\right\Vert _{2}\leq\max_{i\in\left[d_{1}\right]}\left\Vert \bm{U}_{i,:}^{\star}\right\Vert _{2}\left\Vert \bm{\Sigma}^{\star}\right\Vert \left\Vert \bm{V}^{\star}\right\Vert \leq\sigma_{1}^{\star}\left\Vert \bm{U}^{\star}\right\Vert _{2,\infty}\leq\sqrt{\frac{\mu_{1}r\sigma_{1}^{\star2}}{d_{1}}};\\
\left\Vert \bm{A}^{\star\top}\right\Vert _{2,\infty} & =\max_{j\in\left[d_{2}\right]}\left\Vert \bm{V}_{i,:}^{\star}\bm{\Sigma}^{\star}\bm{U}^{\star\top}\right\Vert _{2}\leq\max_{j\in\left[d_{2}\right]}\left\Vert \bm{V}_{j,:}^{\star}\right\Vert _{2}\left\Vert \bm{\Sigma}^{\star}\right\Vert \left\Vert \bm{U}^{\star}\right\Vert \leq\sigma_{1}^{\star}\left\Vert \bm{V}^{\star}\right\Vert _{2,\infty}\leq\sqrt{\frac{\mu_{2}r\sigma_{1}^{\star2}}{d_{2}}};\\
\left\Vert \bm{G}^{\star}\right\Vert _{2,\infty} & =\max_{i\in\left[d_{1}\right]}\left\Vert \bm{U}_{i,:}^{\star}\bm{\Sigma}^{\star2}\bm{U}^{\star\top}\right\Vert _{2}\leq\max_{i\in\left[d_{1}\right]}\left\Vert \bm{U}_{i,:}^{\star}\right\Vert _{2}\left\Vert \bm{\Sigma}^{\star2}\right\Vert \left\Vert \bm{U}^{\star}\right\Vert \leq\sigma_{1}^{\star2}\left\Vert \bm{U}^{\star}\right\Vert _{2,\infty}\leq\sqrt{\frac{\mu_{1}r\sigma_{1}^{\star4}}{d_{1}}}.
\end{align*}
Moreover, the Cauchy-Schwartz inequality allows one to upper bound
\[
\left\Vert \bm{A}^{\star}\right\Vert _{\infty}=\max_{\left(i,j\right)\in\left[d_{1}\right]\times\left[d_{2}\right]}\left|\bm{U}_{i,:}^{\star}\bm{\Sigma}^{\star}\left(\bm{V}_{j,:}^{\star}\right)^{\top}\right|\leq\left\Vert \bm{U}^{\star}\right\Vert _{2,\infty}\left\Vert \bm{\Sigma}^{\star}\right\Vert \left\Vert \bm{V}^{\star}\right\Vert _{2,\infty}\leq\sigma_{1}^{\star}\left\Vert \bm{U}^{\star}\right\Vert _{2,\infty}\left\Vert \bm{V}^{\star}\right\Vert _{2,\infty}.
\]
In view of the simple bounds $\left\Vert \bm{U}^{\star}\right\Vert _{2,\infty}\leq\left\Vert \bm{U}^{\star}\right\Vert \leq1$
and $\left\Vert \bm{V}^{\star}\right\Vert _{2,\infty}\leq\left\Vert \bm{V}^{\star}\right\Vert \leq1$,
we conclude that
\[
\left\Vert \bm{A}^{\star}\right\Vert _{\infty}\leq\sigma_{1}^{\star}\left\Vert \bm{U}^{\star}\right\Vert _{2,\infty}\qquad\text{and}\qquad\left\Vert \bm{A}^{\star}\right\Vert _{\infty}\leq\sigma_{1}^{\star}\left\Vert \bm{V}^{\star}\right\Vert _{2,\infty}.
\]

\subsection{Proof of Lemma~\ref{lemma:row_col_2_norm}}

\label{subsec:pf:lemma:row_col_2_norm}

We shall only prove the results concerning $\sigma_{\mathsf{col}}$;
the results concerning $\sigma_{\mathsf{row}}$ follow immediately
via nearly identical arguments.

In view of the Bernstein inequality, we have
\begin{align*}
\mathbb{P}\left\{ \Big|\sum\nolimits _{i\in[d_{1}]}E_{i,j}^{2}-M_{1}\Big|\geq t\right\}  & \leq2\exp\left(-\frac{3}{8}\min\left\{ \frac{t^{2}}{V_{1}},\frac{t}{L_{1}}\right\} \right),\quad t>0,
\end{align*}
where $M_{1},L_{1}$ and $S_{1}$ are given respectively by
\begin{align*}
M_{1} & :=\sum\nolimits _{i\in[d_{1}]}\mathbb{E}\left[E_{i,l}^{2}\right]\leq\sigma_{\mathsf{col}}^{2},\\
L_{1} & :=\max_{i\in\left[d_{1}\right]}\left|E_{i,l}^{2}-\mathbb{E}\left[E_{i,l}^{2}\right]\right|\leq B^{2}+\sigma_{\infty}^{2}\leq2B^{2},\\
V_{1} & :=\sum\nolimits _{i\in[d_{1}]}\mathsf{Var}\left(E_{i,l}^{2}\right)\leq\sum\nolimits _{i\in[d_{1}]}\mathbb{E}\left[E_{i,l}^{4}\right]\leq B^{2}\sigma_{\mathsf{col}}^{2}.
\end{align*}
Here, we have made use of the fact that $\sigma_{\infty}\leq B$.
As a result, one has
\begin{align*}
\sum\nolimits _{i\in[d_{1}]}E_{i,j}^{2} & \lesssim M_{1}+L_{1}\log d+\sqrt{V_{1}\log d}\lesssim\sigma_{\mathsf{col}}^{2}+B^{2}\log d+B\sigma_{\mathsf{col}}\sqrt{\log d}\\
 & \asymp\sigma_{\mathsf{col}}^{2}+B^{2}\log d
\end{align*}
with probability exceeding $1-O\left(d^{-20}\right)$, where the last
line arises from the AM-GM inequality (namely, $2B\sigma_{\mathsf{col}}\sqrt{\log d}\leq\sigma_{\mathsf{col}}^{2}+B^{2}\log d$).
As an immediate consequence, with probability at least $1-O\left(d^{-20}\right)$,
\begin{align*}
\left\Vert \bm{E}_{:,j}\right\Vert _{2} & =\sqrt{\sum\nolimits _{i\in[d_{1}]}E_{i,j}^{2}}\lesssim\sigma_{\mathsf{col}}+B\sqrt{\log d},\\
\left\Vert \bm{A}_{:,j}^{\mathsf{s}}\right\Vert _{2} & \leq\left\Vert \bm{A}_{:,j}^{\star}\right\Vert _{2}+\left\Vert \bm{E}_{:,j}\right\Vert _{2}\lesssim\left\Vert \bm{A}^{\star\top}\right\Vert _{2,\infty}+\sigma_{\mathsf{col}}+B\sqrt{\log d}.
\end{align*}

Next, we turn to the claim concerning a fixed matrix $\bm{W}_{2}$.
Observe that $\left(\bm{E}_{:,l}\right)^{\top}\bm{W}_{2}=\sum_{i\in\left[d_{1}\right]}E_{i,l}(\bm{W}_{2})_{i,:}$
is a sum of independent zero-mean random vectors. In order to invoke
standard concentration inequalities, we compute
\begin{align*}
L_{2} & :=\max_{i\in\left[d_{1}\right]}\left\Vert E_{i,l}(\bm{W}_{2})_{i,:}\right\Vert _{2}\leq B\left\Vert \bm{W}_{2}\right\Vert _{2,\infty},\\
V_{2} & :=\sum\nolimits _{i\in[d_{1}]}\mathbb{E}\left[E_{i,l}^{2}\right]\left\Vert (\bm{W}_{2})_{i,:}\right\Vert _{2}^{2}\leq\sigma_{\mathsf{col}}^{2}\left\Vert \bm{W}_{2}\right\Vert _{2,\infty}^{2}.
\end{align*}
Invoking the matrix Bernstein inequality yields that with probability
exceeding $1-O\left(d^{-20}\right)$,
\[
\big\|\big(\bm{E}_{:,l}\big)^{\top}\bm{W}_{2}\big\|_{2}\lesssim L_{2}\log d+\sqrt{V_{2}\log d}\lesssim\big(B\log d+\sigma_{\mathsf{col}}\sqrt{\log d}\big)\left\Vert \bm{W}_{2}\right\Vert _{2,\infty}.
\]

\subsection{Proof of Lemma~\ref{lemma:A_op_norm}\label{subsec:pf:A_op_norm}}

First, we can write
\[
\bm{E}=\sum_{i\in\left[d_{1}\right],j\in\left[d_{2}\right]}E_{i,j}\bm{e}_{i}\bm{e}_{j}^{\top}
\]
as a sum of independent zero-mean random matrices (since $\mathbb{E}[E_{i,j}]=0$).
We make the observation that
\begin{align*}
L & :=\max_{i\in\left[d_{1}\right],j\in\left[d_{2}\right]}\left\Vert E_{i,j}\bm{e}_{i}\bm{e}_{j}^{\top}\right\Vert \leq B;\\
V & :=\max\left\{ \Bigg\|\sum_{i\in\left[d_{1}\right],j\in\left[d_{2}\right]}\mathbb{E}\left[E_{i,j}^{2}\right]\bm{e}_{i}\bm{e}_{i}^{\top}\Bigg\|,\Bigg\|\sum_{i\in\left[d_{1}\right],j\in\left[d_{2}\right]}\mathbb{E}\left[E_{i,j}^{2}\right]\bm{e}_{j}\bm{e}_{j}^{\top}\Bigg\|\right\} \leq\sigma_{\mathsf{row}}^{2}+\sigma_{\mathsf{col}}^{2}.
\end{align*}
It then follows from the matrix Bernstein inequality that, with probability
at least $1-O\left(d^{-10}\right),$
\begin{align*}
\left\Vert \bm{E}\right\Vert  & \lesssim L\log d+\sqrt{V\log d}\lesssim B\log d+\left(\sigma_{\mathsf{row}}+\sigma_{\mathsf{col}}\right)\sqrt{\log d}.
\end{align*}

\subsection{Proof of Lemma~\ref{lemma:sum_row_square_sum_col}\label{subsec:pf:sum_row_square_sum_col}}

Let us define a sequence of independent zero-mean random variables
$\left\{ X_{i}\right\} _{1\leq i\leq d_{1}}$ as follows
\[
X_{i}:=\sum_{j\in\left[d_{2}\right]}w_{j}E_{i,j}.
\]
It is easily seen that
\[
\max_{j\in\left[d_{2}\right]}\left|w_{j}E_{i,j}\right|\leq\left\Vert \bm{w}\right\Vert _{\infty}B;\quad\mathbb{E}\left[X_{i}^{2}\right]=\sum_{j\in\left[d_{2}\right]}w_{j}^{2}\sigma_{i,j}^{2}\leq\left\Vert \bm{w}\right\Vert _{2}^{2}\sigma_{\infty}^{2}.
\]
We can therefore apply the Bernstein inequality to show that, with
probability at least $1-O\left(d^{-11}\right)$,
\begin{equation}
\left|X_{i}\right|\lesssim\left(\left\Vert \bm{w}\right\Vert _{\infty}B\right)\log d+\sqrt{\big(\left\Vert \bm{w}\right\Vert _{2}^{2}\sigma_{\infty}^{2}\big)\log d}=:R.\label{eq:X_i_UB}
\end{equation}

Next, let us introduce a sequence of independent random variables
$\left\{ Y_{i}\right\} _{1\leq i\leq d_{1}}$, obtained by truncating
$X_{i}$ 
\[
Y_{i}\triangleq X_{i}\ind\left\{ \left|X_{i}\right|\leq CR\right\} 
\]
for some sufficiently large absolute constant $C>0$. From (\ref{eq:X_i_UB})
and the union bound, we know that $Y_{i}=X_{i}$ holds simultaneously
for all $1\leq i\leq d_{1}$ with probability at least $1-O\left(d^{-10}\right)$.

Further, it is straightforward to compute that
\begin{align*}
M_{2} & :=\sum_{i\in\left[d_{1}\right]}\mathbb{E}\left[Y_{i}^{2}\right]\leq\sum_{i\in\left[d_{1}\right]}\mathbb{E}\left[X_{i}^{2}\right]\leq\sum_{i\in\left[d_{1}\right],j\in\left[d_{2}\right]}w_{j}^{2}\sigma_{i,j}^{2}\leq\left\Vert \bm{w}\right\Vert _{2}^{2}\sigma_{\mathsf{col}}^{2};\\
L_{2} & :=\max_{i\in\left[d_{1}\right]}\left|Y_{i}^{2}-\mathbb{E}\left[Y_{i}^{2}\right]\right|\lesssim R^{2}\lesssim\left\Vert \bm{w}\right\Vert _{\infty}^{2}B^{2}\log^{2}d+\left\Vert \bm{w}\right\Vert _{2}^{2}\sigma_{\infty}^{2}\log d;\\
V_{2} & :=\sum_{i\in\left[d_{1}\right]}\mathsf{Var}\left(Y_{i}^{2}\right)\leq\sum_{i\in\left[d_{1}\right]}\mathbb{E}\left[Y_{i}^{4}\right]\leq\sum_{i\in\left[d_{1}\right]}\mathbb{E}\left[X_{i}^{4}\right]\lesssim\sum_{i\in\left[d_{1}\right]}\sum_{j\in\left[d_{2}\right]}w_{j}^{4}\mathbb{E}\left[E_{i,j}^{4}\right]+\sum_{i\in\left[d_{1}\right]}\sum_{j_{1}\neq j_{2}}w_{j_{1}}^{2}w_{j_{2}}^{2}\mathbb{E}\left[E_{i,j_{1}}^{2}\right]\mathbb{E}\left[E_{i,j_{2}}^{2}\right]\\
 & \lesssim\left\Vert \bm{w}\right\Vert _{\infty}^{2}\left\Vert \bm{w}\right\Vert _{2}^{2}B^{2}\sigma_{\mathsf{col}}^{2}+\left\Vert \bm{w}\right\Vert _{2}^{4}\sigma_{\infty}^{2}\sigma_{\mathsf{col}}^{2}.
\end{align*}
We then apply the Bernstein inequality to conclude that with probability
at least $1-O\left(d^{-10}\right)$:
\begin{align*}
\sum_{i\in\left[d_{1}\right]}Y_{i}^{2} & \lesssim M_{2}+L_{2}\log d+\sqrt{V_{2}\log d}\\
 & \lesssim\left\Vert \bm{w}\right\Vert _{2}^{2}\sigma_{\mathsf{col}}^{2}+\left\Vert \bm{w}\right\Vert _{\infty}^{2}B^{2}\log^{3}d+\left\Vert \bm{w}\right\Vert _{2}^{2}\sigma_{\infty}^{2}\log^{2}d+\left(\left\Vert \bm{w}\right\Vert _{\infty}\left\Vert \bm{w}\right\Vert _{2}B\sigma_{\mathsf{col}}+\left\Vert \bm{w}\right\Vert _{2}^{2}\sigma_{\infty}\sigma_{\mathsf{col}}\right)\sqrt{\log d}\\
 & \asymp\left\Vert \bm{w}\right\Vert _{2}^{2}\left(\sigma_{\mathsf{col}}^{2}+\sigma_{\infty}^{2}\log^{2}d\right)+\left\Vert \bm{w}\right\Vert _{\infty}^{2}B^{2}\log^{3}d,
\end{align*}
where the last line arises from the AM-GM inequality (namely, $2\left\Vert \bm{w}\right\Vert _{\infty}\left\Vert \bm{w}\right\Vert _{2}B\sigma_{\mathsf{col}}\sqrt{\log d}\leq\left\Vert \bm{w}\right\Vert _{\infty}^{2}B^{2}\log d+\left\Vert \bm{w}\right\Vert _{2}^{2}\sigma_{\mathsf{col}}^{2}$
and $2\left\Vert \bm{w}\right\Vert _{2}^{2}\sigma_{\infty}\sigma_{\mathsf{col}}\sqrt{\log d}\leq\left\Vert \bm{w}\right\Vert _{2}^{2}\sigma_{\infty}^{2}\log d+\left\Vert \bm{w}\right\Vert _{2}^{2}\sigma_{\mathsf{col}}^{2}$).

\bibliographystyle{alphaabbr}
\bibliography{bibfileNonconvex}

\end{document}